\documentclass[oneside]{amsart}
\usepackage[LGR,T1]{fontenc}
\usepackage[latin9]{inputenc}
\usepackage{mathrsfs}
\usepackage{mathtools}
\usepackage{bm}
\usepackage{amsbsy}
\usepackage{amstext}
\usepackage{amsthm}
\usepackage{amssymb}
\usepackage{graphicx}
\usepackage{geometry}
\usepackage{wasysym}
\usepackage[pdfusetitle,
 bookmarks=true,bookmarksnumbered=false,bookmarksopen=false,
 breaklinks=false,pdfborder={0 0 1},backref=false,colorlinks=false]
 {hyperref}

\makeatletter

\DeclareRobustCommand{\greektext}{%
  \fontencoding{LGR}\selectfont\def\encodingdefault{LGR}}
\DeclareRobustCommand{\textgreek}[1]{\leavevmode{\greektext #1}}

\providecommand{\tabularnewline}{\\}

\numberwithin{equation}{section}
\numberwithin{figure}{section}
\theoremstyle{plain}
\newtheorem{thm}{\protect\theoremname}
\theoremstyle{plain}
\newtheorem{prop}[thm]{\protect\propositionname}
\theoremstyle{remark}
\newtheorem{rem}[thm]{\protect\remarkname}
\theoremstyle{plain}
\newtheorem{cor}[thm]{\protect\corollaryname}
\theoremstyle{plain}
\newtheorem{lem}[thm]{\protect\lemmaname}
\theoremstyle{definition}
\newtheorem{example}[thm]{\protect\examplename}

\usepackage{shuffle}
\usepackage{tikz}
\usetikzlibrary{arrows.meta,decorations.markings}

\makeatother

\providecommand{\corollaryname}{Corollary}
\providecommand{\examplename}{Example}
\providecommand{\lemmaname}{Lemma}
\providecommand{\propositionname}{Proposition}
\providecommand{\remarkname}{Remark}
\providecommand{\theoremname}{Theorem}

\begin{document}
\address[Minoru Hirose]{Graduate School of Science and Engineering, Kagoshima University, 1-21-35 Korimoto, Kagoshima, Kagoshima 890-0065, Japan}
\email{hirose@sci.kagoshima-u.ac.jp}
\address[Nobuo Sato]{Department of Mathematics, National Taiwan University, No. 1, Sec. 4, Roosevelt Rd., Taipei 10617, Taiwan (R.O.C.)}
\email{nbsato@ntu.edu.tw}
\title{Iterated beta integrals}
\author{Minoru Hirose and Nobuo Sato}
\date{\today}
\begin{abstract}
We introduce \emph{iterated beta integrals}, a new class of iterated
integrals on the universal abelian covering of the punctured projective
line that unifies hyperlogarithms and classical beta integrals while
preserving their fundamental properties. We establish various analytic
properties of these integrals with respect to both the exponent parameters
and the main variables. Their key feature is invariance under simultaneous
translation of the exponent parameters, which generates relations
between integrals over possibly different coverings. This mechanism
recovers notable identities for multiple zeta values and variants---including
Zagier\textquoteright s 2-3-2 formula, Murakami\textquoteright s $t$-value
analogue, Charlton's $t$-value analogue, Zhao\textquoteright s $2$-$1$
formula, and Ohno's relation---and also yields new relations, such
as a proof of a Galois descent phenomenon for multiple omega values.
\end{abstract}

\maketitle
\tableofcontents{}

\section{Introduction}

\subsection{Zhao's 2-1 formula vs Zagier's 2-3-2 formula}

\emph{Multiple zeta values}, or MZVs in short, are real numbers defined
by the nested sum 
\[
\zeta(k_{1},\ldots,k_{d})\coloneqq\sum_{0<m_{1}<\cdots<m_{d}}\frac{1}{m_{1}^{k_{1}}\cdots m_{d}^{k_{d}}},
\]
and they have been actively studied by numerous mathematicians and
physicists because of their rich structures and profound nature. The
tuple $(k_{1},\ldots,k_{d})\in\mathbb{Z}_{>0}^{d}$ for which the
sum is convergent is called an admissible index, whose set is given
by
\[
\mathbb{I}\coloneqq\left\{ (k_{1},\ldots,k_{d})\Big|d>0,k_{1},\ldots,k_{d-1}\geq1,k_{d}>1\right\} .
\]
MZV has a twin sibling named \emph{multiple zeta star values} (MZSV)
defined by 
\[
\zeta^{\star}(k_{1},\ldots,k_{d})\coloneqq\sum_{0<m_{1}\leq\cdots\leq m_{d}}\frac{1}{m_{1}^{k_{1}}\cdots m_{d}^{k_{d}}},
\]
which form another standard generator of the linear space of MZVs.
For an MZV/MZSV of index $(k_{1},\ldots,k_{d})$, the sum $k_{1}+\cdots+k_{d}$
is called \emph{weight} and the number $d$ of entries is called \emph{depth}.
MZVs satisfy rich linear relations over $\mathbb{Q}$ (conjecturally
all homogeneous in weight) and enjoy various amusing combinatorial
structures. One of the simplest-looking yet mysterious families of
such relations is the so-called $2$-$1$ formula. The simplest instance
($=$ depth one case) of the $2$-$1$ formula is
\begin{equation}
\zeta^{\star}(1,\overbrace{2,\ldots,2}^{l})=2\zeta(2l+1)\quad(l>0),\label{eq:2_1_depth_1}
\end{equation}
which was first established by Zlobin \cite{Zlobin}. Later, Ohno
and Zudilin found and proved a ``depth two analog'' of Zlobin's formula
\begin{align*}
\zeta^{\star}(1,\overbrace{2,\ldots,2}^{k},1,\overbrace{2,\ldots,2}^{l}) & =2^{2}\zeta(2k+1,2l+1)+2\zeta(2k+2l+2)\quad(k\geq0,l>0)\\
 & =\sum_{0<m\leq n}\frac{2^{\#\{m,n\}}}{m^{2k+1}n^{2l+1}}
\end{align*}
and conjectured the\emph{ $2$-$1$ formula} \cite{OZ_ZetaStar}
\[
\zeta^{\star}(1,\{2\}^{l_{1}},\ldots,1,\{2\}^{l_{d}})=\sum_{0<m_{1}\leq\cdots\leq m_{d}}\frac{2^{\#\{m_{1},\ldots,m_{d}\}}}{m_{1}^{2l_{1}+1}\cdots m_{d}^{2l_{d}+1}}\quad(l_{1},\ldots,l_{d-1}\geq0,l_{d}>0),
\]
generalizing their result to ``general depth'' case. The $2$-$1$
formula was later proved by Zhao in full generality. In fact, Zhao
proved an even more general equality as follows. First, define
\[
\zeta^{\#}(k_{1},\ldots,k_{d})\coloneqq\sum_{0<m_{1}\leq\cdots\leq m_{d}}2^{\#\{m_{1},\ldots,m_{d}\}}\frac{(-1)^{(k_{1}-1)m_{1}+\cdots+(k_{d}-1)m_{d}}}{m_{1}^{k_{1}}\cdots m_{d}^{k_{d}}}\quad(k_{1},\ldots,k_{d-1}\geq1,k_{d}>1).
\]
Notice that the sign $(-1)^{(k_{1}-1)m_{1}+\cdots+(k_{d}-1)m_{d}}$
is $1$ if all $k_{i}$'s are odd, and so the right-hand side of the
$2$-$1$ formula is exactly equal to $\zeta^{\#}(2l_{1}+1,\ldots,2l_{d}+1)$.
We next define the bijection $\sigma$ on $\mathbb{I}$ (the set of
admissible indices) as follows. First, define a map $\hat{\sigma}$
from $\bigsqcup_{d\geq0}\mathbb{Z}_{>0}^{d}$ to $\bigsqcup_{d\geq0}\left(\mathbb{Z}_{\geq0}\times\mathbb{Z}_{>0}^{d}\right)$
recursively by $\hat{\sigma}(\emptyset)=(0)$ and
\begin{align*}
\hat{\sigma}(\Bbbk,1) & =(\hat{\sigma}(\Bbbk),1)\\
\hat{\sigma}(\Bbbk,2) & =\hat{\sigma}(\Bbbk)_{\uparrow\uparrow}\\
\hat{\sigma}(\Bbbk,3+m) & =\left(\hat{\sigma}(\Bbbk)_{\uparrow},\{1\}^{m},2\right)\quad\left(m\geq0\right),
\end{align*}
where $(k_{1},\dots,k_{r})_{\uparrow}$ means $(k_{1},\dots,k_{r-1},k_{r}+1)$.
For example, we have
\[
\hat{\sigma}(1,\{2\}^{l_{1}},\ldots,1,\{2\}^{l_{d}})=(0,2l_{1}+1,\ldots,2l_{d}+1).
\]
We then define $\sigma:\mathbb{I}\rightarrow\mathbb{I}$ as
\[
\sigma(\Bbbk)\coloneqq\begin{cases}
\boldsymbol{\ell} & \text{ if }\hat{\sigma}(\Bbbk)=(0,\boldsymbol{\ell})\text{ with some }\boldsymbol{\ell}\in\mathbb{I}\\
\boldsymbol{\ell} & \text{ if }\hat{\sigma}(\Bbbk)=\boldsymbol{\ell}\text{ with some }\boldsymbol{\ell}\in\mathbb{I}.
\end{cases}
\]
Notice that this is well-defined and bijective. Additionally, we define
the sign $\delta:\mathbb{I}\rightarrow\{\pm1\}$ as
\[
\delta(k_{1},\ldots,k_{d})\coloneqq\begin{cases}
1 & k_{1}=1\:\text{ (equivalently, }\hat{\sigma}(\Bbbk)=(0,\sigma(\Bbbk))\text{)}\\
-1 & k_{1}>1\:\text{ (equivalently, }\hat{\sigma}(\Bbbk)=\sigma(\Bbbk)\text{)}.
\end{cases}
\]
Then, the full general version of Zhao's formula can be stated as
follows:
\begin{thm}[Zhao's formula \cite{Zhao_21}]
\label{thm:Generalized_2-1}For $\Bbbk\in\mathbb{I}$, we have
\[
\zeta^{\star}(\Bbbk)=\delta(\Bbbk)\zeta^{\#}(\sigma(\Bbbk)).
\]
Particularly, when $\Bbbk=(1,\{2\}^{l_{1}},\ldots,1,\{2\}^{l_{d}})$,
this gives the 2-1 formula
\begin{equation}
\zeta^{\star}(1,\{2\}^{l_{1}},\ldots,1,\{2\}^{l_{d}})=\zeta^{\#}(2l_{1}+1,\ldots,2l_{d}+1).\label{eq:2-1 formula}
\end{equation}
\end{thm}

Zhao's proof of Theorem \ref{thm:Generalized_2-1} is based on establishing
a refinement for which an inductive argument works. More precisely,
he constructed finite sum versions $\zeta_{N}^{\star}$ ($H_{N}^{\star}$
in his paper) and $\zeta_{N}^{\#}$ (a certain sum of $\mathcal{H}_{n}$
in his paper) of $\zeta^{\star}$ and $\zeta^{\#}$ which satisfy
\begin{equation}
\zeta_{N}^{\star}(\Bbbk)=\delta(\Bbbk)\zeta_{N}^{\#}(\sigma(\Bbbk)),\label{eq:Zhao_21_truncated}
\end{equation}
and recover $\zeta^{\star},\zeta^{\#}$ under the limit $N\rightarrow\infty$.
Here, $\zeta_{N}^{\star}$ is simply just the same sum of $\zeta^{\star}$
but truncated at $N$, while the definition of $\zeta_{N}^{\#}$ is
far more nontrivial, involving a quotient of binomial coefficients
(see Section \ref{subsec:Zhao21} for the definitions of $\zeta_{N}^{\star}$
and $\zeta_{N}^{\#}$, where we will also give the proof of \ref{eq:Zhao_21_truncated}
based on the iterated beta integrals). In addition to his ingenious
yet mysterious proof, the general correspondence $\Bbbk\longleftrightarrow\sigma(\Bbbk)$
of the indices is not very explicit, in the sense that it is only
defined recursively. What is the nature of Theorem \ref{thm:Generalized_2-1}?
Is there a clearer way to view the equality? Before answering this
question, let us also recall the following formula for multiple zeta
values:
\begin{thm}[Zagier's 2-3-2 formula \cite{Zagier_232}]
\label{thm:Zagier's 2-3-2}For $a,b\geq0$,
\begin{equation}
\zeta(\overbrace{2,\ldots,2}^{a},3,\overbrace{2,\ldots,2}^{b})=\sum_{\substack{r+s=a+b+1\\
r>0,s\geq0
}
}c_{r}^{a,b}\zeta(2r+1)\frac{\pi^{2s}}{(2s+1)!}\label{eq:Zagier's 2-3-2}
\end{equation}
where 
\[
c_{r}^{a,b}\coloneqq(-1)^{r}2\left\{ \binom{2r}{2a+2}-(1-2^{-2r})\binom{2r}{2b+1}\right\} .
\]
\end{thm}

This formula is called Zagier's 2-3-2 formula, and it is particularly
famous for its crucial role in Brown's celebrated faithfulness theorem
of the motivic Galois action of mixed Tate motive over $\mathbb{Z}$,
proving the linear independence of Hoffman's conjectural basis in
the motivic setting \cite{Brown_MTM_over_Z}. Zagier's 2-3-2 formula
was repeatedly proved by several mathematicians based on various hypergeometric
identities (see, for example, \cite{Li_232}, \cite{Hessami-232},
\cite{Teo-232}). 

Although they are not apparently very similar, the two Theorems \ref{thm:Generalized_2-1}
and \ref{thm:Zagier's 2-3-2} share a common flavor. First, via the
duality formula
\[
\zeta(\overbrace{2,\ldots,2}^{a},3,\overbrace{2,\ldots,2}^{b})=\zeta(\overbrace{2,\ldots,2}^{b},1,\overbrace{2,\ldots,2}^{a+1}),
\]
the left-hand side of (\ref{eq:Zagier's 2-3-2}) is a multiple zeta
value whose index is a sequence of $2$ with a $1$ inserted in the
middle, just like the indices appearing on the left-hand side of (\ref{eq:2_1_depth_1})
(or $d=1$ case of (\ref{eq:2-1 formula})). In both formulas, the
corresponding right-hand sides are essentially \emph{single} zeta
values, up to taking a linear combination and multiplying powers of
$\pi$. A natural question then, is whether there is a generalization
of (\ref{eq:Zagier's 2-3-2}) in which the left-hand side is
\[
\zeta(\{2\}^{l_{0}},3,\{2\}^{l_{1}},\ldots,3,\{2\}^{l_{d}})\quad\left(=\zeta(\{2\}^{l_{d}},1,\{2\}^{l_{d-1}+1},\ldots,1,\{2\}^{l_{0}+1})\right)
\]
and the right-hand side is a multiple zeta value of `depth $d$' in
some sense. As we will see in the sequel, the answer is \emph{yes}.
Moreover, we have a further generalization to MZV of an arbitrary
index. Note that, at this point, the similarity I described above
is still somewhat vague, and not quite legitimate. For example, the
right-hand side of (\ref{eq:2_1_depth_1}) is a single term of Riemann
zeta value, whereas that of (\ref{eq:Zagier's 2-3-2}) is a sum of
products of Riemann zeta values and powers of $\pi$ with slightly
complicated coefficients. To see a true similarity, we need to interpret
the equalities in terms of iterated integrals.

\subsection{Reformulation of Zhao and Zagier into integral equalities}

Let $X$ be a complex curve and $U\subset X$ be an open subset. For
a sequence $\omega_{1},\ldots,\omega_{n}$ of holomorphic differential
$1$-forms on $U$ and a piecewise smooth path $\gamma:[0,1]\rightarrow X$
from $x\in X$ to $y\in X$ such that $\gamma((0,1))\subset U$, let
$I_{\gamma}(x;\omega_{1},\dots,\omega_{n};y)$ (or $I_{\gamma}(x;\omega_{1}\cdots\omega_{n};y)$
if there is no risk of confusion) denote the iterated integral
\[
\int_{0<t_{1}<\cdots<t_{n}<1}\omega_{1}(\gamma(t_{1}))\cdots\omega_{n}(\gamma(t_{n}))
\]
when it converges. By iterated application of Cauchy integral theorem,
$I_{\gamma}$ depends only on its homotopy class of the path $\gamma$\footnote{Iterated integrals are the key objects in the $\pi_{1}$ de Rham theory
established by Chen \cite{Chen_77}, and the homotopy invariance is
not unconditional in general. However, we restrict ourselves to holomorphic
$1$-forms on a curve here, which trivializes the homotopy invariance
conditions.}. When the path $\gamma$ is clear from the context, we tacitly drop
$\gamma$ from the notation.

Now, let $X=\mathbb{P}^{1}(\mathbb{C})$ and $e_{z}(t)\coloneqq\frac{dt}{t-z}$
for $z\in\mathbb{C}$. Then $\zeta^{\star},\zeta^{\#}$ are expressed
by the iterated integrals\footnote{It is more standard to use $e_{1}$ and $e_{0}$ in the expression
for $\zeta^{\star}(k_{1},\ldots,k_{d}),$ but we use $e_{1}$ and
$e_{-1}$ instead (equivalent via an affine transformation) for nicely
writing the formulas later.}
\begin{align*}
\zeta^{\star}(k_{1},\ldots,k_{d}) & =(-1)^{k_{1}+\cdots+k_{d}}I(\infty;(e_{1}-e_{-1})e_{-1}^{k_{1}-1}e_{1}e_{-1}^{k_{2}-1}\cdots e_{1}e_{-1}^{k_{d}-1};1)\\
\zeta^{\#}(k_{1},\ldots,k_{d}) & =(-1)^{k_{1}+\cdots+k_{d}}I(\infty;2(e_{\varepsilon_{1}}-e_{0})e_{0}^{k_{1}-1}(2e_{\varepsilon_{2}}-e_{0})e_{0}^{k_{2}-1}\cdots(2e_{\varepsilon_{d}}-e_{0})e_{0}^{k_{d}-1};\varepsilon_{d+1})
\end{align*}
where the omitted path is the straight path on the real line from
positive infinity to $1$, and $\varepsilon_{1},\ldots,\varepsilon_{d+1}\in\{\pm1\}$
are defined recursively (backward) as $\varepsilon_{d+1}\coloneqq1$
and $\varepsilon_{i}\coloneqq(-1)^{k_{i}-1}\varepsilon_{i+1}$. Now,
additionally, let us define $f_{a,b}$ ($a,b\in\{\pm1\}$) by
\begin{align*}
f_{1,1} & \coloneqq2e_{1}-e_{0}\\
f_{-1,-1} & \coloneqq2e_{-1}-e_{0}\\
f_{1,-1} & =f_{-1,1}\coloneqq e_{0}.
\end{align*}
Then, magically, Zhao's formula turns into the following surprisingly
clean statement.
\begin{thm}[Reformulated version of Zhao's formula]
\label{thm:Zhao_int}For $\varepsilon_{0},\varepsilon_{1},\ldots,\varepsilon_{n+1}\in\{\pm1\}$
with $\varepsilon_{n}\neq\varepsilon_{n+1}=1$, we have
\begin{equation}
I(\infty;(e_{\varepsilon_{0}}-e_{\varepsilon_{1}})e_{\varepsilon_{2}}e_{\varepsilon_{3}}\cdots e_{\varepsilon_{n}};\varepsilon_{n+1})=I(\infty;(f_{\varepsilon_{0},\varepsilon_{2}}-f_{\varepsilon_{1},\varepsilon_{2}})f_{\varepsilon_{2},\varepsilon_{3}}f_{\varepsilon_{3},\varepsilon_{4}}\cdots f_{\varepsilon_{n},\varepsilon_{n+1}};\varepsilon_{n+1}).\label{eq:Zhao_int_ver}
\end{equation}
\end{thm}

Notice that the mysterious bijection $\sigma$ on $\mathbb{I}$ as
well as the sign $\delta$ have totally disappeared from the statement.
What about Zagier's 2-3-2 formula? The left-hand side has a standard
iterated integral expression
\begin{equation}
\zeta(\overbrace{2,\ldots,2}^{k},3,\overbrace{2,\ldots,2}^{l})=(-1)^{k+l+1}I_{\mathrm{dch}}(1;(e_{-1}e_{1})^{k}(e_{-1}e_{1}^{2})(e_{-1}e_{1})^{l};-1).\label{eq:Zagier's 2-3-2 LHS}
\end{equation}
How about the right-hand side? Magic happens again, and we have the
following iterated integral expression for the right-hand side:
\begin{prop}
\label{prop:2-3-2_iterated_int}Let $\gamma$ be a path from $1$
to $-1$ such that $\gamma(0,1)\subset\{z\in\mathbb{C}|\Im(z)>0\}$.
Then, we have
\begin{equation}
\sum_{\substack{r+s=k+l+1\\
r>0,s\geq0
}
}c_{r}^{k,l}\zeta(2r+1)\frac{\pi^{2s}}{(2s+1)!}=\frac{(-1)^{k+l+1}}{\pi i}I_{\gamma}(1;e_{0}^{2k+2}(2e_{1}-e_{0})e_{0}^{2l+1};-1).\label{eq:Zagier's 2-3-2 RHS}
\end{equation}
\end{prop}

\begin{proof}[Sketch of proof]
By applying the path composition formula to the right-hand side,
we can show that the real part of the right-hand side is equal to
the left-hand side. By applying the M\"obius transformation $t\mapsto t^{-1}$
to the right-hand side, we can show that the right-hand side is a
real number.
\end{proof}
By (\ref{eq:Zagier's 2-3-2 LHS}) and (\ref{eq:Zagier's 2-3-2 RHS}),
we now find that Zagier's 2-3-2 is equivalent to the equality
\begin{align*}
 & \quad I_{\mathrm{dch}}(1,\quad\overbrace{e_{-1}\quad e_{1}\;\cdots\;e_{-1}\quad e_{1}}^{2k+2}\;\brokenvert\;\overbrace{e_{1}\quad e_{-1}\;\cdots\;e_{1}\quad e_{-1}\quad e_{1}}^{2l+1}\quad;-1)\\
= & \frac{1}{\pi i}I_{\gamma}(1,\underbrace{f_{1,-1}f_{-1,1}\cdots f_{1,-1}f_{-1,1}}_{2k+2}f_{1,1}\underbrace{f_{1,-1}f_{-1,1}\cdots f_{1,-1}f_{-1,1}f_{1,-1}}_{2l+1};-1).
\end{align*}
What is the pattern here? By a careful observation on the sequence
of $\pm1$ on the two sides, one may be tempted to conjecture the
general equality
\[
I_{\mathrm{dch}}(\varepsilon_{0};e_{\varepsilon_{1}}e_{\varepsilon_{2}}\cdots e_{\varepsilon_{n}};\varepsilon_{n+1})=\frac{1}{\pi i}I_{\gamma}(\varepsilon_{0};f_{\varepsilon_{0},\varepsilon_{1}}f_{\varepsilon_{1},\varepsilon_{2}}\cdots f_{\varepsilon_{n},\varepsilon_{n+1}};\varepsilon_{n+1}).
\]
for $\varepsilon_{0},\varepsilon_{1},\ldots,\varepsilon_{n+1}\in\{\pm1\}$
(under the convergence conditions $\varepsilon_{1}\neq\varepsilon_{0}=1$
and $\varepsilon_{n}\neq\varepsilon_{n+1}=-1$). This speculation
turned out to be correct, and we have the following theorem:
\begin{thm}[Theorem \ref{thm:Case_2a-1} later]
\label{thm:Zagier_int}For $\varepsilon_{0},\varepsilon_{1},\ldots,\varepsilon_{n+1}\in\{\pm1\}$
with $\varepsilon_{1}\neq\varepsilon_{0}=1$ and $\varepsilon_{n}\neq\varepsilon_{n+1}=-1$,
we have
\begin{equation}
I_{\mathrm{dch}}(\varepsilon_{0};e_{\varepsilon_{1}}e_{\varepsilon_{2}}\cdots e_{\varepsilon_{n}};\varepsilon_{n+1})=\frac{1}{\pi i}I_{\gamma}(\varepsilon_{0};f_{\varepsilon_{0},\varepsilon_{1}}f_{\varepsilon_{1},\varepsilon_{2}}\cdots f_{\varepsilon_{n},\varepsilon_{n+1}};\varepsilon_{n+1}).\label{eq:Zagier_int_ver_gen}
\end{equation}
\end{thm}

\begin{rem}
Since
\[
\zeta(\{2\}^{l_{0}},3,\{2\}^{l_{1}},\ldots,3,\{2\}^{l_{d}})=(-1)^{\sum_{i=0}^{d}(l_{i}+1)}I_{\mathrm{dch}}(1;(e_{-1}e_{1})^{l_{0}+1}e_{1}(e_{-1}e_{1})^{l_{1}+1}\cdots e_{1}(e_{-1}e_{1})^{l_{d}+1};-1),
\]
Theorem \ref{thm:Zagier_int} gives
\[
\zeta(\{2\}^{l_{0}},3,\{2\}^{l_{1}},\ldots,3,\{2\}^{l_{d}})=(-1)^{\sum_{i=0}^{d}(l_{i}+1)}\frac{1}{\pi i}I_{\gamma}(1;e_{0}^{2l_{0}+2}f_{1,1}e_{0}^{2l_{1}+2}f_{1,1}\cdots e_{0}^{2l_{d-1}+2}f_{1,1}e_{0}^{2l_{d}+1};-1),
\]
which is exactly the case where $f_{-1,-1}$ does not appear on the
right-hand side.
\end{rem}

By reformulating Zhao's formula and Zagier's formula into Theorem
\ref{thm:Zhao_int} and Theorem \ref{thm:Zagier_int}, respectively,
we now see a striking similarity between the two formulas. At this
point, it is natural to expect a common structure or mechanism that
explains the two theorems simultaneously. As we know that the left-hand
side of (\ref{eq:Zhao_int_ver}) and (\ref{eq:Zagier_int_ver_gen})
are special values of the hyperlogarithms $I_{\mathrm{dch}}(\infty;(e_{z_{0}}-e_{z_{1}})e_{z_{2}}e_{z_{3}}\cdots e_{z_{n}};z_{n+1})$
and $I_{\mathrm{dch}}(z_{0};e_{z_{1}}e_{z_{2}}\cdots e_{z_{n}};z_{n+1})$
evaluated at $z_{0},\ldots,z_{n+1}\in\{\pm1\}$, a potential strategy
is to lift the equations to a functional equation between hyperlogarithms
and some integral that reduces to the left-hand sides of (\ref{eq:Zhao_int_ver})
and (\ref{eq:Zagier_int_ver_gen}) in each case. More precisely, we
may ask whether there are differential $1$-forms $\hat{f}_{x,y}(t)$
defined for general complex numbers $x,y$, such that 
\begin{enumerate}
\item $\hat{f}_{a,b}(t)=f_{a,b}(t)$ for $a,b\in\{\pm1\}$ and
\item 
\begin{align*}
I(\infty;(e_{z_{0}}-e_{z_{1}})e_{z_{2}}\cdots e_{z_{n}};z_{n+1}) & =I(\infty';(\hat{f}_{z_{0},z_{2}}-\hat{f}_{z_{1},z_{2}})\hat{f}_{z_{2},z_{3}}\cdots\hat{f}_{z_{n},z_{n+1}};z_{n+1}'),\\
I(z_{0};e_{z_{1}}e_{z_{2}}\cdots e_{z_{n}};z_{n+1}) & =\frac{1}{\pi i}I(z_{0}';\hat{f}_{z_{0},z_{1}}\hat{f}_{z_{1},z_{2}}\cdots\hat{f}_{z_{n},z_{n+1}};z_{n+1}')
\end{align*}
for suitable choices of paths on the two sides (here $x'$ at the
end points of the integration paths means that it should be determined
by $x$, but it is not clear what it should be for a general value
of $x$).
\end{enumerate}
If we only look at Condition (1), it is not too difficult to find
such $\hat{f}_{x,y}$. For example, if we naively define $\hat{f}_{x,y}$
to be
\[
2e_{\frac{x+y}{2}}-e_{0},
\]
this satisfies Condition (1), while it fails to satisfy Condition
(2) unfortunately. Such $\hat{f}_{x,y}$ exists, but not within the
world of rational $1$-forms, and the `correct' expression turned
out to be
\[
\hat{f}_{x,y}(t)\coloneqq2d\log\left(\sqrt{t^{2}-2xt+1}+\sqrt{t^{2}-2yt+1}\right)-e_{0}(t),
\]
as proved in later sections. Here, the sign of the square roots needs
to be chosen as $\sqrt{t^{2}-2xt+1}=t-x$ for $x\in\{\pm1\}$. Notice
that Condition (1) can be checked easily by quick calculations
\[
\hat{f}_{x,x}(t)=2d\log\left(2(t-x)\right)-e_{0}(t)=f_{x,x}(t)
\]
and
\[
\hat{f}_{x,-x}(t)=2d\log\left((t-x)+(t+x)\right)-e_{0}(t)=f_{x,-x}(t)
\]
for $x\in\{\pm1\}$, while whether $\hat{f}_{x,y}$ also satisfies
Condition (2) is not clear at this point. Although $\hat{f}_{x,y}$
appears to be a bit complicated, it nicely simplifies as
\[
\hat{f}_{x,y}(t)=2d\log\left(\sqrt{u-x}+\sqrt{u-y}\right)=\frac{du}{\sqrt{(u-x)(u-y)}},
\]
via the change of coordinates $u=\frac{t+t^{-1}}{2}$. Furthermore,
in the new $u$-coordinates, the subtlety of ``for suitable choices
of paths on the two sides'' in Condition 2 nicely disappears, and
we have the following:
\begin{thm}
\label{thm:Zhao and Zagier in multi-variable}For an arbitrary simple
path $\gamma$ from $\infty$ to $z_{n+1}$ and $\gamma'$ from $z_{0}$
to $z_{n+1}$, we have 
\begin{enumerate}
\item 
\[
I_{\gamma}(\infty;(e_{z_{0}}-e_{z_{1}})e_{z_{2}}\cdots e_{z_{n}};z_{n+1})=I_{\gamma}(\infty;(F_{z_{0},z_{2}}-F_{z_{1},z_{2}})F_{z_{2},z_{3}}\cdots F_{z_{n},z_{n+1}};z_{n+1})
\]
and
\item 
\[
I_{\gamma'}(z_{0};e_{z_{1}}e_{z_{2}}\cdots e_{z_{n}};z_{n+1})=\frac{1}{\pi i}I_{\gamma'}(z_{0};F_{z_{0},z_{1}}F_{z_{1},z_{2}}\cdots F_{z_{n},z_{n+1}};z_{n+1}).
\]
\end{enumerate}
\end{thm}

This theorem turns out to be a special case of a far more general
theorem in the next section.

\subsection{Interpretation of Zhao and Zagier by the translation invariance of
iterated beta integrals}

Motivated by the functional equations of Theorem \ref{thm:Zhao and Zagier in multi-variable},
we consider the differential form
\[
\left[\substack{x,y\\
\alpha,\beta
}
\right](t):=\frac{dt}{(t-x)^{\alpha}(t-y)^{1-\beta}},
\]
and define the iterated beta integrals
\begin{align*}
B_{\gamma}^{\mathrm{f},\mathrm{f}}(\left.\substack{z_{0}\\
\alpha_{0}
}
\right|\left.\substack{z_{1}\\
\alpha_{1}
}
\right|\cdots\left|\substack{z_{n}\\
\alpha_{n}
}
\right.) & \coloneqq I_{\gamma}\left(z_{0};\left[\substack{z_{0},z_{1}\\
\alpha_{0},\alpha_{1}
}
\right],\left[\substack{z_{1},z_{2}\\
\alpha_{1},\alpha_{2}
}
\right],\ldots,\left[\substack{z_{n-1},z_{n}\\
\alpha_{n-1},\alpha_{n}
}
\right];z_{n}\right)\\
B_{\gamma}^{\infty,\mathrm{f}}(\left.\substack{z_{0}\\
\alpha_{0}
}
\right|\left.\substack{z_{1}\\
\alpha_{1}
}
\right|\cdots\left|\substack{z_{n}\\
\alpha_{n}
}
\right.) & \coloneqq I_{\gamma}\left(\infty;\left[\substack{z_{0},z_{1}\\
\alpha_{0},\alpha_{1}
}
\right],\left[\substack{z_{1},z_{2}\\
\alpha_{1},\alpha_{2}
}
\right],\ldots,\left[\substack{z_{n-1},z_{n}\\
\alpha_{n-1},\alpha_{n}
}
\right];z_{n}\right)
\end{align*}
and
\begin{align*}
\hat{B}_{\gamma}^{\bullet,\mathrm{f}}(\left.\substack{z_{0}\\
\alpha_{0}
}
\right|\left.\substack{z_{1}\\
\alpha_{1}
}
\right|\cdots\left|\substack{z_{n}\\
\alpha_{n}
}
\right.) & \coloneqq\frac{B_{\gamma}^{\bullet,\mathrm{f}}(\left.\substack{z_{0}\\
\alpha_{0}
}
\right|\left.\substack{z_{1}\\
\alpha_{1}
}
\right|\cdots\left|\substack{z_{n}\\
\alpha_{n}
}
\right.)}{B_{\gamma}^{\bullet,\mathrm{f}}(\substack{z_{0}\\
\alpha_{0}
}
\left|\substack{z_{n}\\
\alpha_{n}
}
\right.)}
\end{align*}
for $\bullet\in\{{\rm f},\infty\}$ (we drop $\gamma$ from the notation
if it is clear from the context). Then, we can interpret both sides
of (1) and (2) of Theorem \ref{thm:Zhao and Zagier in multi-variable}
by these notations as follows. 

Noting $F_{x,y}=\left[\substack{x,y\\
\frac{1}{2},\frac{1}{2}
}
\right]$ and $B_{\gamma}^{\mathrm{f},\mathrm{f}}(\substack{z_{0}\\
\frac{1}{2}
}
\left|\substack{z_{n}\\
\frac{1}{2}
}
\right.)=\pi i$, the right-hand side of (2) of Theorem \ref{thm:Zhao and Zagier in multi-variable}
is precisely
\[
\hat{B}^{\mathrm{f},\mathrm{f}}(\left.\substack{z_{0}\\
\frac{1}{2}
}
\right|\left.\substack{z_{1}\\
\frac{1}{2}
}
\right|\cdots\left|\substack{z_{n+1}\\
\frac{1}{2}
}
\right.).
\]
On the other hand,
\begin{align*}
\lim_{\beta\to+0}\hat{B}^{\mathrm{f},\mathrm{f}}(\left.\substack{z_{0}\\
0
}
\right|\left.\substack{z_{1}\\
0
}
\right|\cdots\left|\substack{z_{n}\\
0
}
\right.\left|\substack{z_{n+1}\\
\beta
}
\right.) & =\lim_{\beta\to+0}\frac{I\left(z_{0};e_{z_{1}},\ldots,e_{z_{n}},(t-z_{n+1})^{\beta-1}dt;z_{n+1}\right)}{I\left(z_{0};(t-z_{n+1})^{\beta-1}dt;z_{n+1}\right)}\\
 & =\lim_{\beta\to+0}\frac{-\beta^{-1}I\left(z_{0};e_{z_{1}},\ldots,e_{z_{n-1}},(t-z_{n+1})^{\beta}e_{z_{n}};z_{n+1}\right)}{-\beta^{-1}(z_{0}-z_{n+1})^{\beta}}\\
 & =I\left(z_{0};e_{z_{1}},\ldots,e_{z_{n}};z_{n+1}\right),
\end{align*}
where the last expression is equal to the left-hand side of (2) of
Theorem \ref{thm:Zhao and Zagier in multi-variable}. As we will see
in Section \ref{sec:Analytic-continuation}, the function $\hat{B}_{\gamma}^{\mathrm{f},\mathrm{f}}(\left.\substack{z_{0}\\
\alpha_{0}
}
\right|\left.\substack{z_{1}\\
\alpha_{1}
}
\right|\cdots\left|\substack{z_{n}\\
\alpha_{n}
}
\right.)$ is meromorphically continued for $\boldsymbol{\alpha}=(\alpha_{0},\alpha_{1},\ldots,\alpha_{n})\in\mathbb{C}^{n+1}$
and holomorphic at $\boldsymbol{\alpha}=(0,0,\ldots,0)$. Therefore,
we simply write $\lim_{\beta\to+0}\hat{B}_{\gamma}^{\mathrm{f},\mathrm{f}}(\left.\substack{z_{0}\\
0
}
\right|\left.\substack{z_{1}\\
0
}
\right|\cdots\left|\substack{z_{n-1}\\
0
}
\right.\left|\substack{z_{n}\\
\beta
}
\right.)$ as $\hat{B}_{\gamma}^{\mathrm{f},\mathrm{f}}(\left.\substack{z_{0}\\
0
}
\right|\left.\substack{z_{1}\\
0
}
\right|\cdots\left|\substack{z_{n}\\
0
}
\right.)$. Thus, formula (2) of Theorem \ref{thm:Zhao and Zagier in multi-variable}
is equivalent to
\[
\hat{B}_{\gamma}^{\mathrm{f},\mathrm{f}}(\left.\substack{z_{0}\\
0
}
\right|\left.\substack{z_{1}\\
0
}
\right|\cdots\left|\substack{z_{n+1}\\
0
}
\right.)=\hat{B}_{\gamma}^{\mathrm{f},\mathrm{f}}(\left.\substack{z_{0}\\
\frac{1}{2}
}
\right|\left.\substack{z_{1}\\
\frac{1}{2}
}
\right|\cdots\left|\substack{z_{n+1}\\
\frac{1}{2}
}
\right.).
\]
In a similar manner, formula (1) of Theorem \ref{thm:Zhao and Zagier in multi-variable}
can be also restated as an iterated beta integral identity
\[
\hat{B}_{\gamma}^{\infty,\mathrm{f}}(\left.\substack{z_{0}\\
0
}
\right|\left.\substack{z_{2}\\
0
}
\right|\cdots\left|\substack{z_{n+1}\\
0
}
\right.)-\hat{B}_{\gamma}^{\infty,\mathrm{f}}(\left.\substack{z_{1}\\
0
}
\right|\left.\substack{z_{2}\\
0
}
\right|\cdots\left|\substack{z_{n+1}\\
0
}
\right.)=\hat{B}_{\gamma}^{\infty,\mathrm{f}}(\left.\substack{z_{0}\\
\frac{1}{2}
}
\right|\left.\substack{z_{2}\\
\frac{1}{2}
}
\right|\cdots\left|\substack{z_{n+1}\\
\frac{1}{2}
}
\right.)-\hat{B}_{\gamma}^{\infty,\mathrm{f}}(\left.\substack{z_{1}\\
\frac{1}{2}
}
\right|\left.\substack{z_{2}\\
\frac{1}{2}
}
\right|\cdots\left|\substack{z_{n+1}\\
\frac{1}{2}
}
\right.).
\]
\footnote{Precisely speaking, two sides of the formula should be understood
as the limits
\[
\lim_{\varepsilon\rightarrow0}\left(\hat{B}_{\gamma}^{\infty,\mathrm{f}}(\left.\substack{z_{0}\\
\varepsilon
}
\right|\left.\substack{z_{2}\\
0
}
\right|\cdots\left|\substack{z_{n+1}\\
0
}
\right.)-\hat{B}_{\gamma}^{\infty,\mathrm{f}}(\left.\substack{z_{1}\\
\varepsilon
}
\right|\left.\substack{z_{2}\\
0
}
\right|\cdots\left|\substack{z_{n+1}\\
0
}
\right.)\right)
\]
and 
\[
\lim_{\varepsilon\rightarrow0}\left(\hat{B}_{\gamma}^{\infty,\mathrm{f}}(\left.\substack{z_{0}\\
\varepsilon+1/2
}
\right|\left.\substack{z_{2}\\
1/2
}
\right|\cdots\left|\substack{z_{n+1}\\
1/2
}
\right.)-\hat{B}_{\gamma}^{\infty,\mathrm{f}}(\left.\substack{z_{1}\\
\varepsilon+1/2
}
\right|\left.\substack{z_{2}\\
1/2
}
\right|\cdots\left|\substack{z_{n+1}\\
1/2
}
\right.)\right)
\]
since they are term-wisely divergent.}Both of these two formulas are now stated as relationship between
the values of iterated beta integral with different exponent parameters.
In fact, these formulas are special instances of the following general
theorem:
\begin{thm}[Translation invariance (Theorem \ref{thm: translation invariance})]
 \label{thm:translation invariance intro}The iterated beta integrals
$\hat{B}_{\gamma}^{\bullet,\mathrm{f}}(\left.\substack{z_{0}\\
\alpha_{0}
}
\right|\left.\substack{z_{1}\\
\alpha_{1}
}
\right|\cdots\left|\substack{z_{n}\\
\alpha_{n}
}
\right.)$ are invariant under simultaneous translation of the exponent parameters
$\alpha_{i}$, i.e.,
\[
\hat{B}_{\gamma}^{\bullet,\mathrm{f}}(\left.\substack{z_{0}\\
\alpha_{0}
}
\right|\left.\substack{z_{1}\\
\alpha_{1}
}
\right|\cdots\left|\substack{z_{n}\\
\alpha_{n}
}
\right.)=\hat{B}_{\gamma}^{\bullet,\mathrm{f}}(\left.\substack{z_{0}\\
\alpha_{0}+c
}
\right|\left.\substack{z_{1}\\
\alpha_{1}+c
}
\right|\cdots\left|\substack{z_{n}\\
\alpha_{n}+c
}
\right.)
\]
for $c\in\mathbb{C}.$
\end{thm}

\subsection{Other applications of the translation invariance}

Now, let us recall how we retrieved Zhao's formula (Theorem \ref{thm:Zhao_int})
and our generalization of Zagier's formula (Theorem \ref{thm:Zagier_int})
from Theorem \ref{thm:translation invariance intro}. We specialized
the variables to $z_{0},z_{1},\ldots,z_{n}\in\{\pm1\}$ and considered
the exponents $(\alpha_{0},\alpha_{1},\ldots,\alpha_{n})=(0,0,\ldots,0)$
and $(\alpha_{0},\alpha_{1},\ldots,\alpha_{n})=(1/2,1/2,\ldots,1/2)$.
On the side of $(\alpha_{0},\alpha_{1},\ldots,\alpha_{n})=(0,0,\ldots,0)$,
the differential forms that appear in the integral are 
\[
e_{1}(u)=\frac{du}{u-1},\quad\text{and }e_{-1}(u)=\frac{du}{u+1},
\]
hence gives a multiple zeta value on that side. On the side of $(\alpha_{0},\alpha_{1},\ldots,\alpha_{n})=(1/2,1/2,\ldots,1/2)$,
the differential forms that appear in the integral are 
\[
f_{1,1}(u)=\frac{du}{u-1},f_{-1,-1}(u)=\frac{du}{u+1},\quad\text{and }f_{1,-1}(u)=f_{-1,1}(u)=\frac{du}{\sqrt{u^{2}-1}},
\]
which are not entirely rational. However, since the curve $v^{2}=u^{2}-1$
is rational, with a parametrization 
\[
(u,v)=\left(\frac{t+t^{-1}}{2},\frac{t-t^{-1}}{2}\right),
\]
it can be expressed as rational differential forms
\[
f_{1,1}=2e_{1}(t)-e_{0}(t),f_{-1,-1}(u)=2e_{-1}(t)-e_{0}(t),\quad\text{and }f_{1,-1}(u)=f_{-1,1}(u)=e_{0}(t)
\]
in the $t$-coordinate. 

Generalizing this idea, Theorem \ref{thm:translation invariance intro}
yields various interesting formulas apart from those of Zhao and Zagier:
Just as we did in the Zhao-Zagier case, we associate a complex algebraic
curve $X_{\boldsymbol{z},\boldsymbol{\alpha}}$ with a given set of
parameters $z_{0},z_{1},\ldots,z_{n}\in\mathbb{C}$ and $\alpha_{0},\alpha_{1},\ldots,\alpha_{n}\in\mathbb{Q}$
on which all necessary differential forms are defined, so that Theorem
\ref{thm:translation invariance intro} is viewed as a relation between
iterated integrals on $X_{\boldsymbol{z},(\alpha_{0},\alpha_{1},\ldots,\alpha_{n})}$
and those on $X_{\boldsymbol{z},(\alpha_{0}+c,\alpha_{1}+c,\ldots,\alpha_{n}+c)}$.
In particular, if both $X_{\boldsymbol{z},(\alpha_{0},\alpha_{1},\ldots,\alpha_{n})}$
and $X_{\boldsymbol{z},(\alpha_{0}+c,\alpha_{1}+c,\ldots,\alpha_{n}+c)}$
are rational, we can rewrite the integrals in terms of special values
of hyperlogarithms, thus obtain curious relations like Zhao's formula
and Zagier's formula. These genus zero cases can be classified via
Riemann-Hurwitz formula, and the complete classification is given
in Section \ref{sec:classification_genus_one}.  All cases are discussed
in detail (Sections \ref{sec:Case-A1-Zagier232-Zhao21}, \ref{sec:Case-B1-OmegaValues},
\ref{sec:Case-A2-Ohno} and \ref{sec:Case-B2-others}). One of the
cases has a nice application in the evaluation of the omega values
introduced by Charlton, Heller, Heller and Traizet in \cite{CHHT_omega}.
Another case has an application to Ohno's relation \cite{Ohno_rel}.

\subsection{Structure of the paper}

This article is divided into two parts, Part I and Part II. In Part
I, we introduce iterated beta integrals and develop their basic analytic
theory, with particular emphasis on the translation invariance that
lies at the heart of the paper. In Part II, we classify the patterns
of hyperlogarithm identities arising from this translation invariance.
This classification recovers several known formulas, including formulas
of Zhao and Zagier, and also yields new identities.

More precisely, in Part I, we first introduce incomplete and complete,
as well as finite and infinite, iterated beta integrals, together
with several normalizations (Section \ref{sec:Defi-IteratedBeta}).
We then study their domains of convergence and establish meromorphic
continuation with respect to the exponent parameters, together with
a description of the possible poles (Section \ref{sec:Analytic-continuation}).
Next, we derive special value formulas, including a relation between
complete and incomplete iterated beta integrals; as a special case,
this also realizes hyperlogarithms as a special instance of complete
iterated beta integrals (Section \ref{sec:Special-values-and-hyperlogarithms}).
We also prove a contiguous-type relation showing that iterated beta
integrals whose exponents differ by integers are equal up to a simple
factor, modulo lower-dimensional iterated beta integrals (Section
\ref{sec:Contiguous-type-relation}).  We then establish the total
differential equation for iterated beta integrals, which cleanly extends
the corresponding differential equation for hyperlogarithms (Section
\ref{sec:Differential-equations}), and use it to prove the translation
invariance, the main result of Part I (Section \ref{sec:Translation-invariance}).
After that, we derive series expansion formulas for finite and infinite
iterated beta integrals (Section \ref{sec:Series-expansions}). We
then establish a relation between finite and infinite iterated beta
integrals (Section \ref{sec:Relating-finite-and-infinite-path});
a key ingredient is the fact that iterated beta integrals along the
Pochhammer contour are independent of the choice of base points. We
conclude Part I by proving a family of algebraic relations satisfied
by iterated beta integrals with general parameters (Section \ref{sec:Algebraic-relations}),
and by establishing a monodromy formula (Section \ref{sec:Monodromies}).

Part II is devoted to the classification of the hyperlogarithm identities
obtained from translation invariance. We first compute the genus of
the associated complex curve $X_{\boldsymbol{z},\boldsymbol{\alpha}}$,
and classify all translation-equivalent genus-zero pairs $\left(X_{\boldsymbol{z},\boldsymbol{\alpha}},X_{\boldsymbol{z},\boldsymbol{\alpha}'}\right)$
(Section \ref{sec:classification_genus_one}). This classification
yields two `sporadic' cases, denoted by B1 and B2, and two `infinite
families', denoted by A1 and A2. We then study the family A1 in detail
(Section \ref{sec:Case-A1-Zagier232-Zhao21}). This family already
contains Zhao's formula and Zagier's formula, and, after introducing
a continuous parameter, also leads to a theorem on Hurwitz-type multiple
series. In the same setting, we obtain four variants of Zagier's $2$-$3$-$2$
formula, including Zagier's original formula, Murakami's formula for
Hoffman's $t$-values, and another formula due to Charlton. In Section
\ref{sec:Case-B1-OmegaValues}, we treat the sporadic case B1 and,
as a special case, obtain a formula expressing omega values of certain
indices in terms of alternating multiple zeta values. In Section \ref{sec:Case-A2-Ohno},
we turn to the family A2; here again a continuous parameter can be
introduced, and this yields Ohno's relation for multiple zeta values,
giving a new proof based on a simple symmetry of the integral. Finally,
in Section \ref{sec:Case-B2-others}, we study the sporadic case B2,
which yields a new equality between hyperlogarithms in two variables.

\subsection{A note on complex powers}

Throughout the paper, the notation $x^{\alpha}$ for a complex number
$\alpha$ is frequently used, even when $x$ is not necessarily a
positive real number. Strictly speaking, $x^{\alpha}$ depends on
the specification of a branch and is therefore, in that sense, an
imprecise notation. However, since a rigorous description of the branch
would unnecessarily add technicalities and complicate the exposition
without benefiting the reader, we do not necessarily explicate the
precise branch choice in the description. Readers should interpret
the choice of branches appropriately, for example, by making a coherent
choice along the paths of integration.

\subsection*{Acknowledgements}

The first author is grateful to Ryota Umezawa for informing him about
Akhilesh's paper. The second author is grateful to Steven Charlton
for informing him about the $\Omega$-values. This work was supported
by JSPS KAKENHI Grant Number JP22K03244 and NSTC Grant Numbers 111-2115-M-002-003-MY3
and 113-2115-M-002-007-MY3.

\part{Definitions and properties of iterated beta integrals}

In this part, we will introduce and investigate the fundamental properties
of iterated beta integrals. The iterated beta integral is a common
generalization of the hyperlogarithm and the beta integral. We will
prove various properties, such as a differential formula, a translation
invariance property, and series expressions. 

\section{Definition of iterated beta integrals\label{sec:Defi-IteratedBeta}}

In this section, we define iterated beta integrals. Throughout the
paper, we implicitly assume that the path $\gamma$ of integration
is always `well-behaved', in the sense that it will not circulate
around the endpoints infinitely many times, i.e., the imaginary part
of $\log\left(\gamma(t)-\gamma(p)\right)$ (resp. $\log\gamma(t)$)
is bounded if $\gamma(p)$ is finite (resp. infinite) for $p\in\{0,1\}$.
Let us define the beta differential form

\[
\left[\substack{x,y\\
\alpha,\beta
}
\right](t):=\frac{dt}{(t-x)^{\alpha}(t-y)^{1-\beta}}
\]
and its normalized version
\[
\left\{ \substack{x,y\\
\alpha,\beta
}
\right\} (t)\coloneqq\frac{(x-y)^{\alpha-\beta}dt}{(t-x)^{\alpha}(t-y)^{1-\beta}}
\]
(note that the expressions like this have ambiguities due to the branches
of complex powers). For a path $\gamma$ from $z$ to $z'$ ($z$
and $z'$ may be infinity), we define the \emph{incomplete iterated
beta integral}
\[
B_{\gamma}(z;\left.\substack{z_{0}\\
\alpha_{0}
}
\right|\left.\substack{z_{1}\\
\alpha_{1}
}
\right|\cdots\left|\substack{z_{n}\\
\alpha_{n}
}
\right.;z')
\]
by
\[
I_{\gamma}\left(z;\left[\substack{z_{0},z_{1}\\
\alpha_{0},\alpha_{1}
}
\right],\left[\substack{z_{1},z_{2}\\
\alpha_{1},\alpha_{2}
}
\right],\ldots,\left[\substack{z_{n-1},z_{n}\\
\alpha_{n-1},\alpha_{n}
}
\right];z'\right)
\]
and we regard $B_{\gamma}\!\left(z;\substack{z_{0}\\
\alpha_{0}
}
;z'\right)=1$ for the case $n=0$. Note that, even though the differential forms
$\left[\substack{z_{i},z_{i+1}\\
\alpha_{i},\alpha_{i+1}
}
\right]$ themselves depend on the choice of the branches, the ambiguities
of the incomplete iterated beta integral arising from the choice of
$(t-z_{i})^{\alpha_{i}}$ cancel out naturally for $i=1,\dots,n-1$. 

We introduce the four types of \emph{complete iterated beta integrals}
\[
B_{\gamma}^{\bullet,\circ}\!\left(\left.\substack{z_{0}\\
\alpha_{0}
}
\right|\left.\substack{z_{1}\\
\alpha_{1}
}
\right|\cdots\left|\substack{z_{n}\\
\alpha_{n}
}
\right.\right)\quad\left(\bullet,\circ\in\{\mathrm{f},\infty\}\right)
\]
by
\[
B_{\gamma}\!\left(p;\left.\substack{z_{0}\\
\alpha_{0}
}
\right|\left.\substack{z_{1}\\
\alpha_{1}
}
\right|\cdots\left|\substack{z_{n}\\
\alpha_{n}
}
\right.;q\right)
\]
where
\[
p=\begin{cases}
z_{0} & \text{if }\bullet={\rm f}\\
\infty & \text{if }\bullet=\infty,
\end{cases}\qquad q=\begin{cases}
z_{n} & \text{if }\circ={\rm f}\\
\infty & \text{if }\circ=\infty.
\end{cases}
\]
Similarly, we introduce the left-complete (resp. right-complete) iterated
beta integrals
\[
B_{\gamma}^{\bullet}\!\left(\left.\substack{z_{0}\\
\alpha_{0}
}
\right|\left.\substack{z_{1}\\
\alpha_{1}
}
\right|\cdots\left|\substack{z_{n}\\
\alpha_{n}
}
\right.;z'\right)\quad\left(\text{resp. }B_{\gamma}^{\circ}\!\left(z;\left.\substack{z_{0}\\
\alpha_{0}
}
\right|\left.\substack{z_{1}\\
\alpha_{1}
}
\right|\cdots\left|\substack{z_{n}\\
\alpha_{n}
}
\right.\right)\right)
\]
by 
\begin{align*}
 & B_{\gamma}\!\left(p;\left.\substack{z_{0}\\
\alpha_{0}
}
\right|\left.\substack{z_{1}\\
\alpha_{1}
}
\right|\cdots\left|\substack{z_{n}\\
\alpha_{n}
}
\right.;z'\right)\quad\left(p\in\{z_{0},\infty\}\right),\\
 & \left(\text{resp. }B_{\gamma}\!\left(z;\left.\substack{z_{0}\\
\alpha_{0}
}
\right|\left.\substack{z_{1}\\
\alpha_{1}
}
\right|\cdots\left|\substack{z_{n}\\
\alpha_{n}
}
\right.;q\right)\quad\left(q\in\{z_{n},\infty\}\right)\right)
\end{align*}
according to the aforementioned cases. Our main interest is the complete
iterated beta integrals\footnote{As shown in Theorem \ref{thm: complete vs incomplete}, (incomplete)
iterated beta integrals
\[
B_{\gamma}(z;\left.\substack{z_{0}\\
\alpha_{0}
}
\right|\left.\substack{z_{1}\\
\alpha_{1}
}
\right|\cdots\left|\substack{z_{n}\\
\alpha_{n}
}
\right.;z'),\quad B_{\gamma}^{\bullet}\!(\left.\substack{z_{0}\\
\alpha_{0}
}
\right|\left.\substack{z_{1}\\
\alpha_{1}
}
\right|\cdots\left|\substack{z_{n}\\
\alpha_{n}
}
\right.;z'),\quad B_{\gamma}^{\circ}\!(z;\left.\substack{z_{0}\\
\alpha_{0}
}
\right|\left.\substack{z_{1}\\
\alpha_{1}
}
\right|\cdots\left|\substack{z_{n}\\
\alpha_{n}
}
\right.)
\]
are always expressible by the complete ones.}. Notice that they are not necessarily convergent (the domain of convergence
and analytic continuation with respect to $\alpha_{0},\dots,\alpha_{n}$
will be discussed in Section \ref{sec:Analytic-continuation}). 

Furthermore, for $n\geq1$, we define the \emph{normalized iterated
beta integral} by 
\[
\hat{B}_{\gamma}^{\bullet,\circ}(\left.\substack{z_{0}\\
\alpha_{0}
}
\right|\left.\substack{z_{1}\\
\alpha_{1}
}
\right|\cdots\left|\substack{z_{n}\\
\alpha_{n}
}
\right.)=\frac{B_{\gamma}^{\bullet,\circ}(\left.\substack{z_{0}\\
\alpha_{0}
}
\right|\left.\substack{z_{1}\\
\alpha_{1}
}
\right|\cdots\left|\substack{z_{n}\\
\alpha_{n}
}
\right.)}{B_{\gamma}^{\bullet,\circ}(\left.\substack{z_{0}\\
\alpha_{0}
}
\right|\substack{z_{n}\\
\alpha_{n}
}
)},
\]
which, as we will see later, behaves in an even nicer way. Then, since
the ambiguity for the choices of $(t-z_{0})^{\alpha_{0}}$ and $(t-z_{n})^{\alpha_{n}}$
are also cancelled, $\hat{B}_{\gamma}^{\bullet,\circ}$ can be defined
without choices of the branches of $(t-z_{i})^{\alpha_{i}}$ for all
$i$. Furthermore, when all $z_{0},\dots,z_{n}$ are distinct, we
also introduce the `scripted' notations
\begin{align*}
\mathscr{B}_{\gamma}(z;\left.\substack{z_{0}\\
\alpha_{0}
}
\right|\left.\substack{z_{1}\\
\alpha_{1}
}
\right|\cdots\left|\substack{z_{n}\\
\alpha_{n}
}
\right.;z') & =I_{\gamma}(z;\left\{ \substack{z_{0},z_{1}\\
\alpha_{0},\alpha_{1}
}
\right\} ,\left\{ \substack{z_{1},z_{2}\\
\alpha_{1},\alpha_{2}
}
\right\} ,\ldots,\left\{ \substack{z_{n-1},z_{n}\\
\alpha_{n-1},\alpha_{n}
}
\right\} ;z')\\
 & =B_{\gamma}(z;\left.\substack{z_{0}\\
\alpha_{0}
}
\right|\left.\substack{z_{1}\\
\alpha_{1}
}
\right|\cdots\left|\substack{z_{n}\\
\alpha_{n}
}
\right.;z')\prod_{j=1}^{n}\left(z_{j-1}-z_{j}\right)^{\alpha_{j-1}-\alpha_{j}},
\end{align*}
\[
\mathscr{B}_{\gamma}^{\bullet,\circ}(\left.\substack{z_{0}\\
\alpha_{0}
}
\right|\left.\substack{z_{1}\\
\alpha_{1}
}
\right|\cdots\left|\substack{z_{n}\\
\alpha_{n}
}
\right.)=B_{\gamma}^{\bullet,\circ}(\left.\substack{z_{0}\\
\alpha_{0}
}
\right|\left.\substack{z_{1}\\
\alpha_{1}
}
\right|\cdots\left|\substack{z_{n}\\
\alpha_{n}
}
\right.)\prod_{j=1}^{n}\left(z_{j-1}-z_{j}\right)^{\alpha_{j-1}-\alpha_{j}},
\]
and
\begin{align*}
\hat{\mathscr{B}}_{\gamma}^{\bullet,\circ}\left(\left.\substack{z_{0}\\
\alpha_{0}
}
\right|\left.\substack{z_{1}\\
\alpha_{1}
}
\right|\cdots\left|\substack{z_{n}\\
\alpha_{n}
}
\right.\right) & =\frac{\prod_{j=1}^{n}\left(z_{j-1}-z_{j}\right)^{\alpha_{j-1}-\alpha_{j}}}{\left(z_{0}-z_{n}\right)^{\alpha_{0}-\alpha_{n}}}\hat{B}_{\gamma}^{\bullet,\circ}\left(\left.\substack{z_{0}\\
\alpha_{0}
}
\right|\left.\substack{z_{1}\\
\alpha_{1}
}
\right|\cdots\left|\substack{z_{n}\\
\alpha_{n}
}
\right.\right)\\
 & =\frac{\mathscr{B}_{\gamma}^{\bullet,\circ}(\left.\substack{z_{0}\\
\alpha_{0}
}
\right|\left.\substack{z_{1}\\
\alpha_{1}
}
\right|\cdots\left|\substack{z_{n}\\
\alpha_{n}
}
\right.)}{\mathscr{B}_{\gamma}^{\bullet,\circ}(\left.\substack{z_{0}\\
\alpha_{0}
}
\right|\substack{z_{n}\\
\alpha_{n}
}
)}.
\end{align*}
Note that by definition, the `scripted' iterated beta integrals are
invariant under affine transformations, i.e., 
\[
\mathscr{B}_{\gamma}\left(z;\left.\substack{z_{0}\\
\alpha_{0}
}
\right|\left.\substack{z_{1}\\
\alpha_{1}
}
\right|\cdots\left|\substack{z_{n}\\
\alpha_{n}
}
\right.;z'\right)=\mathscr{B}_{\sigma(\gamma)}\left(\sigma(z);\left.\substack{\sigma(z_{0})\\
\alpha_{0}
}
\right|\left.\substack{\sigma(z_{1})\\
\alpha_{1}
}
\right|\cdots\left|\substack{\sigma(z_{n})\\
\alpha_{n}
}
\right.;\sigma(z')\right),
\]
for $\sigma(z)=az+b$ ($a\neq0$). Iterated beta integrals are a generalization
of the hyperlogarithms (see Theorem \ref{thm:relationship_with_hyperlogarithm}),
and also a generalization of the beta function $\mathrm{B}(\alpha,\beta)=\frac{\Gamma(\alpha)\Gamma(\beta)}{\Gamma(\alpha+\beta)}$
as follows:
\begin{prop}
\label{prop:is_beta_function}Let $z_{0}$ and $z_{1}$ be different
complex variables. Then
\begin{align*}
\mathscr{B}_{\gamma_{{\rm f,f}}}^{{\rm f},{\rm f}}\left(\substack{z_{0}\\
\alpha_{0}
}
\left|\substack{z_{1}\\
\alpha_{1}
}
\right.\right) & =(-1)^{1-\alpha_{0}}{\rm B}(1-\alpha_{0},\alpha_{1}),\\
\mathscr{B}_{\gamma_{{\rm f,\infty}}}^{{\rm f},\infty}\left(\substack{z_{0}\\
\alpha_{0}
}
\left|\substack{z_{1}\\
\alpha_{1}
}
\right.\right) & ={\rm B}(1-\alpha_{0},\alpha_{0}-\alpha_{1}),\\
\mathscr{B}_{\gamma_{{\rm \infty,f}}}^{\infty,{\rm f}}\left(\substack{z_{0}\\
\alpha_{0}
}
\left|\substack{z_{1}\\
\alpha_{1}
}
\right.\right) & =(-1)^{1-\alpha_{0}+\alpha_{1}}{\rm B}(\alpha_{0}-\alpha_{1},\alpha_{1})
\end{align*}
where $\gamma_{{\rm f,f}}$, $\gamma_{{\rm f,\infty}}$ and $\gamma_{{\rm \infty,f}}$
are the simple paths (here, the complex powers of $-1$ are chosen
in accordance with the chosen branches of the differential forms defining
$\mathscr{B}$ ).
\end{prop}

\begin{proof}
From the invariance of iterated beta integrals under affine transformation,
it is enough to consider the case $(z_{0},z_{1})=(1,0)$. Then
\[
\mathscr{B}^{{\rm f},{\rm f}}\left(\substack{z_{0}\\
\alpha_{0}
}
\left|\substack{z_{1}\\
\alpha_{1}
}
\right.\right)=(-1)^{-\alpha_{0}}\int_{1}^{0}\frac{dt}{(1-t)^{\alpha_{0}}t^{1-\alpha_{1}}}=(-1)^{1-\alpha_{0}}{\rm B}(1-\alpha_{0},\alpha_{1}),
\]
\begin{align*}
\mathscr{B}^{{\rm f},\infty}\left(\substack{z_{0}\\
\alpha_{0}
}
\left|\substack{z_{1}\\
\alpha_{1}
}
\right.\right) & =\int_{1}^{\infty}\frac{dt}{(t-1)^{\alpha_{0}}t^{1-\alpha_{1}}}\\
 & =\int_{0}^{1}\frac{du}{u^{\alpha_{0}}(1-u)^{1-\alpha_{0}+\alpha_{1}}}\qquad\left(t=\frac{1}{1-u}\right)\\
 & ={\rm B}(1-\alpha_{0},\alpha_{0}-\alpha_{1}),
\end{align*}
and
\begin{align*}
\mathscr{B}^{\infty,{\rm f}}\left(\substack{z_{0}\\
\alpha_{0}
}
\left|\substack{z_{1}\\
\alpha_{1}
}
\right.\right) & =(-1)^{1-\alpha_{0}+\alpha_{1}}\int_{-\infty}^{0}\frac{dt}{(1-t)^{\alpha_{0}}(-t)^{1-\alpha_{1}}}\\
 & =(-1)^{1-\alpha_{0}+\alpha_{1}}\int_{0}^{1}\frac{du}{u^{1-\alpha_{0}+\alpha_{1}}(1-u)^{1-\alpha_{1}}}\qquad\left(t=\frac{u-1}{u}\right)\\
 & =(-1)^{1-\alpha_{0}+\alpha_{1}}{\rm B}(\alpha_{0}-\alpha_{1},\alpha_{1}).\qedhere
\end{align*}
\end{proof}

\section{\label{sec:Analytic-continuation}Domain of convergence and analytic
continuation with respect to exponent parameters}

In this section, we will give the domain of convergence and the analytic
continuation of the iterated beta integrals with respect to the exponent
parameters $\alpha_{i}$.
\begin{thm}
\label{thm:convergence_and_analytic_continuation_general}Let $z_{0},\dots,z_{n}\in\mathbb{C}$
and $p,q\in\mathbb{C}\cup\{\infty\}$. Put
\[
m_{i}=1+\#\{0<j<i\mid z_{j}\neq p\}\qquad(i=1,\dots,n)
\]
and
\[
m_{i}'=1+\#\{i<j<n\mid z_{j}\neq q\}\qquad(i=0,\dots,n-1).
\]
Then, the defining integral of $B_{\gamma}(p;\left.\substack{z_{0}\\
\alpha_{0}
}
\right|\left.\substack{z_{1}\\
\alpha_{1}
}
\right|\cdots\left|\substack{z_{n}\\
\alpha_{n}
}
\right.;q)$ converges absolutely if
\begin{itemize}
\item $\delta_{p,z_{0}}\Re(\alpha_{0})-\delta_{p,z_{i}}\Re(\alpha_{i})<m_{i}$
for $i=1,\dots,n$ when $p\in\mathbb{C}$,
\item $\Re(\alpha_{0}-\alpha_{i})>0$ for $i=1,\dots,n$ when $p=\infty$,
\item $\delta_{q,z_{n}}\Re(1-\alpha_{n})-\delta_{q,z_{i}}\Re(1-\alpha_{i})<m_{i}'$
for $i=0,\dots,n-1$ when $q\in\mathbb{C}$,
\item $\Re(-\alpha_{n}+\alpha_{i})>0$ for $i=0,\dots,n-1$ when $q=\infty$.
\end{itemize}
Furthermore, as a function of $\alpha_{0},\dots,\alpha_{n}$, $B_{\gamma}(p;\left.\substack{z_{0}\\
\alpha_{0}
}
\right|\left.\substack{z_{1}\\
\alpha_{1}
}
\right|\cdots\left|\substack{z_{n}\\
\alpha_{n}
}
\right.;q)$ is holomorphically continued to the whole $\mathbb{C}^{n+1}$ except
for the possible simple poles at
\begin{itemize}
\item $\delta_{p,z_{0}}\alpha_{0}-\delta_{p,z_{i}}(\alpha_{i})\in\mathbb{Z}_{\geq m_{i}}$
for $i=1,\dots,n$ when $p\in\mathbb{C}$,
\item $\alpha_{0}-\alpha_{i}\in\mathbb{Z}_{\leq0}$ for $i=1,\dots,n$ when
$p=\infty$,
\item $\delta_{q,z_{n}}(1-\alpha_{n})-\delta_{q,z_{i}}(1-\alpha_{i})\in\mathbb{Z}_{\geq m_{i}'}$
for $i=0,\dots,n-1$ when $q\in\mathbb{C}$,
\item $(-\alpha_{n}+\alpha_{i})\in\mathbb{Z}_{\leq0}$ for $i=0,\dots,n-1$
when $q=\infty$.
\end{itemize}
\end{thm}

As a special case of Theorem \ref{thm:convergence_and_analytic_continuation_general},
we get the following.
\begin{cor}[Domain of convergence]
\label{cor:convergence_and_analytic_continuation_complete}When $z_{0},\dots,z_{n}$
are distinct, the defining integral of $B_{\gamma}^{\bullet,\circ}(\left.\substack{z_{0}\\
\alpha_{0}
}
\right|\left.\substack{z_{1}\\
\alpha_{1}
}
\right|\cdots\left|\substack{z_{n}\\
\alpha_{n}
}
\right.)$ converges absolutely if
\begin{itemize}
\item $\Re(\alpha_{0})<1$ when $\bullet=\mathrm{f}$,
\item $\Re(\alpha_{0}-\alpha_{i})>0$ for $i=1,\dots,n$ when $\bullet=\infty$,
\item $\Re(1-\alpha_{n})<1$ when $\circ=\mathrm{f}$,
\item $\Re(-\alpha_{n}+\alpha_{i})>0$ for $i=0,\dots,n-1$ when $\circ=\infty$.
\end{itemize}
Furthermore, as a function of $\alpha_{0},\dots,\alpha_{n}$, $B_{\gamma}^{\bullet,\circ}(\left.\substack{z_{0}\\
\alpha_{0}
}
\right|\left.\substack{z_{1}\\
\alpha_{1}
}
\right|\cdots\left|\substack{z_{n}\\
\alpha_{n}
}
\right.)$ is holomorphically continued to the whole $\mathbb{C}^{n+1}$ except
for the possible simple poles at
\begin{itemize}
\item $\alpha_{0}\in\mathbb{Z}_{\geq1}$ when $\bullet=\mathrm{f}$,
\item $\alpha_{0}-\alpha_{i}\in\mathbb{Z}_{\leq0}$ for $i=1,\dots,n$ when
$\bullet=\infty$,
\item $1-\alpha_{n}\in\mathbb{Z}_{\geq1}$ when $\circ=\mathrm{f}$,
\item $-\alpha_{n}+\alpha_{i}\in\mathbb{Z}_{\leq0}$ for $i=0,\dots,n-1$
when $\circ=\infty$.
\end{itemize}
\end{cor}

Furthermore, there are some cancellations of poles of the numerator
and the denominator of the normalized iterated beta integrals (by
\ref{prop:is_beta_function}), and they typically have fewer poles
as follows: 

\begin{cor}
Let $\bullet,\circ\in\{\mathrm{f},\infty\}$ be $(\bullet,\circ)\neq(\infty,\infty)$.
Assume that $\gamma$ is a simple path and $z_{0}\neq z_{n}$. When
$z_{0},\dots,z_{n}$ are distinct, as a function of $\alpha_{0},\dots,\alpha_{n}$,
$\hat{B}_{\gamma}^{\bullet,\circ}(\left.\substack{z_{0}\\
\alpha_{0}
}
\right|\left.\substack{z_{1}\\
\alpha_{1}
}
\right|\cdots\left|\substack{z_{n}\\
\alpha_{n}
}
\right.)$ is meromorphically continued to the whole $\mathbb{C}^{n+1}$ with
the following possible poles
\begin{itemize}
\item $\alpha_{0}-\alpha_{i}\in\mathbb{Z}_{\leq0}$ for $i=1,\dots,n-1$
when $\bullet=\infty$,
\item $-\alpha_{n}+\alpha_{i}\in\mathbb{Z}_{\leq0}$ for $i=1,\dots,n-1$
when $\circ=\infty$.
\end{itemize}
Especially, $\hat{B}_{\gamma}^{\mathrm{f},\mathrm{f}}(\left.\substack{z_{0}\\
\alpha_{0}
}
\right|\left.\substack{z_{1}\\
\alpha_{1}
}
\right|\cdots\left|\substack{z_{n}\\
\alpha_{n}
}
\right.)$ is entire.
\end{cor}

Theorem \ref{thm:convergence_and_analytic_continuation_general} is
an immediate consequence of the following more general statement:
\begin{lem}
\label{lem:Analytic-continuation-General}Let $x>0$ and $g\left(\substack{t_{1},\dots,t_{n}\\
\beta_{1},\dots,\beta_{n}
}
\right)$ be a holomorphic function on $(t_{1},\dots,t_{n},\beta_{1},\dots,\beta_{n})\in U^{n}\times\mathbb{C}^{n}$
where $U$ is a domain including the closed interval $[0,x]$. Then
the integral
\[
f(\beta_{1},\dots,\beta_{n})=\int_{0<t_{1}<\cdots<t_{n}<x}g\left(\substack{t_{1},\dots,t_{n}\\
\beta_{1},\dots,\beta_{n}
}
\right)\prod_{j=1}^{n}t_{j}^{\beta_{j}-1}dt_{j}
\]
converges if $\Re(\beta_{1}+\cdots+\beta_{i})>0$ for $i=1,\dots,n$.
Furthermore, $f(\beta_{1},\dots,\beta_{n})$ is meromorphically continued
to the whole $\mathbb{C}^{n}$ with possible poles at
\[
\beta_{1}+\cdots+\beta_{i}\in\mathbb{Z}_{\leq0}\qquad(i=1,\dots,n),
\]
all of which are simple poles, and for $0<j_{1}<j_{2}\leq n$ and
$m_{1},m_{2}\in\mathbb{Z}_{\leq0}$ 
\[
\mathrm{Res}_{\beta_{1}+\cdots+\beta_{j_{2}}=m_{2}}\mathrm{Res}_{\beta_{1}+\cdots+\beta_{j_{1}}=m_{1}}f(\beta_{1},\dots,\beta_{n})=0
\]
except for the case $m_{1}\leq m_{2}$.   
\end{lem}

\begin{proof}
The convergence is easy. Let $\omega\left(\substack{t_{1},\dots,t_{n}\\
\beta_{1},\dots,\beta_{n}
}
\right)\coloneqq\prod_{j=1}^{n}t_{j}^{\beta_{j}-1}dt_{j}$. For $n\geq1$, we have
\begin{align*}
 & \int_{0<t_{1}<\cdots<t_{n}<x}g\left(\substack{t_{1},\dots,t_{n}\\
\beta_{1},\dots,\beta_{n}
}
\right)\omega\left(\substack{t_{1},\dots,t_{n}\\
\beta_{1},\dots,\beta_{n}
}
\right)\\
 & =\int_{0<t_{2}<\cdots<t_{n}<x}\left(\int_{0}^{t_{2}}g\left(\substack{t_{1},\dots,t_{n}\\
\beta_{1},\dots,\beta_{n}
}
\right)t_{1}^{\beta_{1}}\frac{dt_{1}}{t_{1}}\right)\omega\left(\substack{t_{2},\dots,t_{n}\\
\beta_{2},\dots,\beta_{n}
}
\right)\\
 & =\frac{1}{\beta_{1}}\left(\int_{0<t_{2}<\cdots<t_{n}<x}\left[g\left(\substack{t_{1},\dots,t_{n}\\
\beta_{1},\dots,\beta_{n}
}
\right)t_{1}^{\beta_{1}}\right]_{0}^{t_{2}}\omega\left(\substack{t_{2},\dots,t_{n}\\
\beta_{2},\dots,\beta_{n}
}
\right)-\int_{0<t_{1}<t_{2}<\cdots<t_{n}<x}\frac{\partial g}{\partial t_{1}}\left(\substack{t_{1},\dots,t_{n}\\
\beta_{1},\dots,\beta_{n}
}
\right)\omega\left(\substack{t_{1},t_{2},\dots,t_{n}\\
\beta_{1}+1,\beta_{2},\dots,\beta_{n}
}
\right)\right)\\
 & =\frac{1}{\beta_{1}}\int_{0<t_{2}<\cdots<t_{n}<x}g\left(\substack{t_{2},t_{2},\dots,t_{n}\\
\beta_{1},\beta_{2},\dots,\beta_{n}
}
\right)\omega\left(\substack{t_{2},t_{3},\dots,t_{n}\\
\beta_{1}+\beta_{2},\beta_{3},\dots,\beta_{n}
}
\right)-\frac{1}{\beta_{1}}\int_{0<t_{1}<t_{2}<\cdots<t_{n}<x}\frac{\partial g}{\partial t_{1}}\left(\substack{t_{1},\dots,t_{n}\\
\beta_{1},\dots,\beta_{n}
}
\right)\omega\left(\substack{t_{1},t_{2},\dots,t_{n}\\
\beta_{1}+1,\beta_{2},\dots,\beta_{n}
}
\right).
\end{align*}
This also holds for $n=1$ if we understand the first term as $\frac{1}{\beta_{1}}g\left(\substack{x\\
\beta_{1}
}
\right)x^{\beta_{1}}$. The analytic continuation and the locations of possible poles follows
from this expression. 
\end{proof}
\begin{proof}[Proof of Theorem \ref{thm:convergence_and_analytic_continuation_general}]
Let $0<t<t'<1$, and $u=\gamma(t),v=\gamma(t')$ the two corresponding
points on the path $\gamma:[0,1]\rightarrow\mathbb{C}\cup\{\infty\}$.
Decompose $\gamma=\gamma_{p,u}\gamma_{u,v}\gamma_{v,q}$ where $\gamma_{x,y}$
denotes the subpath from $x$ to $y$. By the path composition formula,
we have
\begin{align*}
 & B_{\gamma}(p;\left.\substack{z_{0}\\
\alpha_{0}
}
\right|\left.\substack{z_{1}\\
\alpha_{1}
}
\right|\cdots\left|\substack{z_{n}\\
\alpha_{n}
}
\right.;q)\\
 & =\sum_{0\leq i\leq j\leq n}B_{\gamma_{p,u}}(p;\left.\substack{z_{0}\\
\alpha_{0}
}
\right|\left.\substack{z_{1}\\
\alpha_{1}
}
\right|\cdots\left|\substack{z_{i}\\
\alpha_{i}
}
\right.;u)B_{\gamma_{u,v}}(u;\left.\substack{z_{i}\\
\alpha_{i}
}
\right|\left.\substack{z_{i+1}\\
\alpha_{i+1}
}
\right|\cdots\left|\substack{z_{j}\\
\alpha_{j}
}
\right.;v)B_{\gamma_{v,q}}(v;\left.\substack{z_{j}\\
\alpha_{j}
}
\right|\left.\substack{z_{j+1}\\
\alpha_{j+1}
}
\right|\cdots\left|\substack{z_{n}\\
\alpha_{n}
}
\right.;q).
\end{align*}
The middle factor $B_{\gamma_{u,v}}$ is entire in $\alpha_{i},\ldots,\alpha_{j}$,
since its integrand has no singularities along $\gamma_{u,v}$ including
at the endpoints $u,v$. Hence, possible poles can only come from
the first and third factors.

Let us first discuss the first factor. Taking sufficiently small $t$,
$\gamma_{p,u}$ is homotopic to the straight line path from $p$ to
$u$ (resp. the image of straight path from $0$ to $1/u$ under the
inversion $z\mapsto1/z$) when $p$ is finite (resp. infinite). Furthermore,
via an affine transformation (resp. an affine transformation after
the inversion) when $p$ is finite (resp. infinite), $B_{\gamma_{p,u}}$
becomes integrals on a real segment, which fit the conditions of Lemma
\ref{lem:Analytic-continuation-General}. The same argument applies
to the third factor. This gives the domain of convergence as well
as the locations of the poles as stated in Theorem \ref{thm:convergence_and_analytic_continuation_general}.
\end{proof}

\section{Special values and connection to hyperlogarithms\label{sec:Special-values-and-hyperlogarithms}}

In this section, we will give some special values of iterated beta
integrals and establish a simple connection between iterated beta
integrals and hyperlogarithms. By investigating the poles of the complete
iterated beta integrals, we can derive the following evaluation formulas
for finite endpoint case:
\begin{thm}[Special values: finite endpoint case]
\label{thm: complete vs incomplete}Let $z_{0},\dots,z_{n}$ be distinct
complex numbers. Then, 
\begin{enumerate}
\item The residue of $B_{\gamma}^{\mathrm{f}}\left(\left.\substack{z_{0}\\
\alpha_{0}
}
\right|\cdots\left|\substack{z_{n}\\
\alpha_{n}
}
\right.;q\right)$ at the pole $\alpha_{0}=1$ is given by
\[
-(z_{0}-z_{1})^{\alpha_{1}-1}B_{\gamma}\left(z_{0};\left.\substack{z_{1}\\
\alpha_{1}
}
\right|\cdots\left|\substack{z_{n}\\
\alpha_{n}
}
\right.;q\right).
\]
Equivalently, the residue of $B_{\gamma}^{\mathrm{f}}\left(p;\left.\substack{z_{0}\\
\alpha_{0}
}
\right|\cdots\left|\substack{z_{n}\\
\alpha_{n}
}
\right.\right)$ at the pole $\alpha_{n}=0$ is given by
\[
(z_{n}-z_{n-1})^{-\alpha_{n-1}}B_{\gamma}\left(p;\left.\substack{z_{0}\\
\alpha_{0}
}
\right|\cdots\left|\substack{z_{n-1}\\
\alpha_{n-1}
}
\right.;z_{n}\right).
\]
\item $\hat{B}_{\gamma}^{\mathrm{f},\circ}\left(\left.\substack{z_{0}\\
\alpha_{0}
}
\right|\cdots\left|\substack{z_{n}\\
\alpha_{n}
}
\right.\right)$ is holomorphic at $\alpha_{0}=1$, and we have
\[
\hat{B}_{\gamma}^{\mathrm{f},\circ}\left(\left.\substack{z_{0}\\
1
}
\right|\left.\substack{z_{1}\\
\alpha_{1}
}
\right|\cdots\left|\substack{z_{n}\\
\alpha_{n}
}
\right.\right)=\frac{(z_{0}-z_{n})^{1-\alpha_{n}}}{(z_{0}-z_{1})^{1-\alpha_{1}}}B_{\gamma}^{\circ}\left(z_{0};\left.\substack{z_{1}\\
\alpha_{1}
}
\right|\cdots\left|\substack{z_{n}\\
\alpha_{n}
}
\right.\right).
\]
Equivalently, $\hat{B}_{\gamma}^{\bullet,\mathrm{f}}\left(\left.\substack{z_{0}\\
\alpha_{0}
}
\right|\cdots\left|\substack{z_{n}\\
\alpha_{n}
}
\right.\right)$ is holomorphic at $\alpha_{n}=0$, and we have
\begin{align*}
\hat{B}_{\gamma}^{\bullet,\mathrm{f}}\left(\left.\substack{z_{0}\\
\alpha_{0}
}
\right|\cdots\left|\substack{z_{n-1}\\
\alpha_{n-1}
}
\right.\left|\substack{z_{n}\\
0
}
\right.\right) & =\frac{(z_{n}-z_{0})^{\alpha_{0}}}{(z_{n}-z_{n-1})^{\alpha_{n-1}}}B_{\gamma}^{\bullet}\left(\left.\substack{z_{0}\\
\alpha_{0}
}
\right|\cdots\left|\substack{z_{n-1}\\
\alpha_{n-1}
}
\right.;z_{n}\right).
\end{align*}
\item $\hat{\mathscr{B}}_{\gamma}^{\mathrm{f},\circ}\left(\left.\substack{z_{0}\\
\alpha_{0}
}
\right|\cdots\left|\substack{z_{n}\\
\alpha_{n}
}
\right.\right)$ is holomorphic at $\alpha_{0}=1$, and we have
\[
\hat{\mathscr{B}}_{\gamma}^{\mathrm{f},\circ}(\left.\substack{z_{0}\\
1
}
\right|\left.\substack{z_{1}\\
\alpha_{1}
}
\right|\cdots\left|\substack{z_{n}\\
\alpha_{n}
}
\right.)=\mathscr{B}_{\gamma}^{\circ}(z_{0};\left.\substack{z_{1}\\
\alpha_{1}
}
\right|\cdots\left|\substack{z_{n}\\
\alpha_{n}
}
\right.),
\]
Equivalently, $\hat{\mathscr{B}}_{\gamma}^{\bullet,\mathrm{f}}\left(\left.\substack{z_{0}\\
\alpha_{0}
}
\right|\cdots\left|\substack{z_{n}\\
\alpha_{n}
}
\right.\right)$ is holomorphic at $\alpha_{n}=0$, and we have
\begin{align*}
\hat{\mathscr{B}}_{\gamma}^{\bullet,\mathrm{f}}\left(\left.\substack{z_{0}\\
\alpha_{0}
}
\right|\cdots\left|\substack{z_{n-1}\\
\alpha_{n-1}
}
\right.\left|\substack{z_{n}\\
0
}
\right.\right) & =\mathscr{B}_{\gamma}^{\bullet}\left(\left.\substack{z_{0}\\
\alpha_{0}
}
\right|\cdots\left|\substack{z_{n-1}\\
\alpha_{n-1}
}
\right.;z_{n}\right).
\end{align*}
\end{enumerate}
\end{thm}

\begin{proof}
Notice first that the formulas in (2) and (3) follow immediately from
the corresponding formulas of (1). Also, by the symmetry $B_{\gamma}\left(p;\left.\substack{z_{0}\\
\alpha_{0}
}
\right|\cdots\left|\substack{z_{n}\\
\alpha_{n}
}
\right.;q\right)=B_{\gamma^{-1}}\left(q;\left.\substack{z_{n}\\
1-\alpha_{n}
}
\right|\cdots\left|\substack{z_{0}\\
1-\alpha_{0}
}
\right.;p\right),$ the two formulas in (1) are equivalent. By analytic continuation,
it is enough to prove the claims when $\alpha_{i}$'s lie in the domain
for which the considered integral converges. Note that
\[
B_{\gamma}^{\mathrm{f}}\left(\left.\substack{z_{0}\\
\alpha_{0}
}
\right|\cdots\left|\substack{z_{n}\\
\alpha_{n}
}
\right.;q\right)=I_{\gamma}\left(z_{0};(t-z_{0})^{-\alpha_{0}}(t-z_{1})^{\alpha_{1}-1}dt,\left[\substack{z_{1},z_{2}\\
\alpha_{1},\alpha_{2}
}
\right],\dots,\left[\substack{z_{n-1},z_{n}\\
\alpha_{n-1},\alpha_{n}
}
\right];q\right).
\]
Here,
\begin{align*}
\lim_{\alpha_{0}\to1}(1-\alpha_{0})\int_{z_{0}}^{x}(t-z_{0})^{-\alpha_{0}}(t-z_{1})^{\alpha_{1}-1}dt & =\lim_{\alpha_{0}\to1}\int_{z_{0}}^{x}d\left((t-z_{0})^{1-\alpha_{0}}\right)(t-z_{1})^{\alpha_{1}-1}\\
 & =\lim_{\alpha_{0}\to1}\left((x-z_{0})^{1-\alpha_{0}}(x-z_{1})^{\alpha_{1}-1}-\int_{z_{0}}^{x}(t-z_{0})^{1-\alpha_{0}}d\left((t-z_{1})^{\alpha_{1}-1}\right)\right)\\
 & =(x-z_{1})^{\alpha_{1}-1}-\lim_{\alpha_{0}\to1}\int_{z_{0}}^{x}d\left((t-z_{1})^{\alpha_{1}-1}\right)\\
 & =(z_{0}-z_{1})^{\alpha_{1}-1},
\end{align*}
and thus,
\begin{align*}
 & \lim_{\alpha_{0}\to1}(\alpha_{0}-1)B_{\gamma}^{\mathrm{f}}\left(\left.\substack{z_{0}\\
\alpha_{0}
}
\right|\cdots\left|\substack{z_{n}\\
\alpha_{n}
}
\right.;q\right)\\
 & =-\lim_{\alpha_{0}\to1}(1-\alpha_{0})I_{\gamma}\left(z_{0};(t-z_{0})^{-\alpha_{0}}(t-z_{1})^{\alpha_{1}-1}dt,\left[\substack{z_{1},z_{2}\\
\alpha_{1},\alpha_{2}
}
\right],\dots,\left[\substack{z_{n-1},z_{n}\\
\alpha_{n-1},\alpha_{n}
}
\right];q\right)\\
 & =-(z_{0}-z_{1})^{\alpha_{1}-1}B_{\gamma}\left(z_{0};\left.\substack{z_{1}\\
\alpha_{1}
}
\right|\cdots\left|\substack{z_{n}\\
\alpha_{n}
}
\right.;q\right).
\end{align*}
\end{proof}
Furthermore, the double residues (double limit) of the complete iterated
beta integrals are given by the following theorem.

\begin{thm}
\label{thm:double_limit}Let $z_{0},\dots,z_{n}$ be distinct complex
numbers. Then, we have
\[
\lim_{\alpha_{0}\to1}\lim_{\alpha_{n}\to0}(\alpha_{0}-1)\alpha_{n}B_{\gamma}^{\mathrm{f},\mathrm{f}}\left(\left.\substack{z_{0}\\
\alpha_{0}
}
\right|\cdots\left|\substack{z_{n}\\
\alpha_{n}
}
\right.\right)=-(z_{0}-z_{1})^{\alpha_{1}-1}(z_{n}-z_{n-1})^{-\alpha_{n-1}}B_{\gamma}\left(z_{0};\left.\substack{z_{1}\\
\alpha_{1}
}
\right|\cdots\left|\substack{z_{n-1}\\
\alpha_{n-1}
}
\right.;z_{n}\right),
\]
\[
\lim_{\alpha_{0}\to1}\lim_{\alpha_{n}\to0}(\alpha_{0}-\alpha_{n}-1)\hat{B}_{\gamma}^{\mathrm{f},\mathrm{f}}\left(\left.\substack{z_{0}\\
\alpha_{0}
}
\right|\cdots\left|\substack{z_{n}\\
\alpha_{n}
}
\right.\right)=\frac{z_{0}-z_{n}}{(z_{0}-z_{1})^{1-\alpha_{1}}(z_{n}-z_{n-1})^{\alpha_{n-1}}}B_{\gamma}\left(z_{0};\left.\substack{z_{1}\\
\alpha_{1}
}
\right|\cdots\left|\substack{z_{n-1}\\
\alpha_{n-1}
}
\right.;z_{n}\right),
\]
and
\[
\lim_{\alpha_{0}\to1}\lim_{\alpha_{n}\to0}(\alpha_{0}-\alpha_{n}-1)\hat{\mathscr{B}}_{\gamma}^{\mathrm{f},\mathrm{f}}\left(\left.\substack{z_{0}\\
\alpha_{0}
}
\right|\cdots\left|\substack{z_{n}\\
\alpha_{n}
}
\right.\right)=(-1)^{\alpha_{n-1}}\mathscr{B}_{\gamma}\left(z_{0};\left.\substack{z_{1}\\
\alpha_{1}
}
\right|\cdots\left|\substack{z_{n-1}\\
\alpha_{n-1}
}
\right.;z_{n}\right).
\]
\end{thm}

\begin{proof}
By Theorem \ref{thm: complete vs incomplete}, we have
\begin{align*}
 & \lim_{\alpha_{0}\to1}\lim_{\alpha_{n}\to0}(\alpha_{0}-1)\alpha_{n}B_{\gamma}^{\mathrm{f},\mathrm{f}}\left(\left.\substack{z_{0}\\
\alpha_{0}
}
\right|\cdots\left|\substack{z_{n}\\
\alpha_{n}
}
\right.\right)\\
 & =(z_{n}-z_{n-1})^{-\alpha_{n-1}}\lim_{\alpha_{0}\to1}(\alpha_{0}-1)B_{\gamma}^{\mathrm{f}}\left(\left.\substack{z_{0}\\
\alpha_{0}
}
\right|\cdots\left|\substack{z_{n-1}\\
\alpha_{n-1}
}
\right.;z_{n}\right).
\end{align*}
Furthermore, by Theorem \ref{thm: complete vs incomplete}, we have
\[
\lim_{\alpha_{0}\to1}(\alpha_{0}-1)B_{\gamma}^{\mathrm{f}}\left(\left.\substack{z_{0}\\
\alpha_{0}
}
\right|\cdots\left|\substack{z_{n-1}\\
\alpha_{n-1}
}
\right.;z_{n}\right)=-(z_{0}-z_{1})^{\alpha_{1}-1}B_{\gamma}\left(z_{0};\left.\substack{z_{1}\\
\alpha_{1}
}
\right|\cdots\left|\substack{z_{n-1}\\
\alpha_{n-1}
}
\right.;z_{n}\right),
\]
and thus, it follows that
\[
\lim_{\alpha_{0}\to1}\lim_{\alpha_{n}\to0}(\alpha_{0}-1)\alpha_{n}B_{\gamma}^{\mathrm{f},\mathrm{f}}\left(\left.\substack{z_{0}\\
\alpha_{0}
}
\right|\cdots\left|\substack{z_{n}\\
\alpha_{n}
}
\right.\right)=-(z_{0}-z_{1})^{\alpha_{1}-1}(z_{n}-z_{n-1})^{-\alpha_{n-1}}B_{\gamma}\left(z_{0};\left.\substack{z_{1}\\
\alpha_{1}
}
\right|\cdots\left|\substack{z_{n-1}\\
\alpha_{n-1}
}
\right.;z_{n}\right).
\]
Similarly, we have
\begin{align*}
 & \lim_{\alpha_{0}\to1}\lim_{\alpha_{n}\to0}(\alpha_{0}-\alpha_{n}-1)\hat{B}_{\gamma}^{\mathrm{f},\mathrm{f}}\left(\left.\substack{z_{0}\\
\alpha_{0}
}
\right|\cdots\left|\substack{z_{n}\\
\alpha_{n}
}
\right.\right)\\
 & =\lim_{\alpha_{0}\to1}(\alpha_{0}-1)\frac{(z_{n}-z_{0})^{\alpha_{0}}}{(z_{n}-z_{n-1})^{\alpha_{n-1}}}B_{\gamma}^{\mathrm{f}}\left(\left.\substack{z_{0}\\
\alpha_{0}
}
\right|\cdots\left|\substack{z_{n-1}\\
\alpha_{n-1}
}
\right.;z_{n}\right)\\
 & =\frac{z_{0}-z_{n}}{(z_{0}-z_{1})^{1-\alpha_{1}}(z_{n}-z_{n-1})^{\alpha_{n-1}}}B_{\gamma}\left(z_{0};\left.\substack{z_{1}\\
\alpha_{1}
}
\right|\cdots\left|\substack{z_{n-1}\\
\alpha_{n-1}
}
\right.;z_{n}\right)
\end{align*}
and
\begin{align*}
 & \lim_{\alpha_{0}\to1}\lim_{\alpha_{n}\to0}(\alpha_{0}-\alpha_{n}-1)\hat{\mathscr{B}}_{\gamma}^{\mathrm{f},\mathrm{f}}\left(\left.\substack{z_{0}\\
\alpha_{0}
}
\right|\cdots\left|\substack{z_{n}\\
\alpha_{n}
}
\right.\right)\\
 & =\lim_{\alpha_{0}\to1}\lim_{\alpha_{n}\to0}(\alpha_{0}-\alpha_{n}-1)\frac{\prod_{j=1}^{n}\left(z_{j-1}-z_{j}\right)^{\alpha_{j-1}-\alpha_{j}}}{\left(z_{0}-z_{n}\right)^{\alpha_{0}-\alpha_{n}}}\hat{B}_{\gamma}^{\mathrm{f},\mathrm{f}}\left(\left.\substack{z_{0}\\
\alpha_{0}
}
\right|\left.\substack{z_{1}\\
\alpha_{1}
}
\right|\cdots\left|\substack{z_{n}\\
\alpha_{n}
}
\right.\right)\\
 & =(-1)^{\alpha_{n-1}}\prod_{j=2}^{n-1}\left(z_{j-1}-z_{j}\right)^{\alpha_{j-1}-\alpha_{j}}B_{\gamma}\left(z_{0};\left.\substack{z_{1}\\
\alpha_{1}
}
\right|\cdots\left|\substack{z_{n-1}\\
\alpha_{n-1}
}
\right.;z_{n}\right)\\
 & =(-1)^{\alpha_{n-1}}\mathscr{B}_{\gamma}\left(z_{0};\left.\substack{z_{1}\\
\alpha_{1}
}
\right|\cdots\left|\substack{z_{n-1}\\
\alpha_{n-1}
}
\right.;z_{n}\right),
\end{align*}
which completes the proof.
\end{proof}
Notice that Theorem \ref{thm: complete vs incomplete} can also be
viewed as a relationship between complete and incomplete iterated
beta integrals. As a particular instance, it yields the following
relationship between complete iterated beta integrals and hyperlogarithms.
\begin{thm}[Relationship with hyperlogarithms]
\label{thm:relationship_with_hyperlogarithm}Let $z_{0},\dots,z_{n}$
be distinct complex numbers. We have the following:
\begin{enumerate}
\item In terms of $B$,
\[
\lim_{\alpha_{0},\ldots,\alpha_{n}\rightarrow1}\hat{B}_{\gamma}^{\mathrm{f},\mathrm{f}}\left(\left.\substack{z_{0}\\
\alpha_{0}
}
\right|\cdots\left|\substack{z_{n}\\
\alpha_{n}
}
\right.\right)=\lim_{\alpha_{0},\ldots,\alpha_{n}\rightarrow0}\hat{B}_{\gamma}^{\mathrm{f},\mathrm{f}}\left(\left.\substack{z_{0}\\
\alpha_{0}
}
\right|\cdots\left|\substack{z_{n}\\
\alpha_{n}
}
\right.\right)=I_{\gamma}\left(z_{0};e_{z_{1}}\cdots e_{z_{n-1}};z_{n}\right)
\]
and
\[
\lim_{\alpha_{0},\ldots,\alpha_{n}\rightarrow0}\hat{B}_{\gamma}^{\infty,\mathrm{f}}\left(\left.\substack{z_{0}\\
\alpha_{0}+1
}
\right|\left.\substack{z_{1}\\
\alpha_{1}
}
\right|\cdots\left|\substack{z_{n}\\
\alpha_{n}
}
\right.\right)=\frac{z_{n}-z_{0}}{z_{1}-z_{0}}I_{\gamma}\left(\infty;(e_{z_{0}}-e_{z_{1}})e_{z_{2}}\cdots e_{z_{n-1}};z_{n}\right).
\]
\item In terms of $\mathscr{B}$,
\[
\lim_{\alpha_{0},\ldots,\alpha_{n}\rightarrow1}\hat{\mathscr{B}}_{\gamma}^{\mathrm{f},\mathrm{f}}\left(\left.\substack{z_{0}\\
\alpha_{0}
}
\right|\cdots\left|\substack{z_{n}\\
\alpha_{n}
}
\right.\right)=\lim_{\alpha_{0},\ldots,\alpha_{n}\rightarrow0}\hat{\mathscr{B}}_{\gamma}^{\mathrm{f},\mathrm{f}}\left(\left.\substack{z_{0}\\
\alpha_{0}
}
\right|\cdots\left|\substack{z_{n}\\
\alpha_{n}
}
\right.\right)=I_{\gamma}\left(z_{0};e_{z_{1}}\cdots e_{z_{n-1}};z_{n}\right)
\]
and
\[
\lim_{\alpha_{0},\ldots,\alpha_{n}\rightarrow0}\hat{\mathscr{B}}_{\gamma}^{\infty,\mathrm{f}}\left(\left.\substack{z_{0}\\
\alpha_{0}+1
}
\right|\left.\substack{z_{1}\\
\alpha_{1}
}
\right|\cdots\left|\substack{z_{n}\\
\alpha_{n}
}
\right.\right)=I_{\gamma}\left(\infty;(e_{z_{0}}-e_{z_{1}})e_{z_{2}}\cdots e_{z_{n-1}};z_{n}\right).
\]
\end{enumerate}
\end{thm}

\begin{proof}
Note that (2) is an immediate consequence of (1), so we only prove
(1). By Theorem \ref{thm: complete vs incomplete}, we have
\begin{align*}
\hat{B}_{\gamma}^{\mathrm{f},\mathrm{f}}\left(\left.\substack{z_{0}\\
\alpha_{0}
}
\right|\cdots\left|\substack{z_{n-1}\\
\alpha_{n-1}
}
\right.\left|\substack{z_{n}\\
0
}
\right.\right) & =\frac{(z_{n}-z_{0})^{\alpha_{0}}}{(z_{n}-z_{n-1})^{\alpha_{n-1}}}B_{\gamma}^{\mathrm{f}}\left(\left.\substack{z_{0}\\
\alpha_{0}
}
\right|\cdots\left|\substack{z_{n-1}\\
\alpha_{n-1}
}
\right.;z_{n}\right)
\end{align*}
and
\[
\hat{B}_{\gamma}^{\infty,\mathrm{f}}\left(\left.\substack{z_{0}\\
\alpha_{0}+1
}
\right|\left.\substack{z_{1}\\
\alpha_{1}
}
\right|\cdots\left|\substack{z_{n-1}\\
\alpha_{n-1}
}
\right.\left|\substack{z_{n}\\
0
}
\right.\right)=\frac{(z_{n}-z_{0})^{\alpha_{0}+1}}{(z_{n}-z_{n-1})^{\alpha_{n-1}}}B_{\gamma}^{\infty}\left(\left.\substack{z_{0}\\
\alpha_{0}+1
}
\right|\left.\substack{z_{1}\\
\alpha_{1}
}
\right|\cdots\left|\substack{z_{n-1}\\
\alpha_{n-1}
}
\right.;z_{n}\right).
\]
It follows that
\begin{align*}
\lim_{\alpha_{0},\ldots,\alpha_{n}\rightarrow0}\hat{B}_{\gamma}^{\mathrm{f},\mathrm{f}}\left(\left.\substack{z_{0}\\
\alpha_{0}
}
\right|\cdots\left|\substack{z_{n-1}\\
\alpha_{n-1}
}
\right.\left|\substack{z_{n}\\
\alpha_{n}
}
\right.\right) & =B_{\gamma}^{\mathrm{f}}\left(\left.\substack{z_{0}\\
0
}
\right|\left.\substack{z_{1}\\
0
}
\right|\cdots\left|\substack{z_{n-1}\\
0
}
\right.;z_{n}\right)\\
 & =I_{\gamma}\left(z_{0};e_{z_{1}}e_{z_{2}}\cdots e_{z_{n-1}};z_{n}\right)
\end{align*}
and
\begin{align*}
\lim_{\alpha_{0},\ldots,\alpha_{n}\rightarrow0}\hat{B}_{\gamma}^{\infty,\mathrm{f}}\left(\left.\substack{z_{0}\\
\alpha_{0}+1
}
\right|\left.\substack{z_{1}\\
\alpha_{1}
}
\right|\cdots\left|\substack{z_{n}\\
\alpha_{n}
}
\right.\right) & =(z_{n}-z_{0})B_{\gamma}^{\infty}\left(\left.\substack{z_{0}\\
1
}
\right|\left.\substack{z_{1}\\
0
}
\right|\cdots\left|\substack{z_{n-1}\\
0
}
\right.;z_{n}\right)\\
 & =\frac{z_{n}-z_{0}}{z_{1}-z_{0}}I_{\gamma}\left(\infty;(e_{z_{0}}-e_{z_{1}})e_{z_{2}}\cdots e_{z_{n-1}};z_{n}\right),
\end{align*}
respectively.
\end{proof}
In a similar manner, by investigating the poles at $\alpha_{0}=\alpha_{1}$
of $B_{\gamma}^{\infty,\circ}\left(\left.\substack{z_{0}\\
\alpha_{0}
}
\right|\cdots\left|\substack{z_{n}\\
\alpha_{n}
}
\right.\right)$, we may also get the following:
\begin{thm}[Special values: infinite endpoint case]
\label{prop: residue formulas}Let $n\geq1$ and $z_{0},\dots,z_{n}$
be distinct complex numbers. 
\begin{enumerate}
\item The residue of $B_{\gamma}^{\infty}\left(\left.\substack{z_{0}\\
\alpha_{0}
}
\right|\cdots\left|\substack{z_{n}\\
\alpha_{n}
}
\right.;q\right)$ at the pole $\alpha_{0}=\alpha_{1}$ is given by
\[
-B_{\gamma}^{\infty}\left(\left.\substack{z_{1}\\
\alpha_{1}
}
\right|\cdots\left|\substack{z_{n}\\
\alpha_{n}
}
\right.;q\right).
\]
Equivalently, the residue of $B_{\gamma}^{\infty}\left(p;\left.\substack{z_{0}\\
\alpha_{0}
}
\right|\cdots\left|\substack{z_{n}\\
\alpha_{n}
}
\right.\right)$ at the pole $\alpha_{n}=\alpha_{n-1}$ is given by
\[
B_{\gamma}^{\infty}\left(p;\left.\substack{z_{0}\\
\alpha_{0}
}
\right|\cdots\left|\substack{z_{n-1}\\
\alpha_{n-1}
}
\right.\right).
\]
\end{enumerate}
\end{thm}

\begin{proof}
Again by symmetry, the two formulas are equivalent. By the meromorphy
of the function $B_{\gamma}^{\infty}\left(\left.\substack{z_{0}\\
\alpha_{0}
}
\right|\cdots\left|\substack{z_{n}\\
\alpha_{n}
}
\right.;q\right)$, we may assume $\Re(\alpha_{0})>\Re(\alpha_{1})>\cdots>\Re(\alpha_{n})$.
Note that
\begin{align*}
B_{\gamma}^{\infty}\left(\left.\substack{z_{0}\\
\alpha_{0}
}
\right|\cdots\left|\substack{z_{n}\\
\alpha_{n}
}
\right.;q\right) & =I_{\gamma}\left(\infty;(t-z_{0})^{-\alpha_{0}}(t-z_{1})^{\alpha_{1}-1}dt,\left[\substack{z_{1},z_{2}\\
\alpha_{1},\alpha_{2}
}
\right],\dots,\left[\substack{z_{n-1},z_{n}\\
\alpha_{n-1},\alpha_{n}
}
\right];q\right)\\
 & =I_{\gamma}\left(\infty;(t-z_{0})^{-\alpha_{0}+\alpha_{1}-1}g(t)dt;q\right)
\end{align*}
where we put
\[
g(t)=\left(\frac{t-z_{1}}{t-z_{0}}\right)^{\alpha_{1}-1}I_{\gamma'}\left(t;\left[\substack{z_{1},z_{2}\\
\alpha_{1},\alpha_{2}
}
\right],\dots,\left[\substack{z_{n-1},z_{n}\\
\alpha_{n-1},\alpha_{n}
}
\right];q\right).
\]
Here, $t$ varies along the path $\gamma$, and $\gamma'$ is the
part of $\gamma$ from $t$ to $q$. We let $\gamma''$ be the part
of $\gamma$ from $\infty$ to $t$. By the path composition formula,
we have
\begin{align*}
I_{\gamma'}\left(t;\left[\substack{z_{1},z_{2}\\
\alpha_{1},\alpha_{2}
}
\right],\dots,\left[\substack{z_{n-1},z_{n}\\
\alpha_{n-1},\alpha_{n}
}
\right];q\right) & =\sum_{i=1}^{n}I_{(\gamma'')^{-1}}\left(t;\left[\substack{z_{1},z_{2}\\
\alpha_{1},\alpha_{2}
}
\right],\dots,\left[\substack{z_{i-1},z_{i}\\
\alpha_{i-1},\alpha_{i}
}
\right];\infty\right)I_{\gamma}\left(\infty;\left[\substack{z_{i},z_{i+1}\\
\alpha_{i},\alpha_{i+1}
}
\right],\dots,\left[\substack{z_{n-1},z_{n}\\
\alpha_{n-1},\alpha_{n}
}
\right];q\right)\\
 & =I_{\gamma}\left(\infty;\left[\substack{z_{1},z_{2}\\
\alpha_{1},\alpha_{2}
}
\right],\dots,\left[\substack{z_{n-1},z_{n}\\
\alpha_{n-1},\alpha_{n}
}
\right];q\right)\\
 & \qquad+\sum_{i=2}^{n}I_{(\gamma'')^{-1}}\left(t;\left[\substack{z_{1},z_{2}\\
\alpha_{1},\alpha_{2}
}
\right],\dots,\left[\substack{z_{i-1},z_{i}\\
\alpha_{i-1},\alpha_{i}
}
\right];\infty\right)I_{\gamma}\left(\infty;\left[\substack{z_{i},z_{i+1}\\
\alpha_{i},\alpha_{i+1}
}
\right],\dots,\left[\substack{z_{n-1},z_{n}\\
\alpha_{n-1},\alpha_{n}
}
\right];q\right).
\end{align*}
Now we want to show 
\begin{equation}
I_{(\gamma'')^{-1}}\left(t;\left[\substack{z_{1},z_{2}\\
\alpha_{1},\alpha_{2}
}
\right],\dots,\left[\substack{z_{i-1},z_{i}\\
\alpha_{i-1},\alpha_{i}
}
\right];\infty\right)=O(|t|^{\Re(\alpha_{i}-\alpha_{1})})\quad\left(t\rightarrow\infty\right).\label{eq:estimate_I}
\end{equation}
Let $\mathrm{ray}$ denote the straight line path from $t$ to $\infty$
such that $\mathrm{ray}((0,1))=t\mathbb{R}_{>1}$. Since $\gamma''$
is homotopic to $\mathrm{ray}$ when $t$ is sufficiently close to
$\infty$, we have 
\[
I_{(\gamma'')^{-1}}\left(t;\left[\substack{z_{1},z_{2}\\
\alpha_{1},\alpha_{2}
}
\right],\dots,\left[\substack{z_{i-1},z_{i}\\
\alpha_{i-1},\alpha_{i}
}
\right];\infty\right)=I_{\mathrm{ray}}\!\left(t;\left[\substack{z_{1},z_{2}\\
\alpha_{1},\alpha_{2}
}
\right],\dots,\left[\substack{z_{i-1},z_{i}\\
\alpha_{i-1},\alpha_{i}
}
\right];\infty\right).
\]
The right-hand side equals
\[
\int_{|t|<t_{1}<\cdots<t_{i-1}<\infty}\prod_{j=1}^{i-1}\frac{d\left(\lambda t_{j}\right)}{(\lambda t_{j}-z_{j})^{\alpha_{j}}(\lambda t_{j}-z_{j+1})^{1-\alpha_{j+1}}}
\]
where $\lambda=t/|t|$. Notice here that, since 
\[
\frac{1}{(\lambda t_{j}-z_{j})^{\alpha_{j}}(\lambda t_{j}-z_{j+1})^{1-\alpha_{j+1}}}=O(t_{j}^{\Re(\alpha_{j+1}-\alpha_{j}-1)}),
\]
it follows that
\[
I_{\mathrm{ray}}\!\left(t;\left[\substack{z_{1},z_{2}\\
\alpha_{1},\alpha_{2}
}
\right],\dots,\left[\substack{z_{i-1},z_{i}\\
\alpha_{i-1},\alpha_{i}
}
\right];\infty\right)=O\!\left(\int_{|t|<t_{1}<\cdots<t_{i-1}<\infty}\prod_{j=1}^{i-1}t_{j}^{\Re(\alpha_{j+1}-\alpha_{j}-1)}\right)=O\left(|t|^{\Re(\alpha_{i}-\alpha_{1})}\right)\quad\left(t\rightarrow\infty\right).
\]
This proves (\ref{eq:estimate_I}). Thus, 
\begin{align*}
g(t) & =\left(\frac{t-z_{1}}{t-z_{0}}\right)^{\alpha_{1}-1}I_{\gamma}\left(\infty;\left[\substack{z_{1},z_{2}\\
\alpha_{1},\alpha_{2}
}
\right],\dots,\left[\substack{z_{n-1},z_{n}\\
\alpha_{n-1},\alpha_{n}
}
\right];z_{\circ}\right)+\sum_{i=2}^{n}O(|t|^{\Re(\alpha_{i}-\alpha_{1})})\\
 & =\left(\frac{t-z_{1}}{t-z_{0}}\right)^{\alpha_{1}-1}g(\infty)+\sum_{i=2}^{n}O(|t|^{\Re(\alpha_{i}-\alpha_{1})})\\
 & =g(\infty)+O(|t|^{-1})+O(|t|^{\Re(\alpha_{2}-\alpha_{1})})
\end{align*}
as $t$ tends to $\infty.$ Since $\Re(\alpha_{2}-\alpha_{1})<0$,
the integrals
\[
\int_{\infty}^{z_{\circ}}(t-z_{0})^{\beta-1}|t|^{-1}dt
\]
and
\[
\int_{\infty}^{z_{\circ}}(t-z_{0})^{\beta-1}|t|^{\Re(\alpha_{2}-\alpha_{1})}dt
\]
converge for $\Re(\beta)<0$, hence
\[
\lim_{\alpha_{0}\to\alpha_{1}}\int_{\infty}^{z_{\circ}}(t-z_{0})^{-\alpha_{0}+\alpha_{1}-1}\left(g(t)-g(\infty)\right)dt
\]
converges. Hence,
\begin{align*}
 & \lim_{\alpha_{0}\to\alpha_{1}}(\alpha_{0}-\alpha_{1})\int_{\infty}^{z_{\circ}}(t-z_{0})^{-\alpha_{0}+\alpha_{1}-1}g(t)dt\\
 & =\lim_{\alpha_{0}\to\alpha_{1}}(\alpha_{0}-\alpha_{1})\int_{\infty}^{z_{\circ}}(t-z_{0})^{-\alpha_{0}+\alpha_{1}-1}g(\infty)dt\\
 & =-\lim_{\alpha_{0}\to\alpha_{1}}\int_{\infty}^{z_{\circ}}\frac{d}{dt}\left((t-z_{0})^{-\alpha_{0}+\alpha_{1}}g(\infty)\right)dt\\
 & =-\lim_{\alpha_{0}\to\alpha_{1}}(z_{\circ}-z_{0})^{-\alpha_{0}+\alpha_{1}}g(\infty)\\
 & =-g(\infty)\\
 & =-B_{\gamma}^{\infty}\left(\left.\substack{z_{1}\\
\alpha_{1}
}
\right|\cdots\left|\substack{z_{n}\\
\alpha_{n}
}
\right.;z_{\circ}\right).
\end{align*}
This completes the proof.
\end{proof}

\section{Contiguous-type relations\label{sec:Contiguous-type-relation}}

As is well known, the beta function
\[
\mathrm{B}(\alpha,\beta)=\int_{0}^{1}t^{\alpha-1}(1-t)^{\beta-1}dt
\]
satisfies the recurrence relation 
\[
\mathrm{B}(\alpha+1,\beta)=\frac{\alpha}{\alpha+\beta}\mathrm{B}(\alpha,\beta).
\]
The iterated beta integrals satisfy the following generalization of
this recurrence relation:
\begin{thm}[Contiguous relation]
\label{thm: contiguous relations}For $0\leq i\leq n-1$, we have

\begin{multline*}
\alpha_{0}(z_{i}-z_{i+1})B_{\gamma}^{\bullet,\circ}\left(\left.\substack{z_{0}\\
\alpha_{0}+1
}
\right|\cdots\left|\substack{z_{i}\\
\alpha_{i}+1
}
\right|\left.\substack{z_{i+1}\\
\alpha_{i+1}
}
\right|\cdots\left|\substack{z_{n}\\
\alpha_{n}
}
\right.\right)\\
=\left(\alpha_{i+1}-\alpha_{i}\right)B_{\gamma}^{\bullet,\circ}\left(\left.\substack{z_{0}\\
\alpha_{0}
}
\right|\cdots\left|\substack{z_{n}\\
\alpha_{n}
}
\right.\right)+\begin{cases}
B_{\gamma}^{\bullet,\circ}\left(\left.\substack{z_{0}\\
\alpha_{0}
}
\right|\cdots\widehat{\left|\substack{z_{i}\\
\alpha_{i}
}
\right|}\cdots\left|\substack{z_{n}\\
\alpha_{n}
}
\right.\right) & \left(i\neq0\right)\\
0 & \left(i=0\right)
\end{cases}-\begin{cases}
B_{\gamma}^{\bullet,\circ}\left(\left.\substack{z_{0}\\
\alpha_{0}
}
\right|\cdots\widehat{\left|\substack{z_{i+1}\\
\alpha_{i+1}
}
\right|}\cdots\left|\substack{z_{n}\\
\alpha_{n}
}
\right.\right) & \left(i+1\neq n\right)\\
0 & \left(i+1=n\right)
\end{cases}
\end{multline*}
where $\widehat{x}$ denotes the deletion of the entry $x$. In terms
of the normalized ones,
\begin{multline*}
\left(\alpha_{0}-\alpha_{n}\right)\frac{(z_{i}-z_{i+1})}{(z_{0}-z_{n})}\hat{B}_{\gamma}^{\bullet,\circ}\left(\left.\substack{z_{0}\\
\alpha_{0}+1
}
\right|\cdots\left|\substack{z_{i}\\
\alpha_{i}+1
}
\right|\left.\substack{z_{i+1}\\
\alpha_{i+1}
}
\right|\cdots\left|\substack{z_{n}\\
\alpha_{n}
}
\right.\right)\\
=\left(\alpha_{i}-\alpha_{i+1}\right)\hat{B}_{\gamma}^{\bullet,\circ}\left(\left.\substack{z_{0}\\
\alpha_{0}
}
\right|\cdots\left|\substack{z_{n}\\
\alpha_{n}
}
\right.\right)-\begin{cases}
\hat{B}_{\gamma}^{\bullet,\circ}\left(\left.\substack{z_{0}\\
\alpha_{0}
}
\right|\cdots\widehat{\left|\substack{z_{i}\\
\alpha_{i}
}
\right|}\cdots\left|\substack{z_{n}\\
\alpha_{n}
}
\right.\right) & \left(i\neq0\right)\\
0 & \left(i=0\right)
\end{cases}+\begin{cases}
\hat{B}_{\gamma}^{\bullet,\circ}\left(\left.\substack{z_{0}\\
\alpha_{0}
}
\right|\cdots\widehat{\left|\substack{z_{i+1}\\
\alpha_{i+1}
}
\right|}\cdots\left|\substack{z_{n}\\
\alpha_{n}
}
\right.\right) & \left(i+1\neq n\right)\\
0 & \left(i+1=n\right)
\end{cases}
\end{multline*}
and 
\begin{align*}
 & (\alpha_{0}-\alpha_{n})\hat{\mathscr{B}}_{\gamma}^{\bullet,\circ}\left(\left.\substack{z_{0}\\
\alpha_{0}+1
}
\right|\cdots\left|\substack{z_{i}\\
\alpha_{i}+1
}
\right|\left.\substack{z_{i+1}\\
\alpha_{i+1}
}
\right|\cdots\left|\substack{z_{n}\\
\alpha_{n}
}
\right.\right)\\
= & (\alpha_{i}-\alpha_{i+1})\hat{\mathscr{B}}_{\gamma}^{\bullet,\circ}\left(\left.\substack{z_{0}\\
\alpha_{0}
}
\right|\cdots\left|\substack{z_{n}\\
\alpha_{n}
}
\right.\right)-\begin{cases}
\chi_{i-1,i,i+1}\hat{\mathscr{B}}_{\gamma}^{\bullet,\circ}\left(\left.\substack{z_{0}\\
\alpha_{0}
}
\right|\cdots\widehat{\left|\substack{z_{i}\\
\alpha_{i}
}
\right|}\cdots\left|\substack{z_{n}\\
\alpha_{n}
}
\right.\right) & \left(i\neq0\right)\\
0 & \left(i=0\right)
\end{cases}\\
 & \qquad+\begin{cases}
\chi_{i,i+1,i+2}\hat{\mathscr{B}}_{\gamma}^{\bullet,\circ}\left(\left.\substack{z_{0}\\
\alpha_{0}
}
\right|\cdots\widehat{\left|\substack{z_{i+1}\\
\alpha_{i+1}
}
\right|}\cdots\left|\substack{z_{n}\\
\alpha_{n}
}
\right.\right) & \left(i+1\neq n\right)\\
0 & \left(i+1=n\right),
\end{cases}
\end{align*}
where we put
\[
\chi_{i,j,k}\coloneqq\frac{(z_{i}-z_{j})^{\alpha_{i}-\alpha_{j}}(z_{j}-z_{k})^{\alpha_{j}-\alpha_{k}}}{(z_{i}-z_{k})^{\alpha_{i}-\alpha_{k}}}.
\]
\end{thm}

\begin{rem}
Theorem \ref{thm: contiguous relations} says
\begin{equation}
\alpha_{0}(z_{i}-z_{i+1})B_{\gamma}^{\bullet,\circ}\left(\left.\substack{z_{0}\\
\alpha_{0}+1
}
\right|\cdots\left|\substack{z_{i}\\
\alpha_{i}+1
}
\right|\left.\substack{z_{i+1}\\
\alpha_{i+1}
}
\right|\cdots\left|\substack{z_{n}\\
\alpha_{n}
}
\right.\right)\equiv\left(\alpha_{i+1}-\alpha_{i}\right)B_{\gamma}^{\bullet,\circ}\left(\left.\substack{z_{0}\\
\alpha_{0}
}
\right|\cdots\left|\substack{z_{n}\\
\alpha_{n}
}
\right.\right)\label{eq:plus one formula}
\end{equation}
modulo the terms of iterated beta integrals of length shorter by $1$.
Replacing $\alpha_{j}$ with $\alpha_{j}-1$ for $0\leq j\leq i$
in the theorem, one finds that
\begin{equation}
\left(\alpha_{i+1}-\alpha_{i}+1\right)B_{\gamma}^{\bullet,\circ}\left(\left.\substack{z_{0}\\
\alpha_{0}-1
}
\right|\cdots\left|\substack{z_{i}\\
\alpha_{i}-1
}
\right|\left.\substack{z_{i+1}\\
\alpha_{i+1}
}
\right|\cdots\left|\substack{z_{n}\\
\alpha_{n}
}
\right.\right)\equiv(\alpha_{0}-1)(z_{i}-z_{i+1})B_{\gamma}^{\bullet,\circ}\left(\left.\substack{z_{0}\\
\alpha_{0}
}
\right|\cdots\left|\substack{z_{n}\\
\alpha_{n}
}
\right.\right)\label{eq:minus one formula}
\end{equation}
Also, replacing $\alpha_{j}$ with $\alpha_{j}-1$ for $0\leq j\leq i-1$
in (\ref{eq:plus one formula}), one gets
\[
(\alpha_{0}-1)(z_{i}-z_{i+1})B_{\gamma}^{\bullet,\circ}\left(\left.\substack{z_{0}\\
\alpha_{0}
}
\right|\cdots\left|\substack{z_{i-1}\\
\alpha_{i-1}
}
\right.\left|\substack{z_{i}\\
\alpha_{i}+1
}
\right|\left.\substack{z_{i+1}\\
\alpha_{i+1}
}
\right|\cdots\left|\substack{z_{n}\\
\alpha_{n}
}
\right.\right)\equiv\left(\alpha_{i+1}-\alpha_{i}\right)B_{\gamma}^{\bullet,\circ}\left(\left.\substack{z_{0}\\
\alpha_{0}-1
}
\right|\cdots\left|\substack{z_{i-1}\\
\alpha_{i-1}-1
}
\right|\left.\substack{z_{i}\\
\alpha_{i}
}
\right|\cdots\left|\substack{z_{n}\\
\alpha_{n}
}
\right.\right),
\]
and replacing $i$ with $i-1$ in (\ref{eq:minus one formula}),
\[
\left(\alpha_{i}-\alpha_{i-1}+1\right)B_{\gamma}^{\bullet,\circ}\left(\left.\substack{z_{0}\\
\alpha_{0}-1
}
\right|\cdots\left|\substack{z_{i-1}\\
\alpha_{i-1}-1
}
\right|\left.\substack{z_{i}\\
\alpha_{i}
}
\right|\cdots\left|\substack{z_{n}\\
\alpha_{n}
}
\right.\right)\equiv(\alpha_{0}-1)(z_{i-1}-z_{i})B_{\gamma}^{\bullet,\circ}\left(\left.\substack{z_{0}\\
\alpha_{0}
}
\right|\cdots\left|\substack{z_{n}\\
\alpha_{n}
}
\right.\right).
\]
Comparing those equalities, we find that
\[
B_{\gamma}^{\bullet,\circ}\left(\left.\substack{z_{0}\\
\alpha_{0}
}
\right|\cdots\left|\substack{z_{i-1}\\
\alpha_{i-1}
}
\right.\left|\substack{z_{i}\\
\alpha_{i}+1
}
\right|\left.\substack{z_{i+1}\\
\alpha_{i+1}
}
\right|\cdots\left|\substack{z_{n}\\
\alpha_{n}
}
\right.\right)\equiv B_{\gamma}^{\bullet,\circ}\left(\left.\substack{z_{0}\\
\alpha_{0}
}
\right|\cdots\left|\substack{z_{n}\\
\alpha_{n}
}
\right.\right)\times\frac{\left(\alpha_{i+1}-\alpha_{i}\right)(z_{i-1}-z_{i})}{(z_{i}-z_{i+1})\left(\alpha_{i}-\alpha_{i-1}+1\right)}
\]
for $1\leq i\leq n-1$. Also, we can deduce similar formulas for the
case $i=0,n$. In this way, we can reduce
\[
B_{\gamma}^{\bullet,\circ}\left(\left.\substack{z_{0}\\
\alpha_{0}+m_{0}
}
\right|\left.\substack{z_{1}\\
\alpha_{1}+m_{1}
}
\right|\cdots\left|\substack{z_{n}\\
\alpha_{n}+m_{n}
}
\right.\right)\quad\left(m_{0},\ldots,m_{n}\in\mathbb{Z}\right)
\]
to
\[
C(\boldsymbol{z},\boldsymbol{\alpha})\cdot B_{\gamma}^{\bullet,\circ}\left(\left.\substack{z_{0}\\
\alpha_{0}
}
\right|\left.\substack{z_{1}\\
\alpha_{1}
}
\right|\cdots\left|\substack{z_{n}\\
\alpha_{n}
}
\right.\right)
\]
with $C(\boldsymbol{z},\boldsymbol{\alpha})\in\mathbb{Q}\left(\alpha_{0},\ldots,\alpha_{n}\right)^{\times}\cdot\prod_{i=0}^{n-1}\left(z_{i}-z_{i+1}\right)^{\mathbb{Z}}$,
plus some linear combinations of iterated beta integrals whose length
is reduced by $1$. In this sense, integer shifts of the parameters
$\alpha_{i}$ do not produce significant differences, and we will
later restrict ourselves to the case when $0\leq\alpha_{i}\leq1$
($0\leq i\leq n$) when discussing the equalities arising from the
translation invariance of iterated beta integrals. As a special case
when $z_{i}=z_{i+1}$, the theorem gives 
\[
(\alpha_{i+1}-\alpha_{i})\hat{B}_{\gamma}^{\bullet,\circ}\left(\left.\substack{z_{0}\\
\alpha_{0}
}
\right|\cdots\left|\substack{z_{n}\\
\alpha_{n}
}
\right.\right)=\begin{cases}
\hat{B}_{\gamma}^{\bullet,\circ}\left(\left.\substack{z_{0}\\
\alpha_{0}
}
\right|\cdots\widehat{\left|\substack{z_{i+1}\\
\alpha_{i+1}
}
\right|}\cdots\left|\substack{z_{n}\\
\alpha_{n}
}
\right.\right) & \left(i+1\neq n\right)\\
0 & \left(i+1=n\right)
\end{cases}-\begin{cases}
\hat{B}_{\gamma}^{\bullet,\circ}\left(\left.\substack{z_{0}\\
\alpha_{0}
}
\right|\cdots\widehat{\left|\substack{z_{i}\\
\alpha_{i}
}
\right|}\cdots\left|\substack{z_{n}\\
\alpha_{n}
}
\right.\right) & \left(i\neq0\right)\\
0 & \left(i=0\right).
\end{cases}
\]
So if $\alpha_{i}\neq\alpha_{i+1}$, $\hat{B}_{\gamma}^{\bullet,\circ}\left(\left.\substack{z_{0}\\
\alpha_{0}
}
\right|\cdots\left|\substack{z_{n}\\
\alpha_{n}
}
\right.\right)$ with $z_{i}=z_{i+1}$ reduces to $\hat{B}_{\gamma}^{\bullet,\circ}$
of length shorter by $1$.
\end{rem}

As an immediate corollary of the third formula of Theorem \ref{thm: contiguous relations},
we obtain the following simple identity:
\begin{cor}
We have
\[
\sum_{i=0}^{n-1}\hat{\mathscr{B}}_{\gamma}^{\bullet,\circ}\left(\left.\substack{z_{0}\\
\alpha_{0}+1
}
\right|\cdots\left|\substack{z_{i}\\
\alpha_{i}+1
}
\right|\left.\substack{z_{i+1}\\
\alpha_{i+1}
}
\right|\cdots\left|\substack{z_{n}\\
\alpha_{n}
}
\right.\right)=\hat{\mathscr{B}}_{\gamma}^{\bullet,\circ}\left(\left.\substack{z_{0}\\
\alpha_{0}
}
\right|\cdots\left|\substack{z_{n}\\
\alpha_{n}
}
\right.\right).
\]
\end{cor}

Before proving Theorem \ref{thm: contiguous relations}, we prepare
the following lemma.
\begin{lem}
\label{lem: plus one lemma}For integers $i,n$ with $0\leq i\leq n$,
we have 
\begin{multline*}
\alpha_{0}I_{\gamma}\left(z_{\bullet};\left[\substack{z_{0},z_{1}\\
\alpha_{0}+1,\alpha_{1}+1
}
\right],\dots,\left[\substack{z_{i-1},z_{i}\\
\alpha_{i-1}+1,\alpha_{i}+1
}
\right],\left[\substack{z_{i},z_{i+1}\\
\alpha_{i},\alpha_{i+1}
}
\right],\dots,\left[\substack{z_{n-1},z_{n}\\
\alpha_{n-1},\alpha_{n}
}
\right];z_{\circ}\right)\\
=\alpha_{i}B_{\gamma}^{\bullet,\circ}\left(\left.\substack{z_{0}\\
\alpha_{0}
}
\right|\cdots\left|\substack{z_{n}\\
\alpha_{n}
}
\right.\right)-\begin{cases}
B_{\gamma}^{\bullet,\circ}\left(\left.\substack{z_{0}\\
\alpha_{0}
}
\right|\cdots\widehat{\left|\substack{z_{i}\\
\alpha_{i}
}
\right|}\cdots\left|\substack{z_{n}\\
\alpha_{n}
}
\right.\right) & i\neq0,n\\
0 & i=0,n,
\end{cases}
\end{multline*}
where $z_{\bullet}=\begin{cases}
z_{0} & \text{ if }\bullet=\mathrm{f}\\
\infty & \text{ if }\bullet=\infty
\end{cases}$ and $z_{\circ}=\begin{cases}
z_{n} & \text{ if }\circ=\mathrm{f}\\
\infty & \text{ if }\circ=\infty
\end{cases}$.
\end{lem}

\begin{proof}
We prove the claim by induction on $n$ and $i$. Notice that the
claim is trivial when $i=0$. We denote by $\Gamma_{i}$ (resp. $\Gamma^{i}$)
the sequence $\left[\substack{z_{0},z_{1}\\
\alpha_{0}+1,\alpha_{1}+1
}
\right],\dots,\left[\substack{z_{i-1},z_{i}\\
\alpha_{i-1}+1,\alpha_{i}+1
}
\right]$ (resp. $\left[\substack{z_{i},z_{i+1}\\
\alpha_{i},\alpha_{i+1}
}
\right],\dots,\left[\substack{z_{n-1},z_{n}\\
\alpha_{n-1},\alpha_{n}
}
\right]$). Notice that the left-hand side of the equality is expressed as
$\alpha_{0}I_{\gamma}\left(z_{\bullet};\Gamma_{i},\Gamma^{i};z_{\circ}\right)$
under this notation. Suppose $0<i\leq n$. The key identity is obtained
by expressing
\[
X_{n,i}\coloneqq I_{\gamma}\left(z_{\bullet};\Gamma_{i},f'(t)dt,\Gamma^{i+1};z_{\circ}\right)\quad\text{ with }f(t)=(t-z_{i})^{-\alpha_{i}}(t-z_{i+1})^{\alpha_{i+1}}
\]
in two ways. First, since
\[
f'(t)dt=-\alpha_{i}\left[\substack{z_{i},z_{i+1}\\
\alpha_{i}+1,\alpha_{i+1}+1
}
\right]+\alpha_{i+1}\left[\substack{z_{i},z_{i+1}\\
\alpha_{i},\alpha_{i+1}
}
\right],
\]
we have
\[
X_{n,i}=-\alpha_{i}I_{\gamma}\left(z_{\bullet};\Gamma_{i+1},\Gamma^{i+1};z_{\circ}\right)+\alpha_{i+1}I_{\gamma}\left(z_{\bullet};\Gamma_{i},\Gamma^{i};z_{\circ}\right).
\]
On the other hand, if we integrate $f'(t)dt$ part first, we find
\[
X_{n,i}\coloneqq-I_{\gamma}\left(z_{\bullet};\Gamma_{i-1},f\cdot\left[\substack{z_{i-1},z_{i}\\
\alpha_{i-1}+1,\alpha_{i}+1
}
\right],\Gamma^{i+1};z_{\circ}\right)+\begin{cases}
I_{\gamma}\left(z_{\bullet};\Gamma_{i},f\cdot\left[\substack{z_{i+1},z_{i+2}\\
\alpha_{i+1},\alpha_{i+2}
}
\right],\Gamma^{i+2};z_{\circ}\right) & \text{ if }i\leq n-2\\
0 & \text{ if }i=n-1
\end{cases}
\]
Since
\[
f\cdot\left[\substack{z_{i+1},z_{i+2}\\
\alpha_{i+1},\alpha_{i+2}
}
\right]=\left[\substack{z_{i},z_{i+2}\\
\alpha_{i},\alpha_{i+2}
}
\right]\qquad(i\leq n-2)
\]
and
\[
f\cdot\left[\substack{z_{i-1},z_{i}\\
\alpha_{i-1}+1,\alpha_{i}+1
}
\right]=\left[\substack{z_{i-1},z_{i+1}\\
\alpha_{i-1}+1,\alpha_{i+1}+1
}
\right]\qquad(i\geq1),
\]
we have
\begin{align*}
X_{n,i} & =-I_{\gamma}\left(z_{\bullet};\Gamma_{i-1},\left[\substack{z_{i-1},z_{i+1}\\
\alpha_{i-1}+1,\alpha_{i+1}+1
}
\right],\Gamma^{i+1};z_{\circ}\right)+\begin{cases}
I_{\gamma}\left(z_{\bullet};\Gamma_{i},\left[\substack{z_{i},z_{i+2}\\
\alpha_{i},\alpha_{i+2}
}
\right],\Gamma^{i+2};z_{\circ}\right) & \text{ if }i\leq n-2\\
0 & \text{ if }i=n-1
\end{cases}
\end{align*}
By equating the two expressions for $X_{n,i}$ obtained above, it
follows that
\begin{align*}
I_{\gamma}\left(z_{\bullet};\Gamma_{i+1},\Gamma^{i+1};z_{\circ}\right) & =\frac{\alpha_{i+1}}{\alpha_{i}}I_{\gamma}\left(z_{\bullet};\Gamma_{i},\Gamma^{i};z_{\circ}\right)+\frac{1}{\alpha_{i}}I_{\gamma}\left(z_{\bullet};\Gamma_{i-1},\left[\substack{z_{i-1},z_{i+1}\\
\alpha_{i-1}+1,\alpha_{i+1}+1
}
\right],\Gamma^{i+1};z_{\circ}\right)\\
 & \qquad-\begin{cases}
\frac{1}{\alpha_{i}}I_{\gamma}\left(z_{\bullet};\Gamma_{i},\left[\substack{z_{i},z_{i+2}\\
\alpha_{i},\alpha_{i+2}
}
\right],\Gamma^{i+2};z_{\circ}\right) & \text{ if }i\leq n-2\\
0 & \text{ if }i=n-1
\end{cases}
\end{align*}
Using the induction hypothesis for each term, the right-hand side
simplifies to 
\[
\frac{\alpha_{i+1}}{\alpha_{0}}B_{\gamma}^{\bullet,\circ}\left(\left.\substack{z_{0}\\
\alpha_{0}
}
\right|\cdots\left|\substack{z_{n}\\
\alpha_{n}
}
\right.\right)-\begin{cases}
\frac{1}{\alpha_{0}}B_{\gamma}^{\bullet,\circ}\left(\left.\substack{z_{0}\\
\alpha_{0}
}
\right|\cdots\widehat{\left|\substack{z_{i+1}\\
\alpha_{i+1}
}
\right|}\cdots\left|\substack{z_{n}\\
\alpha_{n}
}
\right.\right) & i\neq n-1\\
0 & i=n-1,
\end{cases}
\]
yielding the claim for $(n,i+1)$. For example, when $i\neq n-1$,
the right-hand side can be calculated as
\begin{align*}
 & \quad\frac{\alpha_{i+1}}{\alpha_{i}}\left(\frac{\alpha_{i}}{\alpha_{0}}B_{\gamma}^{\bullet,\circ}\left(\left.\substack{z_{0}\\
\alpha_{0}
}
\right|\cdots\left|\substack{z_{n}\\
\alpha_{n}
}
\right.\right)-\frac{1}{\alpha_{0}}B_{\gamma}^{\bullet,\circ}\left(\left.\substack{z_{0}\\
\alpha_{0}
}
\right|\cdots\left|\substack{z_{i-1}\\
\alpha_{i-1}
}
\right|\left.\substack{z_{i+1}\\
\alpha_{i+1}
}
\right|\cdots\left|\substack{z_{n}\\
\alpha_{n}
}
\right.\right)\right)\\
 & \quad-\frac{1}{\alpha_{i}}\left(\frac{\alpha_{i}}{\alpha_{0}}B_{\gamma}^{\bullet,\circ}\left(\left.\substack{z_{0}\\
\alpha_{0}
}
\right|\cdots\left|\substack{z_{i}\\
\alpha_{i}
}
\right|\left.\substack{z_{i+2}\\
\alpha_{i+2}
}
\right|\cdots\left|\substack{z_{n}\\
\alpha_{n}
}
\right.\right)-\frac{1}{\alpha_{0}}B_{\gamma}^{\bullet,\circ}\left(\left.\substack{z_{0}\\
\alpha_{0}
}
\right|\cdots\left|\substack{z_{i-1}\\
\alpha_{i-1}
}
\right|\left.\substack{z_{i+2}\\
\alpha_{i+2}
}
\right|\cdots\left|\substack{z_{n}\\
\alpha_{n}
}
\right.\right)\right)\\
 & \quad+\frac{1}{\alpha_{i}}\left(\frac{\alpha_{i+1}}{\alpha_{0}}B_{\gamma}^{\bullet,\circ}\left(\left.\substack{z_{0}\\
\alpha_{0}
}
\right|\cdots\left|\substack{z_{i-1}\\
\alpha_{i-1}
}
\right|\left.\substack{z_{i+1}\\
\alpha_{i+1}
}
\right|\cdots\left|\substack{z_{n}\\
\alpha_{n}
}
\right.\right)-\frac{1}{\alpha_{0}}B_{\gamma}^{\bullet,\circ}\left(\left.\substack{z_{0}\\
\alpha_{0}
}
\right|\cdots\left|\substack{z_{i-1}\\
\alpha_{i-1}
}
\right|\left.\substack{z_{i+2}\\
\alpha_{i+2}
}
\right|\cdots\left|\substack{z_{n}\\
\alpha_{n}
}
\right.\right)\right)\\
 & =\frac{\alpha_{i+1}}{\alpha_{0}}B_{\gamma}^{\bullet,\circ}\left(\left.\substack{z_{0}\\
\alpha_{0}
}
\right|\cdots\left|\substack{z_{n}\\
\alpha_{n}
}
\right.\right)-\frac{1}{\alpha_{0}}B_{\gamma}^{\bullet,\circ}\left(\left.\substack{z_{0}\\
\alpha_{0}
}
\right|\cdots\left|\substack{z_{i}\\
\alpha_{i}
}
\right|\left.\substack{z_{i+2}\\
\alpha_{i+2}
}
\right|\cdots\left|\substack{z_{n}\\
\alpha_{n}
}
\right.\right)
\end{align*}
using the induction hypothesis for $(n,i)$, $(n-1,i)$ and $(n-1,i-1)$.
This proves the claim.
\end{proof}
\begin{proof}[Proof of Theorem \ref{thm: contiguous relations}]
Let $\Gamma_{i}$ be the same as in the proof of Lemma \ref{lem: plus one lemma}.
Furthermore, let $z_{\bullet}$ and $z_{\circ}$ be as in Lemma \ref{lem: plus one lemma}.
Then, we have
\begin{align*}
I_{\gamma}\left(z_{\bullet};\Gamma_{i},\Gamma^{i};z_{\circ}\right) & =I_{\gamma}\left(z_{\bullet};\Gamma_{i},\left[\substack{z_{i},z_{i+1}\\
\alpha_{i}+1,\alpha_{i+1}
}
\right](t-z_{i+1}+z_{i+1}-z_{i}),\Gamma^{i+1};z_{\circ}\right)\\
 & =I_{\gamma}\left(z_{\bullet};\Gamma_{i},\left[\substack{z_{i},z_{i+1}\\
\alpha_{i}+1,\alpha_{i+1}+1
}
\right],\Gamma^{i+1};z_{\circ}\right)-(z_{i}-z_{i+1})B_{\gamma}^{\bullet,\circ}\left(\left.\substack{z_{0}\\
\alpha_{0}+1
}
\right|\cdots\left|\substack{z_{i}\\
\alpha_{i}+1
}
\right|\left.\substack{z_{i+1}\\
\alpha_{i+1}
}
\right|\cdots\left|\substack{z_{n}\\
\alpha_{n}
}
\right.\right)\\
 & =I_{\gamma}\left(z_{\bullet};\Gamma_{i+1},\Gamma^{i+1};z_{\circ}\right)-(z_{i}-z_{i+1})B_{\gamma}^{\bullet,\circ}\left(\left.\substack{z_{0}\\
\alpha_{0}+1
}
\right|\cdots\left|\substack{z_{i}\\
\alpha_{i}+1
}
\right|\left.\substack{z_{i+1}\\
\alpha_{i+1}
}
\right|\cdots\left|\substack{z_{n}\\
\alpha_{n}
}
\right.\right).
\end{align*}
Thus, by Lemma \ref{lem: plus one lemma}, it follows that
\begin{align*}
 & \alpha_{0}(z_{i}-z_{i+1})B_{\gamma}^{\bullet,\circ}\left(\left.\substack{z_{0}\\
\alpha_{0}+1
}
\right|\cdots\left|\substack{z_{i}\\
\alpha_{i}+1
}
\right|\left.\substack{z_{i+1}\\
\alpha_{i+1}
}
\right|\cdots\left|\substack{z_{n}\\
\alpha_{n}
}
\right.\right)\\
 & =\alpha_{0}I_{\gamma}\left(z_{\bullet};\Gamma_{i+1},\Gamma^{i+1};z_{\circ}\right)-\alpha_{0}I_{\gamma}\left(z_{\bullet};\Gamma_{i},\Gamma^{i};z_{\circ}\right)\\
 & =(\mathrm{RHS}).
\end{align*}
\end{proof}

\section{Differential equations\label{sec:Differential-equations}}

The iterated beta integrals satisfy a system of differential equations
generalizing the differential equation for hyperlogarithms. For $\bm{z}=(z_{0},\dots,z_{n+1}),\:\bm{\alpha}=(\alpha_{0},\dots,\alpha_{n+1})\in\mathbb{C}^{n+2}$,
define 
\[
\chi\left(\left.\substack{z_{i}\\
\alpha_{i}
}
\right|\substack{z_{j}\\
\alpha_{j}
}
\left|\substack{z_{k}\\
\alpha_{k}
}
\right.\right)\coloneqq\frac{(z_{i}-z_{j})^{\alpha_{i}-\alpha_{j}}(z_{j}-z_{k})^{\alpha_{j}-\alpha_{k}}}{(z_{i}-z_{k})^{\alpha_{i}-\alpha_{k}}}.
\]
By definition, $(-1)^{\alpha_{k}-\alpha_{i}}\chi\left(\left.\substack{z_{i}\\
\alpha_{i}
}
\right|\substack{z_{j}\\
\alpha_{j}
}
\left|\substack{z_{k}\\
\alpha_{k}
}
\right.\right)$ is invariant under the cyclic permutation of $i,j,k$.
\begin{thm}
\label{Thm:differential_eq_complete}Let $\bm{\alpha}=(\alpha_{0},\dots,\alpha_{n+1})\in\mathbb{C}^{n+2}$.
The total differential of scripted normalized iterated beta integrals
with respect to $\bm{z}=(z_{0},\dots,z_{n+1})$ is given by
\begin{align*}
d\hat{\mathscr{B}}_{\gamma}^{\bullet,\circ}\left(\left.\substack{z_{0}\\
\alpha_{0}
}
\right|\left.\substack{z_{1}\\
\alpha_{1}
}
\right|\cdots\left|\substack{z_{n+1}\\
\alpha_{n+1}
}
\right.\right)= & \sum_{i=1}^{n}\hat{\mathscr{B}}_{\gamma}^{\bullet,\circ}\left(\left.\substack{z_{0}\\
\alpha_{0}
}
\right|\left.\substack{z_{1}\\
\alpha_{1}
}
\right|\dots\left|\substack{z_{i-1}\\
\alpha_{i-1}
}
\right|\left.\substack{z_{i+1}\\
\alpha_{i+1}
}
\right|\cdots\left|\substack{z_{n+1}\\
\alpha_{n+1}
}
\right.\right)\cdot\chi\left(\left.\substack{z_{i-1}\\
\alpha_{i-1}
}
\right|\substack{z_{i}\\
\alpha_{i}
}
\left|\substack{z_{i+1}\\
\alpha_{i+1}
}
\right.\right)d\log\left(\frac{z_{i}-z_{i+1}}{z_{i}-z_{i-1}}\right).
\end{align*}
\end{thm}

\begin{rem}
Notice that the case $n=1$ of Theorem \ref{Thm:differential_eq_complete}
gives
\begin{equation}
d\hat{\mathscr{B}}_{\gamma}^{\bullet,\circ}\bigl(\left.\substack{z_{i}\\
\alpha_{i}
}
\right|\substack{z_{j}\\
\alpha_{j}
}
\left|\substack{z_{k}\\
\alpha_{k}
}
\right.\bigr)=\chi\left(\left.\substack{z_{i}\\
\alpha_{i}
}
\right|\substack{z_{j}\\
\alpha_{j}
}
\left|\substack{z_{k}\\
\alpha_{k}
}
\right.\right)d\log\left(\frac{z_{j}-z_{k}}{z_{j}-z_{i}}\right),\label{eq:diff_formula_shortest-1}
\end{equation}
which says that $d\hat{\mathscr{B}}_{\gamma}^{\bullet,\circ}\bigl(\left.\substack{z_{i}\\
\alpha_{i}
}
\right|\substack{z_{j}\\
\alpha_{j}
}
\left|\substack{z_{k}\\
\alpha_{k}
}
\right.\bigr)$ does not depend on the choice of $\gamma$ or $\bullet,\circ\in\{\infty,{\rm f}\}$,
allowing us to omit $\bullet,\circ,\gamma$ from the notation and
write it simply as $d\hat{\mathscr{B}}\bigl(\left.\substack{z_{i}\\
\alpha_{i}
}
\right|\substack{z_{j}\\
\alpha_{j}
}
\left|\substack{z_{k}\\
\alpha_{k}
}
\right.\bigr).$ From (\ref{eq:diff_formula_shortest-1}), it follows that 
\[
d\left((-1)^{\alpha_{0}-\alpha_{1}}\hat{\mathscr{B}}_{\gamma_{2}}^{\mathrm{f},\mathrm{f}}(\substack{z_{1}\\
\alpha_{1}
}
\left|\substack{z_{2}\\
\alpha_{2}
}
\right.\left|\substack{z_{0}\\
\alpha_{0}
}
\right.)+(-1)^{\alpha_{2}-\alpha_{0}}\hat{\mathscr{B}}_{\gamma_{0}}^{\mathrm{f},\mathrm{f}}\left(\substack{z_{0}\\
\alpha_{0}
}
\left|\substack{z_{1}\\
\alpha_{1}
}
\right.\left|\substack{z_{2}\\
\alpha_{2}
}
\right.\right)+(-1)^{\alpha_{1}-\alpha_{2}}\hat{\mathscr{B}}_{\gamma_{1}}^{\mathrm{f},\mathrm{f}}(\substack{z_{2}\\
\alpha_{2}
}
\left|\substack{z_{0}\\
\alpha_{0}
}
\right.\left|\substack{z_{1}\\
\alpha_{1}
}
\right.)\right)=0,
\]
implying that $(-1)^{\alpha_{0}-\alpha_{1}}\hat{\mathscr{B}}_{\gamma_{2}}^{\mathrm{f},\mathrm{f}}(\substack{z_{1}\\
\alpha_{1}
}
\left|\substack{z_{2}\\
\alpha_{2}
}
\right.\left|\substack{z_{0}\\
\alpha_{0}
}
\right.)+(-1)^{\alpha_{2}-\alpha_{0}}\hat{\mathscr{B}}_{\gamma_{0}}^{\mathrm{f},\mathrm{f}}\left(\substack{z_{0}\\
\alpha_{0}
}
\left|\substack{z_{1}\\
\alpha_{1}
}
\right.\left|\substack{z_{2}\\
\alpha_{2}
}
\right.\right)+(-1)^{\alpha_{1}-\alpha_{2}}\hat{\mathscr{B}}_{\gamma_{1}}^{\mathrm{f},\mathrm{f}}(\substack{z_{2}\\
\alpha_{2}
}
\left|\substack{z_{0}\\
\alpha_{0}
}
\right.\left|\substack{z_{1}\\
\alpha_{1}
}
\right.)$ is a constant.
\end{rem}

\begin{rem}
By noting the relation $\hat{\mathscr{B}}_{\gamma}^{\mathrm{f},\mathrm{f}}\bigl(\left.\substack{z_{0}\\
0
}
\right|\left.\substack{z_{1}\\
0
}
\right|\dots\left|\substack{z_{n+1}\\
0
}
\right.\bigr)=I_{\gamma}\bigl(z_{0};e_{z_{1}}\cdots e_{z_{n}};z_{n+1}\bigr)$ (Theorem \ref{thm:relationship_with_hyperlogarithm}) and $\chi\left(\left.\substack{z_{i-1}\\
0
}
\right|\substack{z_{i}\\
0
}
\left|\substack{z_{i+1}\\
0
}
\right.\right)=1$, the above differential equation generalizes that of hyperlogarithms,
i.e.,
\begin{equation}
dI_{\gamma}\bigl(z_{0};e_{z_{1}}\cdots e_{z_{n}};z_{n+1}\bigr)=\sum_{i=1}^{n}I_{\gamma}\bigl(z_{0};e_{z_{1}}\cdots e_{z_{i-1}}e_{z_{i+1}}\cdots e_{z_{n}};z_{n+1}\bigr)\cdot d\log\left(\frac{z_{i}-z_{i+1}}{z_{i}-z_{i-1}}\right)\label{eq:GoncharovDiff}
\end{equation}
by Goncharov \cite[Theorem 2.1]{Gon_MPL_and_MTM}. It might be interesting
to note that, in the hyperlogarithmic case, the term $I_{\gamma}\bigl(z_{0};e_{z_{1}}\cdots e_{z_{i-1}}e_{z_{i+1}}\cdots e_{z_{n}};z_{n+1}\bigr)$
is the iterated integral obtained by removing the $i$-th differential
form $d\log(t-z_{i})$ from $I_{\gamma}\bigl(z_{0};e_{z_{1}}\cdots e_{z_{n}};z_{n+1}\bigr)$,
whereas, in the iterated beta integral case, the term $\hat{\mathscr{B}}_{\gamma}^{\bullet,\circ}\left(\left.\substack{z_{0}\\
\alpha_{0}
}
\right|\left.\substack{z_{1}\\
\alpha_{1}
}
\right|\dots\left|\substack{z_{i-1}\\
\alpha_{i-1}
}
\right|\left.\substack{z_{i+1}\\
\alpha_{i+1}
}
\right|\cdots\left|\substack{z_{n+1}\\
\alpha_{n+1}
}
\right.\right)$ is obtained by replacing the consecutive differential forms $\left\{ \substack{z_{i-1},z_{i}\\
\alpha_{i-1},\alpha_{i}
}
\right\} ,\left\{ \substack{z_{i},z_{i+1}\\
\alpha_{i},\alpha_{i+1}
}
\right\} $ with $\left\{ \substack{z_{i-1},z_{i+1}\\
\alpha_{i-1},\alpha_{i+1}
}
\right\} $.
\end{rem}

By appealing to (\ref{eq:diff_formula_shortest-1}), we obtain the
following cleaner version of the differential formula:
\begin{thm}
\label{thm:Differential_eq_normalized_B_formal}With the settings
above, 
\begin{align*}
d\hat{\mathscr{B}}_{\gamma}^{\bullet,\circ}\bigl(\left.\substack{z_{0}\\
\alpha_{0}
}
\right|\left.\substack{z_{1}\\
\alpha_{1}
}
\right|\dots\left|\substack{z_{n}\\
\alpha_{n}
}
\right.\left|\substack{z_{n+1}\\
\alpha_{n+1}
}
\right.\bigr)= & \sum_{i=1}^{n}\hat{\mathscr{B}}_{\gamma}^{\bullet,\circ}\left(\left.\substack{z_{0}\\
\alpha_{0}
}
\right|\left.\substack{z_{1}\\
\alpha_{1}
}
\right|\dots\left|\substack{z_{i-1}\\
\alpha_{i-1}
}
\right|\left.\substack{z_{i+1}\\
\alpha_{i+1}
}
\right|\cdots\left|\substack{z_{n}\\
\alpha_{n}
}
\right.\left|\substack{z_{n+1}\\
\alpha_{n+1}
}
\right.\right)\cdot d\hat{\mathscr{B}}\bigl(\left.\substack{z_{i-1}\\
\alpha_{i-1}
}
\right|\substack{z_{i}\\
\alpha_{i}
}
\left|\substack{z_{i+1}\\
\alpha_{i+1}
}
\right.\bigr).
\end{align*}
\end{thm}

\begin{rem}
Notice that, by rewriting (\ref{eq:GoncharovDiff}) in the form
\[
dI_{\gamma}\bigl(z_{0};z_{1}\cdots z_{n};z_{n+1}\bigr)=\sum_{i=1}^{n}I_{\gamma}\bigl(z_{0};z_{1}\cdots\widehat{z_{i}}\cdots z_{n};z_{n+1}\bigr)\cdot dI\bigl(z_{i-1};e_{z_{i}};z_{i+1}\bigr),
\]
we find the striking similarity between the differential equations
for iterated beta integrals and hyperlogarithms.
\end{rem}

Additionally, in terms of $\hat{B}$, the differential equation takes
the following form:
\begin{thm}
\label{thm:Differential_eq_normalized_B}With the settings above,
\begin{align*}
\Delta\bigl(\left.\substack{z_{0}\\
\alpha_{0}
}
\right|\left.\substack{z_{1}\\
\alpha_{1}
}
\right|\dots\left|\substack{z_{n+1}\\
\alpha_{n+1}
}
\right.\bigr)\hat{B}_{\gamma}^{\bullet,\circ}\bigl(\left.\substack{z_{0}\\
\alpha_{0}
}
\right|\left.\substack{z_{1}\\
\alpha_{1}
}
\right|\dots\left|\substack{z_{n+1}\\
\alpha_{n+1}
}
\right.\bigr)= & \sum_{i=1}^{n}\hat{B}_{\gamma}^{\bullet,\circ}\left(\left.\substack{z_{0}\\
\alpha_{0}
}
\right|\left.\substack{z_{1}\\
\alpha_{1}
}
\right|\dots\left|\substack{z_{i-1}\\
\alpha_{i-1}
}
\right|\left.\substack{z_{i+1}\\
\alpha_{i+1}
}
\right|\cdots\left|\substack{z_{n+1}\\
\alpha_{n+1}
}
\right.\right)\cdot d\log\left(\frac{z_{i}-z_{i+1}}{z_{i}-z_{i-1}}\right)
\end{align*}
where $\Delta\bigl(\left.\substack{z_{0}\\
\alpha_{0}
}
\right|\left.\substack{z_{1}\\
\alpha_{1}
}
\right|\dots\left|\substack{z_{n+1}\\
\alpha_{n+1}
}
\right.\bigr)\varphi\coloneqq d\varphi+\left(\sum_{i\in\mathbb{Z}/(n+2)}(\alpha_{i}-\alpha_{i+1})d\log\left(z_{i}-z_{i+1}\right)\right)\varphi$. 
\end{thm}

In the following, we give a proof of Theorem \ref{Thm:differential_eq_complete}
and Theorem \ref{thm:Differential_eq_normalized_B}.

\begin{proof}[Proof of Theorem \ref{Thm:differential_eq_complete}]
Put $p\coloneqq\gamma(0)\in\{z_{0},\infty\}$ and $q\coloneqq\gamma(1)\in\{z_{n+1},\infty\}$.
By the identity theorem, we may assume the absolute convergence of
the iterated integral and $\Re(\alpha_{0})<0$ if $p=z_{0}$ and $\Re(1-\alpha_{n+1})<0$
if $q=z_{n+1}$, without loss of generality. We put
\[
g_{i,j}=g_{i,j}(\bm{z},t)=(t-z_{i})^{-\alpha_{i}}(t-z_{j})^{\alpha_{j}-1}(z_{i}-z_{j})^{\alpha_{i}-\alpha_{j}}
\]
and
\[
\omega_{i,j}(t):=g_{i,j}dt=\left\{ \substack{z_{i},z_{j}\\
\alpha_{i},\alpha_{j}
}
\right\} (t).
\]
To avoid cumbersome notation, we will prove an equivalent claim for
$\mathcal{D}=\sum_{i=0}^{n+1}c_{i}\frac{\partial}{\partial z_{i}}$
($c_{i}\in\mathbb{C}$), instead of the total differential. Then the
claim of the theorem is equivalent to
\[
\mathcal{D}\mathscr{B}_{\gamma}^{\bullet,\circ}\left(\left.\substack{z_{0}\\
\alpha_{0}
}
\right|\left.\substack{z_{1}\\
\alpha_{1}
}
\right|\cdots\left|\substack{z_{n+1}\\
\alpha_{n+1}
}
\right.\right)=\sum_{i=1}^{n}\mathscr{B}_{\gamma}^{\bullet,\circ}\left(\left.\substack{z_{0}\\
\alpha_{0}
}
\right|\left.\substack{z_{1}\\
\alpha_{1}
}
\right|\cdots\left|\substack{z_{i-1}\\
\alpha_{i-1}
}
\right|\left.\substack{z_{i+1}\\
\alpha_{i+1}
}
\right|\cdots\left|\substack{z_{n+1}\\
\alpha_{n+1}
}
\right.\right)\cdot\chi\left(\left.\substack{z_{i-1}\\
\alpha_{i-1}
}
\right|\substack{z_{i}\\
\alpha_{i}
}
\left|\substack{z_{i+1}\\
\alpha_{i+1}
}
\right.\right)\cdot\mathcal{D}\log\left(\frac{z_{i}-z_{i+1}}{z_{i}-z_{i-1}}\right).
\]
Now, $\mathcal{D}\mathscr{B}_{\gamma}^{\bullet,\circ}\left(\left.\substack{z_{0}\\
\alpha_{0}
}
\right|\left.\substack{z_{1}\\
\alpha_{1}
}
\right|\cdots\left|\substack{z_{n+1}\\
\alpha_{n+1}
}
\right.\right)$ is equal to
\[
\sum_{i=0}^{n}I_{\gamma}(p;\omega_{0,1},\dots,\omega_{i-1,i},\mathcal{D}\omega_{i,i+1},\omega_{i+1,i+2}\dots,\omega_{n,n+1};q).
\]
Here, the terms that come from the differentiation of the upper and
lower limits of the iterated integral vanish since
\[
\lim_{t\to z_{0}}g_{0,1}(\bm{z},t)=0\qquad\left(\Re(\alpha_{0})<0\right)
\]
and
\[
\lim_{t\to z_{n+1}}g_{n,n+1}(\bm{z},t)=0\qquad\left(\Re(1-\alpha_{n+1})<0\right).
\]
Here, we have
\begin{align*}
\mathcal{D}g_{i,j} & =c_{i}\frac{\partial}{\partial z_{i}}g_{i,j}+c_{j}\frac{\partial}{\partial z_{j}}g_{i,j}\\
 & =c_{i}\left(\frac{\alpha_{i}}{t-z_{i}}+\frac{\alpha_{i}-\alpha_{j}}{z_{i}-z_{j}}\right)g_{i,j}+c_{j}\left(\frac{1-\alpha_{j}}{t-z_{j}}+\frac{\alpha_{i}-\alpha_{j}}{z_{j}-z_{i}}\right)g_{i,j}\\
 & =-c_{i}\left(\frac{-\alpha_{i}}{t-z_{i}}+\frac{\alpha_{j}}{t-z_{j}}\right)\frac{t-z_{j}}{z_{i}-z_{j}}g_{i,j}+c_{j}\left(\frac{1-\alpha_{i}}{t-z_{i}}-\frac{1-\alpha_{j}}{t-z_{j}}\right)\frac{t-z_{i}}{z_{i}-z_{j}}g_{i,j}\\
 & =\frac{\partial f_{i,j}}{\partial t}
\end{align*}
where 
\[
f_{i,j}\coloneqq\left(-\frac{t-z_{j}}{z_{i}-z_{j}}c_{i}+\frac{t-z_{i}}{z_{i}-z_{j}}c_{j}\right)g_{i,j},
\]
and thus,
\begin{align*}
 & \mathcal{D}\mathscr{B}_{\gamma}^{\bullet,\circ}\left(\left.\substack{z_{0}\\
\alpha_{0}
}
\right|\left.\substack{z_{1}\\
\alpha_{1}
}
\right|\cdots\left|\substack{z_{n+1}\\
\alpha_{n+1}
}
\right.\right)\\
 & =\sum_{i=0}^{n}I_{\gamma}(p;\omega_{0,1},\dots,\omega_{i-1,i},\frac{\partial f_{i,i+1}}{\partial t}dt,\omega_{i+1,i+2},\dots,\omega_{n,n+1};q).
\end{align*}
Since ${\displaystyle \lim_{t\rightarrow0}f_{0,1}(\gamma(t))=\lim_{t\rightarrow1}f_{n,n+1}(\gamma(t))=0}$,
we have
\begin{align*}
 & \sum_{i=0}^{n}I_{\gamma}(p;\omega_{0,1},\dots,\omega_{i-1,i},\frac{\partial f_{i,i+1}}{\partial t}dt,\omega_{i+1,i+2},\dots,\omega_{n,n+1};q)\\
 & =\sum_{i=0}^{n-1}I_{\gamma}(p;\omega_{0,1},\dots,\omega_{i-1,i},f_{i,i+1}\omega_{i+1,i+2},\omega_{i+2,i+3},\dots,\omega_{n,n+1};q)\\
 & \ -\sum_{i=1}^{n}I_{\gamma}(p;\omega_{0,1},\dots,\omega_{i-2,i-1},\omega_{i-1,i}f_{i,i+1},\omega_{i+1,i+2},\dots,\omega_{n,n+1};q)\\
 & =\sum_{i=1}^{n}I_{\gamma}(p;\omega_{0,1},\dots,\omega_{i-2,i-1},f_{i-1,i}\omega_{i,i+1}-\omega_{i-1,i}f_{i,i+1},\omega_{i+1,i+2}\dots,\omega_{n,n+1};q).
\end{align*}
Noting $\omega_{i,j}=g_{i,j}dt$ etc., 
\begin{align*}
f_{i,j}\cdot\omega_{j,k}-\omega_{i,j}\cdot f_{j,k} & =\left(-\frac{t-z_{j}}{z_{i}-z_{j}}c_{i}+\frac{t-z_{i}}{z_{i}-z_{j}}c_{j}\right)g_{i,j}\cdot g_{j,k}dt-g_{i,j}dt\cdot\left(-\frac{t-z_{k}}{z_{j}-z_{k}}c_{j}+\frac{t-z_{j}}{z_{j}-z_{k}}c_{k}\right)g_{j,k}\\
 & =\left(-\frac{t-z_{j}}{z_{i}-z_{j}}c_{i}+\left(\frac{t-z_{i}}{z_{i}-z_{j}}+\frac{t-z_{k}}{z_{j}-z_{k}}\right)c_{j}-\frac{t-z_{j}}{z_{j}-z_{k}}c_{k}\right)g_{i,j}g_{j,k}dt\\
 & =\left(-\frac{c_{i}-c_{j}}{z_{i}-z_{j}}+\frac{c_{j}-c_{k}}{z_{j}-z_{k}}\right)(t-z_{j})\frac{g_{i,j}g_{j,k}}{g_{i,k}}\omega_{i,k}\\
 & =\mathcal{D}\left(\log\left(\frac{z_{j}-z_{k}}{z_{j}-z_{i}}\right)\right)\chi\left(\left.\substack{z_{i}\\
\alpha_{i}
}
\right|\substack{z_{j}\\
\alpha_{j}
}
\left|\substack{z_{k}\\
\alpha_{k}
}
\right.\right)\omega_{i,k}.
\end{align*}
Hence, 
\begin{align*}
 & \sum_{i=1}^{n}I_{\gamma}(p;\omega_{0,1},\dots,\omega_{i-2,i-1},f_{i-1,i}\omega_{i,i+1}-\omega_{i-1,i}f_{i,i+1},\omega_{i+1,i+2}\dots,\omega_{n,n+1};q).\\
= & \sum_{i=1}^{n}I_{\gamma}(p;\omega_{0,1},\dots,\omega_{i-2,i-1},\omega_{i-1,i+1},\omega_{i+1,i+2}\dots,\omega_{n,n+1};q)\cdot\chi\left(\left.\substack{z_{i-1}\\
\alpha_{i-1}
}
\right|\substack{z_{i}\\
\alpha_{i}
}
\left|\substack{z_{i+1}\\
\alpha_{i+1}
}
\right.\right)\cdot\mathcal{D}\left(\log\left(\frac{z_{i}-z_{i+1}}{z_{i}-z_{i-1}}\right)\right),
\end{align*}
which completes the proof.
\end{proof}
\begin{proof}[Proof of Theorem \ref{thm:Differential_eq_normalized_B}]
The theorem immediately follows from Theorem \ref{thm:Differential_eq_normalized_B_formal}
by noting $\Delta\bigl(\left.\substack{z_{0}\\
\alpha_{0}
}
\right|\left.\substack{z_{1}\\
\alpha_{1}
}
\right|\dots\left|\substack{z_{n+1}\\
\alpha_{n+1}
}
\right.\bigr)(\varphi)=g^{-1}d(g\varphi)$ where $g=\prod_{i\in\mathbb{Z}/(n+2)}(z_{i}-z_{i+1})^{\alpha_{i}-\alpha_{i+1}}$.
\end{proof}

\section{Translation invariance\label{sec:Translation-invariance}}

The highlight of the iterated beta integral is the following translation
invariance.
\begin{thm}
\label{thm: translation invariance}Let $n\geq0$, $z_{0},\dots,z_{n+1}\in\mathbb{C}$,
and $\gamma$ be a nontrivial path from $p\in\{z_{0},\infty\}$ to
$q\in\{z_{n+1},\infty\}$ on $\mathbb{C}\setminus\{z_{0},\dots,z_{n+1}\}$
such that $z_{1},\dots,z_{n}$ belong to the same connected component
of $\mathbb{P}^{1}\setminus\gamma$ as one of the points in $\{z_{0},z_{n+1},\infty\}\setminus\{p,q\}$.
Then, $\hat{B}_{\gamma}^{\bullet,\circ}\bigl(\left.\substack{z_{0}\\
\alpha_{0}
}
\right|\left.\substack{z_{1}\\
\alpha_{1}
}
\right|\dots\left|\substack{z_{n+1}\\
\alpha_{n+1}
}
\right.\bigr)$ is invariant under the translation of the exponent parameters, i.e.,
\[
\hat{B}_{\gamma}^{\bullet,\circ}\bigl(\left.\substack{z_{0}\\
\alpha_{0}
}
\right|\left.\substack{z_{1}\\
\alpha_{1}
}
\right|\dots\left|\substack{z_{n+1}\\
\alpha_{n+1}
}
\right.\bigr)=\hat{B}_{\gamma}^{\bullet,\circ}\bigl(\left.\substack{z_{0}\\
\alpha_{0}+\lambda
}
\right|\left.\substack{z_{1}\\
\alpha_{1}+\lambda
}
\right|\dots\left|\substack{z_{n+1}\\
\alpha_{n+1}+\lambda
}
\right.\bigr)
\]
for $\lambda\in\mathbb{C}$. 
\end{thm}

\begin{rem}
By multiplying $\frac{\prod_{i=0}^{n}\left(z_{i}-z_{i+1}\right)^{\alpha_{i}-\alpha_{i+1}}}{\left(z_{0}-z_{n+1}\right)^{\alpha_{0}-\alpha_{n+1}}}$
to both sides, the claim is equivalent to the translation invariance
of $\hat{\mathscr{B}}_{\gamma}^{\bullet,\circ}$, i.e.,
\[
\hat{\mathscr{B}}_{\gamma}^{\bullet,\circ}\bigl(\left.\substack{z_{0}\\
\alpha_{0}
}
\right|\left.\substack{z_{1}\\
\alpha_{1}
}
\right|\dots\left|\substack{z_{n+1}\\
\alpha_{n+1}
}
\right.\bigr)=\hat{\mathscr{B}}_{\gamma}^{\bullet,\circ}\bigl(\left.\substack{z_{0}\\
\alpha_{0}+\lambda
}
\right|\left.\substack{z_{1}\\
\alpha_{1}+\lambda
}
\right|\dots\left|\substack{z_{n+1}\\
\alpha_{n+1}+\lambda
}
\right.\bigr)\quad\left(\lambda\in\mathbb{C}\right).
\]
\end{rem}

\begin{proof}[Proof of Theorem \ref{thm: translation invariance}]
Put $\alpha_{i}'\coloneqq\alpha_{i}+\lambda$,
\[
z_{\bullet}\coloneqq\begin{cases}
z_{0} & \text{if }\bullet={\rm f}\\
\infty & \text{if }\bullet=\infty
\end{cases}\quad\text{and }z_{\circ}\coloneqq\begin{cases}
z_{n+1} & \text{if }\circ={\rm f}\\
\infty & \text{if }\circ=\infty.
\end{cases}
\]
As in the proof of Theorem \ref{thm:Differential_eq_normalized_B},
we put 
\[
g_{i,j}(\bm{z},t)=(t-z_{i})^{-\alpha_{i}}(t-z_{j})^{\alpha_{j}-1}(z_{i}-z_{j})^{\alpha_{i}-\alpha_{j}}
\]
so that
\[
\omega_{i,j}(t):=g_{i,j}(\bm{z},t)dt=\left\{ \substack{z_{i},z_{j}\\
\alpha_{i},\alpha_{j}
}
\right\} (t).
\]
We will prove the claim by induction on $n$. The case $n=0$ is obvious,
since both sides are equal to $1$. Assume $n>0$. Then, by Theorem
\ref{thm:Differential_eq_normalized_B_formal} and the induction hypothesis,
\begin{align*}
 & d\left(\hat{\mathscr{B}}_{\gamma}^{\bullet,\circ}\left(\left.\substack{z_{0}\\
\alpha_{0}
}
\right|\left.\substack{z_{1}\\
\alpha_{1}
}
\right|\cdots\left|\substack{z_{n+1}\\
\alpha_{n+1}
}
\right.\right)-\hat{\mathscr{B}}_{\gamma}^{\bullet,\circ}\left(\left.\substack{z_{0}\\
\alpha_{0}+\lambda
}
\right|\left.\substack{z_{1}\\
\alpha_{1}+\lambda
}
\right|\cdots\left|\substack{z_{n+1}\\
\alpha_{n+1}+\lambda
}
\right.\right)\right)\\
 & =\sum_{i=1}^{n}\hat{\mathscr{B}}_{\gamma}^{\bullet,\circ}\left(\left.\substack{z_{0}\\
\alpha_{0}
}
\right|\left.\substack{z_{1}\\
\alpha_{1}
}
\right|\dots\widehat{\left|\substack{z_{i}\\
\alpha_{i}
}
\right|}\cdots\left|\substack{z_{n}\\
\alpha_{n}
}
\right.\left|\substack{z_{n+1}\\
\alpha_{n+1}
}
\right.\right)\cdot d\hat{\mathscr{B}}\bigl(\left.\substack{z_{i-1}\\
\alpha_{i-1}
}
\right|\substack{z_{i}\\
\alpha_{i}
}
\left|\substack{z_{i+1}\\
\alpha_{i+1}
}
\right.\bigr)\\
 & \quad-\sum_{i=1}^{n}\hat{\mathscr{B}}_{\gamma}^{\bullet,\circ}\left(\left.\substack{z_{0}\\
\alpha_{0}'
}
\right|\left.\substack{z_{1}\\
\alpha_{1}'
}
\right|\dots\widehat{\left|\substack{z_{i}\\
\alpha_{i}'
}
\right|}\cdots\left|\substack{z_{n}\\
\alpha_{n}'
}
\right.\left|\substack{z_{n+1}\\
\alpha_{n+1}'
}
\right.\right)\cdot d\hat{\mathscr{B}}\bigl(\left.\substack{z_{i-1}\\
\alpha_{i-1}'
}
\right|\substack{z_{i}\\
\alpha_{i}'
}
\left|\substack{z_{i+1}\\
\alpha_{i+1}'
}
\right.\bigr)\\
 & =\sum_{i=1}^{n}\hat{\mathscr{B}}_{\gamma}^{\bullet,\circ}\left(\left.\substack{z_{0}\\
\alpha_{0}
}
\right|\left.\substack{z_{1}\\
\alpha_{1}
}
\right|\dots\widehat{\left|\substack{z_{i}\\
\alpha_{i}
}
\right|}\cdots\left|\substack{z_{n}\\
\alpha_{n}
}
\right.\left|\substack{z_{n+1}\\
\alpha_{n+1}
}
\right.\right)\cdot\left(d\hat{\mathscr{B}}\bigl(\left.\substack{z_{i-1}\\
\alpha_{i-1}
}
\right|\substack{z_{i}\\
\alpha_{i}
}
\left|\substack{z_{i+1}\\
\alpha_{i+1}
}
\right.\bigr)-d\hat{\mathscr{B}}\bigl(\left.\substack{z_{i-1}\\
\alpha_{i-1}'
}
\right|\substack{z_{i}\\
\alpha_{i}'
}
\left|\substack{z_{i+1}\\
\alpha_{i+1}'
}
\right.\bigr)\right).
\end{align*}
Since $\chi\left(\left.\substack{z_{i-1}\\
\alpha_{i-1}
}
\right|\substack{z_{i}\\
\alpha_{i}
}
\left|\substack{z_{i+1}\\
\alpha_{i+1}
}
\right.\right)=\chi\left(\left.\substack{z_{i-1}\\
\alpha_{i-1}'
}
\right|\substack{z_{i}\\
\alpha_{i}'
}
\left|\substack{z_{i+1}\\
\alpha_{i+1}'
}
\right.\right)$,
\[
d\hat{\mathscr{B}}\bigl(\left.\substack{z_{i-1}\\
\alpha_{i-1}
}
\right|\substack{z_{i}\\
\alpha_{i}
}
\left|\substack{z_{i+1}\\
\alpha_{i+1}
}
\right.\bigr)-d\hat{\mathscr{B}}\bigl(\left.\substack{z_{i-1}\\
\alpha_{i-1}'
}
\right|\substack{z_{i}\\
\alpha_{i}'
}
\left|\substack{z_{i+1}\\
\alpha_{i+1}'
}
\right.\bigr)=\left(\chi\left(\left.\substack{z_{i-1}\\
\alpha_{i-1}
}
\right|\substack{z_{i}\\
\alpha_{i}
}
\left|\substack{z_{i+1}\\
\alpha_{i+1}
}
\right.\right)-\chi\left(\left.\substack{z_{i-1}\\
\alpha_{i-1}'
}
\right|\substack{z_{i}\\
\alpha_{i}'
}
\left|\substack{z_{i+1}\\
\alpha_{i+1}'
}
\right.\right)\right)d\log\left(\frac{z_{i}-z_{i+1}}{z_{i}-z_{i-1}}\right)=0
\]
by (\ref{eq:diff_formula_shortest-1}). It follows that
\[
d\left(\hat{\mathscr{B}}_{\gamma}^{\bullet,\circ}\left(\left.\substack{z_{0}\\
\alpha_{0}
}
\right|\left.\substack{z_{1}\\
\alpha_{1}
}
\right|\cdots\left|\substack{z_{n+1}\\
\alpha_{n+1}
}
\right.\right)-\hat{\mathscr{B}}_{\gamma}^{\bullet,\circ}\left(\left.\substack{z_{0}\\
\alpha_{0}'
}
\right|\left.\substack{z_{1}\\
\alpha_{1}'
}
\right|\cdots\left|\substack{z_{n+1}\\
\alpha_{n+1}'
}
\right.\right)\right)=0,
\]
and thus $\hat{\mathscr{B}}_{\gamma}^{\bullet,\circ}\left(\left.\substack{z_{0}\\
\alpha_{0}
}
\right|\left.\substack{z_{1}\\
\alpha_{1}
}
\right|\cdots\left|\substack{z_{n+1}\\
\alpha_{n+1}
}
\right.\right)-\hat{\mathscr{B}}_{\gamma}^{\bullet,\circ}\left(\left.\substack{z_{0}\\
\alpha_{0}'
}
\right|\left.\substack{z_{1}\\
\alpha_{1}'
}
\right|\cdots\left|\substack{z_{n+1}\\
\alpha_{n+1}'
}
\right.\right)$ does not depend on $z_{0},z_{1},\ldots,z_{n+1}$. For $i\in\{1,\ldots,n\}$,
consider the limit as $z_{i}\rightarrow x$ for $x\in\{z_{i-1},z_{i+1},\infty\}\setminus\{z_{\bullet},z_{\circ}\}$
(we can take this limit without deforming the path $\gamma$, since
$\gamma$ does not enclose any of $z_{1},\ldots,z_{n}$). By the meromorphy
of $\hat{\mathscr{B}}_{\gamma}^{\bullet,\circ}$ in $\boldsymbol{\alpha}=(\alpha_{0},\alpha_{1},\ldots,\alpha_{n+1})\in\mathbb{C}^{n+2}$,
it suffices to show that this limit is zero for $\boldsymbol{\alpha}$
in some open subset of $\mathbb{C}^{n+2}$ by the identity theorem.
To see $\hat{\mathscr{B}}_{\gamma}^{\bullet,\circ}\left(\left.\substack{z_{0}\\
\alpha_{0}
}
\right|\left.\substack{z_{1}\\
\alpha_{1}
}
\right|\cdots\left|\substack{z_{n+1}\\
\alpha_{n+1}
}
\right.\right)=\hat{\mathscr{B}}_{\gamma}^{\bullet,\circ}\left(\left.\substack{z_{0}\\
\alpha_{0}'
}
\right|\left.\substack{z_{1}\\
\alpha_{1}'
}
\right|\cdots\left|\substack{z_{n+1}\\
\alpha_{n+1}'
}
\right.\right)$, notice that the only parts of $\hat{\mathscr{B}}_{\gamma}^{\bullet,\circ}\left(\left.\substack{z_{0}\\
\alpha_{0}
}
\right|\left.\substack{z_{1}\\
\alpha_{1}
}
\right|\cdots\left|\substack{z_{n+1}\\
\alpha_{n+1}
}
\right.\right)$ that depend on $z_{i}$ are $\left\{ \substack{z_{i-1},z_{i}\\
\alpha_{i-1},\alpha_{i}
}
\right\} (t)$ and $\left\{ \substack{z_{i},z_{i+1}\\
\alpha_{i},\alpha_{i+1}
}
\right\} (t)$, where
\[
g_{i-1,i}(\boldsymbol{z},t_{i})g_{i,i+1}(\boldsymbol{z},t_{i+1})=\begin{cases}
O\left(\left(\frac{1}{z_{i}}\right)^{1-\alpha_{i-1}+\alpha_{i+1}}\right) & \text{when }z_{i}\rightarrow\infty\\
O((z_{i}-z_{i-1})^{\alpha_{i-1}-\alpha_{i}}) & \text{when }z_{i}\rightarrow z_{i-1}\\
O((z_{i}-z_{i+1})^{\alpha_{i}-\alpha_{i+1}}) & \text{when }z_{i}\rightarrow z_{i+1}.
\end{cases}
\]
Thus, we find that both $\hat{\mathscr{B}}_{\gamma}^{\bullet,\circ}\left(\left.\substack{z_{0}\\
\alpha_{0}
}
\right|\left.\substack{z_{1}\\
\alpha_{1}
}
\right|\cdots\left|\substack{z_{n+1}\\
\alpha_{n+1}
}
\right.\right)$ and $\hat{\mathscr{B}}_{\gamma}^{\bullet,\circ}\left(\left.\substack{z_{0}\\
\alpha_{0}'
}
\right|\left.\substack{z_{1}\\
\alpha_{1}'
}
\right|\cdots\left|\substack{z_{n+1}\\
\alpha_{n+1}'
}
\right.\right)$ behave as
\[
\begin{cases}
O\left(\left(\frac{1}{z_{i}}\right)^{1-\alpha_{i-1}+\alpha_{i+1}}\right)\quad\left(z_{i}\rightarrow\infty\right) & \text{if }\infty\notin\{z_{\bullet},z_{\circ}\}\\
O((z_{i}-z_{i-1})^{\alpha_{i-1}-\alpha_{i}})\quad\left(z_{i}\rightarrow z_{i-1}\right) & \text{if }z_{i-1}\notin\{z_{\bullet},z_{\circ}\}\\
O((z_{i}-z_{i+1})^{\alpha_{i}-\alpha_{i+1}})\quad\left(z_{i}\rightarrow z_{i+1}\right) & \text{if }z_{i+1}\notin\{z_{\bullet},z_{\circ}\}.
\end{cases}
\]
Therefore,
\[
\lim_{z_{i}\rightarrow x}\hat{\mathscr{B}}_{\gamma}^{\bullet,\circ}\left(\left.\substack{z_{0}\\
\alpha_{0}
}
\right|\left.\substack{z_{1}\\
\alpha_{1}
}
\right|\cdots\left|\substack{z_{n+1}\\
\alpha_{n+1}
}
\right.\right)=\lim_{z_{i}\rightarrow x}\hat{\mathscr{B}}_{\gamma}^{\bullet,\circ}\left(\left.\substack{z_{0}\\
\alpha_{0}'
}
\right|\left.\substack{z_{1}\\
\alpha_{1}'
}
\right|\cdots\left|\substack{z_{n+1}\\
\alpha_{n+1}'
}
\right.\right)=0
\]
if $\boldsymbol{\alpha}$ lies in the ranges
\[
\begin{cases}
\Re(1-\alpha_{i-1}+\alpha_{i+1})>0 & x=\infty,\\
\Re(\alpha_{i-1}-\alpha_{i})>0 & x=z_{i-1},\\
\Re(\alpha_{i}-\alpha_{i+1})>0 & x=z_{i+1},
\end{cases}
\]
respectively. Hence, we conclude that 
\[
\hat{\mathscr{B}}_{\gamma}^{\bullet,\circ}\left(\left.\substack{z_{0}\\
\alpha_{0}
}
\right|\left.\substack{z_{1}\\
\alpha_{1}
}
\right|\cdots\left|\substack{z_{n+1}\\
\alpha_{n+1}
}
\right.\right)=\hat{\mathscr{B}}_{\gamma}^{\bullet,\circ}\left(\left.\substack{z_{0}\\
\alpha_{0}'
}
\right|\left.\substack{z_{1}\\
\alpha_{1}'
}
\right|\cdots\left|\substack{z_{n+1}\\
\alpha_{n+1}'
}
\right.\right)
\]
for $\boldsymbol{\alpha}$ in the aforementioned ranges. This completes
the proof.
\end{proof}
For a simple path $\gamma$, Theorem \ref{thm: translation invariance}
may also be stated in the following manner.
\begin{cor}
\label{cor:simple_path_case}Let $n\geq0$, $z_{0},\dots,z_{n+1}\in\mathbb{C}$.
Then:
\begin{enumerate}
\item Let $\gamma$ be a simple path from $z_{0}$ to $z_{n+1}$ on $\mathbb{C}\setminus\{z_{0},\dots,z_{n+1}\}$.
Then,
\[
\frac{(-1)^{\alpha_{0}}}{\Gamma(1-\alpha_{0})\Gamma(\alpha_{n+1})}B_{\gamma}^{{\rm f},{\rm f}}\bigl(\left.\substack{z_{0}\\
\alpha_{0}
}
\right|\left.\substack{z_{1}\\
\alpha_{1}
}
\right|\dots\left|\substack{z_{n+1}\\
\alpha_{n+1}
}
\right.\bigr)=\frac{(-1)^{\alpha_{0}+\lambda}}{\Gamma(1-\alpha_{0}-\lambda)\Gamma(\alpha_{n+1}+\lambda)}B_{\gamma}^{{\rm f},{\rm f}}\bigl(\left.\substack{z_{0}\\
\alpha_{0}+\lambda
}
\right|\left.\substack{z_{1}\\
\alpha_{1}+\lambda
}
\right|\dots\left|\substack{z_{n+1}\\
\alpha_{n+1}+\lambda
}
\right.\bigr)
\]
for $\lambda\in\mathbb{C}$. 
\item Let $\gamma$ be a simple path from $\infty$ to $z_{n+1}$ on $\mathbb{C}\setminus\{z_{0},\dots,z_{n+1}\}$
. Then, 
\[
\frac{\Gamma(\alpha_{0})}{\Gamma(\alpha_{n+1})}B_{\gamma}^{\infty,{\rm f}}\bigl(\left.\substack{z_{0}\\
\alpha_{0}
}
\right|\left.\substack{z_{1}\\
\alpha_{1}
}
\right|\dots\left|\substack{z_{n+1}\\
\alpha_{n+1}
}
\right.\bigr)=\frac{\Gamma(\alpha_{0}+\lambda)}{\Gamma(\alpha_{n+1}+\lambda)}B_{\gamma}^{\infty,{\rm f}}\bigl(\left.\substack{z_{0}\\
\alpha_{0}+\lambda
}
\right|\left.\substack{z_{1}\\
\alpha_{1}+\lambda
}
\right|\dots\left|\substack{z_{n+1}\\
\alpha_{n+1}+\lambda
}
\right.\bigr)
\]
for $\lambda\in\mathbb{C}$. 
\end{enumerate}
\end{cor}

\begin{proof}
Since we have
\begin{align*}
\mathscr{B}_{\gamma}^{{\rm f},{\rm f}}\left(\substack{z_{0}\\
\alpha_{0}
}
\left|\substack{z_{n+1}\\
\alpha_{n+1}
}
\right.\right) & =(-1)^{1-\alpha_{0}}{\rm B}(1-\alpha_{0},\alpha_{n+1}),\\
\mathscr{B}_{\gamma}^{\infty,{\rm f}}\left(\substack{z_{0}\\
\alpha_{0}
}
\left|\substack{z_{n+1}\\
\alpha_{n+1}
}
\right.\right) & =(-1)^{1-\alpha_{0}+\alpha_{n+1}}{\rm B}(\alpha_{0}-\alpha_{n+1},\alpha_{n+1})
\end{align*}
by Proposition \ref{prop:is_beta_function}, the claims follow from
Theorem \ref{thm: translation invariance}.
\end{proof}
\begin{cor}
\label{cor:diagonal_reduces_to_hyperlogarithm}Let $z_{0},\dots,z_{n+1}\in\mathbb{C}$. 
\begin{enumerate}
\item Let $\gamma$ be a simple path from $z_{0}$ to $z_{n+1}$ on $\mathbb{C}\setminus\{z_{0},\dots,z_{n+1}\}$.
Then,
\[
\frac{(-1)^{\alpha}\sin(\pi\alpha)}{\pi}B_{\gamma}^{{\rm f},{\rm f}}\bigl(\left.\substack{z_{0}\\
\alpha
}
\right|\left.\substack{z_{1}\\
\alpha
}
\right|\dots\left|\substack{z_{n+1}\\
\alpha
}
\right.\bigr)=I_{\gamma}\left(z_{0};e_{z_{1}}\cdots e_{z_{n}};z_{n+1}\right)
\]
for $\alpha\in\mathbb{C}$. 
\item Let $\gamma$ be a simple path from $\infty$ to $z_{n+1}$ on $\mathbb{C}\setminus\{z_{0},\dots,z_{n+1}\}$
. Then, 
\[
\alpha(z_{1}-z_{0})B_{\gamma}^{\infty,{\rm f}}\bigl(\left.\substack{z_{0}\\
\alpha+1
}
\right|\left.\substack{z_{1}\\
\alpha
}
\right|\dots\left|\substack{z_{n+1}\\
\alpha
}
\right.\bigr)=I_{\gamma}\left(\infty;(e_{z_{0}}-e_{z_{1}})e_{z_{2}}\cdots e_{z_{n}};z_{n+1}\right)
\]
for $\alpha\in\mathbb{C}$. 
\end{enumerate}
\end{cor}

\begin{proof}
Consider the case $(\alpha_{0},\alpha_{1},\ldots,\alpha_{n+1})=(\alpha,\alpha,\ldots,\alpha)$
in (1) and the case $(\alpha_{0},\alpha_{1},\ldots,\alpha_{n+1})=(\alpha+1,\alpha,\ldots,\alpha)$
in (2) of Corollary \ref{cor:simple_path_case}. Since 
\[
\frac{1}{\Gamma(1-\alpha)\Gamma(\alpha)}=\frac{\sin(\pi\alpha)}{\pi}\quad\text{and }\frac{\Gamma(\alpha+1)}{\Gamma(\alpha)}=\alpha,
\]
Corollary \ref{cor:simple_path_case} says that the quantities
\[
\frac{(-1)^{\alpha}\sin(\pi\alpha)}{\pi}B_{\gamma}^{{\rm f},{\rm f}}\bigl(\left.\substack{z_{0}\\
\alpha
}
\right|\left.\substack{z_{1}\\
\alpha
}
\right|\dots\left|\substack{z_{n+1}\\
\alpha
}
\right.\bigr)
\]
and
\[
\alpha B_{\gamma}^{\infty,{\rm f}}\bigl(\left.\substack{z_{0}\\
\alpha+1
}
\right|\left.\substack{z_{1}\\
\alpha
}
\right|\dots\left|\substack{z_{n+1}\\
\alpha
}
\right.\bigr)
\]
do not depend on $\alpha\in\mathbb{C}$. By the residue formula (1)
of Proposition \ref{thm: complete vs incomplete}, 
\[
\lim_{\alpha_{n+1}\rightarrow0}\alpha_{n+1}B_{\gamma}^{\bullet,\mathrm{f}}\left(\left.\substack{z_{0}\\
\alpha_{0}
}
\right|\left.\substack{z_{1}\\
\alpha_{1}
}
\right|\cdots\left|\substack{z_{n+1}\\
\alpha_{n+1}
}
\right.\right)=(z_{n+1}-z_{n})^{-\alpha_{n}}B_{\gamma}^{\bullet}\left(\left.\substack{z_{0}\\
\alpha_{0}
}
\right|\left.\substack{z_{1}\\
\alpha_{1}
}
\right|\cdots\left|\substack{z_{n}\\
\alpha_{n}
}
\right.;z_{n+1}\right).
\]
Thus,
\begin{align*}
\lim_{\alpha\rightarrow0}\frac{(-1)^{\alpha}\sin(\pi\alpha)}{\pi}B_{\gamma}^{{\rm f},{\rm f}}\bigl(\left.\substack{z_{0}\\
\alpha
}
\right|\left.\substack{z_{1}\\
\alpha
}
\right|\dots\left|\substack{z_{n+1}\\
\alpha
}
\right.\bigr) & =\lim_{\alpha\rightarrow0}\frac{(-1)^{\alpha}\sin(\pi\alpha)}{\pi\alpha}\cdot\lim_{\alpha\rightarrow0}\alpha B_{\gamma}^{{\rm f},{\rm f}}\bigl(\left.\substack{z_{0}\\
\alpha
}
\right|\left.\substack{z_{1}\\
\alpha
}
\right|\dots\left|\substack{z_{n+1}\\
\alpha
}
\right.\bigr)\\
 & =B_{\gamma}^{{\rm f}}\left(\left.\substack{z_{0}\\
0
}
\right|\cdots\left|\substack{z_{n}\\
0
}
\right.;z_{n+1}\right)\\
 & =I_{\gamma}\left(z_{0};e_{z_{1}}\cdots e_{z_{n}};z_{n+1}\right),
\end{align*}
which proves (1). For (2), we have
\begin{align*}
\alpha B_{\gamma}^{\infty,{\rm f}}\bigl(\left.\substack{z_{0}\\
\alpha+1
}
\right|\left.\substack{z_{1}\\
\alpha
}
\right|\dots\left|\substack{z_{n+1}\\
\alpha
}
\right.\bigr) & =B_{\gamma}^{\infty,{\rm f}}\bigl(\left.\substack{z_{0}\\
2
}
\right|\left.\substack{z_{1}\\
1
}
\right|\dots\left|\substack{z_{n+1}\\
1
}
\right.\bigr)\qquad(\text{by the independence on }\alpha)\\
 & =I_{\gamma}\bigl(\infty;\frac{dt}{(t-z_{0})^{2}},e_{z_{1}},\ldots,e_{z_{n}};z_{n+1}\bigr)\\
 & =-I_{\gamma}\bigl(\infty;\left(\frac{\partial}{\partial t}\frac{1}{t-z_{0}}\right)dt,e_{z_{1}},\ldots,e_{z_{n}};z_{n+1}\bigr)\\
 & =-I_{\gamma}\bigl(\infty;\frac{1}{t-z_{0}}e_{z_{1}},e_{z_{2}}\ldots,e_{z_{n}};z_{n+1}\bigr).
\end{align*}
Since 
\[
\frac{1}{t-z_{0}}e_{z_{1}}=\frac{1}{z_{0}-z_{1}}\left(e_{z_{0}}-e_{z_{1}}\right),
\]
we obtain the claim. 
\end{proof}

\section{Series expansions\label{sec:Series-expansions}}

In this section, we give a power series expansion for the iterated
beta integrals $\hat{B}_{\mathrm{dch}}^{{\rm f},{\rm f}}\left(\left.\substack{z_{0}\\
\alpha_{0}
}
\right|\left.\substack{z_{1}\\
\alpha_{1}
}
\right|\cdots\left|\substack{z_{k+1}\\
\alpha_{k+1}
}
\right.\right)$ and $\hat{B}_{\mathrm{ray}}^{\infty,{\rm f}}\!\left(\left.\substack{z_{0}\\
\alpha_{0}
}
\right|\left.\substack{z_{1}\\
\alpha_{1}
}
\right|\cdots\left|\substack{z_{k+1}\\
\alpha_{k+1}
}
\right.\right)$ under some conditions. Here, $\mathrm{dch}$ denotes a finite straight
line and $\mathrm{ray}$ denotes a straight line from $\infty$ to
a nonzero complex number $z_{k+1}$ along the half line $z_{k+1}\mathbb{R}_{\geq1}$. 

\subsection{Series expansion for $\hat{B}^{{\rm f},{\rm f}}$}

For $n\geq0$, define $c(n,z)\in\mathbb{Q}(z)$ by 
\[
\frac{t}{t-z}=\sum_{n=0}^{\infty}c(n,z)t^{n}\in\mathbb{Q}(z)[[t]].
\]
In other words,
\[
c(n,z)=\begin{cases}
-z^{-n} & z\neq0,n>0\\
0 & z\neq0,n=0\\
\delta_{n,0} & z=0.
\end{cases}
\]

\begin{lem}
\label{lem:summation_lemma}Let $z,t$ be real numbers and $s,\alpha$
complex numbers satisfying $0<t$, $z\in\{0\}\cup\mathbb{R}_{\leq t}$,
and $\Re(s)>-1$. Furthermore, we assume $\Re(s+\alpha)>0$ if $z=0$.
Then, we have
\[
\frac{1}{(t-z)^{\alpha}}\int_{0}^{t}\frac{u^{s}du}{(u-z)^{1-\alpha}}=\sum_{n=0}^{\infty}c(n,z)\frac{\Gamma(s+1)\Gamma(s+\alpha+n)}{\Gamma(s+\alpha+1)\Gamma(s+n+1)}t^{n+s}.
\]
\end{lem}

\begin{proof}
We may assume $\Re(\alpha)<0$ and $\Re(s+\alpha)>0$ without loss
of generality. Put
\[
I=\int_{0<u<v<t}u^{s+\alpha}(v-u)^{-\alpha-1}(v-z)^{\alpha-1}dudv.
\]
Then, 
\begin{align*}
I= & \int_{0<u<t}u^{s+\alpha}\left(\int_{u<v<t}(v-u)^{-\alpha-1}(v-z)^{\alpha-1}dv\right)du\\
= & \int_{0<u<t}u^{s+\alpha}\frac{1}{\alpha(z-u)}\left[(v-u)^{-\alpha}(v-z)^{\alpha}\right]_{v=u}^{v=t}du\\
= & \frac{(t-z)^{\alpha}}{\alpha}\int_{0<u<t}\frac{u^{s+\alpha}(t-u)^{-\alpha}}{(z-u)}du\\
= & -\frac{(t-z)^{\alpha}}{\alpha}\int_{0<u<t}u^{s+\alpha-1}(t-u)^{-\alpha}\cdot\frac{u}{u-z}du\\
= & -\frac{(t-z)^{\alpha}}{\alpha}\int_{0<u<t}u^{s+\alpha-1}(t-u)^{-\alpha}\sum_{n=0}^{\infty}c(n,z)u^{n}du\\
= & -\frac{(t-z)^{\alpha}}{\alpha}\sum_{n=0}^{\infty}c(n,z)\int_{0<u<t}u^{n+s+\alpha-1}(t-u)^{-\alpha}du\\
= & -\frac{(t-z)^{\alpha}}{\alpha}\sum_{n=0}^{\infty}c(n,z)\frac{\Gamma(n+s+\alpha)\Gamma(1-\alpha)}{\Gamma(n+s+1)}t^{n+s}\\
= & (t-z)^{\alpha}\sum_{n=0}^{\infty}c(n,z)\frac{\Gamma(n+s+\alpha)\Gamma(-\alpha)}{\Gamma(n+s+1)}t^{n+s}.
\end{align*}
On the other hand,
\begin{align*}
I & =\int_{0<u<v<t}u^{s+\alpha}(v-u)^{-\alpha-1}(v-z)^{\alpha-1}dudv.\\
 & =\int_{0<v<t}\left(\int_{0<u<v}u^{s+\alpha}(v-u)^{-\alpha-1}du\right)(v-z)^{\alpha-1}dv.\\
 & =\frac{\Gamma(s+\alpha+1)\Gamma(-\alpha)}{\Gamma(s+1)}\int_{0<v<t}v^{s}(v-z)^{\alpha-1}dv.
\end{align*}
Equating the two expressions, we get 
\[
(t-z)^{-\alpha}\int_{0<v<t}v^{s}(v-z)^{\alpha-1}dv=\sum_{n=0}^{\infty}c(n,z)\frac{\Gamma(n+s+\alpha)\Gamma(s+1)}{\Gamma(n+s+1)\Gamma(s+\alpha+1)}t^{n+s}.
\]
\end{proof}
\begin{thm}
\label{thm:expansion_iterated_beta_finite} If $z_{0}\neq z_{k+1}$
and $z_{1},\dots,z_{k}\in\{z\mid z=z_{0}\ \mathrm{or}\ \left|z-z_{0}\right|>\left|z_{k+1}-z_{0}\right|\}$,
we have
\begin{multline}
\hat{B}_{\mathrm{dch}}^{{\rm f},{\rm f}}\left(\left.\substack{z_{0}\\
\alpha_{0}
}
\right|\left.\substack{z_{1}\\
\alpha_{1}
}
\right|\cdots\left|\substack{z_{k+1}\\
\alpha_{k+1}
}
\right.\right)\\
=\Gamma(1+\alpha_{k+1}-\alpha_{0})\sum_{0=m_{0}\leq\cdots\leq m_{k}}\frac{\prod_{i=1}^{k}c(m_{i}-m_{i-1},z_{i}-z_{0})\Gamma(m_{i}+\alpha_{i}-\alpha_{0})}{\prod_{i=0}^{k}\Gamma(m_{i}+1+\alpha_{i+1}-\alpha_{0})}\left(z_{k+1}-z_{0}\right)^{m_{k}}.\label{eq:SerExpFF}
\end{multline}
Equivalently, 
\begin{equation}
\hat{B}_{\mathrm{dch}}^{{\rm f},{\rm f}}\left(\left.\substack{z_{0}\\
\alpha_{0}
}
\right|\left.\substack{z_{1}\\
\alpha_{1}
}
\right|\cdots\left|\substack{z_{k+1}\\
\alpha_{k+1}
}
\right.\right)=(-1)^{d}\sum_{0=l_{0}<l_{1}<\cdots<l_{d}}\prod_{i=1}^{d}\prod_{l=l_{i-1}+1}^{l_{i}}\left(\frac{l+\alpha_{s_{i}}-\alpha_{0}}{l+\alpha_{s_{d+1}}-\alpha_{0}}\right)\cdot\frac{\prod_{i=1}^{d}\left(x_{i}^{-1}x_{i+1}\right)^{l_{i}}}{\prod_{i=0}^{d}\prod_{\substack{s=s_{i}+\delta_{i,0}}
}^{s_{i+1}-1}(l_{i}+\alpha_{s}-\alpha_{0})}\qquad(s_{0}:=0)\label{eq:SerExpFF2}
\end{equation}
when $(z_{0},\dots,z_{k+1})=(0,\{0\}^{s_{1}-1},x_{1},\{0\}^{s_{2}-s_{1}-1},x_{2},\ldots,\{0\}^{s_{d+1}-s_{d}-1},x_{d+1})$
with $x_{i}\neq0$ ($1\leq i\leq d$) (the equivalence follows from
the invariance of $\hat{B}^{{\rm f},{\rm f}}$ under the simultaneous
affine transformation of $z$-variables). 

\end{thm}

\begin{proof}
We may assume $z_{0}=0$, $z_{k+1}\in\mathbb{R}_{>0}$, and $z_{1},\dots,z_{k}\in\mathbb{R}_{\leq0}$
without loss of generality by the identity theorem. Put $\beta_{i}:=\alpha_{i}-\alpha_{0}$
for $0\leq i\leq k+1$ and
\[
f_{i}(t)=\frac{1}{(t-z_{i})^{\alpha_{i}}}B_{\mathrm{dch}}^{\mathrm{f}}(\left.\substack{z_{0}\\
\alpha_{0}
}
\right|\left.\substack{z_{1}\\
\alpha_{1}
}
\right|\cdots\left|\substack{z_{i}\\
\alpha_{i}
}
\right.;t)
\]
for $i=0,\dots,k$. Then, $f_{i}(t)$ satisfies a recursive formula
\[
f_{i}(t)=\begin{cases}
t^{-\alpha_{0}} & i=0\\
L_{i}(f_{i-1}(t)) & i>0,
\end{cases}
\]
where $L_{i}$ is defined by
\[
L_{i}(f(t))=\frac{1}{(t-z_{i})^{\alpha_{i}}}\int_{0<u<t}\frac{f(u)du}{(u-z_{i})^{1-\alpha_{i}}}.
\]
Then, by Lemma \ref{lem:summation_lemma},
\[
L_{i}(t^{s})=\sum_{n=0}^{\infty}c(n,z_{i})\frac{\Gamma(s+1)\Gamma(s'+\alpha_{i})}{\Gamma(s+\alpha_{i}+1)\Gamma(s'+1)}t^{s'}\qquad\left(s'=s+n\right).
\]
By linearity of $L_{i}$, we have
\begin{align*}
f_{k}(t) & =L_{k}\circ L_{k-1}\circ\cdots\circ L_{1}(t^{-\alpha_{0}})\\
 & =\sum_{n_{1},\dots,n_{k}=0}^{\infty}\left(\prod_{i=1}^{k}c(n_{i},z_{i})\frac{\Gamma(s_{i-1}+1)\Gamma(s_{i}+\alpha_{i})}{\Gamma(s_{i-1}+1+\alpha_{i})\Gamma(s_{i}+1)}\right)t^{s_{k}}\qquad(s_{i}\coloneqq-\alpha_{0}+n_{1}+\cdots+n_{i})\\
 & =\sum_{n_{1},\dots,n_{k}=0}^{\infty}\frac{\Gamma(s_{0}+1)}{\Gamma(s_{k}+1)}\left(\prod_{i=1}^{k}\frac{c(n_{i},z_{i})\Gamma(s_{i}+\alpha_{i})}{\Gamma(s_{i-1}+1+\alpha_{i})}\right)t^{s_{k}}\qquad(s_{i}\coloneqq-\alpha_{0}+n_{1}+\cdots+n_{i})\\
 & =\sum_{0=m_{0}\leq\cdots\leq m_{k}}\frac{\Gamma(1-\alpha_{0})}{\Gamma(m_{k}+1-\alpha_{0})}\left(\prod_{i=1}^{k}\frac{c(m_{i}-m_{i-1},z_{i})\Gamma(m_{i}+\beta_{i})}{\Gamma(m_{i-1}+1+\beta_{i})}\right)t^{m_{k}-\alpha_{0}}\qquad(m_{i}\coloneqq n_{1}+\cdots+n_{i}).
\end{align*}
Since
\begin{align*}
\hat{B}_{\mathrm{dch}}^{{\rm f},{\rm f}}\left(\left.\substack{z_{0}\\
\alpha_{0}
}
\right|\left.\substack{z_{1}\\
\alpha_{1}
}
\right|\cdots\left|\substack{z_{k+1}\\
\alpha_{k+1}
}
\right.\right) & =\frac{B_{\mathrm{dch}}^{\mathrm{f},\mathrm{f}}\left(\left.\substack{z_{0}\\
\alpha_{0}
}
\right|\left.\substack{z_{1}\\
\alpha_{1}
}
\right|\cdots\left|\substack{z_{k+1}\\
\alpha_{k+1}
}
\right.\right)}{B_{\mathrm{dch}}^{\mathrm{f},\mathrm{f}}\left(\substack{z_{0}\\
\alpha_{0}
}
\left|\substack{z_{k+1}\\
\alpha_{k+1}
}
\right.\right)}\\
 & =\frac{\int_{0}^{z_{k+1}}f_{k}(t)(z_{k+1}-t)^{\alpha_{k+1}-1}dt}{\int_{0}^{z_{k+1}}t^{-\alpha_{0}}(z_{k+1}-t)^{\alpha_{k+1}-1}dt}\\
 & =\frac{\Gamma(1+\beta_{k+1})}{\Gamma(1-\alpha_{0})\Gamma(\alpha_{k+1})}z_{k+1}^{-\beta_{k+1}}\int_{0}^{z_{k+1}}\frac{f_{k}(t)dt}{(z_{k+1}-t)^{1-\alpha_{k+1}}},
\end{align*}
it follows that
\begin{align*}
 & \hat{B}_{\mathrm{dch}}^{{\rm f},{\rm f}}\left(\left.\substack{z_{0}\\
\alpha_{0}
}
\right|\left.\substack{z_{1}\\
\alpha_{1}
}
\right|\cdots\left|\substack{z_{k+1}\\
\alpha_{k+1}
}
\right.\right)\\
 & =\frac{\Gamma(1+\beta_{k+1})}{\Gamma(1-\alpha_{0})\Gamma(\alpha_{k+1})}z_{k+1}^{-\beta_{k+1}}\sum_{0=m_{0}\leq\cdots\leq m_{k}}\frac{\Gamma(1-\alpha_{0})}{\Gamma(m_{k}+1-\alpha_{0})}\left(\prod_{i=1}^{k}\frac{c(m_{i}-m_{i-1},z_{i})\Gamma(m_{i}+\beta_{i})}{\Gamma(m_{i-1}+1+\beta_{i})}\right)\\
 & \qquad\qquad\qquad\qquad\qquad\qquad\qquad\qquad\qquad\qquad\qquad\qquad\times\int_{0}^{z_{k+1}}\frac{t^{m_{k}-\alpha_{0}}}{(z_{k+1}-t)^{1-\alpha_{k+1}}}dt\\
 & =\frac{\Gamma(1+\beta_{k+1})}{\Gamma(1-\alpha_{0})\Gamma(\alpha_{k+1})}z_{k+1}^{-\beta_{k+1}}\sum_{0=m_{0}\leq\cdots\leq m_{k}}\frac{\Gamma(1-\alpha_{0})}{\Gamma(m_{k}+1-\alpha_{0})}\left(\prod_{i=1}^{k}\frac{c(m_{i}-m_{i-1},z_{i})\Gamma(m_{i}+\beta_{i})}{\Gamma(m_{i-1}+1+\beta_{i})}\right)\\
 & \qquad\qquad\qquad\qquad\qquad\qquad\qquad\qquad\qquad\qquad\qquad\qquad\times z_{k+1}^{m_{k}+\beta_{k+1}}\frac{\Gamma(1+m_{k}-\alpha_{0})\Gamma(\alpha_{k+1})}{\Gamma(1+m_{k}+\beta_{k+1})}\\
 & =\sum_{0=m_{0}\leq\cdots\leq m_{k}<\infty}\frac{\Gamma(1+\beta_{k+1})}{\Gamma(1+m_{k}+\beta_{k+1})}\left(\prod_{i=1}^{k}\frac{c(m_{i}-m_{i-1},z_{i})\Gamma(m_{i}+\beta_{i})}{\Gamma(m_{i-1}+1+\beta_{i})}\right)z_{k+1}^{m_{k}}\\
 & =\Gamma(1+\beta_{k+1})\sum_{0=m_{0}\leq\cdots\leq m_{k}}\frac{\prod_{i=1}^{k}c(m_{i}-m_{i-1},z_{i})\Gamma(m_{i}+\beta_{i})}{\prod_{i=0}^{k}\Gamma(m_{i}+1+\beta_{i+1})}z_{k+1}^{m_{k}}.
\end{align*}
This completes the proof.
\end{proof}
\begin{rem}
Theorem \ref{thm:expansion_iterated_beta_finite} is a generalization
of the series expression for hyperlogarithms with finite endpoints.
By putting $\alpha_{0}=\cdots=\alpha_{k+1}=0$ in (\ref{eq:SerExpFF}),
we obtain
\begin{align*}
I_{\mathrm{dch}}(z_{0};e_{z_{1}}\cdots e_{z_{k}};z_{k+1}) & =\sum_{0=m_{0}\leq\cdots\leq m_{k}}\prod_{i=1}^{k}\frac{c(m_{i}-m_{i-1},z_{i}-z_{0})}{m_{i}}\left(z_{k+1}-z_{0}\right)^{m_{k}}.
\end{align*}
By putting $\alpha_{0}=\cdots=\alpha_{k+1}=0$, $s_{1}=1$, $k_{i}=s_{i+1}-s_{i}$
for $i=1,\dots,d$ in (\ref{eq:SerExpFF2}), we obtain
\[
I_{\mathrm{dch}}(0;e_{x_{1}}e_{0}^{k_{1}-1}e_{x_{2}}e_{0}^{k_{2}-1}\cdots e_{x_{d}}e_{0}^{k_{d}-1};x_{d+1})=(-1)^{d}\sum_{0<l_{1}<\cdots<l_{d}}\prod_{i=1}^{d}\frac{(x_{i}^{-1}x_{i+1})^{l_{i}}}{l_{i}^{k_{i}}}.
\]
\end{rem}

\begin{example}
We have
\[
\hat{B}_{\mathrm{dch}}^{{\rm f},{\rm f}}\left(\left.\substack{0\\
\alpha_{0}
}
\right|\left.\substack{x_{1}\\
\alpha_{1}
}
\right|\left.\substack{0\\
\alpha_{2}
}
\right|\left.\substack{x_{2}\\
\alpha_{3}
}
\right|\left.\substack{0\\
\alpha_{4}
}
\right|\left.\substack{0\\
\alpha_{5}
}
\right|\substack{x_{3}\\
\alpha_{6}
}
\right)=\sum_{0<l_{1}<l_{2}}\frac{\left(\prod_{l=1}^{l_{1}}\frac{l+\beta_{1}}{l+\beta_{6}}\right)\left(\prod_{l=l_{1}+1}^{l_{2}}\frac{l+\beta_{3}}{l+\beta_{6}}\right)\left(\frac{x_{2}}{x_{1}}\right)^{l_{1}}\left(\frac{x_{3}}{x_{2}}\right)^{l_{2}}}{(l_{1}+\beta_{1})(l_{1}+\beta_{2})(l_{2}+\beta_{3})(l_{2}+\beta_{4})(l_{2}+\beta_{5})}
\]
where $\beta_{j}=\alpha_{j}-\alpha_{0}$. This example is provided
only for illustration and is not of particular importance.
\end{example}

\begin{example}
When
\[
(z_{0},\dots,z_{k+1})=(0,x_{1},\{0\}^{k_{1}-1},x_{2},\{0\}^{k_{2}-1},\ldots,x_{d},\{0\}^{k_{d}-1},x_{d+1})
\]
and all $\alpha_{1}-\alpha_{0},\dots,\alpha_{k}-\alpha_{0}$ are equal
to $\beta$, we have
\[
\hat{B}_{\mathrm{dch}}^{{\rm f},{\rm f}}\left(\left.\substack{z_{0}\\
\alpha_{0}
}
\right|\left.\substack{z_{1}\\
\alpha_{1}
}
\right|\cdots\left|\substack{z_{k+1}\\
\alpha_{k+1}
}
\right.\right)=(-1)^{d}\sum_{0<l_{1}<\cdots<l_{d}}\left(\prod_{l=1}^{l_{d}}\frac{l+\beta}{l+\alpha_{k+1}-\alpha_{0}}\right)\cdot\frac{\prod_{i=1}^{d}\left(x_{i}^{-1}x_{i+1}\right)^{l_{i}}}{\prod_{i=1}^{d}(l_{i}+\beta)^{k_{i}}}
\]
where $\beta=\alpha_{1}-\alpha_{0}$.
\end{example}

\begin{example}
Let
\[
(z_{0},\dots,z_{k+1})=(0,1,\{0\}^{k_{1}-1},\ldots,1,\{0\}^{k_{d-1}-1},1,\{0\}^{k_{d}-1},\sin^{2}y)
\]
and
\[
(\alpha_{0},\dots,\alpha_{k+1})=(-n+1/2,\{1/2\}^{k},0).
\]
We also put $\epsilon_{i}=z_{i}\in\{0,1\}$ for $i=1,\dots,k$ to
emphasize that $\epsilon_{i}\in\{0,1\}$ and to match the notation
in \cite{Akhilesh}. By Theorem \ref{thm:expansion_iterated_beta_finite},
we have
\begin{align*}
\hat{B}_{\mathrm{dch}}^{{\rm f},{\rm f}}\left(\left.\substack{z_{0}\\
\alpha_{0}
}
\right|\left.\substack{z_{1}\\
\alpha_{1}
}
\right|\cdots\left|\substack{z_{k+1}\\
\alpha_{k+1}
}
\right.\right) & =(-1)^{d}\frac{{2n \choose n}}{(4\sin^{2}y)^{n}}\sum_{n<n_{1}<\cdots<n_{d}}\frac{(4\sin^{2}y)^{n_{d}}}{n_{1}^{k_{1}}\cdots n_{d}^{k_{d}}{2n_{d} \choose n_{d}}}.
\end{align*}
By (2) of Theorem \ref{thm: complete vs incomplete}, we have
\begin{align*}
\hat{B}_{\mathrm{dch}}^{{\rm f},{\rm f}}\left(\left.\substack{z_{0}\\
\alpha_{0}
}
\right|\left.\substack{z_{1}\\
\alpha_{1}
}
\right|\cdots\left|\substack{z_{k+1}\\
\alpha_{k+1}
}
\right.\right) & =\frac{(\sin^{2}y)^{-n+1/2}}{(\sin^{2}y-\epsilon_{k})^{1/2}}\cdot B_{\mathrm{dch}}^{{\rm f}}\left(\left.\substack{z_{0}\\
\alpha_{0}
}
\right|\left.\substack{z_{1}\\
\alpha_{1}
}
\right|\cdots\left|\substack{z_{k}\\
\alpha_{k}
}
\right.;\sin^{2}y\right)\\
 & =(-1)^{d}(\sin^{2}y)^{-n+1/2}s_{\epsilon_{k}}(\sin^{2}y)\cdot I(0;t^{n}s_{0,1}(t)dt,s_{\epsilon_{1},\epsilon_{2}}(t)dt,s_{\epsilon_{2},\epsilon_{3}}(t)dt,\dots,s_{\epsilon_{k-1},\epsilon_{k}}(t)dt;\sin^{2}y)
\end{align*}
where
\[
s_{\epsilon}(t)=\begin{cases}
\frac{1}{\sqrt{t}} & \epsilon=0\\
\frac{1}{\sqrt{1-t}} & \epsilon=1
\end{cases}
\]
and $s_{\epsilon,\eta}(t)=s_{\epsilon}(t)s_{\eta}(t)$. By change
of variables $t=\sin^{2}\theta$, since
\[
s_{\epsilon,\eta}(t)dt=2(\tan\theta)^{\epsilon+\eta-1}d\theta,
\]
we have
\begin{align*}
 & \hat{B}_{\mathrm{dch}}^{{\rm f},{\rm f}}\left(\left.\substack{z_{0}\\
\alpha_{0}
}
\right|\left.\substack{z_{1}\\
\alpha_{1}
}
\right|\cdots\left|\substack{z_{k+1}\\
\alpha_{k+1}
}
\right.\right)\\
 & =(-1)^{d}2^{k}(\sin y)^{-2n}(\tan y)^{\epsilon_{k}}\\
 & \qquad\times I(0;(\sin\theta)^{2n}d\theta,(\tan\theta)^{\epsilon_{1}+\epsilon_{2}-1}d\theta,\dots,(\tan\theta)^{\epsilon_{k-1}+\epsilon_{k}-1}d\theta;y).
\end{align*}
Thus, we have
\begin{align*}
 & \sum_{n<n_{1}<\cdots<n_{d}}\frac{(4\sin^{2}y)^{n_{d}}}{n_{1}^{k_{1}}\cdots n_{d}^{k_{d}}{2n_{d} \choose n_{d}}}\\
 & =2^{k}{2n \choose n}^{-1}(\tan y)^{\epsilon_{k}}\\
 & \qquad\times I(0;(4\sin^{2}\theta)^{n}d\theta,(\tan\theta)^{\epsilon_{1}+\epsilon_{2}-1}d\theta,\dots,(\tan\theta)^{\epsilon_{k-1}+\epsilon_{k}-1}d\theta;y),
\end{align*}
which coincides with the formula proved by P. Akhilesh \cite[Theorem 4]{Akhilesh}.
\end{example}

\subsection{Series expansion for $\hat{B}^{\infty,{\rm f}}$}
\begin{lem}
\label{lem:summation_lemma_infinity}Let $z,t$ be real numbers and
$s,\alpha$ be complex numbers satisfying $0<t$, $|z|<t$, and $\Re(s+\alpha)<0$.
Then we have
\[
\frac{1}{(t-z)^{\alpha}}\int_{\infty}^{t}\frac{u^{s}du}{(u-z)^{1-\alpha}}=-\sum_{n=0}^{\infty}z^{n}t^{s-n}\frac{\Gamma(-s+n)\Gamma(-s-\alpha)}{\Gamma(-s-\alpha+1+n)\Gamma(-s)}.
\]
\end{lem}

\begin{proof}
We may assume $\Re(\alpha)<0$ and $\Re(s)<0$ without loss of generality.
Note that we have
\[
\int_{u=\infty}^{t}u^{s+\alpha}\left(\int_{v=u}^{t}(u-v)^{-\alpha-1}(v-z)^{\alpha-1}dv\right)du=\int_{v=\infty}^{t}\left(\int_{u=\infty}^{v}u^{s+\alpha}(u-v)^{-\alpha-1}du\right)(v-z)^{\alpha-1}dv
\]
Then, 
\begin{align*}
\int_{u=\infty}^{t}u^{s+\alpha}\left(\int_{v=u}^{t}(u-v)^{-\alpha-1}(v-z)^{\alpha-1}dv\right)du & =\int_{\infty}^{t}u^{s+\alpha}\frac{1}{\alpha(u-z)}\left[(u-v)^{-\alpha}(v-z)^{\alpha}\right]_{v=u}^{v=t}du\\
 & =\frac{(t-z)^{\alpha}}{\alpha}\int_{\infty}^{t}\frac{u^{s+\alpha}(u-t)^{-\alpha}}{u-z}du\\
 & =\frac{(t-z)^{\alpha}}{\alpha}\int_{\infty}^{t}u^{s+\alpha-1}(u-t)^{-\alpha}\cdot\frac{u}{u-z}du\\
 & =\frac{(t-z)^{\alpha}}{\alpha}\int_{\infty}^{t}u^{s+\alpha-1}(u-t)^{-\alpha}\cdot\sum_{n=0}^{\infty}z^{n}\left(\frac{1}{u}\right)^{n}du\\
 & =\frac{(t-z)^{\alpha}}{\alpha}\sum_{n=0}^{\infty}z^{n}\int_{\infty}^{t}u^{s+\alpha-1-n}(u-t)^{-\alpha}du\\
 & =-\frac{(t-z)^{\alpha}}{\alpha}\sum_{n=0}^{\infty}z^{n}t^{s-n}\frac{\Gamma(1-\alpha)\Gamma(-s+n)}{\Gamma(-s-\alpha+1+n)}.
\end{align*}
On the other hand,
\[
\int_{v=\infty}^{t}\left(\int_{u=\infty}^{v}u^{s+\alpha}(u-v)^{-\alpha-1}du\right)(v-z)^{\alpha-1}dv=-\frac{\Gamma(-\alpha)\Gamma(-s)}{\Gamma(-s-\alpha)}\int_{v=\infty}^{t}v^{s}(v-z)^{\alpha-1}dv.
\]
Equating the two expressions, we get 
\begin{align*}
(t-z)^{-\alpha}\int_{v=\infty}^{t}v^{s}(v-z)^{\alpha-1}dv & =\frac{1}{\alpha}\sum_{n=0}^{\infty}z^{n}t^{s-n}\frac{\Gamma(1-\alpha)\Gamma(-s+n)}{\Gamma(-s-\alpha+1+n)}\frac{\Gamma(-s-\alpha)}{\Gamma(-\alpha)\Gamma(-s)}\\
 & =-\sum_{n=0}^{\infty}z^{n}t^{s-n}\frac{\Gamma(-s+n)\Gamma(-s-\alpha)}{\Gamma(-s-\alpha+1+n)\Gamma(-s)}.
\end{align*}
\end{proof}
\begin{thm}
\label{thm:expansion_iterated_beta_at_infinity}Suppose that $(z_{0},\ldots,z_{k+1})$
lies in the domain $|z_{1}-z_{0}|,\ldots,|z_{k}-z_{0}|<|z_{k+1}-z_{0}|$
and $\Re(\alpha_{i}-\alpha_{0})<0$ for $i=1,\dots,k$. Then,
\begin{equation}
\hat{B}_{\mathrm{ray}}^{\infty,\mathrm{f}}\!\left(\left.\substack{z_{0}\\
\alpha_{0}
}
\right|\left.\substack{z_{1}\\
\alpha_{1}
}
\right|\cdots\left|\substack{z_{k+1}\\
\alpha_{k+1}
}
\right.\right)=\frac{(-1)^{k}}{\Gamma(\alpha_{0}-\alpha_{k+1})}\sum_{0=m_{0}\leq\cdots\leq m_{k}}\frac{\prod_{i=1}^{k+1}\Gamma(m_{i-1}+\alpha_{0}-\alpha_{i})(z_{i}-z_{0})^{m_{i}-m_{i-1}}}{\prod_{i=1}^{k}\Gamma(m_{i}+\alpha_{0}-\alpha_{i}+1)},\label{eq:SerExpInf}
\end{equation}
where we set $m_{k+1}\coloneqq0$. Equivalently, 
\begin{equation}
\hat{B}_{\mathrm{ray}}^{\infty,\mathrm{f}}\!\left(\left.\substack{z_{0}\\
\alpha_{0}
}
\right|\left.\substack{z_{1}\\
\alpha_{1}
}
\right|\cdots\left|\substack{z_{k+1}\\
\alpha_{k+1}
}
\right.\right)=(-1)^{k}\sum_{0=l_{0}\leq l_{1}\leq\cdots\leq l_{d}}\prod_{i=1}^{d}\prod_{l=l_{i-1}}^{l_{i}-1}\left(\frac{l-\alpha_{s_{d+1}}+\alpha_{0}}{l-\alpha_{s_{i}}+\alpha_{0}}\right)\cdot\frac{\prod_{i=1}^{d}(x_{i}^{-1}x_{i+1})^{-l_{i}}}{\prod_{i=0}^{d}\prod_{s=s_{i}+\delta_{i,0}}^{s_{i+1}-1}(l_{i}-\alpha_{s}+\alpha_{0})}\qquad(s_{0}:=0).\label{eq:SerExpInf2}
\end{equation}
when $(z_{0},\dots,z_{k+1})=(0,\{0\}^{s_{1}-1},x_{1},\{0\}^{s_{2}-s_{1}-1},x_{2},\ldots,\{0\}^{s_{d+1}-s_{d}-1},x_{d+1})$
with $x_{i}\neq0$ ($1\leq i\leq d$) (the equivalence follows from
the invariance of $\hat{B}^{\infty,{\rm f}}$ under the simultaneous
affine transformation of $z$-variables). 
\end{thm}

\begin{proof}
We may assume $z_{0}=0$, $z_{k+1}\in\mathbb{R}_{>0}$, and $z_{1},\dots,z_{k}\in\mathbb{R}$
without loss of generality. Put $\beta_{i}:=\alpha_{i}-\alpha_{0}$
for $0\leq i\leq k+1$ and
\[
f_{i}(t)=\frac{1}{(t-z_{i})^{\alpha_{i}}}B_{\mathrm{ray}}\left(\infty;\left.\substack{z_{0}\\
\alpha_{0}
}
\right|\left.\substack{z_{1}\\
\alpha_{1}
}
\right|\cdots\left|\substack{z_{i}\\
\alpha_{i}
}
\right.;t\right)
\]
for $i=0,\dots,k$ and $z_{k+1}<t<\infty$. Then $f_{i}(t)$ satisfies
a recursive formula
\[
f_{i}(t)=\begin{cases}
t^{-\alpha_{0}} & i=0\\
L_{i}(f_{i-1}(t)) & i>0
\end{cases}
\]
where $L_{i}$ is defined by 
\[
L_{i}(f(t))=\frac{1}{(t-z_{i})^{\alpha_{i}}}\int_{\infty}^{t}\frac{f(u)du}{(u-z_{i})^{1-\alpha_{i}}}.
\]
Then, by Lemma \ref{lem:summation_lemma_infinity},
\[
L_{i}(t^{s})=-\sum_{n=0}^{\infty}z_{i}^{n}\frac{\Gamma(-s')\Gamma(-s-\alpha_{i})}{\Gamma(-s'-\alpha_{i}+1)\Gamma(-s)}t^{s'}\qquad\left(s'=s-n\right).
\]
By linearity of $L_{i}$, we have
\begin{align*}
f_{k}(t) & =L_{k}\circ L_{k-1}\circ\cdots\circ L_{1}(t^{-\alpha_{0}})\\
 & =(-1)^{k}\sum_{n_{1},\dots,n_{k}=0}^{\infty}\left(\prod_{i=1}^{k}z_{i}^{n_{i}}\frac{\Gamma(-s_{i})\Gamma(-s_{i-1}-\alpha_{i})}{\Gamma(-s_{i}-\alpha_{i}+1)\Gamma(-s_{i-1})}\right)t^{s_{k}}\qquad(s_{i}\coloneqq-\alpha_{0}-n_{1}-\cdots-n_{i})\\
 & =(-1)^{k}\sum_{n_{1},\dots,n_{k}=0}^{\infty}\frac{\Gamma(-s_{k})}{\Gamma(-s_{0})}\left(\prod_{i=1}^{k}z_{i}^{n_{i}}\frac{\Gamma(-s_{i-1}-\alpha_{i})}{\Gamma(-s_{i}-\alpha_{i}+1)}\right)t^{s_{k}}\qquad(s_{i}\coloneqq-\alpha_{0}-n_{1}-\cdots-n_{i})\\
 & =(-1)^{k}\sum_{0=m_{0}\leq\cdots\leq m_{k}}\frac{\Gamma(m_{k}+\alpha_{0})}{\Gamma(\alpha_{0})}\left(\prod_{i=1}^{k}z_{i}^{m_{i}-m_{i-1}}\frac{\Gamma(m_{i-1}+\alpha_{0}-\alpha_{i})}{\Gamma(m_{i}+\alpha_{0}-\alpha_{i}+1)}\right)t^{-\alpha_{0}-m_{k}}\qquad(m_{i}\coloneqq n_{1}+\cdots+n_{i}).
\end{align*}
Since
\begin{align*}
\hat{B}_{\mathrm{ray}}^{\infty,\mathrm{f}}\!\left(\left.\substack{z_{0}\\
\alpha_{0}
}
\right|\left.\substack{z_{1}\\
\alpha_{1}
}
\right|\cdots\left|\substack{z_{k+1}\\
\alpha_{k+1}
}
\right.\right) & =\frac{B_{\mathrm{ray}}\!\left(\infty;\left.\substack{z_{0}\\
\alpha_{0}
}
\right|\left.\substack{z_{1}\\
\alpha_{1}
}
\right|\cdots\left|\substack{z_{k+1}\\
\alpha_{k+1}
}
\right.;z_{k+1}\right)}{B_{\mathrm{ray}}\!\left(\infty;\substack{z_{0}\\
\alpha_{0}
}
\left|\substack{z_{k+1}\\
\alpha_{k+1}
}
\right.;z_{k+1}\right)}\\
 & =\frac{\int_{\infty}^{z_{k+1}}f_{k}(t)(t-z_{k+1})^{\alpha_{k+1}-1}dt}{\int_{\infty}^{z_{k+1}}t^{-\alpha_{0}}(t-z_{k+1})^{\alpha_{k+1}-1}dt}\\
 & =\frac{-\Gamma(\alpha_{0})}{\Gamma(-\beta_{k+1})\Gamma(\alpha_{k+1})}z_{k+1}^{-\beta_{k+1}}\int_{\infty}^{z_{k+1}}f_{k}(t)(t-z_{k+1})^{\alpha_{k+1}-1}dt,
\end{align*}
it follows that
\begin{align*}
 & \hat{B}_{\mathrm{ray}}^{\infty,\mathrm{f}}\!\left(\left.\substack{z_{0}\\
\alpha_{0}
}
\right|\left.\substack{z_{1}\\
\alpha_{1}
}
\right|\cdots\left|\substack{z_{k+1}\\
\alpha_{k+1}
}
\right.\right)\\
 & =\frac{(-1)^{k+1}\Gamma(\alpha_{0})}{\Gamma(-\beta_{k+1})\Gamma(\alpha_{k+1})}z_{k+1}^{-\beta_{k+1}}\sum_{0=m_{0}\leq\cdots\leq m_{k}}\frac{\Gamma(m_{k}+\alpha_{0})}{\Gamma(\alpha_{0})}\left(\prod_{i=1}^{k}z_{i}^{m_{i}-m_{i-1}}\frac{\Gamma(m_{i-1}+\alpha_{0}-\alpha_{i})}{\Gamma(m_{i}+\alpha_{0}-\alpha_{i}+1)}\right)\\
 & \qquad\qquad\qquad\qquad\qquad\qquad\qquad\qquad\qquad\qquad\qquad\qquad\times\int_{\infty}^{z_{k+1}}t^{-\alpha_{0}-m_{k}}(t-z_{k+1})^{\alpha_{k+1}-1}dt\\
 & =\frac{(-1)^{k}\Gamma(\alpha_{0})}{\Gamma(-\beta_{k+1})\Gamma(\alpha_{k+1})}z_{k+1}^{-\beta_{k+1}}\sum_{0=m_{0}\leq\cdots\leq m_{k}}\frac{\Gamma(m_{k}+\alpha_{0})}{\Gamma(\alpha_{0})}\left(\prod_{i=1}^{k}z_{i}^{m_{i}-m_{i-1}}\frac{\Gamma(m_{i-1}+\alpha_{0}-\alpha_{i})}{\Gamma(m_{i}+\alpha_{0}-\alpha_{i}+1)}\right)\\
 & \qquad\qquad\qquad\qquad\qquad\qquad\qquad\qquad\qquad\qquad\qquad\qquad\times z_{k+1}^{\beta_{k+1}-m_{k}}\frac{\Gamma(\alpha_{k+1})\Gamma(-\beta_{k+1}+m_{k})}{\Gamma(\alpha_{0}+m_{k})}\\
 & =\frac{(-1)^{k}}{\Gamma(-\beta_{k+1})}\sum_{0=m_{0}\leq\cdots\leq m_{k}}\left(\frac{\prod_{i=1}^{k}z_{i}^{m_{i}-m_{i-1}}\cdot\prod_{i=1}^{k+1}\Gamma(m_{i-1}-\beta_{i})}{\prod_{i=1}^{k}\Gamma(m_{i}-\beta_{i}+1)}\right)z_{k+1}^{-m_{k}}
\end{align*}
This completes the proof.
\end{proof}
\begin{rem}
Theorem \ref{thm:expansion_iterated_beta_at_infinity} is generalization
of series expression for hyperlogarithms with (infinite, finite) endpoints.

By specializing (\ref{eq:SerExpInf}) to $\alpha_{0}=1,\alpha_{1}=\cdots=\alpha_{k+1}=0$,
we have 
\[
\frac{z_{k+1}-z_{0}}{z_{1}-z_{0}}I_{\mathrm{ray}}(\infty;(e_{z_{0}}-e_{z_{1}})e_{z_{2}}\cdots e_{z_{k}};z_{k+1})=(-1)^{k}\sum_{0=m_{0}\leq\cdots\leq m_{k}}\frac{\prod_{i=1}^{k+1}(z_{i}-z_{0})^{m_{i}-m_{i-1}}}{\prod_{i=1}^{k}(m_{i}+1)}
\]
which can be rewritten as
\[
I_{\mathrm{ray}}(\infty;(e_{z_{0}}-e_{z_{1}})e_{z_{2}}\cdots e_{z_{k}};z_{k+1})=(-1)^{k}\sum_{1\leq n_{1}\leq\cdots\leq n_{k}}\prod_{i=1}^{k}\frac{(z_{i}-z_{0})^{n_{i}-n_{i-1}}}{n_{i}}\cdot(z_{k+1}-z_{0})^{-n_{k}}\quad(n_{0}:=0)
\]
by putting $n_{i}=m_{i}+1$.

In the same way, specializing (\ref{eq:SerExpInf2}) to $\alpha_{1}=\cdots=\alpha_{k+1}=-1,s_{1}=1$
yields

\[
\frac{x_{d+1}}{x_{1}}I_{\mathrm{ray}}(\infty;(e_{0}-e_{x_{1}})e_{0}^{s_{2}-s_{1}-1}e_{x_{2}}\cdots e_{0}^{s_{d+1}-s_{d}-1};x_{d+1})=(-1)^{k}\sum_{0\leq l_{1}\leq\cdots\leq l_{d}}\frac{\prod_{i=1}^{d}(x_{i}^{-1}x_{i+1})^{-l_{i}}}{\prod_{i=1}^{d}\prod_{s=s_{i}}^{s_{i+1}-1}(l_{i}+1)}.
\]
Furthermore, setting $k_{i}=s_{i+1}-s_{i}$ and $n_{i}=l_{i}+1$,
this equality can be rewritten as
\[
I_{\mathrm{ray}}(\infty;(e_{0}-e_{x_{1}})e_{0}^{k_{1}-1}e_{x_{2}}e_{0}^{k_{2}-1}\cdots e_{x_{d}}e_{0}^{k_{d}-1};x_{d+1})=(-1)^{k_{1}+\cdots+k_{d}}\sum_{1\leq n_{1}\leq\cdots\leq n_{d}}\prod_{i=1}^{d}\frac{(x_{i}^{-1}x_{i+1})^{-n_{i}}}{n_{i}^{k_{i}}}.
\]
\end{rem}

\section{Relating finite and infinite iterated beta integrals $B_{\gamma}^{{\rm f},{\rm f}}$,
$B_{\gamma}^{{\rm f},\infty}$, and $B_{\gamma}^{\infty,{\rm f}}$\label{sec:Relating-finite-and-infinite-path}}

The iterated beta integrals $B_{\gamma}^{{\rm f},{\rm f}}$, $B_{\gamma}^{{\rm f},\infty}$,
and $B_{\gamma}^{\infty,{\rm f}}$ satisfy the same system of differential
equations. For this reason, it may be natural to expect some simple
relationship between them. In this section, we will provide a formula
expressing the `finite' iterated beta integral $B_{\gamma}^{{\rm f},{\rm f}}$
in terms of `infinite' ones $B_{\gamma}^{{\rm f},\infty}$ and $B_{\gamma}^{\infty,{\rm f}}$.

Let $z_{0},\dots,z_{n}\in\mathbb{C}$ ($n\geq1$) be complex numbers
such that $z_{0}\neq z_{n}$ and $z_{1},\dots,z_{n-1}\in\mathbb{C}\setminus\{z_{0},z_{n}\}$,
and let $D$ be the connected domain containing $\{z_{1},\dots,z_{n-1}\}$.
Let $\alpha,\beta_{\mathrm{up}},\beta_{\mathrm{down}}$ be the paths
on $\mathbb{C}\setminus(\{z_{0},z_{n}\}\cup D)$ illustrated by Figure
\ref{fig:alpha_beta}, and $P$ the Pochhammer contour illustrated
by Figure \ref{fig:Pochhammer}.

\begin{figure}[bph]
\includegraphics[clip,scale=0.65]{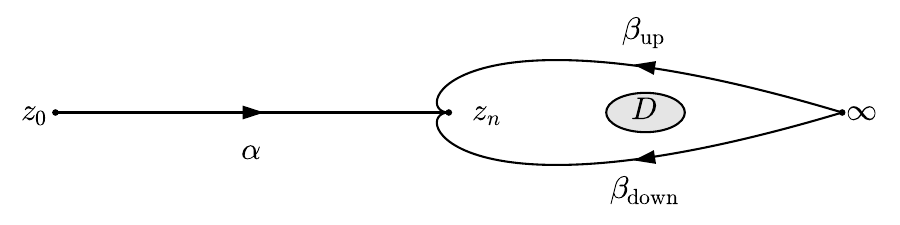}\caption{The paths $\alpha,\beta_{\mathrm{up}},\beta_{\mathrm{down}}$}
\label{fig:alpha_beta}
\end{figure}

\begin{figure}[bph]
\includegraphics[scale=0.6]{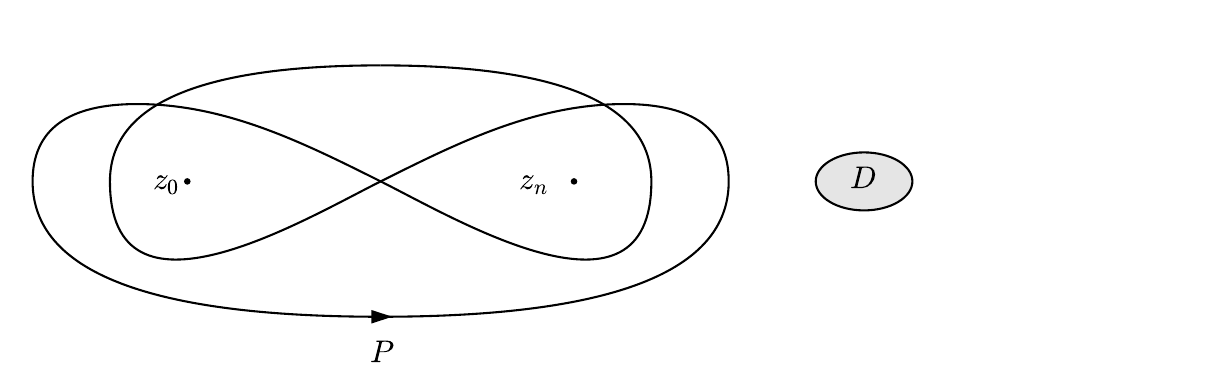}\caption{The Pochhammer contour $P$}
\label{fig:Pochhammer}
\end{figure}
By choosing a basepoint $v\in P$, we can consider the iterated integral
\[
I_{P_{v}}\left(v;\left[\substack{z_{0},z_{1}\\
\alpha_{0},\alpha_{1}
}
\right],\left[\substack{z_{1},z_{2}\\
\alpha_{1},\alpha_{2}
}
\right],\ldots,\left[\substack{z_{n-1},z_{n}\\
\alpha_{n-1},\alpha_{n}
}
\right];v\right)
\]
where $P_{v}$ is the closed path from $v$ to $v$ along $P$. We
will later show in Corollary \ref{cor:basepoint_freeness_for_Pv}
that this iterated integral does not depend on the choice of the basepoint
$v$, so we may denote it as 
\[
B_{P}\left(\left.\substack{z_{0}\\
\alpha_{0}
}
\right|\left.\substack{z_{1}\\
\alpha_{1}
}
\right|\cdots\left|\substack{z_{n}\\
\alpha_{n}
}
\right.\right).
\]
Also, let 
\[
\hat{B}_{P}\left(\left.\substack{z_{0}\\
\alpha_{0}
}
\right|\left.\substack{z_{1}\\
\alpha_{1}
}
\right|\cdots\left|\substack{z_{n}\\
\alpha_{n}
}
\right.\right)\coloneqq\frac{B_{P}\left(\left.\substack{z_{0}\\
\alpha_{0}
}
\right|\left.\substack{z_{1}\\
\alpha_{1}
}
\right|\cdots\left|\substack{z_{n}\\
\alpha_{n}
}
\right.\right)}{B_{P}\left(\substack{z_{0}\\
\alpha_{0}
}
\left|\substack{z_{n}\\
\alpha_{n}
}
\right.\right)}.
\]
The goal of this section is to prove the following relationship between
the iterated beta integrals along $\alpha$, $\beta_{\mathrm{up}}$,
$\beta_{\mathrm{down}}$ and $P$.
\begin{thm}
\label{thm: infinite and finite paths comparison}We have
\begin{enumerate}
\item 
\[
B_{P}\left(\left.\substack{z_{0}\\
\alpha_{0}
}
\right|\left.\substack{z_{1}\\
\alpha_{1}
}
\right|\cdots\left|\substack{z_{n}\\
\alpha_{n}
}
\right.\right)=-\left(1-e^{-2\pi i\alpha_{0}}\right)\left(1-e^{2\pi i\alpha_{n}}\right)B_{\alpha}^{\mathrm{f},\mathrm{f}}\left(\left.\substack{z_{0}\\
\alpha_{0}
}
\right|\left.\substack{z_{1}\\
\alpha_{1}
}
\right|\cdots\left|\substack{z_{n}\\
\alpha_{n}
}
\right.\right).
\]
\item 
\[
B_{\alpha}^{\mathrm{f},\mathrm{f}}\left(\left.\substack{z_{0}\\
\alpha_{0}
}
\right|\left.\substack{z_{1}\\
\alpha_{1}
}
\right|\cdots\left|\substack{z_{n}\\
\alpha_{n}
}
\right.\right)=\frac{1}{1-e^{-2\pi i\alpha_{0}}}B_{\beta_{\mathrm{up}}}^{\infty,\mathrm{f}}\left(\left.\substack{z_{0}\\
\alpha_{0}
}
\right|\left.\substack{z_{1}\\
\alpha_{1}
}
\right|\cdots\left|\substack{z_{n}\\
\alpha_{n}
}
\right.\right)+\frac{1}{1-e^{2\pi i\alpha_{0}}}B_{\beta_{\mathrm{down}}}^{\infty,\mathrm{f}}\left(\left.\substack{z_{0}\\
\alpha_{0}
}
\right|\left.\substack{z_{1}\\
\alpha_{1}
}
\right|\cdots\left|\substack{z_{n}\\
\alpha_{n}
}
\right.\right).
\]
\item 
\begin{align*}
\hat{B}_{P}\left(\left.\substack{z_{0}\\
\alpha_{0}
}
\right|\left.\substack{z_{1}\\
\alpha_{1}
}
\right|\cdots\left|\substack{z_{n}\\
\alpha_{n}
}
\right.\right)= & \hat{B}_{\alpha}^{\mathrm{f},\mathrm{f}}\left(\left.\substack{z_{0}\\
\alpha_{0}
}
\right|\left.\substack{z_{1}\\
\alpha_{1}
}
\right|\cdots\left|\substack{z_{n}\\
\alpha_{n}
}
\right.\right)\\
= & \frac{1}{1-e^{-2\pi i(\alpha_{0}-\alpha_{n})}}\hat{B}_{\beta_{\mathrm{up}}}^{\infty,\mathrm{f}}\left(\left.\substack{z_{0}\\
\alpha_{0}
}
\right|\left.\substack{z_{1}\\
\alpha_{1}
}
\right|\cdots\left|\substack{z_{n}\\
\alpha_{n}
}
\right.\right)+\frac{1}{1-e^{2\pi i(\alpha_{0}-\alpha_{n})}}\hat{B}_{\beta_{\mathrm{down}}}^{\infty,\mathrm{f}}\left(\left.\substack{z_{0}\\
\alpha_{0}
}
\right|\left.\substack{z_{1}\\
\alpha_{1}
}
\right|\cdots\left|\substack{z_{n}\\
\alpha_{n}
}
\right.\right).
\end{align*}
\end{enumerate}
\end{thm}

The following proposition plays a key role.
\begin{prop}
\label{prop: basepoint freeness}Let $L$ be a loop (i.e., an immersion
of $S^{1}$) on the universal abelian covering space of $\mathbb{C}\setminus(\{z_{0},z_{n}\}\cup D)$
whose projection to $\mathbb{C}\setminus(\{z_{0},z_{n}\}\cup D)$
is contained in a simply connected open domain $U\subset\mathbb{C}\setminus D$
containing $z_{0}$, $z_{n}$. Let $v\in L$ be a basepoint and $L_{v}$
the closed path from $v$ to itself along $L.$ Then, the integral
\[
I_{L_{v}}\left(v;\left[\substack{z_{0},z_{1}\\
\alpha_{0},\alpha_{1}
}
\right],\left[\substack{z_{1},z_{2}\\
\alpha_{1},\alpha_{2}
}
\right],\ldots,\left[\substack{z_{n-1},z_{n}\\
\alpha_{n-1},\alpha_{n}
}
\right];v\right)
\]
depends only on the homotopy class of $L$ and does not depend on
the choice of the basepoint $v$.
\end{prop}

\begin{proof}
Let $u\in L$ and $\omega_{i}=\left[\substack{z_{i-1},z_{i}\\
\alpha_{i-1},\alpha_{i}
}
\right]$ for $i=1,\dots,n$. It suffices to show
\[
I_{C_{u,v}L_{v}C_{u,v}^{-1}}(u;\omega_{1},\dots,\omega_{n};u)=I_{L_{v}}\left(v;\omega_{1},\dots,\omega_{n};v\right)
\]
where $C_{u,v}$ is one of the arcs from $u$ to $v$ along $L$.
By the path composition formula, we have
\begin{align*}
 & I_{C_{u,v}L_{v}C_{u,v}^{-1}}\left(u;\omega_{1},\dots,\omega_{n};u\right)\\
 & =\sum_{0\leq i\leq j\leq n}I_{C_{u,v}}\left(u;\omega_{1},\dots,\omega_{i};v\right)I_{L_{v}}\left(v;\omega_{i+1},\dots,\omega_{j};v\right)I_{C_{u,v}^{-1}}\left(v;\omega_{j+1},\dots,\omega_{n};u\right).
\end{align*}

Now, we show the vanishing of $I_{L_{v}}\left(v;\omega_{i+1},\dots,\omega_{j};v\right)$
except for the case $i=j$ or $(i,j)=(0,n)$. Note that $\omega_{i}$
($i\neq0$) can be viewed as holomorphic differential forms on the
universal abelian covering space $X$ of $U\setminus\{z_{n}\}$. Let
$v_{X}$ be the image of $v$ in $X$, and $\mathcal{L}\in\pi_{1}(X,v_{X})$
the image of $L_{v}$. Since $X$ is simply connected, $\pi_{1}(X,v_{X})$
is trivial, hence $\mathcal{L}$ is also trivial. Thus,
\[
I_{L_{v}}\left(v;\omega_{i+1},\dots,\omega_{j};v\right)=0\qquad(0<i<j\leq n).
\]
Similarly, we have
\[
I_{L_{v}}\left(v;\omega_{i+1},\dots,\omega_{j};v\right)=0\qquad(0\leq i<j<n).
\]

Thus,
\begin{align*}
 & I_{C_{u,v}L_{v}C_{u,v}^{-1}}\left(u;\omega_{1},\dots,\omega_{n};u\right)\\
 & =\sum_{\substack{0\leq i\leq j\leq n\\
i=j\text{ or }(i,j)=(0,n)
}
}I_{C_{u,v}}\left(u;\omega_{1},\dots,\omega_{i};v\right)I_{L_{v}}\left(v;\omega_{i+1},\dots,\omega_{j};v\right)I_{C_{u,v}^{-1}}\left(v;\omega_{j+1},\dots,\omega_{n};u\right)\\
 & =I_{C_{u,v}C_{u,v}^{-1}}(\omega_{1},\dots,\omega_{n})+I_{L_{v}}\left(v;\omega_{1},\dots,\omega_{n};v\right)\\
 & =I_{L_{v}}\left(v;\omega_{1},\dots,\omega_{n};v\right),
\end{align*}
which completes the proof.
\end{proof}
\begin{cor}
\label{cor:basepoint_freeness_for_Pv}The expression
\[
B_{P}\left(\left.\substack{z_{0}\\
\alpha_{0}
}
\right|\left.\substack{z_{1}\\
\alpha_{1}
}
\right|\cdots\left|\substack{z_{n}\\
\alpha_{n}
}
\right.\right):=I_{P_{v}}\left(v;\left[\substack{z_{0},z_{1}\\
\alpha_{0},\alpha_{1}
}
\right],\left[\substack{z_{1},z_{2}\\
\alpha_{1},\alpha_{2}
}
\right],\ldots,\left[\substack{z_{n-1},z_{n}\\
\alpha_{n-1},\alpha_{n}
}
\right];v\right)
\]
is well-defined, i.e., the iterated integral on the right-hand side
does not depend on the choice of $v$.
\end{cor}

\begin{proof}
It follows from the fact that $P_{v}$ is a closed path on the universal
abelian covering space of $\mathbb{C}\setminus(\{z_{0},z_{n}\}\cup D)$.
\end{proof}
\begin{proof}[Proof of Theorem \ref{thm: infinite and finite paths comparison}]
 Let $\omega_{i}$ ($i=1,\dots,n$) be the same as in the proof of
Proposition \ref{prop: basepoint freeness}, and let $C_{0}$ (resp.
$C_{1}$, $C_{\infty}$) be a small closed path from $z_{0}$ (resp.
$z_{n},\infty$) to itself which encircles $z_{0}$ (resp. $z_{n}$,
$\infty$) once counterclockwisely. The Pochhammer contour $P$ is
homotopic to a certain lift of the path
\[
C_{0}\alpha C_{1}\alpha^{-1}C_{0}^{-1}\alpha C_{1}^{-1}\alpha^{-1}
\]
(see Figure \ref{fig:finite_path}).
\begin{figure}[bph]
\includegraphics[scale=0.6]{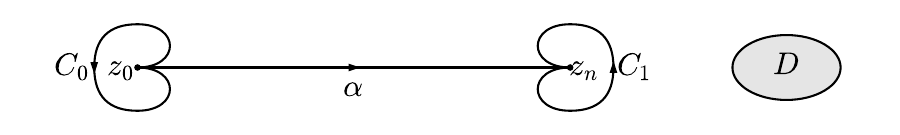}\caption{The paths $C_{0},C_{1},\alpha$}
\label{fig:finite_path}
\end{figure}

Let us consider the lift
\[
C_{0}\tilde{\alpha}_{1,0}C_{1}\tilde{\alpha}_{1,1}^{-1}C_{0}^{-1}\tilde{\alpha}_{0,1}C_{1}^{-1}\tilde{\alpha}_{0,0}^{-1}
\]
of $P$ where $\tilde{\alpha}_{a,b}$ are lifts of $\alpha$ such
that
\[
I_{\tilde{\alpha}_{a,b}}(z_{0};(t-z_{0})^{-\alpha_{0}}(t-z_{n})^{\alpha_{n}-1};z_{n})=e^{2\pi i(-a\alpha_{0}+b\alpha_{n})}I_{\tilde{\alpha}_{0,0}}(z_{0};(t-z_{0})^{-\alpha_{0}}(t-z_{n})^{\alpha_{n}-1};z_{n}).
\]
Let $\tilde{\beta}_{\mathrm{up},a,b}$ (resp. $\tilde{\beta}_{\mathrm{down},a,b}$)
be the lift of $\beta_{\mathrm{up}}$ (resp. $\beta_{\mathrm{down}}$)
whose terminal points coincide with those of $\tilde{\alpha}_{a,b}$.
Then, 
\[
\tilde{\alpha}_{a,b}^{-1}C_{0}\tilde{\alpha}_{a+1,b}\sim\tilde{\beta}_{\mathrm{up},a,b}^{-1}C_{\infty}^{-1}\tilde{\beta}_{\mathrm{down},a+1,b}
\]
where $\sim$ means the homotopy equivalence. Thus,
\begin{align*}
\tilde{\alpha}_{0,0}^{-1}C_{0}\tilde{\alpha}_{1,0}C_{1}\tilde{\alpha}_{1,1}^{-1}C_{0}^{-1}\tilde{\alpha}_{0,1}C_{1}^{-1} & =\left(\tilde{\alpha}_{0,0}^{-1}C_{0}\tilde{\alpha}_{1,0}\right)C_{1}\left(\tilde{\alpha}_{0,1}^{-1}C_{0}\tilde{\alpha}_{1,1}\right)^{-1}C_{1}^{-1}\\
 & \sim\left(\tilde{\beta}_{\mathrm{up},0,0}^{-1}C_{\infty}^{-1}\tilde{\beta}_{\mathrm{down},1,0}\right)C_{1}\left(\tilde{\beta}_{\mathrm{up},0,1}^{-1}C_{\infty}^{-1}\tilde{\beta}_{\mathrm{down},1,1}\right)^{-1}C_{1}^{-1}\\
 & =\tilde{\beta}_{\mathrm{up},0,0}^{-1}C_{\infty}^{-1}\tilde{\beta}_{\mathrm{down},1,0}C_{1}\tilde{\beta}_{\mathrm{down},1,1}^{-1}C_{\infty}\tilde{\beta}_{\mathrm{up},0,1}C_{1}^{-1}.
\end{align*}
Thus, by Proposition \ref{prop: basepoint freeness}, we have
\[
I_{C_{0}\tilde{\alpha}_{1,0}C_{1}\tilde{\alpha}_{1,1}^{-1}C_{0}^{-1}\tilde{\alpha}_{0,1}C_{1}^{-1}\tilde{\alpha}_{0,0}^{-1}}(z_{0};\omega_{1},\dots,\omega_{n};z_{0})=I_{\tilde{\beta}_{\mathrm{up},0,0}^{-1}C_{\infty}^{-1}\tilde{\beta}_{\mathrm{down},1,0}C_{1}\tilde{\beta}_{\mathrm{down},1,1}^{-1}C_{\infty}\tilde{\beta}_{\mathrm{up},0,1}C_{1}^{-1}}(z_{n};\omega_{1},\dots,\omega_{n};z_{n}).
\]

Let us calculate the right-hand side and the left-hand side. First,
the left-hand side equals
\begin{align*}
 & I_{C_{0}\tilde{\alpha}_{1,0}C_{1}\tilde{\alpha}_{1,1}^{-1}C_{0}^{-1}\tilde{\alpha}_{0,1}C_{1}^{-1}\tilde{\alpha}_{0,0}^{-1}}(z_{0};\omega_{1},\dots,\omega_{n};z_{0})\\
 & =\sum_{0\leq i\leq n}I_{C_{0}\tilde{\alpha}_{1,0}C_{1}\tilde{\alpha}_{1,1}^{-1}}(z_{0};\omega_{1},\dots,\omega_{i};z_{0})\cdot I_{C_{0}^{-1}\tilde{\alpha}_{0,1}C_{1}^{-1}\tilde{\alpha}_{0,0}^{-1}}(z_{0};\omega_{i+1},\dots,\omega_{n};z_{0}).
\end{align*}
Here, $I_{C_{0}\tilde{\alpha}_{1,0}C_{1}\tilde{\alpha}_{1,1}^{-1}}(z_{0};\omega_{1},\dots,\omega_{i};z_{0})=0$
except when $i=0$ or $i=n$. Thus,
\begin{align*}
 & I_{C_{0}\tilde{\alpha}_{1,0}C_{1}\tilde{\alpha}_{1,1}^{-1}C_{0}^{-1}\tilde{\alpha}_{0,1}C_{1}^{-1}\tilde{\alpha}_{0,0}^{-1}}(z_{0};\omega_{1},\dots,\omega_{n};z_{0})\\
 & =I_{C_{0}\tilde{\alpha}_{1,0}C_{1}\tilde{\alpha}_{1,1}^{-1}}(z_{0};\omega_{1},\dots,\omega_{n};z_{0})+I_{C_{0}^{-1}\tilde{\alpha}_{0,1}C_{1}^{-1}\tilde{\alpha}_{0,0}^{-1}}(z_{0};\omega_{1},\dots,\omega_{n};z_{0}).
\end{align*}
Since the iterated integrals along $C_{0}$ and $C_{1}$ vanish,
\begin{align*}
 & I_{C_{0}\tilde{\alpha}_{1,0}C_{1}\tilde{\alpha}_{1,1}^{-1}}(z_{0};\omega_{1},\dots,\omega_{n};z_{0})\\
 & =\sum_{j=0}^{n}I_{\tilde{\alpha}_{1,0}}(z_{0};\omega_{1},\dots,\omega_{j};z_{n})\cdot I_{\tilde{\alpha}_{1,1}^{-1}}(z_{n};\omega_{j+1},\dots,\omega_{n};z_{0})\\
 & =\sum_{j=0}^{n}I_{\tilde{\alpha}_{1,1}}(z_{0};\omega_{1},\dots,\omega_{j};z_{n})\cdot I_{\tilde{\alpha}_{1,1}^{-1}}(z_{n};\omega_{j+1},\dots,\omega_{n};z_{0})\\
 & \quad+\sum_{j=0}^{n}\left(I_{\tilde{\alpha}_{1,0}}(z_{0};\omega_{1},\dots,\omega_{j};z_{n})-I_{\tilde{\alpha}_{1,1}}(z_{0};\omega_{1},\dots,\omega_{j};z_{n})\right)\cdot I_{\tilde{\alpha}_{1,1}^{-1}}(z_{n};\omega_{j+1},\dots,\omega_{n};z_{0})\\
 & =0+\left(I_{\tilde{\alpha}_{1,0}}(z_{0};\omega_{1},\dots,\omega_{n};z_{n})-I_{\tilde{\alpha}_{1,1}}(z_{0};\omega_{1},\dots,\omega_{n};z_{n})\right)\\
 & =\left(e^{2\pi i(-\alpha_{0})}-e^{2\pi i(-\alpha_{0}+\alpha_{n})}\right)I_{\tilde{\alpha}}(z_{0};\omega_{1},\dots,\omega_{n};z_{n})
\end{align*}
where we put $\tilde{\alpha}=\tilde{\alpha}_{0,0}$. Similarly, we
have
\[
I_{C_{0}^{-1}\tilde{\alpha}_{0,1}C_{1}^{-1}\tilde{\alpha}_{0,0}^{-1}}(z_{0};\omega_{1},\dots,\omega_{n};z_{0})=\left(e^{2\pi i\alpha_{n}}-1\right)I_{\tilde{\alpha}}(z_{0};\omega_{1},\dots,\omega_{n};z_{n}).
\]
Thus, we find
\[
I_{C_{0}\tilde{\alpha}_{1,0}C_{1}\tilde{\alpha}_{1,1}^{-1}C_{0}^{-1}\tilde{\alpha}_{0,1}C_{1}^{-1}\tilde{\alpha}_{0,0}^{-1}}(z_{0};\omega_{1},\dots,\omega_{n};z_{0})=\left(1-e^{-2\pi i\alpha_{0}}\right)\left(e^{2\pi i\alpha_{n}}-1\right)I_{\tilde{\alpha}}(z_{0};\omega_{1},\dots,\omega_{n};z_{n})
\]

In a similar manner, the right-hand side can be computed as
\begin{align*}
 & I_{\tilde{\beta}_{\mathrm{up},0,0}^{-1}C_{\infty}^{-1}\tilde{\beta}_{\mathrm{down},1,0}C_{1}\tilde{\beta}_{\mathrm{down},1,1}^{-1}C_{\infty}\tilde{\beta}_{\mathrm{up},0,1}C_{1}^{-1}}(z_{n};\omega_{1},\dots,\omega_{n};z_{n})\\
 & =I_{\tilde{\beta}_{\mathrm{down},1,0}C_{1}\tilde{\beta}_{\mathrm{down},1,1}^{-1}C_{\infty}\tilde{\beta}_{\mathrm{up},0,1}C_{1}^{-1}\tilde{\beta}_{\mathrm{up},0,0}^{-1}C_{\infty}^{-1}}(\infty;\omega_{1},\dots,\omega_{n};\infty)\\
 & =\sum_{j=0}^{n}I_{\tilde{\beta}_{\mathrm{down},1,0}C_{1}\tilde{\beta}_{\mathrm{down},1,1}^{-1}C_{\infty}}(\infty;\omega_{1},\dots,\omega_{j};\infty)\cdot I_{\tilde{\beta}_{\mathrm{up},0,1}C_{1}^{-1}\tilde{\beta}_{\mathrm{up},0,0}^{-1}C_{\infty}^{-1}}(\infty;\omega_{j+1},\dots,\omega_{n};\infty)\\
 & =I_{\tilde{\beta}_{\mathrm{down},1,0}C_{1}\tilde{\beta}_{\mathrm{down},1,1}^{-1}C_{\infty}}(\infty;\omega_{1},\dots,\omega_{n};\infty)+I_{\tilde{\beta}_{\mathrm{up},0,1}C_{1}^{-1}\tilde{\beta}_{\mathrm{up},0,0}^{-1}C_{\infty}^{-1}}(\infty;\omega_{1},\dots,\omega_{n};\infty)\\
 & =\left(e^{2\pi i(-\alpha_{0})}-e^{2\pi i(-\alpha_{0}+\alpha_{n})}\right)I_{\tilde{\beta}_{\mathrm{down}}}(\infty;\omega_{1},\dots,\omega_{n};z_{n})+\left(e^{2\pi i\alpha_{n}}-1\right)I_{\tilde{\beta}_{\mathrm{up}}}(\infty;\omega_{1},\dots,\omega_{n};z_{n}).
\end{align*}
Equating the two sides, we find
\[
\left(1-e^{-2\pi i\alpha_{0}}\right)I_{\tilde{\alpha}}(z_{0};\omega_{1},\dots,\omega_{n};z_{n})=-e^{-2\pi i\alpha_{0}}I_{\tilde{\beta}_{\mathrm{down}}}(\infty;\omega_{1},\dots,\omega_{n};\infty)+I_{\tilde{\beta}_{\mathrm{up}}}(\infty;\omega_{1},\dots,\omega_{n};\infty).
\]
Noting
\begin{align*}
I_{\tilde{\alpha}}(z_{0};\omega_{1},\dots,\omega_{n};z_{n}) & =B_{\alpha}^{\mathrm{f},\mathrm{f}}\left(\left.\substack{z_{0}\\
\alpha_{0}
}
\right|\left.\substack{z_{1}\\
\alpha_{1}
}
\right|\cdots\left|\substack{z_{n}\\
\alpha_{n}
}
\right.\right),\\
I_{\tilde{\beta}_{\mathrm{down}}}(\infty;\omega_{1},\dots,\omega_{n};\infty) & =B_{\beta_{\mathrm{down}}}^{\infty,\mathrm{f}}\left(\left.\substack{z_{0}\\
\alpha_{0}
}
\right|\left.\substack{z_{1}\\
\alpha_{1}
}
\right|\cdots\left|\substack{z_{n}\\
\alpha_{n}
}
\right.\right),\\
I_{\tilde{\beta}_{\mathrm{up}}}(\infty;\omega_{1},\dots,\omega_{n};\infty) & =B_{\beta_{\mathrm{up}}}^{\infty,\mathrm{f}}\left(\left.\substack{z_{0}\\
\alpha_{0}
}
\right|\left.\substack{z_{1}\\
\alpha_{1}
}
\right|\cdots\left|\substack{z_{n}\\
\alpha_{n}
}
\right.\right),
\end{align*}
we get (1) and (2) of Theorem \ref{thm: infinite and finite paths comparison}. 

The first equality of (3) follows immediately from (1). For the second
equality of (3), consider the $n=2$ case of (2):

\[
\left(e^{\pi i\alpha_{0}}-e^{-\pi i\alpha_{0}}\right)B_{\alpha}^{\mathrm{f},\mathrm{f}}\left(\left.\substack{z_{0}\\
\alpha_{0}
}
\right|\left.\substack{z_{n}\\
\alpha_{n}
}
\right.\right)=e^{\pi i\alpha_{0}}B_{\beta_{\mathrm{up}}}^{\infty,\mathrm{f}}\left(\left.\substack{z_{0}\\
\alpha_{0}
}
\right|\left.\substack{z_{n}\\
\alpha_{n}
}
\right.\right)-e^{-\pi i\alpha_{0}}B_{\beta_{\mathrm{down}}}^{\infty,\mathrm{f}}\left(\left.\substack{z_{0}\\
\alpha_{0}
}
\right|\left.\substack{z_{n}\\
\alpha_{n}
}
\right.\right).
\]
Here,
\[
B_{\beta_{\mathrm{up}}}^{\infty,\mathrm{f}}\left(\left.\substack{z_{0}\\
\alpha_{0}
}
\right|\left.\substack{z_{n}\\
\alpha_{n}
}
\right.\right)=e^{-2\pi i\alpha_{n}}B_{\beta_{\mathrm{down}}}^{\infty,\mathrm{f}}\left(\left.\substack{z_{0}\\
\alpha_{0}
}
\right|\left.\substack{z_{n}\\
\alpha_{n}
}
\right.\right),
\]
and thus we have
\begin{align*}
\hat{B}_{\alpha}^{\mathrm{f},\mathrm{f}}\left(\left.\substack{z_{0}\\
\alpha_{0}
}
\right|\left.\substack{z_{1}\\
\alpha_{1}
}
\right|\cdots\left|\substack{z_{n}\\
\alpha_{n}
}
\right.\right) & =\frac{B_{\alpha}^{\mathrm{f},\mathrm{f}}\left(\left.\substack{z_{0}\\
\alpha_{0}
}
\right|\left.\substack{z_{1}\\
\alpha_{1}
}
\right|\cdots\left|\substack{z_{n}\\
\alpha_{n}
}
\right.\right)}{B_{\alpha}^{\mathrm{f},\mathrm{f}}\left(\left.\substack{z_{0}\\
\alpha_{0}
}
\right|\left.\substack{z_{n}\\
\alpha_{n}
}
\right.\right)}\\
 & =\frac{e^{\pi i\alpha_{0}}B_{\beta_{\mathrm{up}}}^{\infty,\mathrm{f}}\left(\left.\substack{z_{0}\\
\alpha_{0}
}
\right|\left.\substack{z_{1}\\
\alpha_{1}
}
\right|\cdots\left|\substack{z_{n}\\
\alpha_{n}
}
\right.\right)-e^{-\pi i\alpha_{0}}B_{\beta_{\mathrm{down}}}^{\infty,\mathrm{f}}\left(\left.\substack{z_{0}\\
\alpha_{0}
}
\right|\left.\substack{z_{1}\\
\alpha_{1}
}
\right|\cdots\left|\substack{z_{n}\\
\alpha_{n}
}
\right.\right)}{e^{\pi i\alpha_{0}}B_{\beta_{\mathrm{up}}}^{\infty,\mathrm{f}}\left(\left.\substack{z_{0}\\
\alpha_{0}
}
\right|\left.\substack{z_{n}\\
\alpha_{n}
}
\right.\right)-e^{-\pi i\alpha_{0}}B_{\beta_{\mathrm{down}}}^{\infty,\mathrm{f}}\left(\left.\substack{z_{0}\\
\alpha_{0}
}
\right|\left.\substack{z_{n}\\
\alpha_{n}
}
\right.\right)}\\
 & =\frac{1}{1-e^{-2\pi i(\alpha_{0}-\alpha_{n})}}\hat{B}_{\beta_{\mathrm{up}}}^{\infty,\mathrm{f}}\left(\left.\substack{z_{0}\\
\alpha_{0}
}
\right|\left.\substack{z_{1}\\
\alpha_{1}
}
\right|\cdots\left|\substack{z_{n}\\
\alpha_{n}
}
\right.\right)+\frac{1}{1-e^{2\pi i(\alpha_{0}-\alpha_{n})}}\hat{B}_{\beta_{\mathrm{down}}}^{\infty,\mathrm{f}}\left(\left.\substack{z_{0}\\
\alpha_{0}
}
\right|\left.\substack{z_{1}\\
\alpha_{1}
}
\right|\cdots\left|\substack{z_{n}\\
\alpha_{n}
}
\right.\right).
\end{align*}
\end{proof}

\section{Certain algebraic relations\label{sec:Algebraic-relations}}

Iterated beta integrals satisfy the following algebraic relations:
\begin{thm}
\label{thm:algebraic_relation}Let $m,n$ be integers with $0<m<n$,
and $z_{0},\ldots,z_{n},\alpha_{0},\ldots,\alpha_{n}\in\mathbb{C}$
be complex numbers such that $z_{i}\neq z_{m}$ for $i\neq m$. Let
$\gamma$ be a path from $z_{0}$ to $z_{n}$. Let $\gamma_{m}$ be
a simple path from $\infty$ to $z_{m}$ on $\mathbb{P}^{1}(\mathbb{C})\setminus\left\{ z_{0},\ldots,\widehat{z_{m}},\ldots,z_{n}\right\} $
which does not intersect with $\gamma$. Then, we have
\[
\hat{\mathscr{B}}_{\gamma}^{{\rm f},{\rm f}}\left(\left.\substack{z_{0}\\
\alpha_{0}
}
\right|\left.\substack{z_{1}\\
\alpha_{1}
}
\right|\cdots\left|\substack{z_{n}\\
\alpha_{n}
}
\right.\right)=\sum_{\substack{0\leq i<m\\
m<j\leq k\leq n
}
}(-1)^{k-1-m}\hat{\mathscr{B}}_{\gamma}^{{\rm f},{\rm f}}\left(\left.\substack{z_{0}\\
\alpha_{0}
}
\right|\cdots\left|\substack{z_{i}\\
\alpha_{i}
}
\right|\left.\substack{z_{k}\\
\alpha_{k}
}
\right|\cdots\left|\substack{z_{n}\\
\alpha_{n}
}
\right.\right)\hat{\mathscr{B}}_{\gamma_{m}^{-1}}^{{\rm f},\infty}\left(\left.\substack{z_{m}\\
\alpha_{m}
}
\right|\cdots\left|\substack{z_{j}\\
\alpha_{j}
}
\right.\right)\hat{\mathscr{B}}_{\gamma_{m}}^{\infty,{\rm f}}\left(\left.\substack{z_{j}\\
\alpha_{j}+1
}
\right|\cdots\left|\substack{z_{k}\\
\alpha_{k}+1
}
\right|\left.\substack{z_{i}\\
\alpha_{i}
}
\right|\cdots\left|\substack{z_{m}\\
\alpha_{m}
}
\right.\right).
\]
\end{thm}

\begin{proof}
The differential formula (Theorem \ref{Thm:differential_eq_complete})
says,
\[
d\hat{\mathscr{B}}_{\gamma}^{{\rm f},{\rm f}}\left(\left.\substack{z_{0}\\
\alpha_{0}
}
\right|\left.\substack{z_{1}\\
\alpha_{1}
}
\right|\cdots\left|\substack{z_{n}\\
\alpha_{n}
}
\right.\right)=\sum_{i=1}^{n-1}\hat{\mathscr{B}}_{\gamma}^{{\rm f},{\rm f}}\left(\left.\substack{z_{0}\\
\alpha_{0}
}
\right|\cdots\hat{\left|\substack{z_{i}\\
\alpha_{i}
}
\right|}\cdots\left|\substack{z_{n}\\
\alpha_{n}
}
\right.\right)\cdot\chi_{i-1,i,i+1}d\log\left(\frac{z_{i}-z_{i+1}}{z_{i}-z_{i-1}}\right)
\]
where we put
\[
\chi_{i,j,k}:=\chi\left(\left.\substack{z_{i}\\
\alpha_{i}
}
\right|\substack{z_{j}\\
\alpha_{j}
}
\left|\substack{z_{k}\\
\alpha_{k}
}
\right.\right).
\]
Especially, the coefficient of $dz_{m}$ gives
\begin{align*}
\frac{\partial}{\partial z_{m}}\hat{\mathscr{B}}_{\gamma}^{{\rm f},{\rm f}}\left(\left.\substack{z_{0}\\
\alpha_{0}
}
\right|\left.\substack{z_{1}\\
\alpha_{1}
}
\right|\cdots\left|\substack{z_{n}\\
\alpha_{n}
}
\right.\right) & =\hat{\mathscr{B}}_{\gamma}^{{\rm f},{\rm f}}\bigl(\left.\substack{z_{0}\\
\alpha_{0}
}
\right|\cdots\hat{\left|\substack{z_{m-1}\\
\alpha_{m-1}
}
\right|}\cdots\left|\substack{z_{n}\\
\alpha_{n}
}
\right.\bigr)\cdot\chi_{m-2,m-1,m}\frac{\partial}{\partial z_{m}}\log\left(\frac{z_{m-1}-z_{m}}{z_{m-1}-z_{m-2}}\right)\\
 & \quad+\hat{\mathscr{B}}_{\gamma}^{{\rm f},{\rm f}}\bigl(\left.\substack{z_{0}\\
\alpha_{0}
}
\right|\cdots\hat{\left|\substack{z_{m}\\
\alpha_{m}
}
\right|}\cdots\left|\substack{z_{n}\\
\alpha_{n}
}
\right.\bigr)\cdot\chi_{m-1,m,m+1}\frac{\partial}{\partial z_{m}}\log\left(\frac{z_{m}-z_{m+1}}{z_{m}-z_{m-1}}\right)\\
 & \quad+\hat{\mathscr{B}}_{\gamma}^{{\rm f},{\rm f}}\bigl(\left.\substack{z_{0}\\
\alpha_{0}
}
\right|\cdots\hat{\left|\substack{z_{m+1}\\
\alpha_{m+1}
}
\right|}\cdots\left|\substack{z_{n}\\
\alpha_{n}
}
\right.\bigr)\cdot\chi_{m,m+1,m+2}\frac{\partial}{\partial z_{m}}\log\left(\frac{z_{m+1}-z_{m+2}}{z_{m+1}-z_{m}}\right),
\end{align*}
where we ignore the first (resp. third) term if $m=1$ (resp. $m=n$).
Here, 
\[
\chi_{m-2,m-1,m}\frac{\partial}{\partial z_{m}}\log\left(\frac{z_{m-1}-z_{m}}{z_{m-1}-z_{m-2}}\right)=h_{m-2,m-1}(z_{m}),
\]
\[
\chi_{m-1,m,m+1}\frac{\partial}{\partial z_{m}}\log\left(\frac{z_{m}-z_{m+1}}{z_{m}-z_{m-1}}\right)=h_{m+1,m-1}(z_{m}),
\]
and
\begin{align*}
 & \chi_{m,m+1,m+2}\frac{\partial}{\partial z_{m}}\log\left(\frac{z_{m+1}-z_{m+2}}{z_{m+1}-z_{m}}\right)=h_{m+1,m+2}(z_{m}),
\end{align*}
where
\[
h_{i,j}(t)\coloneqq\begin{cases}
\left\{ \substack{z_{i},z_{j}\\
\beta_{i},\beta_{j}
}
\right\} (t) & \substack{i=j-1\\
0<j<m
}
\\
\left\{ \substack{z_{i},z_{j}\\
1+\beta_{i},\beta_{j}
}
\right\} (t) & \substack{m<i\leq n\\
0\leq j<m
}
\\
\left\{ \substack{z_{i},z_{j}\\
1+\beta_{i},1+\beta_{j}
}
\right\} (t) & \substack{m<i<n\\
j=i+1.
}
\end{cases}\quad\text{with }\beta_{\ell}=\alpha_{\ell}-\alpha_{m}\:\left(0\leq\ell\leq n\right).
\]
Notice that $h_{i,j}(t)$ is defined only for such $(i,j)$ that $i,j\in\{0,\ldots,\hat{m},\ldots,n\}$
(so none of $h_{i,j}(t)$ contains $z_{m}$) and either $j-i=1$ or
$j<m<i$. Setting $\Omega_{i,j}(t)=h_{i,j}(t)dt$, we have
\begin{align*}
\frac{\partial}{\partial z_{m}}\hat{\mathscr{B}}_{\gamma}^{{\rm f},{\rm f}}\left(\left.\substack{z_{0}\\
\alpha_{0}
}
\right|\left.\substack{z_{1}\\
\alpha_{1}
}
\right|\cdots\left|\substack{z_{n}\\
\alpha_{n}
}
\right.\right)dz_{m} & =\hat{\mathscr{B}}_{\gamma}^{{\rm f},{\rm f}}\bigl(\left.\substack{z_{0}\\
\alpha_{0}
}
\right|\cdots\hat{\left|\substack{z_{m-1}\\
\alpha_{m-1}
}
\right|}\cdots\left|\substack{z_{n}\\
\alpha_{n}
}
\right.\bigr)\cdot\Omega_{m-2,m-1}(z_{m})\\
 & \quad+\hat{\mathscr{B}}_{\gamma}^{{\rm f},{\rm f}}\bigl(\left.\substack{z_{0}\\
\alpha_{0}
}
\right|\cdots\hat{\left|\substack{z_{m}\\
\alpha_{m}
}
\right|}\cdots\left|\substack{z_{n}\\
\alpha_{n}
}
\right.\bigr)\cdot\Omega_{m+1,m-1}(z_{m})\\
 & \quad+\hat{\mathscr{B}}_{\gamma}^{{\rm f},{\rm f}}\bigl(\left.\substack{z_{0}\\
\alpha_{0}
}
\right|\cdots\hat{\left|\substack{z_{m+1}\\
\alpha_{m+1}
}
\right|}\cdots\left|\substack{z_{n}\\
\alpha_{n}
}
\right.\bigr)\cdot\Omega_{m+1,m+2}(z_{m}).
\end{align*}
Replacing $z_{m}$ with $t$ and integrating both sides with respect
to $t$ from $\infty$ to $z_{m}$ along $\gamma_{m}$, we obtain
\begin{align}
\hat{\mathscr{B}}_{\gamma}^{{\rm f},{\rm f}}\left(\left.\substack{z_{0}\\
\alpha_{0}
}
\right|\left.\substack{z_{1}\\
\alpha_{1}
}
\right|\cdots\left|\substack{z_{n}\\
\alpha_{n}
}
\right.\right) & =I_{\gamma_{m}}(\infty;\hat{\mathscr{B}}_{\gamma}^{{\rm f},{\rm f}}\bigl(\left.\substack{z_{0}\\
\alpha_{0}
}
\right|\cdots\hat{\left|\substack{z_{m-1}\\
\alpha_{m-1}
}
\right|}\cdots\left|\substack{z_{n}\\
\alpha_{n}
}
\right.\bigr)\bigr|_{z_{m}\mapsto t}\cdot\Omega_{m-2,m-1}(t);z_{m})\label{eq:first_step_conf}\\
 & \quad+\hat{\mathscr{B}}_{\gamma}^{{\rm f},{\rm f}}\bigl(\left.\substack{z_{0}\\
\alpha_{0}
}
\right|\cdots\hat{\left|\substack{z_{m}\\
\alpha_{m}
}
\right|}\cdots\left|\substack{z_{n}\\
\alpha_{n}
}
\right.\bigr)\cdot I_{\gamma_{m}}(\infty;\Omega_{m+1,m-1}(t);z_{m})\nonumber \\
 & \quad+I_{\gamma_{m}}(\infty;\hat{\mathscr{B}}_{\gamma}^{{\rm f},{\rm f}}\bigl(\left.\substack{z_{0}\\
\alpha_{0}
}
\right|\cdots\hat{\left|\substack{z_{m+1}\\
\alpha_{m+1}
}
\right|}\cdots\left|\substack{z_{n}\\
\alpha_{n}
}
\right.\bigr)\bigr|_{z_{m}\mapsto t}\cdot\Omega_{m+1,m+2}(t);z_{m}).\nonumber 
\end{align}
Applying this formula to $\hat{\mathscr{B}}_{\gamma}^{{\rm f},{\rm f}}\left(\left.\substack{z_{0}\\
\alpha_{0}
}
\right|\cdots\hat{\left|\substack{z_{m-1}\\
\alpha_{m-1}
}
\right|}\cdots\left|\substack{z_{n}\\
\alpha_{n}
}
\right.\right)$ and $\hat{\mathscr{B}}_{\gamma}^{{\rm f},{\rm f}}\left(\left.\substack{z_{0}\\
\alpha_{0}
}
\right|\cdots\hat{\left|\substack{z_{m+1}\\
\alpha_{m+1}
}
\right|}\cdots\left|\substack{z_{n}\\
\alpha_{n}
}
\right.\right)$, we find
\begin{align*}
\hat{\mathscr{B}}_{\gamma}^{{\rm f},{\rm f}}\left(\left.\substack{z_{0}\\
\alpha_{0}
}
\right|\cdots\hat{\left|\substack{z_{m-1}\\
\alpha_{m-1}
}
\right|}\cdots\left|\substack{z_{n}\\
\alpha_{n}
}
\right.\right) & =I_{\gamma_{m}}(\infty;\hat{\mathscr{B}}_{\gamma}^{{\rm f},{\rm f}}\bigl(\left.\substack{z_{0}\\
\alpha_{0}
}
\right|\cdots\hat{\left|\substack{z_{m-2}\\
\alpha_{m-2}
}
\right|}\hat{\left.\substack{z_{m-1}\\
\alpha_{m-1}
}
\right|}\cdots\left|\substack{z_{n}\\
\alpha_{n}
}
\right.\bigr)\bigr|_{z_{m}\mapsto t}\cdot\Omega_{m-3,m-2}(t);z_{m})\\
 & \quad+\hat{\mathscr{B}}_{\gamma}^{{\rm f},{\rm f}}\bigl(\left.\substack{z_{0}\\
\alpha_{0}
}
\right|\cdots\hat{\left|\substack{z_{m-1}\\
\alpha_{m-1}
}
\right|}\hat{\left.\substack{z_{m}\\
\alpha_{m}
}
\right|}\cdots\left|\substack{z_{n}\\
\alpha_{n}
}
\right.\bigr)\cdot I_{\gamma_{m}}(\infty;\Omega_{m+1,m-2}(t);z_{m})\\
 & \quad+I_{\gamma_{m}}(\infty;\hat{\mathscr{B}}_{\gamma}^{{\rm f},{\rm f}}\bigl(\left.\substack{z_{0}\\
\alpha_{0}
}
\right|\cdots\hat{\left|\substack{z_{m-1}\\
\alpha_{m-1}
}
\right|}\left.\substack{z_{m}\\
\alpha_{m}
}
\right|\hat{\left.\substack{z_{m+1}\\
\alpha_{m+1}
}
\right|}\cdots\left|\substack{z_{n}\\
\alpha_{n}
}
\right.\bigr)\bigr|_{z_{m}\mapsto t}\cdot\Omega_{m+1,m+2}(t);z_{m}),
\end{align*}
and
\begin{align*}
\hat{\mathscr{B}}_{\gamma}^{{\rm f},{\rm f}}\left(\left.\substack{z_{0}\\
\alpha_{0}
}
\right|\cdots\hat{\left|\substack{z_{m+1}\\
\alpha_{m+1}
}
\right|}\cdots\left|\substack{z_{n}\\
\alpha_{n}
}
\right.\right) & =I_{\gamma_{m}}(\infty;\hat{\mathscr{B}}_{\gamma}^{{\rm f},{\rm f}}\bigl(\left.\substack{z_{0}\\
\alpha_{0}
}
\right|\cdots\hat{\left|\substack{z_{m-1}\\
\alpha_{m-1}
}
\right|}\left.\substack{z_{m}\\
\alpha_{m}
}
\right|\hat{\left.\substack{z_{m+1}\\
\alpha_{m+1}
}
\right|}\cdots\left|\substack{z_{n}\\
\alpha_{n}
}
\right.\bigr)\bigr|_{z_{m}\mapsto t}\cdot\Omega_{m-2,m-1}(t);z_{m})\\
 & \quad+\hat{\mathscr{B}}_{\gamma}^{{\rm f},{\rm f}}\bigl(\left.\substack{z_{0}\\
\alpha_{0}
}
\right|\cdots\hat{\left|\substack{z_{m}\\
\alpha_{m}
}
\right|}\hat{\left.\substack{z_{m+1}\\
\alpha_{m+1}
}
\right|}\cdots\left|\substack{z_{n}\\
\alpha_{n}
}
\right.\bigr)\cdot I_{\gamma_{m}}(\infty;\Omega_{m+2,m-1}(t);z_{m})\\
 & \quad+I_{\gamma_{m}}(\infty;\hat{\mathscr{B}}_{\gamma}^{{\rm f},{\rm f}}\bigl(\left.\substack{z_{0}\\
\alpha_{0}
}
\right|\cdots\hat{\left|\substack{z_{m+1}\\
\alpha_{m+1}
}
\right|}\hat{\left.\substack{z_{m+2}\\
\alpha_{m+2}
}
\right|}\cdots\left|\substack{z_{n}\\
\alpha_{n}
}
\right.\bigr)\bigr|_{z_{m}\mapsto t}\cdot\Omega_{m+2,m+3}(t);z_{m}).
\end{align*}
Plugging these into (\ref{eq:first_step_conf}), 
\begin{align*}
\hat{\mathscr{B}}_{\gamma}^{{\rm f},{\rm f}}\left(\left.\substack{z_{0}\\
\alpha_{0}
}
\right|\left.\substack{z_{1}\\
\alpha_{1}
}
\right|\cdots\left|\substack{z_{n}\\
\alpha_{n}
}
\right.\right) & =I_{\gamma_{m}}(\infty;\hat{\mathscr{B}}_{\gamma}^{{\rm f},{\rm f}}\bigl(\left.\substack{z_{0}\\
\alpha_{0}
}
\right|\cdots\hat{\left|\substack{z_{m-2}\\
\alpha_{m-2}
}
\right|}\hat{\left.\substack{z_{m-1}\\
\alpha_{m-1}
}
\right|}\cdots\left|\substack{z_{n}\\
\alpha_{n}
}
\right.\bigr)\bigr|_{z_{m}\mapsto t}\cdot\Omega_{m-3,m-2}(t),\Omega_{m-2,m-1};z_{m})\\
 & \quad+\hat{\mathscr{B}}_{\gamma}^{{\rm f},{\rm f}}\bigl(\left.\substack{z_{0}\\
\alpha_{0}
}
\right|\cdots\hat{\left|\substack{z_{m-1}\\
\alpha_{m-1}
}
\right|}\hat{\left.\substack{z_{m}\\
\alpha_{m}
}
\right|}\cdots\left|\substack{z_{n}\\
\alpha_{n}
}
\right.\bigr)\cdot I_{\gamma_{m}}(\infty;\Omega_{m+1,m-2},\Omega_{m-2,m-1};z_{m})\\
 & \quad+I_{\gamma_{m}}(\infty;\hat{\mathscr{B}}_{\gamma}^{{\rm f},{\rm f}}\bigl(\left.\substack{z_{0}\\
\alpha_{0}
}
\right|\cdots\hat{\left|\substack{z_{m-1}\\
\alpha_{m-1}
}
\right|}\left.\substack{z_{m}\\
\alpha_{m}
}
\right|\hat{\left.\substack{z_{m+1}\\
\alpha_{m+1}
}
\right|}\cdots\left|\substack{z_{n}\\
\alpha_{n}
}
\right.\bigr)\bigr|_{z_{m}\mapsto t}\cdot\Omega_{m+1,m+2}(t),\Omega_{m-2,m-1};z_{m})\\
 & \quad+\hat{\mathscr{B}}_{\gamma}^{{\rm f},{\rm f}}\bigl(\left.\substack{z_{0}\\
\alpha_{0}
}
\right|\cdots\hat{\left|\substack{z_{m}\\
\alpha_{m}
}
\right|}\cdots\left|\substack{z_{n}\\
\alpha_{n}
}
\right.\bigr)\cdot I_{\gamma_{m}}(\infty;\Omega_{m+1,m-1}(t);z_{m})\\
 & \quad+I_{\gamma_{m}}(\infty;\hat{\mathscr{B}}_{\gamma}^{{\rm f},{\rm f}}\bigl(\left.\substack{z_{0}\\
\alpha_{0}
}
\right|\cdots\hat{\left|\substack{z_{m-1}\\
\alpha_{m-1}
}
\right|}\left.\substack{z_{m}\\
\alpha_{m}
}
\right|\hat{\left.\substack{z_{m+1}\\
\alpha_{m+1}
}
\right|}\cdots\left|\substack{z_{n}\\
\alpha_{n}
}
\right.\bigr)\bigr|_{z_{m}\mapsto t}\cdot\Omega_{m-2,m-1}(t),\Omega_{m+1,m+2};z_{m})\\
 & \quad+\hat{\mathscr{B}}_{\gamma}^{{\rm f},{\rm f}}\bigl(\left.\substack{z_{0}\\
\alpha_{0}
}
\right|\cdots\hat{\left|\substack{z_{m}\\
\alpha_{m}
}
\right|}\hat{\left.\substack{z_{m+1}\\
\alpha_{m+1}
}
\right|}\cdots\left|\substack{z_{n}\\
\alpha_{n}
}
\right.\bigr)\cdot I_{\gamma_{m}}(\infty;\Omega_{m+2,m-1},\Omega_{m+1,m+2};z_{m})\\
 & \quad+I_{\gamma_{m}}(\infty;\hat{\mathscr{B}}_{\gamma}^{{\rm f},{\rm f}}\bigl(\left.\substack{z_{0}\\
\alpha_{0}
}
\right|\cdots\hat{\left|\substack{z_{m+1}\\
\alpha_{m+1}
}
\right|}\hat{\left.\substack{z_{m+2}\\
\alpha_{m+2}
}
\right|}\cdots\left|\substack{z_{n}\\
\alpha_{n}
}
\right.\bigr)\bigr|_{z_{m}\mapsto t}\cdot\Omega_{m+2,m+3}(t),\Omega_{m+1,m+2};z_{m}).
\end{align*}
Repeating this process until all the $\hat{\mathscr{B}}_{\gamma}^{{\rm f},{\rm f}}$'s
inside $I(\infty;-;z_{m})$ on the right-hand side disappear (that
is, until no $\hat{\mathscr{B}}_{\gamma}^{{\rm f},{\rm f}}$ contains
$z_{m}$ anymore), one arrives at the formula
\begin{align*}
 & \hat{\mathscr{B}}_{\gamma}^{{\rm f},{\rm f}}\left(\left.\substack{z_{0}\\
\alpha_{0}
}
\right|\left.\substack{z_{1}\\
\alpha_{1}
}
\right|\cdots\left|\substack{z_{n}\\
\alpha_{n}
}
\right.\right)\\
= & \sum_{\substack{0\leq i<m\\
m<k\leq n
}
}\hat{\mathscr{B}}_{\gamma}^{{\rm f},{\rm f}}\left(\left.\substack{z_{0}\\
\alpha_{0}
}
\right|\cdots\left|\substack{z_{i}\\
\alpha_{i}
}
\right|\left.\substack{z_{k}\\
\alpha_{k}
}
\right|\cdots\left|\substack{z_{n}\\
\alpha_{n}
}
\right.\right)\\
 & \qquad\times I_{\gamma_{m}}\left(\infty;\Omega_{k,i},(\Omega_{i,i+1},\Omega_{i+1,i+2},\ldots,\Omega_{m-2,m-1})\shuffle(\Omega_{k-1,k},\Omega_{k-2,k-1},\ldots,\Omega_{m+1,m+2});z_{m}\right),
\end{align*}
where $\shuffle$ denotes the shuffle product of two sequences, i.e.,
the linear combinations of all shuffles of the two sequences.

Now, for $0\leq i<m$ and $m<j\leq n$, put
\[
\boldsymbol{u}_{i}=(\Omega_{i,i+1},\Omega_{i+1,i+2},\ldots,\Omega_{m-2,m-1})\;\text{and}\;\boldsymbol{v}_{j}=(\Omega_{j-1,j},\Omega_{j-2,j-1},\ldots,\Omega_{m+1,m+2})
\]
Then,
\begin{align*}
 & (\Omega_{k,i},\boldsymbol{u}_{i}\shuffle\boldsymbol{v}_{k})\\
 & =((\Omega_{k,i},\boldsymbol{u}_{i})\shuffle\boldsymbol{v}_{k})-(\Omega_{k-1,k},(\Omega_{k,i},\boldsymbol{u}_{i})\shuffle\boldsymbol{v}_{k-1})\\
 & =((\Omega_{k,i},\boldsymbol{u}_{i})\shuffle\boldsymbol{v}_{k})-((\Omega_{k-1,k},\Omega_{k,i},\boldsymbol{u}_{i})\shuffle\boldsymbol{v}_{k-1})+(\Omega_{k-2,k-1},(\Omega_{k-1,k},\Omega_{k,i},\boldsymbol{u}_{i})\shuffle\boldsymbol{v}_{k-2})\\
 & =\cdots\\
 & =\sum_{m<j\leq k}(-1)^{k-j}(\Omega_{j,j+1},\dots,\Omega_{k-2,k-1},\Omega_{k-1,k},\Omega_{k,i},\boldsymbol{u}_{i})\shuffle\boldsymbol{v}_{j}
\end{align*}
Thus,
\begin{align*}
 & \hat{\mathscr{B}}_{\gamma}^{{\rm f},{\rm f}}\left(\left.\substack{z_{0}\\
\alpha_{0}
}
\right|\left.\substack{z_{1}\\
\alpha_{1}
}
\right|\cdots\left|\substack{z_{n}\\
\alpha_{n}
}
\right.\right)\\
 & =\sum_{\substack{0\leq i<m\\
m<k\leq n
}
}\hat{\mathscr{B}}_{\gamma}^{{\rm f},{\rm f}}\left(\left.\substack{z_{0}\\
\alpha_{0}
}
\right|\cdots\left|\substack{z_{i}\\
\alpha_{i}
}
\right|\left.\substack{z_{k}\\
\alpha_{k}
}
\right|\cdots\left|\substack{z_{n}\\
\alpha_{n}
}
\right.\right)\sum_{m<j\leq k}(-1)^{k-j}I_{\gamma_{m}}(\infty;\Omega_{j,j+1},\dots,\Omega_{k-2,k-1},\Omega_{k-1,k},\Omega_{k,i},\boldsymbol{u}_{i})\cdot I_{\gamma_{m}}(\infty;\boldsymbol{v}_{j};z_{m})\\
 & =\sum_{\substack{0\leq i<m\\
m<k\leq n
}
}\hat{\mathscr{B}}_{\gamma}^{{\rm f},{\rm f}}\left(\left.\substack{z_{0}\\
\alpha_{0}
}
\right|\cdots\left|\substack{z_{i}\\
\alpha_{i}
}
\right|\left.\substack{z_{k}\\
\alpha_{k}
}
\right|\cdots\left|\substack{z_{n}\\
\alpha_{n}
}
\right.\right)\\
 & \qquad\times\sum_{m<j\leq k}(-1)^{k-1-m}I_{\gamma_{m}}(\infty;\Omega_{j,j+1},\dots,\Omega_{k-1,k},\Omega_{k,i},\Omega_{i,i+1},\ldots,\Omega_{m-2,m-1};z_{m})\cdot I_{\gamma_{m}^{-1}}(z_{m};\Omega_{m+1,m+2},\dots,\Omega_{j-1,j};\infty)\\
 & =\sum_{\substack{0\leq i<m\\
m<k\leq n
}
}\hat{\mathscr{B}}_{\gamma}^{{\rm f},{\rm f}}\left(\left.\substack{z_{0}\\
\alpha_{0}
}
\right|\cdots\left|\substack{z_{i}\\
\alpha_{i}
}
\right|\left.\substack{z_{k}\\
\alpha_{k}
}
\right|\cdots\left|\substack{z_{n}\\
\alpha_{n}
}
\right.\right)\\
 & \qquad\times\sum_{m<j\leq k}(-1)^{k-1-m}\mathscr{B}_{\gamma_{m}}^{\infty}\left(\left.\substack{z_{j}\\
1+\beta_{j}
}
\right|\cdots\left|\substack{z_{k}\\
1+\beta_{k}
}
\right|\left.\substack{z_{i}\\
\beta_{i}
}
\right|\cdots\left|\substack{z_{m-1}\\
\beta_{m-1}
}
\right.;z_{m}\right)\cdot\mathscr{B}_{\gamma_{m}^{-1}}^{\infty}\left(z_{m};\left.\substack{z_{m+1}\\
1+\beta_{m+1}
}
\right|\cdots\left|\substack{z_{j}\\
1+\beta_{j}
}
\right.\right).
\end{align*}
By the relations between complete and incomplete iterated beta integrals
(Theorem \ref{thm: complete vs incomplete} (3)), the last quantity
equals 
\begin{align*}
 & \sum_{\substack{0\leq i<m\\
m<k\leq n
}
}\hat{\mathscr{B}}_{\gamma}^{{\rm f},{\rm f}}\left(\left.\substack{z_{0}\\
\alpha_{0}
}
\right|\cdots\left|\substack{z_{i}\\
\alpha_{i}
}
\right|\left.\substack{z_{k}\\
\alpha_{k}
}
\right|\cdots\left|\substack{z_{n}\\
\alpha_{n}
}
\right.\right)\\
 & \qquad\times\sum_{m<j\leq k}(-1)^{k-1-m}\hat{\mathscr{B}}_{\gamma_{m}}^{\infty,{\rm f}}\left(\left.\substack{z_{j}\\
1+\beta_{j}
}
\right|\cdots\left|\substack{z_{k}\\
1+\beta_{k}
}
\right|\left.\substack{z_{i}\\
\beta_{i}
}
\right|\cdots\left|\substack{z_{m-1}\\
\beta_{m-1}
}
\right.\left|\substack{z_{m}\\
\beta_{m}
}
\right.\right)\cdot\hat{\mathscr{B}}_{\gamma_{m}^{-1}}^{{\rm f},\infty}\left(\left.\substack{z_{m}\\
1+\beta_{m}
}
\right|\left.\substack{z_{m+1}\\
1+\beta_{m+1}
}
\right|\cdots\left|\substack{z_{j}\\
1+\beta_{j}
}
\right.\right).
\end{align*}
Here, notice that $\beta_{m}=0$ by definition. Finally, applying
the translation invariance formula (Theorem \ref{thm: translation invariance})
for $\beta_{i}=\alpha_{i}-\alpha_{m}\mapsto\alpha_{i}$ ($0\leq i\leq n$),
this further equals
\[
\sum_{\substack{0\leq i<m\\
m<k\leq n
}
}\hat{\mathscr{B}}_{\gamma}^{{\rm f},{\rm f}}\left(\left.\substack{z_{0}\\
\alpha_{0}
}
\right|\cdots\left|\substack{z_{i}\\
\alpha_{i}
}
\right|\left.\substack{z_{k}\\
\alpha_{k}
}
\right|\cdots\left|\substack{z_{n}\\
\alpha_{n}
}
\right.\right)\sum_{m<j\leq k}(-1)^{k-1-m}\hat{\mathscr{B}}_{\gamma_{m}}^{\infty,{\rm f}}\left(\left.\substack{z_{j}\\
1+\alpha_{j}
}
\right|\cdots\left|\substack{z_{k}\\
1+\alpha_{k}
}
\right|\left.\substack{z_{i}\\
\alpha_{i}
}
\right|\cdots\left|\substack{z_{m}\\
\alpha_{m}
}
\right.\right)\cdot\hat{\mathscr{B}}_{\gamma_{m}^{-1}}^{{\rm f},\infty}\left(\left.\substack{z_{m}\\
\alpha_{m}
}
\right|\cdots\left|\substack{z_{j}\\
\alpha_{j}
}
\right.\right),
\]
which completes the proof of Theorem \ref{thm:algebraic_relation}.
\end{proof}
\begin{rem}
Theorem \ref{thm:algebraic_relation} can also be rewritten in terms
of iterated beta integrals with entries of the form $\left(\substack{z_{i}\\
\alpha_{i}
}
\right)$ using the formula 
\begin{align*}
\hat{\mathscr{B}}_{\gamma_{m}}^{\infty,{\rm f}}\left(\left.\substack{z_{j}\\
1+\alpha_{j}
}
\right|\cdots\left|\substack{z_{k}\\
1+\alpha_{k}
}
\right|\left.\substack{z_{i}\\
\alpha_{i}
}
\right|\cdots\left|\substack{z_{m}\\
\alpha_{m}
}
\right.\right) & =\frac{\alpha_{k}-\alpha_{i}}{\alpha_{j}-\alpha_{m}}\hat{\mathscr{B}}_{\gamma_{m}}^{\infty,{\rm f}}\left(\left.\substack{z_{j}\\
\alpha_{j}
}
\right|\cdots\left|\substack{z_{k}\\
\alpha_{k}
}
\right|\left.\substack{z_{i}\\
\alpha_{i}
}
\right|\cdots\left|\substack{z_{m}\\
\alpha_{m}
}
\right.\right)\\
 & -\frac{1}{\alpha_{j}-\alpha_{m}}\chi_{k-1,k,i}\hat{\mathscr{B}}_{\gamma_{m}}^{\infty,{\rm f}}\left(\left.\substack{z_{j}\\
\alpha_{j}
}
\right|\cdots\left|\substack{z_{k-1}\\
\alpha_{k-1}
}
\right.\widehat{\left|\substack{z_{k}\\
\alpha_{k}
}
\right|}\left.\substack{z_{i}\\
\alpha_{i}
}
\right|\cdots\left|\substack{z_{m}\\
\alpha_{m}
}
\right.\right)\\
 & +\frac{1}{\alpha_{j}-\alpha_{m}}\chi_{k,i,i+1}\hat{\mathscr{B}}_{\gamma_{m}}^{\infty,{\rm f}}\left(\left.\substack{z_{j}\\
\alpha_{j}
}
\right|\cdots\left|\substack{z_{k}\\
\alpha_{k}
}
\right.\widehat{\left|\substack{z_{i}\\
\alpha_{i}
}
\right|}\left.\substack{z_{i+1}\\
\alpha_{i+1}
}
\right|\cdots\left|\substack{z_{m}\\
\alpha_{m}
}
\right.\right)
\end{align*}
of Theorem \ref{thm: contiguous relations}. Moreover, via Theorem
\ref{thm: infinite and finite paths comparison}, $\hat{\mathscr{B}}_{\gamma_{m}}^{{\rm f},{\rm f}}$
can be rewritten in terms of $\hat{\mathscr{B}}_{\gamma_{m}}^{\infty,{\rm f}}$'s,
so the theorem can be viewed as genuine relations among $\hat{\mathscr{B}}_{\gamma_{m}}^{\infty,{\rm f}}$'s. 
\end{rem}

\section{Monodromy of iterated beta integrals\label{sec:Monodromies}}

In this section, we give a formula for the monodromy of iterated beta
integrals.
\begin{thm}[Monodromy of iterated beta integrals]
 Let $z_{0},\dots,z_{n}$ be complex numbers, and suppose that $p\in\{1,\dots,n-1\}$
satisfies $z_{q}\neq z_{p}$ for all $q\in\{0,\dots,n\}\setminus\{p\}$
and $\alpha_{p}\notin\mathbb{Z}$. Let $\bullet,\circ\in\{\mathrm{f},\infty\}$.
Put
\[
z:=\begin{cases}
z_{0} & \bullet=\mathrm{f}\\
\infty & \bullet=\infty
\end{cases},\qquad z':=\begin{cases}
z_{n} & \circ=\mathrm{f}\\
\infty & \circ=\infty
\end{cases}.
\]
Let $\beta$ be a path from $z$ to $z_{p}$, and $\gamma$ be a path
from $z_{p}$ to $z'$, and $C$ be a small closed path which encircles
$z_{p}$ counterclockwisely. Further, let $\tilde{\beta}\tilde{\gamma}_{1}$
and $\tilde{\beta}\tilde{C}\tilde{\gamma}_{2}$ be some lifts of $\beta\gamma$
and $\beta C\gamma$. Then, we have
\begin{align*}
B_{\tilde{\beta}\tilde{C}\tilde{\gamma}_{2}}^{\bullet,\circ}\left(\left.\substack{z_{0}\\
\alpha_{0}
}
\right|\left.\substack{z_{1}\\
\alpha_{1}
}
\right|\cdots\left|\substack{z_{n}\\
\alpha_{n}
}
\right.\right)-B_{\tilde{\beta}\tilde{\gamma}_{1}}^{\bullet,\circ}\left(\left.\substack{z_{0}\\
\alpha_{0}
}
\right|\left.\substack{z_{1}\\
\alpha_{1}
}
\right|\cdots\left|\substack{z_{n}\\
\alpha_{n}
}
\right.\right) & =(e^{-2\pi i\alpha_{p}}-1)B_{\tilde{\beta}}^{\bullet,\mathrm{f}}\left(\left.\substack{z_{0}\\
\alpha_{0}
}
\right|\left.\substack{z_{1}\\
\alpha_{1}
}
\right|\cdots\left|\substack{z_{p}\\
\alpha_{p}
}
\right.\right)B_{\tilde{\gamma}_{1}}^{\mathrm{f},\circ}\left(\left.\substack{z_{p}\\
\alpha_{p}
}
\right|\left.\substack{z_{p+1}\\
\alpha_{p+1}
}
\right|\cdots\left|\substack{z_{n}\\
\alpha_{n}
}
\right.\right).
\end{align*}
\end{thm}

\begin{proof}
By the path composition formula,
\begin{align*}
 & B_{\tilde{\beta}\tilde{C}\tilde{\gamma}_{2}}^{\bullet,\circ}\left(\left.\substack{z_{0}\\
\alpha_{0}
}
\right|\left.\substack{z_{1}\\
\alpha_{1}
}
\right|\cdots\left|\substack{z_{n}\\
\alpha_{n}
}
\right.\right)-B_{\tilde{\beta}\tilde{\gamma}_{1}}^{\bullet,\circ}\left(\left.\substack{z_{0}\\
\alpha_{0}
}
\right|\left.\substack{z_{1}\\
\alpha_{1}
}
\right|\cdots\left|\substack{z_{n}\\
\alpha_{n}
}
\right.\right)\\
 & =\sum_{0\leq j\leq\ell\leq n}B_{\tilde{\beta}}^{\bullet}\left(\left.\substack{z_{0}\\
\alpha_{0}
}
\right|\cdots\left|\substack{z_{j}\\
\alpha_{j}
}
\right.;z_{p}\right)B_{C}\left(z_{p};\left.\substack{z_{j}\\
\alpha_{j}
}
\right|\cdots\left|\substack{z_{\ell}\\
\alpha_{\ell}
}
\right.;z_{p}\right)B_{\tilde{\gamma}_{2}}^{\circ}\left(z_{p};\left.\substack{z_{\ell}\\
\alpha_{\ell}
}
\right|\cdots\left|\substack{z_{n}\\
\alpha_{n}
}
\right.\right)\\
 & \quad-\sum_{0\leq j\leq n}B_{\tilde{\beta}}^{\bullet}\left(\left.\substack{z_{0}\\
\alpha_{0}
}
\right|\cdots\left|\substack{z_{j}\\
\alpha_{j}
}
\right.;z_{p}\right)B_{\tilde{\gamma}_{1}}^{\circ}\left(z_{p};\left.\substack{z_{j}\\
\alpha_{j}
}
\right|\cdots\left|\substack{z_{n}\\
\alpha_{n}
}
\right.\right)\\
 & =\sum_{j=0}^{n}B_{\tilde{\beta}}^{\bullet}\left(\left.\substack{z_{0}\\
\alpha_{0}
}
\right|\cdots\left|\substack{z_{j}\\
\alpha_{j}
}
\right.;z_{p}\right)\left(B_{\tilde{\gamma}_{2}}^{\circ}\left(z_{p};\left.\substack{z_{j}\\
\alpha_{j}
}
\right|\cdots\left|\substack{z_{n}\\
\alpha_{n}
}
\right.\right)-B_{\tilde{\gamma}_{1}}^{\circ}\left(z_{p};\left.\substack{z_{j}\\
\alpha_{j}
}
\right|\cdots\left|\substack{z_{n}\\
\alpha_{n}
}
\right.\right)\right)\\
 & \quad+\sum_{0\leq j<\ell\leq n}B_{\tilde{\beta}}^{\bullet}\left(\left.\substack{z_{0}\\
\alpha_{0}
}
\right|\cdots\left|\substack{z_{j}\\
\alpha_{j}
}
\right.;z_{p}\right)B_{C}\left(z_{p};\left.\substack{z_{j}\\
\alpha_{j}
}
\right|\cdots\left|\substack{z_{\ell}\\
\alpha_{\ell}
}
\right.;z_{p}\right)B_{\tilde{\gamma}_{2}}^{\circ}\left(z_{p};\left.\substack{z_{\ell}\\
\alpha_{\ell}
}
\right|\cdots\left|\substack{z_{n}\\
\alpha_{n}
}
\right.\right).
\end{align*}
By the assumption that $z_{q}\neq z_{p}$ for all $q\in\{0,\dots,n\}\setminus\{p\}$
and $\alpha_{p}\notin\mathbb{Z}$, 
\[
B_{C}\left(z_{p};\left.\substack{z_{j}\\
\alpha_{j}
}
\right|\cdots\left|\substack{z_{\ell}\\
\alpha_{\ell}
}
\right.;z_{p}\right)=\begin{cases}
1 & j=\ell\\
0 & j<\ell.
\end{cases}
\]
Hence,
\begin{align*}
 & B_{\tilde{\beta}\tilde{C}\tilde{\gamma}_{2}}^{\bullet,\circ}\left(\left.\substack{z_{0}\\
\alpha_{0}
}
\right|\left.\substack{z_{1}\\
\alpha_{1}
}
\right|\cdots\left|\substack{z_{n}\\
\alpha_{n}
}
\right.\right)-B_{\tilde{\beta}\tilde{\gamma}_{1}}^{\bullet,\circ}\left(\left.\substack{z_{0}\\
\alpha_{0}
}
\right|\left.\substack{z_{1}\\
\alpha_{1}
}
\right|\cdots\left|\substack{z_{n}\\
\alpha_{n}
}
\right.\right)\\
 & =\sum_{j=0}^{n}B_{\tilde{\beta}}^{\bullet}\left(\left.\substack{z_{0}\\
\alpha_{0}
}
\right|\cdots\left|\substack{z_{j}\\
\alpha_{j}
}
\right.;z_{p}\right)\left(B_{\tilde{\gamma}_{2}}^{\circ}\left(z_{p};\left.\substack{z_{j}\\
\alpha_{j}
}
\right|\cdots\left|\substack{z_{n}\\
\alpha_{n}
}
\right.\right)-B_{\tilde{\gamma}_{1}}^{\circ}\left(z_{p};\left.\substack{z_{j}\\
\alpha_{j}
}
\right|\cdots\left|\substack{z_{n}\\
\alpha_{n}
}
\right.\right)\right).
\end{align*}
Here,
\begin{align*}
 & B_{\tilde{\gamma}_{2}}^{\circ}\left(z_{p};\left.\substack{z_{j}\\
\alpha_{j}
}
\right|\cdots\left|\substack{z_{n}\\
\alpha_{n}
}
\right.\right)-B_{\tilde{\gamma}_{1}}^{\circ}\left(z_{p};\left.\substack{z_{j}\\
\alpha_{j}
}
\right|\cdots\left|\substack{z_{n}\\
\alpha_{n}
}
\right.\right)\\
 & =\begin{cases}
(e^{-2\pi i\alpha_{p}}-1)B_{\tilde{\gamma}_{1}}^{\circ}\left(z_{p};\left.\substack{z_{p}\\
\alpha_{p}
}
\right|\cdots\left|\substack{z_{n}\\
\alpha_{n}
}
\right.\right) & j=p\\
0 & j\neq p.
\end{cases}
\end{align*}
It readily follows that
\begin{align*}
B_{\tilde{\beta}\tilde{C}\tilde{\gamma}_{2}}^{\bullet,\circ}\left(\left.\substack{z_{0}\\
\alpha_{0}
}
\right|\left.\substack{z_{1}\\
\alpha_{1}
}
\right|\cdots\left|\substack{z_{n}\\
\alpha_{n}
}
\right.\right)-B_{\tilde{\beta}\tilde{\gamma}_{1}}^{\bullet,\circ}\left(\left.\substack{z_{0}\\
\alpha_{0}
}
\right|\left.\substack{z_{1}\\
\alpha_{1}
}
\right|\cdots\left|\substack{z_{n}\\
\alpha_{n}
}
\right.\right) & =(e^{-2\pi i\alpha_{p}}-1)B_{\tilde{\beta}}^{\bullet,\mathrm{f}}\left(\left.\substack{z_{0}\\
\alpha_{0}
}
\right|\left.\substack{z_{1}\\
\alpha_{1}
}
\right|\cdots\left|\substack{z_{p}\\
\alpha_{p}
}
\right.\right)B_{\tilde{\gamma}_{1}}^{\mathrm{f},\circ}\left(\left.\substack{z_{p}\\
\alpha_{p}
}
\right|\left.\substack{z_{p+1}\\
\alpha_{p+1}
}
\right|\cdots\left|\substack{z_{n}\\
\alpha_{n}
}
\right.\right).
\end{align*}
\end{proof}

\part{Consequences of translation invariance}

\section{\label{sec:classification_genus_one}Classification of genus zero
cases}

In this section, we will investigate applications of the iterated
beta integrals. In particular, we will classify hyperlogarithm relations
coming from the translation invariance. Let $\boldsymbol{z}=(z_{0},\ldots,z_{n+1})\in\mathbb{C}^{n+2}$
and $\boldsymbol{\alpha}=(\alpha_{0},\dots,\alpha_{n+1})\in\mathbb{C}^{n+2}$
be complex parameters. Recall that the iterated beta integrals are
defined as iterated integrals of the differential forms 
\[
\omega_{i}(t)=\frac{dt}{(t-z_{i})^{\alpha_{i}}(t-z_{i+1})^{1-\alpha_{i+1}}}\quad\left(i=0,\dots,n\right),
\]
which have monodromies $e^{-2\pi i\alpha_{i}}$, $e^{2\pi i\alpha_{i+1}}$,
and $e^{2\pi i(\alpha_{i}-\alpha_{i+1})}$ around $z_{i}$, $z_{i+1}$,
and $\infty$, respectively. Let $S\subset\mathbb{P}^{1}$ be the
set $\{z_{0},\dots,z_{n+1},\infty\}$ and $Y$ a universal abelian
covering space of $\mathbb{P}^{1}\setminus S$. Then, the homology
group $H_{1}(\mathbb{P}^{1}\setminus S,\mathbb{Z})\simeq\mathbb{Z}^{\#S-1}$
acts on $Y$. Let 
\[
\rho:H_{1}(\mathbb{P}^{1}\setminus S,\mathbb{Z})\to\left(\mathbb{C}^{\times}\right)^{n+1};\gamma\mapsto\rho(\gamma)=(\rho_{0}(\gamma),\ldots,\rho_{n}(\gamma))
\]
be the group representation associated to $(\omega_{0},\ldots,\omega_{n})$
defined by
\[
\rho_{i}(\gamma)=\frac{\omega_{i}(\gamma(t))}{\omega_{i}(t)}.
\]
Then, the iterated beta integral $I(z_{0};\omega_{0},\ldots,\omega_{n};z_{n+1})$
can be viewed as an iterated integral on the quotient $X=X_{\boldsymbol{z},\boldsymbol{\alpha}}\coloneqq Y/G$
with $G=\ker\rho.$ Now, assume that $\alpha_{0},\dots,\alpha_{n+1}$
are all rational and let $M=M_{\boldsymbol{\alpha}}$ be their common
denominator, i.e., the minimal positive integer satisfying $\alpha_{0}M,\dots,\alpha_{n+1}M\in\mathbb{Z}$.
Then, $G$ is a subgroup of $H_{1}(\mathbb{P}^{1}\setminus S,\mathbb{Z})$
of finite index, since $\mathrm{im}\,\rho\subset\mu_{M}^{n+1}$ where
$\mu_{M}\subset\mathbb{C}^{\times}$ is the group formed by the $M$-th
roots of unity, hence $X$ becomes an algebraic curve. Let us calculate
the geometric genus $g(X)$ of $X$ using the Riemann-Hurwitz formula
for the branched covering
\[
\Pi:X\to\mathbb{P}^{1}.
\]
First, the covering degree of $\Pi$ is equal to $N\coloneqq\#(\mathrm{im}\,\rho).$
Assume that $\#\{z_{0},\dots,z_{n+1}\}\geq2$. For $z\in S\subset\mathbb{P}^{1}$,
the ramification indices $e(z)$ at $z$ are given by
\[
e(z)=\begin{cases}
\min\{\nu\in\mathbb{Z}_{>0}\mid\nu\alpha_{i}\in\mathbb{Z}\text{ for \ensuremath{i} such that }z_{i}=z\} & z\neq\infty\\
\min\{\nu\in\mathbb{Z}_{>0}\mid\nu(\alpha_{i}-\alpha_{i+1})\in\mathbb{Z}\text{ for \ensuremath{0\leq i\leq n}}\} & z=\infty.
\end{cases}
\]
Notice that if $X_{\boldsymbol{z},\boldsymbol{\alpha}}$ is of genus
zero, the iterated beta integrals reduce to hyperlogarithms under
suitable change of variables, the case we are most interested in.
The following proposition gives the complete list of $(\boldsymbol{z};\boldsymbol{\alpha})$
such that $X_{\boldsymbol{z},\boldsymbol{\alpha}}$ is of genus zero:
\begin{prop}
\label{prop: complete list of genus 0 parameters}The complete list
of $(\boldsymbol{z};\boldsymbol{\alpha})=(z_{0},\ldots,z_{n+1};\alpha_{0},\dots,\alpha_{n+1})\in\mathbb{C}^{n+2}\times\mathbb{Q}^{n+2}$
for which $\#\{z_{0},\dots,z_{n+1}\}\geq2$ and $X_{\boldsymbol{z},\boldsymbol{\alpha}}$
is birational to $\mathbb{P}^{1}$ is given as follows ($S_{\mathrm{fin}}\coloneqq\left\{ z_{0},\ldots,z_{n+1}\right\} $).
\begin{itemize}
\item 3-branch point cases
\begin{enumerate}
\item[(b1)] $S_{\mathrm{fin}}=\left\{ p_{1},p_{2},p_{3}\right\} $ ($p_{1},p_{2},p_{3}$:
distinct) and
\begin{align*}
\alpha_{i} & \in\frac{1}{2}+\mathbb{Z}\quad\left(z_{i}\in S_{\mathrm{fin}}\right).
\end{align*}
In this case, $e(p_{1})=e(p_{2})=e(p_{3})=2,$ $e(\infty)=1$ and
the covering degree is $4$.
\item[(b2)] $S_{\mathrm{fin}}=\left\{ p_{1},p_{2}\right\} \sqcup S'$ ($p_{1},p_{2}$:
distinct, $S'\neq\emptyset$) and
\[
\alpha_{i}\in\begin{cases}
\frac{1}{2}\mathbb{Z} & z_{i}\in\left\{ p_{1},p_{2}\right\} \\
\mathbb{Z} & z_{i}\in S'
\end{cases}
\]
where at least one of the $\alpha_{i}$'s must have denominator exactly
$2$. In this case, $e(p_{1})=e(p_{2})=e(\infty)=2,$ $e(x)=1$ ($x\in S'$)
and the covering degree is $4$.
\end{enumerate}
\item $2$-branch point cases
\begin{enumerate}
\item[(a1)] $S_{\mathrm{fin}}=\left\{ p_{1},p_{2}\right\} $ ($p_{1},p_{2}$:
distinct) and
\begin{align*}
\alpha_{i} & \in\frac{k}{N}+\mathbb{Z}\quad\left(z_{i}\in S_{\mathrm{fin}}\right)
\end{align*}
with $N\in\mathbb{Z}_{>1}$ and $k$ coprime to $N$ ($k$ needs to
be the same for all $i$). In this case, $e(p_{1})=e(p_{2})=N,$ $e(\infty)=1$
and the covering degree is $N$.
\item[(a2)] $S_{\mathrm{fin}}=\left\{ p\right\} \sqcup S'$ ($S'\neq\emptyset$)
and
\[
\alpha_{i}\in\begin{cases}
\frac{1}{N}\mathbb{Z} & z_{i}=p\\
\mathbb{Z} & z_{i}\in S'
\end{cases}
\]
where $N$ is a positive integer greater than $1$ and at least one
of the $\alpha_{i}$'s must have denominator exactly $N$. In this
case, $e(p)=e(\infty)=N,$ $e(x)=1$ ($x\in S'$) and the covering
degree is $N$.
\end{enumerate}
\item Trivial ($0$-branch point) case:
\begin{enumerate}
\item[(c)] Arbitrary $S_{\mathrm{fin}}$ and
\begin{align*}
\alpha_{i} & \in\mathbb{Z}\quad\left(z_{i}\in S_{\mathrm{fin}}\right).
\end{align*}
In this case, $e(x)=1$ ($x\in S_{\mathrm{fin}}$) and the covering
degree is $1$. 
\end{enumerate}
\end{itemize}
\end{prop}

\begin{proof}
By the Riemann-Hurwitz formula,
\[
2g(X)-2=N\cdot(2g(\mathbb{P}^{1})-2)+\sum_{z\in S}\frac{N}{e(z)}(e(z)-1)=-2N+N\sum_{z\in S}\left(1-\frac{1}{e(z)}\right),
\]
or equivalently,
\[
g(X)=1-N+\frac{N}{2}\sum_{z\in S}\left(1-\frac{1}{e(z)}\right).
\]
Now, let us specify ourselves to the genus zero case. Put $S_{>1}\coloneqq\{z\in S\mid e(z)>1\}$.
To classify the parameter $(\alpha_{0},\dots,\alpha_{n+1})$ that
gives $g(X)=0$, observe that, if $\#S_{>1}\geq4$, then
\[
g(X)\geq1-N+\frac{N}{2}\cdot4\cdot\left(1-\frac{1}{2}\right)=1.
\]
Therefore, $g(X)=0$ implies $\#S_{>1}\leq3$. 

Let us first investigate the case $\#S_{>1}=3.$ Suppose that $S_{>1}=\{p,q,r\}$
with $e(p)\leq e(q)\leq e(r)$. Since the formula above gives
\[
\frac{1}{e(p)}+\frac{1}{e(q)}+\frac{1}{e(r)}=1+\frac{2}{N}>1,
\]
$(e(p),e(q),e(r))$ have to be either $(2,2,m)$ with $m\geq2$, $(2,3,3)$,
$(2,3,4)$ or $(2,3,5)$. It is easy to check that the cases $(2,3,3)$,
$(2,3,4)$, $(2,3,5)$ or $(2,2,m)$ with $m>2$ are impossible since
\begin{equation}
\#\{z\in S\mid e(z)\equiv0\pmod q\}\neq1\label{eq:ez_cond}
\end{equation}
for any prime power $q$. Thus, the only possible case is $e(p)=e(q)=e(r)=2$
(and thus $N=4$), which can be classified into the two cases (b1)
and (b2). 

Next, consider the case $\#S_{>1}=2$ and let $S_{>1}=\{p,q\}$. By
putting $d(p)=\frac{N}{e(p)}\in\mathbb{Z}_{\geq1}$ and $d(q)=\frac{N}{e(q)}\in\mathbb{Z}_{\geq1}$,
we have
\[
0=g(X)=1-N+\frac{N}{2}\left(\left(1-\frac{d(p)}{N}\right)+\left(1-\frac{d(q)}{N}\right)\right)=\frac{2-d(p)-d(q)}{2},
\]
by which we find $d(p)=d(q)=1$. Hence, we conclude that $e(p)=e(q)=N\geq2$
which gives the two cases (a1) and (a2). 

Finally, the case $\#S_{>1}=1$ is impossible by (\ref{eq:ez_cond}),
and $\#S_{>1}=0$ gives
\[
\alpha_{i}\in\mathbb{Z}\quad(0\leq i\leq n+1)
\]
meaning that iterated beta integrals are nothing but hyperlogarithms
in this case.
\end{proof}
Now, let us consider when the translation invariance property yields
non-trivial relations among hyperlogarithms. The key idea is that,
if we have two genus zero $X_{\boldsymbol{z},\boldsymbol{\alpha}}$
and $X_{\boldsymbol{z},\boldsymbol{\alpha}'}$ such that $\boldsymbol{\alpha}$
and $\boldsymbol{\alpha}'$ are translations of each other, the associated
iterated beta integrals are equal by the translation invariance. Also,
since $X_{\boldsymbol{z},\boldsymbol{\alpha}}$ and $X_{\boldsymbol{z},\boldsymbol{\alpha}'}$
are both of genus zero, after a change of variables, those iterated
beta integrals can be turned into hyperlogarithms. By Proposition
\ref{prop: complete list of genus 0 parameters}, we can again classify
all such patterns as follows:
\begin{prop}
\label{prop: complete list of iterated beta identity types}The complete
list of $\boldsymbol{z}=(z_{0},\ldots,z_{n+1})\in\mathbb{C}^{n+2}$,
$\boldsymbol{\alpha}=(\alpha_{0},\dots,\alpha_{n+1})\in\mathbb{Q}^{n+2}$
and $\boldsymbol{\alpha}'=(\alpha_{0}+\lambda,\dots,\alpha_{n+1}+\lambda)$
with $\lambda\in\mathbb{Q}\setminus\mathbb{Z}$ for which $X_{\boldsymbol{z},\boldsymbol{\alpha}}$
and $X_{\boldsymbol{z},\boldsymbol{\alpha}'}$ are both birational
to $\mathbb{P}^{1}$ and $\#\{z_{0},\dots,z_{n+1}\}\geq2$ is given
as follows ($S_{\mathrm{fin}}\coloneqq\left\{ z_{0},\ldots,z_{n+1}\right\} $): 
\begin{itemize}
\item 3-branch point cases
\begin{enumerate}
\item[(B1)] $S_{\mathrm{fin}}=\left\{ p_{1},p_{2},p_{3}\right\} $ ($p_{1},p_{2},p_{3}$:
distinct),
\begin{align*}
\alpha_{i} & \in\frac{1}{2}+\mathbb{Z}\quad\left(z_{i}\in S_{\mathrm{fin}}\right)
\end{align*}
and $\lambda\in\frac{1}{2}+\mathbb{Z}$.
\item[(B2)]  $S_{\mathrm{fin}}=\left\{ p_{1},p_{2},q_{1},q_{2}\right\} $ ($p_{1},p_{2},q_{1},q_{2}$:
distinct) and
\[
\alpha_{i}\in\begin{cases}
\frac{1}{2}+\mathbb{Z} & z_{i}\in\left\{ p_{1},p_{2}\right\} \\
\mathbb{Z} & z_{i}\in\left\{ q_{1},q_{2}\right\} 
\end{cases}
\]
and $\lambda\in\frac{1}{2}+\mathbb{Z}$.
\end{enumerate}
\item 2-branch point cases
\begin{enumerate}
\item[(A1)] $S_{\mathrm{fin}}=\left\{ p_{1},p_{2}\right\} $ ($p_{1},p_{2}$:
distinct) and
\begin{align*}
\alpha_{i} & \in\frac{k}{N}+\mathbb{Z}\quad\left(z_{i}\in S_{\mathrm{fin}}\right)
\end{align*}
with some $k$ coprime to $N$ and $\lambda\in\frac{1}{N}\mathbb{Z}\setminus\mathbb{Z}$.
\item[(A2)]  $S_{\mathrm{fin}}=\left\{ p,q\right\} $ and
\[
\alpha_{i}\in\begin{cases}
\frac{k}{N}+\mathbb{Z} & z_{i}=p\\
\mathbb{Z} & z_{i}=q
\end{cases}
\]
with some $k$ coprime to $N$ and $\lambda\in-\frac{k}{N}+\mathbb{Z}$.
\end{enumerate}
\end{itemize}
\end{prop}

\section{Case A1: Application to Zagier's 2-3-2 formula and Zhao's 2-1 formula\label{sec:Case-A1-Zagier232-Zhao21}}

In this section we give a detailed study of Case A1 in the classification
of Section \ref{sec:classification_genus_one}. We first investigate
the general case, and then turn to the special case $\alpha_{i}=1/2$.
The case $\alpha_{i}=1/2$ is closely related to Zagier's 2-3-2 formula
and Zhao's 2-1 formula.

\subsection{General remarks}

Here, we employ the same settings and symbols as in Section \ref{sec:classification_genus_one}.
If $S_{\mathrm{fin}}=\left\{ p_{1},p_{2}\right\} $ ($p_{1},p_{2}$:
distinct) and 
\begin{align*}
\alpha_{i} & =\frac{k}{N}\quad\left(z_{i}\in S_{\mathrm{fin}}\right)
\end{align*}
with some $k$ coprime to $N$ and $\lambda\in\frac{1}{N}\mathbb{Z}$,
the associated complex curve is a connected component of
\begin{align*}
X_{\boldsymbol{z},\boldsymbol{\alpha}} & =\left\{ (t,u_{1},u_{2})\in\mathbb{C}^{3}\left|t\notin\{p_{1},p_{2}\},\,u_{1}^{N}=(t-p_{1})^{k}(t-p_{2})^{N-k},u_{2}^{N}=(t-p_{2})^{k}(t-p_{1})^{N-k}\right.\right\} 
\end{align*}
and the $1$-forms $\left[\substack{p_{i},p_{j}\\
\alpha,\alpha
}
\right]$ ($1\leq i,j\leq2$) we want to rationalize are
\[
\frac{dt}{(t-p_{i})^{k/N}(t-p_{j})^{1-k/N}}=\begin{cases}
\frac{dt}{t-p_{i}} & 1\leq i=j\leq2\\
\frac{dt}{u_{i}} & 1\leq i\neq j\leq2.
\end{cases}
\]
A rational map $\varphi:\mathbb{P}^{1}\rightarrow X_{\boldsymbol{z},\boldsymbol{\alpha}};$
$\xi\mapsto\left(t(\xi),u_{1}(\xi),u_{2}(\xi)\right)$ is given, for
example, by
\begin{align*}
t(\xi) & =\frac{p_{1}\xi^{N}-p_{2}}{\xi^{N}-1},\\
u_{1}(\xi) & =(p_{2}-p_{1})\frac{\xi^{N-k}}{\xi^{N}-1},\\
u_{2}(\xi) & =(p_{2}-p_{1})\frac{\xi^{k}}{\xi^{N}-1}.
\end{align*}
When considering general $k$ and $N$, it is convenient to introduce
the map $\psi:\widetilde{\mathbb{C}^{\times}}\rightarrow X_{\boldsymbol{z},\boldsymbol{\alpha}};s\mapsto\left(t(s),u_{1}(s),u_{2}(s)\right)$
given by
\begin{align*}
t(s) & =\frac{p_{1}s-p_{2}}{s-1},\\
u_{1}(s) & =(p_{2}-p_{1})\frac{s^{1-\alpha}}{s-1},\\
u_{2}(s) & =(p_{2}-p_{1})\frac{s^{\alpha}}{s-1},
\end{align*}
where $\widetilde{\mathbb{C}^{\times}}$ is the universal covering
space of $\mathbb{C}^{\times}$ since we do not need the assumption
that $\alpha\in\mathbb{Q}$ in this setting.

We first discuss the generic case in Section \ref{subsec:Hurwitz_type_series}
and then investigate in more detail the special case $\alpha=1/2$
in Sections \ref{subsec:Preliminary-for-ZagierZhao}, \ref{subsec:Zagier232},
and \ref{subsec:Zhao21}.

\subsection{A theorem on Hurwitz-type series\label{subsec:Hurwitz_type_series}}

In this section, we put
\[
(p_{1},p_{2})=(0,1)
\]
for simplicity. Then, we have
\[
\psi^{*}\left[\substack{x,y\\
\alpha,\alpha
}
\right]=\frac{s^{-x\alpha-y(1-\alpha)}}{1-s}ds\quad\left(x,y\in\{0,1\}\right).
\]
The inverse images of $t\in\{0,1,\infty\}$ under $\psi$ are given
by
\[
\psi^{-1}(t)=\begin{cases}
\left\{ \infty\right\}  & t=0\\
\left\{ 0\right\}  & t=1\\
\left\{ 1\right\}  & t=\infty.
\end{cases}
\]

Thus, Corollary \ref{cor:diagonal_reduces_to_hyperlogarithm} gives
the following results. 
\begin{thm}
\label{thm:Case_2a-4}Let $x_{0},\ldots,x_{n+1}\in\left\{ 0,1\right\} $.
Assume that $x_{0}\neq x_{1}$, and $x_{n}\neq x_{n+1}$. Let $\tilde{0}=\infty$
and $\tilde{1}=0$. Let $\gamma$ be a simple path from $x_{0}$ to
$x_{n+1}$ on $\mathbb{C}\setminus\{0,1\}$. Then, we have
\[
\frac{(-1)^{\alpha}\sin(\pi\alpha)}{\pi}I_{\psi^{-1}(\gamma)}\left(\widetilde{x_{0}};f_{x_{0},x_{1}},f_{x_{1},x_{2}},\ldots,f_{x_{n},x_{n+1}};\widetilde{x_{n+1}}\right)=I_{\gamma}\left(x_{0};e_{x_{1}},e_{x_{2}},\ldots,e_{x_{n}};x_{n+1}\right),
\]
where 
\[
f_{x,y}=\frac{s^{-x\alpha-y(1-\alpha)}}{1-s}ds\quad\left(x,y\in\{0,1\}\right).
\]
\end{thm}

\begin{thm}
\label{thm:Case_2a_generic_inf}Let $x_{0},\ldots,x_{n+1}\in\left\{ 0,1\right\} $.
Assume that $x_{0}\neq x_{1}$. Let $\tilde{0}=\infty$ and $\tilde{1}=0$.
Let $\gamma$ be a simple path from $x_{0}$ to $\infty$ on $\mathbb{C}\setminus\{0,1\}$.
Then, the quantity
\[
(1-\alpha)I_{\psi^{-1}(\gamma)}\left(\widetilde{x_{0}};f_{x_{0},x_{1}},\dots,f_{x_{n-1},x_{n}},f'_{x_{n},x_{n+1}};1\right)
\]
where 
\[
f_{x,y}(s)=\frac{s^{-x\alpha-y(1-\alpha)}}{1-s}ds,\ f_{x,y}'(s)=(1-s)f_{x,y}(s)\quad\left(x,y\in\{0,1\}\right)
\]
is constant with respect to $\alpha$.
\end{thm}

From Theorem \ref{thm:Case_2a_generic_inf}, we can get the following.
\begin{thm}
\label{thm:Hurwitz_identity}Let $y_{1},\ldots,y_{n}\in\{0,1\}$ with
$y_{1}=1,\:y_{n}=0$. Then, the quantity 
\[
\sum_{m_{1}=0}^{\infty}\sum_{m_{2}=m_{1}+y_{1}}^{\infty}\cdots\sum_{m_{n}=m_{n-1}+y_{n-1}}^{\infty}\frac{\beta}{\left(m_{n}+\beta\right)\prod_{j=1}^{n}\left(m_{j}+\beta y_{j}\right)}
\]
does not depend on $\beta\in\mathbb{C}\setminus\mathbb{Z}_{<0}$.
In particular, when $\beta=0$, the quantity above reduces to
\[
\sum_{m_{1}=y_{1}}^{\infty}\sum_{m_{2}=m_{1}+y_{2}}^{\infty}\cdots\sum_{m_{n-1}=m_{n-2}+y_{n-1}}^{\infty}\frac{1}{m_{1}\cdots m_{n-2}m_{n-1}^{2}}.
\]
\end{thm}

\begin{proof}
We apply Theorem \ref{thm:Case_2a_generic_inf} to the case $x_{0}=1$,
$x_{1}=0$, $x_{n}=1$, and $x_{n+1}=0$. Put
\[
\omega_{i,i+1}\coloneqq f_{x_{i},x_{i+1}}\qquad\left(i=0,\ldots,n-1\right)
\]
and 
\[
\omega_{n,n+1}'\coloneqq f'_{x_{n},x_{n+1}}.
\]
Then,
\[
I_{\psi^{-1}(\gamma)}\left(\tilde{x_{0}};f_{x_{0},x_{1}},\dots,f_{x_{n-1},x_{n}},(1-s)f_{x_{n},x_{n+1}};1\right)=I\bigl(0;\omega_{0,1}(s),\ldots,\omega_{n-1,n}(s),\omega_{n,n+1}'(s);1\bigr).
\]
For $i=1,\ldots,n$ and $x,y\in\{0,1\}$, we define a linear operator
$L_{x,y}$ acting on one-forms by
\begin{align*}
L_{x,y}(\omega(s)) & \coloneqq\left(\frac{s^{-\left(x\alpha+y(1-\alpha)\right)}}{1-s}\int_{0}^{s}\omega(s)\right)ds.
\end{align*}
Then, $L_{x,y}(s^{m-\alpha(1-x)-x}ds)=\frac{1}{m+(1-\alpha)(1-x)}\sum_{m'=m+(1-x)}^{\infty}s^{m'-\alpha(1-y)-y}ds.$
Thus,
\begin{align*}
 & I\bigl(0;\omega_{0,1}(s),\ldots,\omega_{n-1,n}(s),\omega_{n,n+1}'(s);1\bigr)\\
 & =I\bigl(0;(1-s)\left(L_{x_{n},x_{n+1}}\circ L_{x_{n-1},x_{n}}\cdots\circ L_{x_{1},x_{2}}\right)\left(\omega_{0,1}(s)\right);1\bigr)\\
 & =I\bigl(0;(1-s)\left(L_{x_{n},x_{n+1}}\circ L_{x_{n-1},x_{n}}\cdots\circ L_{x_{1},x_{2}}\right)\left(\sum_{m_{1}=0}^{\infty}s^{m_{1}-\alpha}ds\right);1\bigr)\\
 & =I\bigl(0;(1-s)\left(L_{x_{n},x_{n+1}}\circ L_{x_{n-1},x_{n}}\cdots\circ L_{x_{1},x_{2}}\right)\left(\sum_{m_{1}=0}^{\infty}s^{m_{1}-\alpha(1-x_{1})-x_{1}}ds\right);1\bigr)\\
 & =I\bigl(0;(1-s)\left(L_{x_{n},x_{n+1}}\circ L_{x_{n-1},x_{n}}\cdots\circ L_{x_{2},x_{3}}\right)\left(\sum_{m_{1}=0}^{\infty}\frac{1}{m_{1}+(1-\alpha)(1-x_{1})}\sum_{m_{2}=m_{1}+(1-x_{1})}^{\infty}s^{m_{2}-\alpha(1-x_{2})-x_{2}}ds\right);1\bigr).
\end{align*}
Repeating the same calculation, the last quantity becomes
\[
I\bigl(0;(1-s)\left(\sum_{m_{1}=0}^{\infty}\sum_{m_{2}=m_{1}+(1-x_{1})}^{\infty}\cdots\sum_{m_{n}=m_{n-1}+(1-x_{n-1})}^{\infty}\sum_{m_{n+1}=m_{n}+(1-x_{n})}^{\infty}\frac{s^{m_{n+1}-\alpha(1-x_{n+1})-x_{n+1}}}{\prod_{i=1}^{n}(m_{i}+(1-\alpha)(1-x_{i}))}ds\right);1\bigr),
\]
which is equal to
\begin{align*}
 & I\bigl(0;(1-s)\left(\sum_{m_{1}=0}^{\infty}\sum_{m_{2}=m_{1}+(1-x_{1})}^{\infty}\cdots\sum_{m_{n}=m_{n-1}+(1-x_{n-1})}^{\infty}\sum_{m_{n+1}=m_{n}}^{\infty}\frac{s^{m_{n+1}-\alpha}}{\prod_{i=1}^{n}(m_{i}+(1-\alpha)(1-x_{i}))}ds\right);1\bigr)\\
 & =I\bigl(0;\left(\sum_{m_{1}=0}^{\infty}\sum_{m_{2}=m_{1}+(1-x_{1})}^{\infty}\cdots\sum_{m_{n}=m_{n-1}+(1-x_{n-1})}^{\infty}\frac{s^{m_{n}-\alpha}}{\prod_{i=1}^{n}(m_{i}+(1-\alpha)(1-x_{i}))}ds\right);1\bigr)\\
 & =\sum_{m_{1}=0}^{\infty}\sum_{m_{2}=m_{1}+(1-x_{1})}^{\infty}\cdots\sum_{m_{n}=m_{n-1}+(1-x_{n-1})}^{\infty}\frac{1}{\left(m_{n}+1-\alpha\right)\prod_{i=1}^{n}(m_{i}+(1-\alpha)(1-x_{i}))}.
\end{align*}
By putting $\beta=1-\alpha$ and $y_{i}=1-x_{i}$, we obtain the claim.
\end{proof}
\begin{rem}
By comparing the cases $\beta=0$ and $\beta=1/2$ of Theorem \ref{thm:Hurwitz_identity},
we get the identity
\begin{align*}
 & \frac{1}{2}\sum_{m_{1}=0}^{\infty}\sum_{m_{2}=m_{1}+y_{1}}^{\infty}\cdots\sum_{m_{n}=m_{n-1}+y_{n-1}}^{\infty}\frac{1}{\left(m_{n}+1/2\right)\prod_{j=1}^{n}\left(m_{j}+y_{j}/2\right)}\\
 & =\sum_{m_{1}=y_{1}}^{\infty}\sum_{m_{2}=m_{1}+y_{2}}^{\infty}\cdots\sum_{m_{n-1}=m_{n-2}+y_{n-1}}^{\infty}\frac{1}{m_{1}\cdots m_{n-2}m_{n-1}^{2}},
\end{align*}
which is equivalent to Zhao's 2-1 formula (Theorem \ref{thm:Generalized_2-1})
under the M\"obius transformation formula for multiple polylogarithms.
\end{rem}

\subsection{Preliminaries for Zagier's 2-3-2 formula and Zhao's 2-1 formula\label{subsec:Preliminary-for-ZagierZhao}}

Throughout Sections \ref{subsec:Preliminary-for-ZagierZhao}, \ref{subsec:Zagier232},
and \ref{subsec:Zhao21}, we let $(p_{1},p_{2})=(1,-1)$ for simplicity.
Define $\chi:\mathbb{P}^{1}\rightarrow X_{\boldsymbol{z},\boldsymbol{\alpha}};$
$\tau\mapsto\left(t(\tau),u_{1}(\tau),u_{2}(\tau)\right)$ by
\begin{align*}
t(\tau) & =\frac{\tau+\tau^{-1}}{2},\\
u_{1}(\tau)=u_{2}(\tau) & =\frac{-\tau+\tau^{-1}}{2}.
\end{align*}
The pull-backs of the rational $1$-forms $\left\{ \substack{x,y\\
1/2,1/2
}
\right\} $ ($1\leq i,j\leq2$) by $\chi$ are given by
\[
\chi^{*}\left[\substack{x,y\\
1/2,1/2
}
\right]=\frac{dt(\tau)}{(t(\tau)-x)^{1/2}(t(\tau)-y)^{1/2}}=2d\log(\sqrt{t(\tau)-x}+\sqrt{t(\tau)-y})=\begin{cases}
2e_{1}-e_{0} & (x,y)=(1,1)\\
2e_{-1}-e_{0} & (x,y)=(-1,-1)\\
e_{0} & (x,y)=(1,-1),(-1,1).
\end{cases}
\]
The inverse images of $t\in\{1,-1,\infty\}$ under $\chi$ are given
by 
\[
\chi^{-1}(x)=\begin{cases}
\left\{ 1\right\}  & x=1\\
\left\{ -1\right\}  & x=-1\\
\left\{ 0,\infty\right\}  & x=\infty.
\end{cases}
\]

\subsection{Zagier's 2-3-2 formula and its analogues\label{subsec:Zagier232}}

Via the rational map $\chi$, Corollary \ref{cor:diagonal_reduces_to_hyperlogarithm}
gives the following. 
\begin{thm}
\label{thm:Case_2a-1}Let $x_{0},\ldots,x_{n+1}\in\left\{ \pm1\right\} $.
Assume that $1=x_{0}\neq x_{1}$, and $x_{n}\neq x_{n+1}=-1$. Let
$\gamma$ be a simple path from $x_{0}=1$ to $x_{n+1}=-1$ on $\mathbb{C}\setminus\{0,\pm1\}$
which makes a half--turn counterclockwise around the origin. Then,
we have
\[
(\pi i)^{-1}I_{\gamma}\left(x_{0};h_{x_{0},x_{1}},h_{x_{1},x_{2}},\ldots,h_{x_{n},x_{n+1}};x_{n+1}\right)=I_{\chi(\gamma)}\left(x_{0};e_{x_{1}},e_{x_{2}},\ldots,e_{x_{n}};x_{n+1}\right),
\]
where 
\[
h_{x,y}=\begin{cases}
2e_{1}-e_{0} & (x,y)=(1,1)\\
2e_{-1}-e_{0} & (x,y)=(-1,-1)\\
e_{0} & (x,y)=(1,-1),(-1,1).
\end{cases}
\]
\end{thm}

In fact, the case $(x_{0},\dots,x_{n+1})=(\{1,-1\}^{a+1},1,\{1,-1\}^{b+1})$
of Theorem \ref{thm:Case_2a-1} implies Zagier's 2-3-2 formula. Theorem
\ref{thm:Case_2a-1} may be slightly generalized in the following
form:
\begin{thm}
\label{thm:Case_2a_finite_gen}Let $s,s'\in\{0,1\}$. With the same
settings as in Theorem \ref{thm:Case_2a-1}, we have
\begin{align*}
 & (i\pi)^{s+s'-1}I_{\gamma}\left(x_{0};h_{x_{s},x_{s+1}},h_{x_{s+1},x_{s+2}},\ldots,h_{x_{n-s'},x_{n-s'+1}};x_{n+1}\right)\\
 & =I_{\chi(\gamma)}\left(x_{0};(\tfrac{2}{x_{0}-t})^{s/2}e_{x_{1}},e_{x_{2}},\ldots,e_{x_{n-1}},e_{x_{n}}(\tfrac{2}{t-x_{n+1}})^{s'/2};x_{n+1}\right).
\end{align*}
\end{thm}

\begin{proof}
We have
\[
\frac{(-1)^{s/2}}{\Gamma(1-s/2)\Gamma(-s'/2)}B_{\chi(\gamma)}^{{\rm f},{\rm f}}\bigl(\left.\substack{x_{0}\\
s/2
}
\right|\left.\substack{x_{1}\\
0
}
\right|\dots\left|\substack{x_{n}\\
0
}
\right.\left|\substack{x_{n+1}\\
-s'/2
}
\right.\bigr)=\frac{(-1)^{(s+1)/2}}{\Gamma((1-s)/2)\Gamma((1-s')/2)}B_{\chi(\gamma)}^{{\rm f},{\rm f}}\bigl(\left.\substack{x_{0}\\
(1+s)/2
}
\right|\left.\substack{x_{1}\\
1/2
}
\right|\dots\left|\substack{x_{n}\\
1/2
}
\right.\left|\substack{x_{n+1}\\
(1-s')/2
}
\right.\bigr)
\]
by setting $(\alpha_{0},\dots,\alpha_{n+1})=(s/2,0,\dots,0,-s'/2)$,
$\lambda=1/2$, and $z_{n}=x_{n}$ in the first identity of Corollary
\ref{cor:simple_path_case}. Then, the left-hand side is calculated
as
\begin{align*}
 & \frac{(-1)^{s/2}}{\Gamma(1-s/2)\Gamma(-s'/2)}B_{\chi(\gamma)}^{{\rm f},{\rm f}}\bigl(\left.\substack{x_{0}\\
s/2
}
\right|\left.\substack{x_{1}\\
0
}
\right|\dots\left|\substack{x_{n}\\
0
}
\right.\left|\substack{x_{n+1}\\
-s'/2
}
\right.\bigr)\\
 & =\frac{(-1)^{s/2}}{\Gamma(1-s/2)\Gamma(-s'/2)}I_{\chi(\gamma)}(x_{0};(\tfrac{1}{t-x_{0}})^{s/2}e_{x_{1}},\dots,e_{x_{n}},(\tfrac{1}{t-x_{n+1}})^{1+s'/2}dt;x_{n+1})\\
 & =\frac{-(-1)^{s/2}}{\Gamma(1-s/2)\Gamma(1-s'/2)}I_{\chi(\gamma)}(x_{0};(\tfrac{1}{t-x_{0}})^{s/2}e_{x_{1}},\dots,e_{x_{n}}(\tfrac{1}{t-x_{n+1}})^{s'/2};x_{n+1}),
\end{align*}
and the right-hand side is calculated as
\begin{align*}
 & \frac{1}{\Gamma((1-s)/2)\Gamma((1-s')/2)}B_{\chi(\gamma)}^{{\rm f},{\rm f}}\bigl(\left.\substack{x_{0}\\
(1+s)/2
}
\right|\left.\substack{x_{1}\\
1/2
}
\right|\dots\left|\substack{x_{n}\\
1/2
}
\right.\left|\substack{x_{n+1}\\
(1-s')/2
}
\right.\bigr)\\
 & =\begin{cases}
\frac{1}{\Gamma(1/2)^{2}}B_{\chi(\gamma)}^{{\rm f},{\rm f}}\bigl(\left.\substack{x_{0}\\
1/2
}
\right|\left.\substack{x_{1}\\
1/2
}
\right|\dots\left|\substack{x_{n}\\
1/2
}
\right.\left|\substack{x_{n+1}\\
1/2
}
\right.\bigr) & (s,s')=(0,0)\\
\frac{(-2)^{-1/2}}{\Gamma(1/2)}B_{\chi(\gamma)}^{{\rm f},{\rm f}}\bigl(\left.\substack{x_{0}\\
1/2
}
\right|\left.\substack{x_{1}\\
1/2
}
\right|\dots\left|\substack{x_{n}\\
1/2
}
\right.;x_{n+1}\bigr) & (s,s')=(0,1)\\
\frac{2^{-1/2}}{\Gamma(1/2)}B_{\chi(\gamma)}^{{\rm f},{\rm f}}\bigl(x_{0};\left.\substack{x_{1}\\
1/2
}
\right|\dots\left|\substack{x_{n}\\
1/2
}
\right.\left|\substack{x_{n+1}\\
1/2
}
\right.\bigr) & (s,s')=(1,0)\\
2^{-1/2}(-2)^{-1/2}B_{\chi(\gamma)}^{{\rm f},{\rm f}}\bigl(x_{0};\left.\substack{x_{1}\\
1/2
}
\right|\dots\left|\substack{x_{n}\\
1/2
}
\right.;x_{n+1}\bigr) & (s,s')=(1,1)
\end{cases}\\
 & =\frac{2^{-s/2}(-2)^{-s'/2}}{\Gamma(1/2)^{2-s-s'}}B_{\chi(\gamma)}^{{\rm f},{\rm f}}\bigl(x_{0};\left.\substack{x_{s}\\
1/2
}
\right|\left.\substack{x_{s+1}\\
1/2
}
\right|\dots\left|\substack{x_{n-s'}\\
1/2
}
\right.\left|\substack{x_{n+1-s'}\\
1/2
}
\right.;x_{n+1}\bigr)
\end{align*}
by using Theorem \ref{thm: complete vs incomplete} for the cases
$s=1$ or $s'=1$. Thus, we have
\[
(\pi i)^{s+s'-1}B_{\chi(\gamma)}^{{\rm f},{\rm f}}\bigl(x_{0};\left.\substack{x_{s}\\
1/2
}
\right|\left.\substack{x_{s+1}\\
1/2
}
\right|\dots\left|\substack{x_{n-s'}\\
1/2
}
\right.\left|\substack{x_{n+1-s'}\\
1/2
}
\right.;x_{n+1}\bigr)=I_{\chi(\gamma)}(x_{0};(\tfrac{2}{x_{0}-t})^{s/2}e_{x_{1}},\dots,e_{x_{n}}(\tfrac{2}{t-x_{n+1}})^{s'/2};x_{n+1}).
\]
Since
\[
B_{\chi(\gamma)}^{{\rm f},{\rm f}}\bigl(x_{0};\left.\substack{x_{s}\\
1/2
}
\right|\left.\substack{x_{s+1}\\
1/2
}
\right|\dots\left|\substack{x_{n-s'}\\
1/2
}
\right.\left|\substack{x_{n+1-s'}\\
1/2
}
\right.;x_{n+1}\bigr)=I_{\gamma}\left(x_{0};h_{x_{s},x_{s+1}},h_{x_{s+1},x_{s+2}},\ldots,h_{x_{n-s'},x_{n-s'+1}};x_{n+1}\right),
\]
we obtain the claim.
\end{proof}
Let $\zeta(k_{1},\dots,k_{d})_{u,v}$ be the double tails of $\zeta(k_{1},\dots,k_{d})$
{[}Double tails of multiple zeta values{]} defined by
\begin{align*}
\zeta(k_{1},\dots,k_{d})_{u,v} & =(-1)^{d}I(0;t^{u}e_{a_{1}},e_{a_{2}},\dots,e_{a_{k-1}},e_{a_{k}}(1-t)^{v};1)
\end{align*}
where
\[
(a_{1},\dots,a_{k})=(1,\{0\}^{k_{1}-1},\dots,1,\{0\}^{k_{d}-1}).
\]
It admits the series expression
\[
\zeta(k_{1},\dots,k_{d})_{u,v}=\sum_{0<m_{1}<\cdots<m_{d}}\frac{1}{(u+m_{1})^{k_{1}}\cdots(u+m_{d})^{k_{d}}}\frac{\Gamma(u+1)\Gamma(m_{d}+v+1)}{\Gamma(u+m_{d}+v+1)}
\]
and satisfies the duality identity
\[
\zeta(\Bbbk)_{u,v}=\zeta(\Bbbk^{\dagger})_{v,u}
\]
where $\Bbbk^{\dagger}$ is the dual index of $\Bbbk$. The case $v=0$
of $\zeta(k_{1},\dots,k_{d})_{u,v}$ is the Hurwitz multiple zeta
values, and especially, the case $(u,v)=(-1/2,0)$ is equal to the
(modified) Hoffman's $t$-value
\[
\zeta(k_{1},\dots,k_{d})_{-1/2,0}=2^{k_{1}+\cdots+k_{d}}\sum_{\substack{0<n_{1}<\cdots<n_{d}\\
n_{j}\equiv1\pmod 2
}
}\frac{1}{n_{1}^{k_{1}}\cdots n_{d}^{k_{d}}}=:\tilde{t}(k_{1},\dots,k_{d}).
\]

By specializing Theorem \ref{thm:Case_2a_finite_gen} to the case
\[
(x_{0},\dots,x_{n+1})=(\{1,-1\}^{a+1},1,\{1,-1\}^{b+1}),
\]
 we get the following.
\begin{thm}
For $a,b\in\mathbb{Z}_{\geq0}$ and $s,s'\in\{0,1\}$, we have
\begin{align*}
 & \zeta(\{2\}^{a},3,\{2\}^{b})_{-s/2,-s'/2}\\
 & =2\sum_{\substack{k\geq1,\,\ell\geq\min(s,s')\\
k+\ell=a+b+1
}
}(-1)^{k}\left((-1)^{s}\binom{2k}{2a+2-s}-(-1)^{s'}\left(1-\frac{1}{2^{2k}}\right)\binom{2k}{2b+1-s'}\right)\zeta(2k+1)\frac{\pi^{2\ell}}{(2\ell-s-s'+1)!}\\
 & \quad+\delta_{s',1}\delta_{b,0}2^{s+1}\log2\frac{\pi^{2a+2}}{(2a+2-s)!}.
\end{align*}
\end{thm}

\begin{proof}
When
\[
(x_{0},\dots,x_{n+1})=(\{1,-1\}^{a+1},1,\{1,-1\}^{b+1}),
\]
the right-hand side of Theorem \ref{thm:Case_2a_finite_gen} is equal
to
\[
(-1)^{a+b+1}\zeta(\{2\}^{a},3,\{2\}^{b})_{-s/2,-s'/2}
\]
while the left-hand side is equal to
\[
(i\pi)^{s+s'-1}I(1;\overbrace{e_{0},\dots,e_{0}}^{2a+2-s},(2e_{1}-e_{0}),\overbrace{e_{0},\dots,e_{0}}^{2b+1-s'};-1).
\]
Let $0^{\pm}$ denote the tangential basepoints at $0$ with the tangential
vectors $\pm1$. For $m\geq1$, $n\geq0$, by the path composition
formula,
\begin{align*}
I_{\gamma}\left(1;e_{0}^{m}e_{1}e_{0}^{n};-1\right) & =\sum_{\substack{k\geq0,\ell\geq0\\
k+\ell=n
}
}I(1;e_{0}^{m}e_{1}e_{0}^{k};0^{+})I(0^{+};e_{0}^{\ell};0^{-})+\sum_{\substack{k\geq0,\ell\geq0\\
k+\ell=m
}
}I(0^{+};e_{0}^{\ell};0^{-})I(0^{-};e_{0}^{k}e_{1}e_{0}^{n};-1)\\
 & =\sum_{\substack{k\geq0,\ell\geq0\\
k+\ell=n
}
}I(1;e_{0}^{m}e_{1}e_{0}^{k};0^{+})\frac{(i\pi)^{\ell}}{\ell!}+\sum_{\substack{k\geq0,\ell\geq0\\
k+\ell=m
}
}I(0^{-};e_{0}^{k}e_{1}e_{0}^{n};-1)\frac{(i\pi)^{\ell}}{\ell!}.
\end{align*}
Here, we have used the fact 
\[
I(0^{+};w;0^{-})=\begin{cases}
\frac{(i\pi)^{m}}{m!} & w=e_{0}^{m}\text{ with }m\geq0\\
0 & \text{otherwise}
\end{cases}
\]
and $I(0^{-};e_{0}^{r};-1)=I(1;e_{0}^{r};0^{+})=0$ if $r>0$. If
$m>0,n,k\geq0$, then
\begin{align*}
I(1;e_{0}^{m}e_{1}e_{0}^{k};0^{+}) & =(-1)^{k+m+1}I(0^{+};e_{0}^{k}e_{1}e_{0}^{m};1)\\
 & =(-1)^{m+1}\binom{m+k}{m}I(0^{+};e_{1}e_{0}^{m+k};1)\\
 & =(-1)^{m}\binom{m+k}{m}\zeta(m+k+1)
\end{align*}
and
\begin{align*}
I(0^{-};e_{0}^{k}e_{1}e_{0}^{n};-1) & =I(0^{+};e_{0}^{k}e_{-1}e_{0}^{n};1)\\
 & =(-1)^{k}\binom{n+k}{n}I(0^{+};e_{-1}e_{0}^{n+k};1)\\
 & =(-1)^{k+1}\binom{n+k}{n}\zeta(\overline{n+k+1})
\end{align*}
where
\[
\zeta(\bar{r})=\sum_{n=1}^{\infty}\frac{(-1)^{n}}{n^{r}}=\begin{cases}
-\left(1-\frac{2}{2^{r}}\right)\zeta(r) & r>1\\
-\log2 & r=1.
\end{cases}
\]
Thus,
\begin{align*}
I_{\gamma}\left(1;e_{0}^{m}e_{1}e_{0}^{n};-1\right) & =\sum_{\substack{k\geq0,\ell\geq0\\
k+\ell=n
}
}(-1)^{m}\binom{m+k}{m}\zeta(m+k+1)\frac{(i\pi)^{\ell}}{\ell!}+\sum_{\substack{k\geq0,\ell\geq0\\
k+\ell=m
}
}(-1)^{k+1}\binom{n+k}{n}\zeta(\overline{n+k+1})\frac{(i\pi)^{\ell}}{\ell!}\\
 & =\sum_{\substack{k\geq m,\ell\geq0\\
k+\ell=m+n
}
}(-1)^{m}\binom{k}{m}\zeta(k+1)\frac{(i\pi)^{\ell}}{\ell!}+\sum_{\substack{k\geq n,\ell\geq0\\
k+\ell=m+n
}
}(-1)^{k+n+1}\binom{k}{n}\zeta(\overline{k+1})\frac{(i\pi)^{\ell}}{\ell!}
\end{align*}
Hence, the real part of the left-hand side is equal to
\begin{align*}
 & \Re\left((i\pi)^{s+s'-1}I\left(1;\overbrace{e_{0},\dots,e_{0}}^{2a+2-s},(2e_{1}-e_{0}),\overbrace{e_{0},\dots,e_{0}}^{2b+1-s'};-1\right)\right)\\
 & =2\Re\left((i\pi)^{s+s'-1}I\left(1;\overbrace{e_{0},\dots,e_{0}}^{2a+2-s},e_{1},\overbrace{e_{0},\dots,e_{0}}^{2b+1-s'};-1\right)\right)\\
 & =2\Re\left(\sum_{\substack{k\geq2a+2-s,\ell\geq0\\
k+\ell=2a+2b+3-s-s'
}
}(-1)^{s}\binom{k}{2a+2-s}\zeta(k+1)\frac{(i\pi)^{\ell+s+s'-1}}{\ell!}\right.\\
 & \qquad\qquad\left.+\sum_{\substack{k\geq2b+1-s',\ell\geq0\\
k+\ell=2a+2b+3-s-s'
}
}(-1)^{k+s'}\binom{k}{2b+1-s'}\zeta(\overline{k+1})\frac{(i\pi)^{\ell+s+s'-1}}{\ell!}\right)\\
 & =2\Re\left(\sum_{\substack{k\geq2a+2-s,\ell\geq s+s'-1\\
k:\text{even},\ \ell:\text{even}\\
k+\ell=2a+2b+2
}
}(-1)^{s}\binom{k}{2a+2-s}\zeta(k+1)\frac{(i\pi)^{\ell}}{(\ell-s-s'+1)!}\right.\\
 & \qquad\qquad\left.+\sum_{\substack{k\geq2b+1-s',\ell\geq s+s'-1\\
k:\text{even},\ \ell:\text{even}\\
k+\ell=2a+2b+2
}
}(-1)^{k+s'}\binom{k}{2b+1-s'}\zeta(\overline{k+1})\frac{(i\pi)^{\ell}}{(\ell-s-s'+1)!}\right)\\
 & =2\sum_{\substack{k\geq1,\,\ell\geq\min(s,s')\\
k+\ell=a+b+1
}
}(-1)^{s+\ell}\binom{2k}{2a+2-s}\zeta(2k+1)\frac{\pi^{2\ell}}{(2\ell-s-s'+1)!}\\
 & \qquad+2\sum_{\substack{k\geq0,\,\ell\geq\min(s,s')\\
k+\ell=a+b+1
}
}(-1)^{\ell+s'}\binom{2k}{2b+1-s'}\zeta(\overline{2k+1})\frac{\pi^{2\ell}}{(2\ell-s-s'+1)!}.
\end{align*}
Hence, we have
\begin{align*}
 & \zeta(\{2\}^{a},3,\{2\}^{b})_{-s/2,-s'/2}\\
 & =2\sum_{\substack{k\geq1,\,\ell\geq\min(s,s')\\
k+\ell=a+b+1
}
}(-1)^{s+k}\binom{2k}{2a+2-s}\zeta(2k+1)\frac{\pi^{2\ell}}{(2\ell-s-s'+1)!}\\
 & \quad+2\sum_{\substack{k\geq0,\,\ell\geq\min(s,s')\\
k+\ell=a+b+1
}
}(-1)^{k+s'}\binom{2k}{2b+1-s'}\zeta(\overline{2k+1})\frac{\pi^{2\ell}}{(2\ell-s-s'+1)!}\\
 & =2\sum_{\substack{k\geq1,\,\ell\geq\min(s,s')\\
k+\ell=a+b+1
}
}(-1)^{k}\left((-1)^{s}\binom{2k}{2a+2-s}-(-1)^{s'}\left(1-\frac{1}{2^{2k}}\right)\binom{2k}{2b+1-s'}\right)\zeta(2k+1)\frac{\pi^{2\ell}}{(2\ell-s-s'+1)!}\\
 & \quad+\delta_{s',1}\delta_{b,0}2^{s+1}\log2\frac{\pi^{2a+2}}{(2a+2-s)!}.
\end{align*}
\end{proof}
The cases $(s,s')=(0,0)$, $(1,0)$, $(0,1)$, and $(1,1)$ of the
theorem above implies:
\begin{itemize}
\item Zagier's formula \cite{Zagier_232}
\[
\zeta(\{2\}^{a},3,\{2\}^{b})=2\sum_{\substack{k\geq1,\,\ell\geq0\\
k+\ell=a+b+1
}
}(-1)^{k}\left(\binom{2k}{2a+2}-\left(1-\frac{1}{2^{2k}}\right)\binom{2k}{2b+1}\right)\zeta(2k+1)\frac{\pi^{2\ell}}{(2\ell+1)!},
\]
\item Murakami's formula \cite[Theorem 3]{Murakami_tvalue}
\[
\tilde{t}(\{2\}^{a},3,\{2\}^{b})=2\sum_{\substack{k\geq1,\,\ell\geq0\\
k+\ell=a+b+1
}
}(-1)^{k}\left(-\binom{2k}{2a+1}-\left(1-\frac{1}{2^{2k}}\right)\binom{2k}{2b+1}\right)\zeta(2k+1)\frac{\pi^{2\ell}}{(2\ell)!},
\]
\item Charlton's formula \cite[Theorem 1.1]{charlton_t_212}
\begin{align*}
\tilde{t}(\{2\}^{b},1,\{2\}^{a+1}) & =2\sum_{\substack{k\geq1,\,\ell\geq0\\
k+\ell=a+b+1
}
}(-1)^{k}\left(\binom{2k}{2a+2}+\left(1-\frac{1}{2^{2k}}\right)\binom{2k}{2b}\right)\zeta(2k+1)\frac{\pi^{2\ell}}{(2\ell)!}\\
 & \quad+\delta_{b,0}\frac{2(\log2)\pi^{2a+2}}{(2a+2)!},
\end{align*}
\item a seemingly new formula
\begin{align*}
\zeta(\{2\}^{a},3,\{2\}^{b})_{-1/2,-1/2} & =2\sum_{\substack{k,\ell\geq1\\
k+\ell=a+b+1
}
}(-1)^{k}\left(-\binom{2k}{2a+1}+\left(1-\frac{1}{2^{2k}}\right)\binom{2k}{2b}\right)\zeta(2k+1)\frac{\pi^{2\ell}}{(2\ell-1)!}\\
 & \quad+\delta_{b,0}\frac{4\log2\pi^{2a+2}}{(2a+1)!}.
\end{align*}
\end{itemize}

\subsection{Zhao's 2-1 formula and its generalizations\label{subsec:Zhao21}}

Let us first see the deduction of the original Zhao's 2-1 formula.
First, by Corollary \ref{cor:simple_path_case}, we have
\begin{align*}
 & -\Gamma(\alpha_{0}-1)I_{\gamma}(\infty;(t-z_{0})^{1-\alpha_{0}}e_{z_{1}},\dots,e_{z_{n-1}},e_{z_{n}};z_{n+1})\\
 & =\frac{\Gamma(\alpha_{0}+\lambda-1)}{\Gamma(\lambda)}I_{\gamma}(\infty;(t-z_{0})^{1-\alpha_{0}}\left[\substack{z_{0},z_{1}\\
\lambda,\lambda
}
\right],\left[\substack{z_{1},z_{2}\\
\lambda,\lambda
}
\right],\dots,\left[\substack{z_{n-1},z_{n}\\
\lambda,\lambda
}
\right],\left[\substack{z_{n},z_{n+1}\\
\lambda,\lambda
}
\right];z_{n+1}).
\end{align*}
By putting $\alpha_{0}=2$ and $\lambda=\frac{1}{2}$, we have
\begin{align*}
 & -I_{\gamma}(\infty;\frac{1}{t-z_{0}}e_{z_{1}},\dots,e_{z_{n-1}},e_{z_{n}};z_{n+1})\\
 & =\frac{1}{2}I_{\gamma}(\infty;\frac{1}{t-z_{0}}\left[\substack{z_{0},z_{1}\\
1/2,1/2
}
\right],\left[\substack{z_{1},z_{2}\\
1/2,1/2
}
\right],\dots,\left[\substack{z_{n-1},z_{n}\\
1/2,1/2
}
\right],\left[\substack{z_{n},z_{n+1}\\
1/2,1/2
}
\right];z_{n+1}).
\end{align*}
Now, assume that $z_{0},\dots,z_{n+1}\in\{\pm1\}$, and furthermore,
$z_{0}=-1$, $z_{1}=1$, $z_{n}=-1$, and $z_{n+1}=1$. By change
of variables $t=\frac{\tau+\tau^{-1}}{2}$, we have
\begin{align*}
 & I_{\gamma}(\infty;\frac{1}{t-z_{0}}e_{z_{1}},e_{z_{2}},\dots,e_{z_{n}};z_{n+1})\\
 & =I_{\gamma}(\infty;\frac{1}{t-z_{0}}h_{z_{1},z_{2}},\dots,h_{z_{n},z_{n+1}};z_{n+1})
\end{align*}
with
\[
h_{x,y}=\begin{cases}
2e_{1}-e_{0} & (x,y)=(1,1)\\
2e_{-1}-e_{0} & (x,y)=(-1,-1)\\
e_{0} & (x,y)=(1,-1),(-1,1)
\end{cases}.
\]
Then we have
\[
\frac{1}{t+1}h_{1,1}=\frac{dt}{t(t-1)},\quad\frac{1}{t+1}h_{1,-1}=\frac{dt}{t(t+1)}.
\]
Now, define $\boldsymbol{k}=(k_{1},\dots,k_{d})$ and $\boldsymbol{l}=(l_{1},\dots,l_{r})$
by
\[
(z_{1},z_{2},\dots,z_{n})=(1,\{-1\}^{k_{1}-1},\dots,1,\{-1\}^{k_{d}-1})
\]
and
\[
(\frac{1}{t}e_{a_{1}},\{e_{0}\}^{l_{1}-1},2e_{a_{2}}-e_{0},\{e_{0}\}^{l_{2}-1},\dots,2e_{a_{r}}-e_{0},\{e_{0}\}^{l_{r}-1})=(\frac{1}{t-z_{0}}h_{z_{1},z_{2}},\dots,h_{z_{n},z_{n+1}})
\]
where $a_{1},\dots,a_{r}$ are either $1$ or $-1$. Note that we
have $a_{j+1}=(-1)^{l_{j}-1}a_{j}$ for $j=1,\dots,r-1$ by definition.
We also note that $\boldsymbol{l}=\sigma(\boldsymbol{k})$ . Then
we have 
\[
2I_{\gamma}(\infty;\frac{1}{t-z_{0}}e_{z_{1}},e_{z_{2}},\dots,e_{z_{n}};z_{n+1})=\pm\zeta^{\star}(\boldsymbol{k})
\]
and
\[
2I_{\gamma}(\infty;\frac{1}{t-z_{0}}h_{z_{1},z_{2}},\dots,h_{z_{n},z_{n+1}};z_{n+1})=\pm\zeta^{\#}(\boldsymbol{l}).
\]
This implies Zhao's 2-1 formula (Theorem \ref{thm:Generalized_2-1}).

We make a few remarks on generalization of this. By considering the
case $\alpha_{0}=2$, $\lambda=1/2$, $\alpha_{n+1}=1-\beta$, $z_{0}=-1$,
$z_{n+1}=1$, and
\[
(z_{1},z_{2},\dots,z_{n})=(1,\{-1\}^{k_{1}-1},\dots,1,\{-1\}^{k_{d}-1})
\]
of 
\begin{align*}
 & \frac{-\Gamma(\alpha_{0}-1)}{\Gamma(\alpha_{n+1})}I_{\gamma}(\infty;(t-z_{0})^{1-\alpha_{0}}e_{z_{1}},\dots,e_{z_{n-1}},e_{z_{n}}(t-z_{n+1})^{\alpha_{n+1}-1};z_{n+1})\\
 & =\frac{\Gamma(\alpha_{0}+\lambda-1)}{\Gamma(\alpha_{n+1}+\lambda-1)}I_{\gamma}(\infty;(t-z_{0})^{1-\alpha_{0}}\left[\substack{z_{0},z_{1}\\
\lambda,\lambda
}
\right],\left[\substack{z_{1},z_{2}\\
\lambda,\lambda
}
\right],\dots,\left[\substack{z_{n-1},z_{n}\\
\lambda,\lambda
}
\right],\left[\substack{z_{n},z_{n+1}\\
\lambda,\lambda
}
\right](t-z_{n+1})^{\alpha_{n+1}-1};z_{n+1}),
\end{align*}
we get the following:
\begin{thm}
For $k_{1},\dots,k_{d}\in\mathbb{Z}_{\geq1}$ with $k_{d}\geq2$ and
$0\leq\beta\leq1/2$, we have
\begin{align*}
 & \sum_{0<m_{1}\leq\cdots\leq m_{d}}\frac{1}{m_{1}^{k_{1}}\cdots m_{d}^{k_{d}}}\cdot\frac{\Gamma(m_{d}+\beta)}{\Gamma(m_{d})}\\
 & =\delta(\Bbbk)\frac{2^{2\beta-1}\Gamma(2-2\beta)\sqrt{\pi}}{\Gamma(3/2-\beta)}\sum_{0<m_{1}\leq\cdots\leq m_{r}}\frac{(-1)^{m_{1}(l_{1}-1)+\cdots+m_{r}(l_{r}-1)}2^{\#\{m_{1},\dots,m_{r}\}}}{m_{1}^{l_{1}}\cdots m_{r}^{l_{r}}}\cdot\frac{m_{r}\Gamma(m_{r}+\beta)}{\Gamma(m_{r}+1-\beta)}
\end{align*}
where
\[
(l_{1},\dots,l_{r})=\sigma(k_{1},\dots,k_{d}).
\]
\end{thm}

The case $\beta=0$ of the above theorem implies Zhao's 2-1 formula.
The case $\beta=1/2$ implies

\[
\sum_{0<m_{1}\leq\cdots\leq m_{d}}\frac{m_{d}{2m_{d} \choose m_{d}}}{m_{1}^{k_{1}}\cdots m_{d}^{k_{d}}4^{m_{d}}}=\delta(\Bbbk)\sum_{0<m_{1}\leq\cdots\leq m_{r}}\frac{(-1)^{m_{1}(l_{1}-1)+\cdots+m_{r}(l_{r}-1)}2^{\#\{m_{1},\dots,m_{r}\}}m_{r}}{m_{1}^{l_{1}}\cdots m_{r}^{l_{r}}}.
\]
This is generalized to
\begin{equation}
\sum_{0<m_{1}\leq\cdots\leq m_{d}}\frac{m_{d}{2m_{d} \choose m_{d}}}{m_{1}^{k_{1}}\cdots m_{d}^{k_{d}}(x+x^{-1})^{2m_{d}}}=\delta(\Bbbk)\sum_{0<m_{1}\leq\cdots\leq m_{r}}\frac{(-1)^{m_{1}(l_{1}-1)+\cdots+m_{r}(l_{r}-1)}2^{\#\{m_{1},\dots,m_{r}\}}m_{r}}{m_{1}^{l_{1}}\cdots m_{r}^{l_{r}}}x^{2m_{r}},\label{eq:cent_binom}
\end{equation}
which follows from the case $\alpha_{0}=2$, $\lambda=1/2$, $\alpha_{n+1}=\frac{1}{2}$,
$z_{0}=-1$, $z_{n+1}=\frac{x^{2}+x^{-2}}{2}$. This identity is equivalent
to (\ref{eq:Zhao_21_truncated}) by the following argument. Let
\[
H_{m}:=\sum_{0<m_{1}\leq\cdots\leq m_{d}=m}\frac{1}{m_{1}^{k_{1}}\cdots m_{d}^{k_{d}}},\quad H_{m}^{\#}:=\delta(\Bbbk)\sum_{0<m_{1}\leq\cdots\leq m_{r}=m}\frac{(-1)^{m_{1}(l_{1}-1)+\cdots+m_{r}(l_{r}-1)}2^{\#\{m_{1},\dots,m_{r}\}}}{m_{1}^{l_{1}}\cdots m_{r}^{l_{r}}}
\]
for $m\geq1$. Since
\[
\frac{m_{d}{2m_{d} \choose m_{d}}}{(x+x^{-1})^{2m_{d}}}=\sum_{s=0}^{\infty}(-1)^{s}m_{d}{2m_{d} \choose m_{d}}{2m_{d}-1+s \choose 2m_{d}-1}x^{2(m_{d}+s)},
\]
the comparison of the coefficients of $x^{2n}$ in (\ref{eq:cent_binom})
implies
\[
\sum_{\substack{0\leq s,\,1\leq m\\
s+m=n
}
}H_{m}\cdot(-1)^{s}m{2m \choose m}{2m-1+s \choose 2m-1}=n\cdot H_{n}^{\#}
\]
and thus, for $N\geq1$, 
\begin{align}
\sum_{n=1}^{N}\frac{{N \choose n}}{{N+n \choose n}}H_{n}^{\#} & =\sum_{n=1}^{N}\frac{{N \choose n}}{{N+n \choose n}}\frac{1}{n}\sum_{\substack{0\leq s,\,1\leq m\\
s+m=n
}
}H_{m}\cdot(-1)^{s}m{2m \choose m}{2m-1+s \choose 2m-1}\nonumber \\
 & =\sum_{m=1}^{N}H_{m}\left(\sum_{s=0}^{N-m}\frac{{N \choose s+m}}{{N+s+m \choose s+m}}\frac{m}{s+m}(-1)^{s}{2m \choose m}{2m-1+s \choose 2m-1}\right).\label{eq:Hsharp_eq_Hstar}
\end{align}
Since, 
\[
\sum_{s=0}^{N-m}\frac{{N \choose s+m}}{{N+s+m \choose s+m}}\frac{m}{s+m}(-1)^{s}{2m \choose m}{2m-1+s \choose 2m-1}=1
\]
by Lemma \ref{lem:binom_identity} below, (\ref{eq:Hsharp_eq_Hstar})
implies
\[
\sum_{m=1}^{N}H_{m}=\sum_{n=1}^{N}\frac{{N \choose n}}{{N+n \choose n}}H_{n}^{\#},
\]
i.e.,
\[
\zeta_{N}^{\star}(\boldsymbol{k})=\delta(\boldsymbol{k})\zeta_{N}^{\#}(\boldsymbol{l})\qquad(\boldsymbol{l}=\sigma(\boldsymbol{k}))
\]
where we put
\begin{align*}
\zeta_{N}^{\star}(k_{1},\dots,k_{d}) & :=\sum_{0<m_{1}\leq\cdots\leq m_{d}\leq N}\frac{1}{m_{1}^{k_{1}}\cdots m_{d}^{k_{d}}},\\
\zeta_{N}^{\#}(l_{1},\dots,l_{r}) & :=\sum_{0<m_{1}\leq\cdots\leq m_{r}\leq N}\frac{{N \choose m_{r}}}{{N+m_{r} \choose m_{r}}}\frac{(-1)^{m_{1}(l_{1}-1)+\cdots+m_{r}(l_{r}-1)}2^{\#\{m_{1},\dots,m_{r}\}}}{m_{1}^{l_{1}}\cdots m_{r}^{l_{r}}}.
\end{align*}

\begin{lem}
\label{lem:binom_identity}For $1\leq m\leq N$, we have
\[
\sum_{s=0}^{N-m}\frac{{N \choose s+m}}{{N+s+m \choose s+m}}\frac{m}{s+m}(-1)^{s}{2m \choose m}{2m-1+s \choose 2m-1}=1.
\]
\end{lem}

\begin{proof}
Let $S$ be the left-hand side. Then, we have
\[
S=\frac{2N!^{2}}{(N-m)!(m-1)!^{2}}\sum_{s=0}^{N-m}(-1)^{s}{N-m \choose s}\left(\frac{1}{s+m}\prod_{j=2m}^{N+m}\frac{1}{s+j}\right).
\]
By partial fractional decomposition with respect to $s$
\[
\frac{1}{s+m}\prod_{j=2m}^{N+m}\frac{1}{j+s}=\frac{(m-1)!}{N!}\frac{1}{s+m}+\sum_{r=0}^{N-m}\frac{(-1)^{r+1}}{(m+r)r!(N-m-r)!}\frac{1}{s+2m+r}
\]
and the identity
\[
\sum_{s=0}^{N-m}(-1)^{s}{N-m \choose s}\frac{1}{s+a}=\frac{(N-m)!(a-1)!}{(N-m+a)!},
\]
we have
\begin{align*}
S & =\frac{2N!^{2}}{(N-m)!(m-1)!^{2}}\cdot\frac{(m-1)!}{N!}\cdot\frac{(N-m)!(m-1)!}{(N-m+m)!}\\
 & +\frac{2N!^{2}}{(N-m)!(m-1)!^{2}}\sum_{r=0}^{N-m}\frac{(-1)^{r+1}}{(m+r)r!(N-m-r)!}\cdot\frac{(N-m)!(2m+r-1)!}{(N+m+r)!}\\
 & =2-\frac{2N!^{2}}{(N-m)!(m-1)!^{2}}\sum_{r=0}^{N-m}(-1)^{r}{N-m \choose r}\frac{1}{r+m}\prod_{j=2m}^{N+m}\frac{1}{r+j}\\
 & =2-S,
\end{align*}
which implies $S=1$.
\end{proof}

\section{Case B1: Application to omega values appearing in Willmore energy
of certain Lawson surfaces\label{sec:Case-B1-OmegaValues}}

In this section we give a detailed study of Case B1 in the classification
of Section \ref{sec:classification_genus_one}. We first investigate
the general case, and then turn to the special case closely related
to omega values appearing in Willmore energy of certain Lawson surfaces.

\subsection{General case}

If $S_{\mathrm{fin}}=\left\{ p_{1},p_{2},p_{3}\right\} $ ($p_{1},p_{2},p_{3}$:
distinct) and $\alpha_{i}=\frac{1}{2}$ ($1\leq i\leq3$), the associated
complex curve is
\[
X_{\boldsymbol{z},\boldsymbol{\alpha}}=\left\{ (x,u_{1},u_{2},u_{3})\in\mathbb{C}^{4}\left|u_{i}^{2}=\prod_{j\in\{1,2,3\}\setminus\{i\}}(x-p_{j})\quad\left(i\in\{1,2,3\}\right)\right.\right\} 
\]
and the $1$-forms $\left\{ \substack{p_{i},p_{j}\\
1/2,1/2
}
\right\} $ ($1\leq i,j\leq3$) we want to rationalize are
\[
\frac{dx}{\sqrt{(x-p_{i})(x-p_{j})}}=\begin{cases}
\frac{dx}{x-p_{i}} & i=j\\
\frac{dx}{u_{k}} & i\neq j
\end{cases}
\]
where $k$ is chosen so that $\{i,j,k\}=\{1,2,3\}$. A rational map
$\varphi:\mathbb{P}^{1}\rightarrow X_{\boldsymbol{z},\boldsymbol{\alpha}};$
$\xi\mapsto\left(x(\xi),u_{1}(\xi),u_{2}(\xi),u_{3}(\xi)\right)$
is given, for example, by
\[
\left(\begin{array}{c}
x(\xi)\\
u_{1}(\xi)\\
u_{2}(\xi)\\
u_{3}(\xi)
\end{array}\right)=\left(\begin{array}{cccc}
1 & 1 & 1 & 1\\
0 & -1 & 1 & 1\\
0 & 1 & -1 & 1\\
0 & 1 & 1 & -1
\end{array}\right)\left(\begin{array}{c}
\xi\\
Q_{1}(\xi)\\
Q_{2}(\xi)\\
Q_{3}(\xi)
\end{array}\right)
\]
where $Q_{i}(\xi)=\frac{1}{4}\frac{P_{i}(\xi)}{\xi-p_{i}}$ with $P_{i}(\xi)=(\xi-p_{i})^{2}-(p_{j}-p_{i})(p_{k}-p_{i})$
($j$ and $k$ are chosen in such a way that $\{i,j,k\}=\{1,2,3\}$).
The pull-backs of the rational $1$-forms $\left\{ \substack{p_{i},p_{j}\\
1/2,1/2
}
\right\} $ ($1\leq i,j\leq3$) are given by 
\[
\varphi^{*}\left\{ \substack{p_{i},p_{j}\\
1/2,1/2
}
\right\} =\frac{dx(\xi)}{\sqrt{(x(\xi)-p_{i})(x(\xi)-p_{j})}}=d\log\begin{cases}
\frac{P_{i}(\xi)^{2}}{(\xi-p_{1})(\xi-p_{2})(\xi-p_{3})} & 1\leq i=j\leq3\\
\frac{(\xi-p_{i})(\xi-p_{j})}{\xi-p_{k}} & 1\leq i\neq j\leq3.
\end{cases}
\]
Here, in the second case, $k$ is chosen as $\{i,j,k\}=\{1,2,3\}$.
Summarizing this, we obtain the following table ($\alpha_{p_{1}}=\alpha_{p_{2}}=\alpha_{p_{3}}=1/2$):
\begin{flushleft}
\begin{tabular}{|c||c|c|c|}
\hline 
The table of $\varphi^{*}\left\{ \substack{x,y\\
1/2,1/2
}
\right\} $ & $y=p_{1}$ & $y=p_{2}$ & $y=p_{3}$\tabularnewline
\hline 
\hline 
$x=p_{1}$ & $d\log\left(\frac{P_{1}(\xi)^{2}}{(\xi-p_{1})(\xi-p_{2})(\xi-p_{3})}\right)$ & $d\log\left(\frac{(\xi-p_{1})(\xi-p_{2})}{\xi-p_{3}}\right)$ & $d\log\left(\frac{(\xi-p_{1})(\xi-p_{3})}{\xi-p_{2}}\right)$\tabularnewline
\hline 
$x=p_{2}$ & $d\log\left(\frac{(\xi-p_{2})(\xi-p_{1})}{\xi-p_{3}}\right)$ & $d\log\left(\frac{P_{2}(\xi)^{2}}{(\xi-p_{1})(\xi-p_{2})(\xi-p_{3})}\right)$ & $d\log\left(\frac{(\xi-p_{2})(\xi-p_{3})}{\xi-p_{1}}\right)$\tabularnewline
\hline 
$x=p_{3}$ & $d\log\left(\frac{(\xi-p_{3})(\xi-p_{1})}{\xi-p_{2}}\right)$ & $d\log\left(\frac{(\xi-p_{3})(\xi-p_{2})}{\xi-p_{1}}\right)$ & $d\log\left(\frac{P_{3}(\xi)^{2}}{(\xi-p_{1})(\xi-p_{2})(\xi-p_{3})}\right)$\tabularnewline
\hline 
\end{tabular}
\par\end{flushleft}

On the other hand, if $\alpha_{i}'=\alpha_{i}+\frac{1}{2}=1$ ($1\leq i\leq3$),
the associated complex curve is (4-point punctured) $\mathbb{P}^{1}$
in the first place and the associated table becomes as follows:
\begin{flushleft}
\begin{tabular}{|c||c|}
\hline 
The table of $\varphi^{*}\left\{ \substack{x,y\\
1,1
}
\right\} $ & $y\in\{p_{1},p_{2},p_{3}\}$\tabularnewline
\hline 
\hline 
$x\in\{p_{1},p_{2},p_{3}\}$ & $d\log(\xi-x)$\tabularnewline
\hline 
\end{tabular}
\par\end{flushleft}

Notice that the inverse image of $x=p_{i}$ under $\varphi$ are exactly
the two roots $p_{i}^{\pm}\coloneqq p_{i}\pm\sqrt{(p_{j}-p_{i})(p_{k}-p_{i})}$
of $P_{i}$ ($1\leq i\leq3$) and $\varphi^{-1}(\infty)=\left\{ p_{1},p_{2},p_{3},\infty\right\} $.
Hence, we obtain the following formula:
\begin{thm}
\label{thm:Case_1a}Let $\widetilde{p_{i}}$ be one of the two roots
of $P_{i}$. Then, for $x_{0},\ldots,x_{n+1}\in\left\{ p_{1},p_{2},p_{3}\right\} $,
we have
\[
\frac{I_{\gamma}\left(\widetilde{x}_{0};f_{x_{0},x_{1}},f_{x_{1},x_{2}},\ldots,f_{x_{n},x_{n+1}};\widetilde{x}_{n+1}\right)}{I_{\gamma}\left(\widetilde{x}_{0};f_{x_{0},x_{n+1}};\widetilde{x}_{n+1}\right)}=I_{\varphi(\gamma)}\left(x_{0};e_{x_{1}},e_{x_{2}},\ldots,e_{x_{n}};x_{n+1}\right),
\]
where 
\[
f_{p_{i},p_{j}}(\xi)=d\log\begin{cases}
\frac{P_{i}(\xi)^{2}}{(\xi-p_{1})(\xi-p_{2})(\xi-p_{3})} & i=j\\
\frac{(\xi-p_{i})(\xi-p_{j})}{\xi-p_{k}} & i\neq j.
\end{cases}
\]
Here, in the second case, $k$ is chosen as $\{i,j,k\}=\{1,2,3\}$.
\end{thm}

\subsection{Application in Charlton's observation on omega values appearing in
Willmore energy of certain Lawson surfaces}

It is worth noting that the special case $x_{0}\neq x_{1}\neq x_{2}\neq\cdots\neq x_{n+1}$
of Theorem \ref{thm:Case_1a} only involves three simple differential
forms
\begin{align}
 & e_{p_{1}}(\xi)+e_{p_{2}}(\xi)-e_{p_{3}}(\xi),\nonumber \\
 & e_{p_{2}}(\xi)+e_{p_{3}}(\xi)-e_{p_{1}}(\xi),\label{eq:three forms}\\
 & e_{p_{3}}(\xi)+e_{p_{1}}(\xi)-e_{p_{2}}(\xi),\nonumber 
\end{align}
which makes the theorem look particularly simple. Furthermore, this
case yields a useful reduction formula for the $\Omega$-values discussed
in \cite{CHHT_omega}. To show the connection to the $\Omega$-values,
it would be convenient to describe Theorem \ref{thm:Case_1a} in terms
of a new parameter $\lambda\in\mathbb{C}^{\times}$ as follows. Notice
that by the M\"obius transform
\[
\tau(\xi)=\lambda\frac{\frac{\xi-p_{2}}{p_{1}-p_{2}}-\frac{\lambda^{2}+1}{\lambda^{2}-1}}{\frac{\xi-p_{2}}{p_{1}-p_{2}}+\frac{\lambda^{2}+1}{\lambda^{2}-1}},
\]
the four points $(\infty,p_{1},p_{2},p_{3})$ are translated into
$(\lambda,-\lambda^{-1},-\lambda,\lambda^{-1})$, where $\lambda$
is a complex number satisfying
\[
\frac{p_{3}-p_{2}}{p_{1}-p_{2}}=\left(\frac{\lambda^{2}+1}{\lambda^{2}-1}\right)^{2}.
\]
The preimages $(\varphi\circ\tau^{-1})^{-1}(p_{i})$ are given by
\begin{align*}
\tau(\varphi^{-1}(p_{1})) & =\left\{ \pm\sqrt{-1}\right\} \\
\tau(\varphi^{-1}(p_{2})) & =\left\{ 0,\infty\right\} \\
\tau(\varphi^{-1}(p_{3})) & =\left\{ \pm1\right\} .
\end{align*}
In the new $\tau$-coordinate, the differential forms nicely simplify
to the following:
\begin{flushleft}
\begin{tabular}{|c||c|c|c|}
\hline 
The table of $(\varphi\circ\tau^{-1})^{*}\left\{ \substack{x,y\\
\alpha_{x},\alpha_{y}
}
\right\} $ & $y=p_{1}$ & $y=p_{2}$ & $y=p_{3}$\tabularnewline
\hline 
\hline 
$x=p_{1}$ & $d\log\left(\frac{(\tau^{2}+1)^{2}}{(\tau^{2}-\lambda^{2})(\tau^{2}-\lambda^{-2})}\right)$ & $d\log\left(\frac{(\tau+\lambda)(\tau+\lambda^{-1})}{(\tau-\lambda)(\tau-\lambda^{-1})}\right)$ & $d\log\left(\frac{(\tau+\lambda^{-1})(\tau-\lambda^{-1})}{(\tau+\lambda)(\tau-\lambda)}\right)$\tabularnewline
\hline 
$x=p_{2}$ & $d\log\left(\frac{(\tau+\lambda)(\tau+\lambda^{-1})}{(\tau-\lambda)(\tau-\lambda^{-1})}\right)$ & $d\log\left(\frac{\tau^{2}}{(\tau^{2}-\lambda^{2})(\tau^{2}-\lambda^{-2})}\right)$ & $d\log\left(\frac{(\tau-\lambda^{-1})(\tau+\lambda)}{(\tau+\lambda^{-1})(\tau-\lambda)}\right)$\tabularnewline
\hline 
$x=p_{3}$ & $d\log\left(\frac{(\tau+\lambda^{-1})(\tau-\lambda^{-1})}{(\tau+\lambda)(\tau-\lambda)}\right)$ & $d\log\left(\frac{(\tau-\lambda^{-1})(\tau+\lambda)}{(\tau+\lambda^{-1})(\tau-\lambda)}\right)$ & $d\log\left(\frac{(\tau^{2}-1)^{2}}{(\tau^{2}-\lambda^{2})(\tau^{2}-\lambda^{-2})}\right)$\tabularnewline
\hline 
\end{tabular}
\par\end{flushleft}

The three differential forms in (\ref{eq:three forms}) are now expressed
as
\[
d\log\left(\frac{(\tau+\lambda)(\tau+\lambda^{-1})}{(\tau-\lambda)(\tau-\lambda^{-1})}\right),\;d\log\left(\frac{(\tau-\lambda^{-1})(\tau+\lambda)}{(\tau+\lambda^{-1})(\tau-\lambda)}\right),\;d\log\left(\frac{(\tau+\lambda^{-1})(\tau-\lambda^{-1})}{(\tau+\lambda)(\tau-\lambda)}\right)
\]
in $\lambda$, and the omega values (up to sign) are defined as iterated
integrals of these differential forms with $\lambda=\rho\coloneqq e^{\pi i/4}$,
a particularly symmetric case. To align with the original definition
of the omega values, we set $g_{x}$ ($x=0,\pm1$) as
\[
g_{x}\coloneqq\begin{cases}
-d\log\left(\frac{(\tau+\rho)(\tau+\rho^{-1})}{(\tau-\rho)(\tau-\rho^{-1})}\right)=e_{\rho}+e_{\rho^{7}}-e_{\rho^{3}}-e_{\rho^{5}} & x=-1\\
-d\log\left(\frac{(\tau-\rho^{-1})(\tau+\rho)}{(\tau+\rho^{-1})(\tau-\rho)}\right)=e_{\rho}+e_{\rho^{3}}-e_{\rho^{5}}-e_{\rho^{7}} & x=1\\
-d\log\left(\frac{(\tau+\rho^{-1})(\tau-\rho^{-1})}{(\tau+\rho)(\tau-\rho)}\right)=e_{\rho}+e_{\rho^{5}}-e_{\rho^{3}}-e_{\rho^{7}} & x=0.
\end{cases}
\]
Thus, by letting $\gamma$ be a straight line path from $\widetilde{x}_{0}=0$
to $\widetilde{x}_{n+1}=1$, and specializing to the case $x_{i}\neq x_{i+1}$
for $i\in\left\{ 0,1,\ldots,n\right\} $ in Theorem \ref{thm:Case_1a},
we obtain the following formula.
\begin{prop}
\label{prop:Omega_value_evaluation}Suppose that $x_{0},x_{1},\ldots,x_{n+1}\in\{0,1,-1\}$
satisfy $0=x_{0}\neq x_{1}\neq x_{2}\neq\cdots\neq x_{n+1}=1$. Then,
$x_{i}+x_{i+1}\in\left\{ 0,\pm1\right\} $ for $0\leq i\leq n$ and
we have
\[
\frac{1}{\pi i}I\left(0;g_{x_{0}+x_{1}},g_{x_{1}+x_{2}},\ldots,g_{x_{n}+x_{n+1}};1\right)=(-1)^{n}I\left(0;e_{x_{1}},e_{x_{2}},\ldots,e_{x_{n}};1\right).
\]
\end{prop}

\begin{rem}
The correspondence between the entries of the two sides of this formula
can be visually described using a triangle. Consider a triangle whose
vertices are labelled by $0,1,-1$ and edges between the vertices
labelled by $x$ and $y$ are labelled by $x+y$ (equivalently, labelled
by $-z$ where $z$ is the label of the vertex not connected to the
edge). For a path $0=x_{0}\rightarrow x_{1}\rightarrow x_{2}\rightarrow\cdots\rightarrow x_{n+1}=1$
connecting vertex $0$ and vertex $1$ in $n$-steps, the sequence
that appear on the right-hand side encodes the sequence of vertices
along this path, whereas the left-hand side encodes the sequence of
edges along this path. The significance of Proposition \ref{prop:Omega_value_evaluation}
is that although the left-hand side is apparently a period of mixed
Tate motives over $\mathbb{Z}[\zeta_{8},\frac{1}{2}]$, Proposition
\ref{prop:Omega_value_evaluation} claims that it is in fact a period
of mixed Tate motives over $\mathbb{Z}[\frac{1}{2}]$ which is a much
smaller space. Proposition \ref{prop:Omega_value_evaluation} has
different versions associated with other pairs of endpoints $p,q\in\left\{ \pm1,\pm\sqrt{-1},0,\infty\right\} $.
For example, if $\left(p,q\right)=\left(0,\infty\right)$, one may
get, for instance,
\[
\frac{1}{\pi i}I_{\gamma}\left(0;g_{x_{0}+x_{1}},g_{x_{1}+x_{2}},\ldots,g_{x_{n}+x_{n+1}};\infty\right)=(-1)^{n}I_{\gamma'}\left(0;e_{x_{1}},e_{x_{2}},\ldots,e_{x_{n}};0\right),
\]
where $\gamma$ is a path from $0$ to $+\infty$ that stays inside
the cone $\mathbb{R}_{>0}+\mathbb{R}_{>0}\rho^{7}$ and $\gamma'$
is a path from $0$ to itself that stays inside the right half complex
plane and encircles $1$ once counterclockwisely.

Let $\xi_{1,g}$ be the Lawson surface of genus $g$ \cite{Lawson_CMS}
and 
\[
\mathrm{Area}(\xi_{1,g})=8\pi\left(1-\sum_{m=0}^{\infty}\frac{\alpha_{2m+1}}{(2g+2)^{2m+1}}\right)
\]
is the Taylor expansion of its area at $g=\infty$. In \cite{CHHT_omega},
Charlton, Heller, Heller, and Traizet established an algorithm to
compute the coefficients $\alpha_{i}$ in terms of multiple $\Omega$-values
defined by 
\[
\Omega_{i_{1},i_{2},\ldots,i_{r}}=I\left(0;\omega_{i_{1}},\omega_{i_{2}},\ldots,\omega_{i_{r}};1\right)
\]
where 
\[
\omega_{1}=g_{0},\;\omega_{2}=g_{-1},\;\omega_{3}=g_{1}
\]
with $\rho=e^{\pi i/4}$. The coefficient $\alpha_{1}=\log2=I\left(0;e_{-1};1\right)$
was already evaluated in the paper \cite[Proposition 3.2 and Theorem 4.7]{HHT_area}.
Via their algorithm, the next coefficient is computed
\begin{equation}
\alpha_{3}=-\left(\frac{\Omega_{2,1}}{i\pi}\right)^{3}+\frac{1}{2}\left(\frac{\Omega_{2,1}}{i\pi}\right)\left(\frac{-\Omega_{2,2,3}+6\Omega_{3,1,1}+3\Omega_{3,3,3}}{i\pi}\right)-\frac{1}{2}\left(\frac{6\Omega_{2,1,1,1}+\Omega_{2,1,3,3}+\Omega_{2,2,2,1}-\Omega_{3,1,2,3}+\Omega_{3,3,2,1}}{i\pi}\right),\label{eq:alpha_3}
\end{equation}
which is further nicely simplified to
\[
\alpha_{3}=\frac{9}{4}\zeta(3)
\]
\cite[Section 8]{CHHT_omega}. Moreover, $\Omega_{i_{1},i_{2},\ldots,i_{r}}$
of only special indices $(i_{1},i_{2},\ldots,i_{r})$ appear in the
simplified expression $\alpha_{i}$ for general $i$. Our Proposition
\ref{prop:Omega_value_evaluation} gives an explicit reduction formula
of those special $\Omega$-values into alternating multiple zeta values.
For example,
\begin{align*}
\frac{\Omega_{2,1}}{i\pi} & =\frac{I\left(0;g_{-1},g_{0};1\right)}{i\pi}=\frac{I\left(0;g_{0+(-1)},g_{(-1)+1};1\right)}{i\pi}=-I\left(0;e_{-1};1\right),\\
\frac{\Omega_{2,2,3}}{i\pi} & =\frac{I\left(0;g_{-1},g_{-1},g_{1};1\right)}{i\pi}=\frac{I\left(0;g_{0+(-1)},g_{(-1)+0},g_{0+1};1\right)}{i\pi}=I\left(0;e_{-1},e_{0};1\right),\\
\frac{\Omega_{3,1,2,3}}{i\pi} & =\frac{I\left(0;g_{1},g_{0},g_{-1},g_{1};1\right)}{i\pi}=\frac{I\left(0;g_{0+1},g_{1+(-1)},g_{(-1)+0},g_{0+1};1\right)}{i\pi}=-I\left(0;e_{1},e_{-1},e_{0};1\right)
\end{align*}
and so on. Plugging these into (\ref{eq:alpha_3}) and linearly expanding
the expression by the shuffle product, we find
\[
\alpha_{3}=I\left(0;W_{3};1\right)
\]
where
\[
W_{3}=6e_{-1}^{3}-6e_{1}e_{-1}^{2}+e_{-1}^{2}e_{0}+e_{-1}e_{0}e_{-1}-e_{-1}e_{1}e_{0}-e_{1}e_{0}e_{-1}-2e_{1}e_{-1}e_{0}.
\]
\end{rem}

\section{Case A2: Application to Ohno's relation\label{sec:Case-A2-Ohno}}

Ohno's relation is the following theorem:
\begin{thm}[Ohno's relation, \cite{Ohno_rel} ]
For an admissible index $\boldsymbol{k}=(k_{1},\dots,k_{d})$ and
nonnegative integer $\ell$, we put
\[
O_{\ell}(\boldsymbol{k})=\sum_{\ell_{1}+\cdots+\ell_{d}=\ell}\zeta(k_{1}+\ell_{1},\dots,k_{d}+\ell_{d}).
\]
Then, for $\ell\geq0$ and an admissible index $\boldsymbol{k}$,
we have
\[
O_{\ell}(\boldsymbol{k})=O_{\ell}(\boldsymbol{k}^{\dagger})
\]
where $\boldsymbol{k}^{\dagger}$ is the dual index of $\boldsymbol{k}$.
\end{thm}

Note that Ohno's relation can be written by the generating series
as
\[
O(\boldsymbol{k};\alpha)=O(\boldsymbol{k}^{\dagger};\alpha)
\]
where
\begin{align*}
O(\boldsymbol{k};\alpha) & :=\sum_{\ell=0}^{\infty}O_{\ell}(\boldsymbol{k})\alpha^{\ell}\\
 & =\sum_{0<m_{1}<\cdots<m_{d}}\prod_{j=1}^{d}\frac{1}{m_{j}^{k_{j}-1}(m_{j}-\alpha)}.
\end{align*}

In this section, we will see that Ohno's relation follows from the
special case of translation invariance of iterated beta integrals.
We begin with a simple example. By translation invariance, applying
the change of variables $t\mapsto1-t$, and reversing the path, we
have
\begin{align*}
-\hat{B}^{\mathrm{f},\mathrm{f}}\left(\left.\substack{0\\
\alpha+1
}
\right|\left.\substack{1\\
1
}
\right|\left.\substack{0\\
\alpha+1
}
\right|\left.\substack{0\\
\alpha+1
}
\right|\left.\substack{1\\
1
}
\right.\right) & =-\hat{B}^{\mathrm{f},\mathrm{f}}\left(\left.\substack{0\\
0
}
\right|\left.\substack{1\\
-\alpha
}
\right|\left.\substack{0\\
0
}
\right|\left.\substack{0\\
0
}
\right|\left.\substack{1\\
-\alpha
}
\right.\right)\\
 & =-\hat{B}^{\mathrm{f},\mathrm{f}}\left(\left.\substack{1\\
0
}
\right|\left.\substack{0\\
-\alpha
}
\right|\left.\substack{1\\
0
}
\right|\left.\substack{1\\
0
}
\right|\left.\substack{0\\
-\alpha
}
\right.\right)\\
 & =\hat{B}^{\mathrm{f},\mathrm{f}}\left(\left.\substack{0\\
\alpha+1
}
\right|\left.\substack{1\\
1
}
\right|\left.\substack{1\\
1
}
\right|\left.\substack{0\\
\alpha+1
}
\right|\left.\substack{1\\
1
}
\right.\right).
\end{align*}
Thus,
\[
-\hat{B}^{\mathrm{f},\mathrm{f}}\left(\left.\substack{0\\
\alpha+1
}
\right|\left.\substack{1\\
1
}
\right|\left.\substack{0\\
\alpha+1
}
\right|\left.\substack{0\\
\alpha+1
}
\right|\left.\substack{1\\
1
}
\right.\right)=\hat{B}^{\mathrm{f},\mathrm{f}}\left(\left.\substack{0\\
\alpha+1
}
\right|\left.\substack{1\\
1
}
\right|\left.\substack{1\\
1
}
\right|\left.\substack{0\\
\alpha+1
}
\right|\left.\substack{1\\
1
}
\right.\right).
\]
Here, the left-hand side is equal to
\begin{align}
\frac{I\left(0;\frac{dt}{t^{\alpha+1}},\frac{t^{\alpha}dt}{1-t},\frac{dt}{t},\frac{dt}{t^{\alpha+1}};1\right)}{I\left(0;\frac{dt}{t^{\alpha+1}};1\right)} & =-\alpha I\left(0;\frac{dt}{t^{\alpha+1}},\frac{t^{\alpha}dt}{1-t},\frac{dt}{t},\frac{dt}{t^{\alpha+1}};1\right),\label{eq:Ohno_ex_left}
\end{align}
and the right-hand side is equal to
\begin{align}
\frac{I\left(0;\frac{dt}{t^{\alpha+1}},\frac{dt}{1-t},\frac{t^{\alpha}dt}{1-t},\frac{dt}{t^{\alpha+1}};1\right)}{I\left(0;\frac{dt}{t^{\alpha+1}};1\right)} & =-\alpha I\left(0;\frac{dt}{t^{\alpha+1}},\frac{dt}{1-t},\frac{t^{\alpha}dt}{1-t},\frac{dt}{t^{\alpha+1}};1\right).\label{eq:Ohno_ex_right}
\end{align}
Now, let us calculate the series expression. For (\ref{eq:Ohno_ex_left}),
we have
\begin{align*}
-\alpha I\left(0;\frac{dt}{t^{\alpha+1}},\frac{t^{\alpha}dt}{1-t},\frac{dt}{t},\frac{dt}{t^{\alpha+1}};1\right) & =I\left(0;\frac{dt}{1-t},\frac{dt}{t},\frac{dt}{t^{\alpha+1}};1\right)\\
 & =\sum_{m=1}^{\infty}\frac{1}{m}I(0;t^{m}\frac{dt}{t},\frac{dt}{t^{\alpha+1}};1)\\
 & =\sum_{m=1}^{\infty}\frac{1}{m^{2}}I(0;t^{m}\frac{dt}{t^{\alpha+1}};1)\\
 & =\sum_{m=1}^{\infty}\frac{1}{m^{2}(m-\alpha)}=O(3;\alpha).
\end{align*}
For (\ref{eq:Ohno_ex_right}), we have
\begin{align*}
-\alpha I(0;\frac{dt}{t^{\alpha+1}},\frac{dt}{1-t},\frac{t^{\alpha}dt}{1-t},\frac{dt}{t^{\alpha+1}};1) & =I(0;t^{-\alpha}\frac{dt}{1-t},t^{\alpha}\frac{dt}{1-t},\frac{dt}{t^{\alpha+1}};1)\\
 & =\sum_{m=1}^{\infty}\frac{1}{m-\alpha}I(0;t^{m}\frac{dt}{1-t},\frac{dt}{t^{\alpha+1}};1)\\
 & =\sum_{m=1}^{\infty}\sum_{n=m+1}^{\infty}\frac{1}{(m-\alpha)n}I(0;t^{n}\frac{dt}{t^{\alpha+1}};1)\\
 & =\sum_{m=1}^{\infty}\sum_{n=m+1}^{\infty}\frac{1}{(m-\alpha)n(n-\alpha)}=O(1,2;\alpha).
\end{align*}
Thus, we get $O(3;\alpha)=O(1,2;\alpha)$.

The general case follows by exactly the same method. In fact, the
special case
\begin{align*}
(z_{0},z_{1},\dots,z_{k+1}) & =(0,1,\{0\}^{k_{1}-1},\dots,1,\{0\}^{k_{d}-1},1)\\
(\alpha_{0},\dots,\alpha_{k+1}) & =(\alpha+1,1,\{\alpha+1\}^{k_{1}-1},\dots,1,\{\alpha+1\}^{k_{d}-1},1)
\end{align*}
of the series expansion formula (\ref{eq:SerExpFF2}) in Theorem \ref{thm:expansion_iterated_beta_finite}
gives
\[
O(\boldsymbol{k};\alpha)=(-1)^{d}\hat{B}^{\mathrm{f},\mathrm{f}}\left(\left.\substack{z_{0}\\
\alpha_{0}
}
\right|\left.\substack{z_{1}\\
\alpha_{1}
}
\right|\cdots\left|\substack{z_{k+1}\\
\alpha_{k+1}
}
\right.\right).
\]
Thus, by translation invariance, applying the change of variables
$t\mapsto1-t$, and reversing the path, we have
\begin{align*}
\hat{B}^{\mathrm{f},\mathrm{f}}\left(\left.\substack{z_{0}\\
\alpha_{0}
}
\right|\left.\substack{z_{1}\\
\alpha_{1}
}
\right|\cdots\left|\substack{z_{k+1}\\
\alpha_{k+1}
}
\right.\right) & =\hat{B}^{\mathrm{f},\mathrm{f}}\left(\left.\substack{z_{0}\\
\alpha_{0}-\alpha-1
}
\right|\left.\substack{z_{1}\\
\alpha_{1}-\alpha-1
}
\right|\cdots\left|\substack{z_{k+1}\\
\alpha_{k+1}-\alpha-1
}
\right.\right)\\
 & =\hat{B}^{\mathrm{f},\mathrm{f}}\left(\left.\substack{1-z_{0}\\
\alpha_{0}-\alpha-1
}
\right|\left.\substack{1-z_{1}\\
\alpha_{1}-\alpha-1
}
\right|\cdots\left|\substack{1-z_{k+1}\\
\alpha_{k+1}-\alpha-1
}
\right.\right)\\
 & =(-1)^{k_{1}+\cdots+k_{d}}\hat{B}^{\mathrm{f},\mathrm{f}}\left(\left.\substack{1-z_{k+1}\\
2+\alpha-\alpha_{k+1}
}
\right|\cdots\left|\substack{1-z_{1}\\
2+\alpha-\alpha_{1}
}
\right.\left|\substack{1-z_{0}\\
2+\alpha-\alpha_{0}
}
\right.\right).
\end{align*}
This proves $O(\boldsymbol{k};\alpha)=O(\boldsymbol{k}^{\dagger};\alpha)$.

\section{Case B2: Other hyperlogarithm identities\label{sec:Case-B2-others}}

Finally, we will discuss the last case, i.e., Case B2 of the classification
given in Section \ref{sec:classification_genus_one}. 

If $S_{\mathrm{fin}}=\left\{ p_{1},p_{2},q_{1},q_{2}\right\} $ ($p_{1},p_{2},q_{1},q_{2}$:
distinct) and
\[
\alpha_{i}=\begin{cases}
\frac{1}{2} & z_{i}\in\left\{ p_{1},p_{2}\right\} \\
0 & z_{i}\in\left\{ q_{1},q_{2}\right\} ,
\end{cases}
\]
the associated complex curve is
\[
X_{\boldsymbol{z},\boldsymbol{\alpha}}=\left\{ (t,u_{1},u_{2})\in\mathbb{C}^{3}\left|u_{i}^{2}=t-p_{i}\quad\left(i\in\{1,2\}\right)\right.\right\} .
\]
A rational map $\varphi:\mathbb{P}^{1}\rightarrow X_{\boldsymbol{z},\boldsymbol{\alpha}};$
$\xi\mapsto\left(t(\xi),u_{1}(\xi),u_{2}(\xi)\right)$ is then given,
for example, by
\begin{align*}
t(\xi) & =p_{1}\left(\frac{\xi+\xi^{-1}}{2}\right)^{2}-p_{2}\left(\frac{\xi-\xi^{-1}}{2}\right)^{2}\\
u_{1}(\xi) & =\sqrt{p_{1}-p_{2}}\cdot\frac{\xi-\xi^{-1}}{2}\\
u_{2}(\xi) & =\sqrt{p_{1}-p_{2}}\cdot\frac{\xi+\xi^{-1}}{2}.
\end{align*}
The pull-backs of the rational $1$-forms $\left[\substack{z_{i},z_{i+1}\\
\alpha_{i},\alpha_{i+1}
}
\right]$ ($z_{i},z_{i+1}\in\{p_{1},p_{2},q_{1},q_{2}\}$) are given by the
following table ($\alpha_{p_{1}}=\alpha_{p_{2}}=1/2$ and $\alpha_{q_{1}}=\alpha_{q_{2}}=0$):
\begin{flushleft}
\begin{tabular}{|c||c|c||c|}
\hline 
The table of $\varphi^{*}\left[\substack{x,y\\
\alpha_{x},\alpha_{y}
}
\right]$ & $y=p_{1}$ & $y=p_{2}$ & $y\in\{q_{1},q_{2}\}$\tabularnewline
\hline 
\hline 
$x=p_{1}$ & $2d\log(\xi-\xi^{-1})$ & $2d\log\xi$ & $\epsilon_{y,1}d\log\left(\frac{\left(\xi-\lambda_{y}\right)\left(\xi+\lambda_{y}^{-1}\right)}{\left(\xi+\lambda_{y}\right)\left(\xi-\lambda_{y}^{-1}\right)}\right)$\tabularnewline
\hline 
$x=p_{2}$ & $2d\log\xi$ & $2d\log(\xi+\xi^{-1})$ & $\epsilon_{y,2}d\log\left(\frac{\left(\xi-\lambda_{y}\right)\left(\xi-\lambda_{y}^{-1}\right)}{\left(\xi+\lambda_{y}\right)\left(\xi+\lambda_{y}^{-1}\right)}\right)$\tabularnewline
\hline 
\hline 
$x\in\{q_{1},q_{2}\}$ & $\sqrt{p_{1}-p_{2}}d(\xi-\xi^{-1})$ & $\sqrt{p_{1}-p_{2}}d(\xi+\xi^{-1})$ & $d\log\left(t(\xi)-y\right)$\tabularnewline
\hline 
\end{tabular}
\par\end{flushleft}

\begin{flushleft}
\par\end{flushleft}

\begin{flushleft}
Here, $\pm\lambda_{y},\pm\lambda_{y}^{-1}$ are the four solutions
to $t(\xi)=y$ when viewed as a quartic equation in $\xi$, and 
\[
\epsilon_{y,1}=\frac{2}{(\lambda_{y}-\lambda_{y}^{-1})\sqrt{p_{1}-p_{2}}}\in\left\{ \frac{\pm1}{\sqrt{y-p_{1}}}\right\} ,
\]
\[
\epsilon_{y,2}=\frac{2}{(\lambda_{y}+\lambda_{y}^{-1})\sqrt{p_{1}-p_{2}}}\in\left\{ \frac{\pm1}{\sqrt{y-p_{2}}}\right\} .
\]
Also, the inverse image of $x=p_{i}$, $x=q_{i}$ and $x=\infty$
under $\varphi$ are given by
\[
\varphi^{-1}(x)=\begin{cases}
\left\{ \pm1\right\}  & x=p_{1}\\
\left\{ \pm\sqrt{-1}\right\}  & x=p_{2}\\
\left\{ \pm\lambda_{x},\pm\lambda_{x}^{-1}\right\}  & x\in\{q_{1},q_{2}\}\\
\left\{ 0,\infty\right\}  & x=\infty.
\end{cases}
\]
Similarly, if 
\[
\alpha_{i}'\coloneqq\alpha_{i}+\frac{1}{2}=\begin{cases}
1 & z_{i}\in\left\{ p_{1},p_{2}\right\} \\
\frac{1}{2} & z_{i}\in\left\{ q_{1},q_{2}\right\} ,
\end{cases}
\]
the associated complex curve is
\[
X_{\boldsymbol{z},\boldsymbol{\alpha}'}=\left\{ (t,u_{1},u_{2})\in\mathbb{C}^{3}\left|u_{i}^{2}=t-q_{i}\quad\left(i\in\{1,2\}\right)\right.\right\} 
\]
and a rational map $\psi:\mathbb{P}^{1}\rightarrow X_{\boldsymbol{z},\boldsymbol{\alpha}'};$
$\xi\mapsto\left(t'(\xi),u_{1}'(\xi),u_{2}'(\xi)\right)$ is given
by
\begin{align*}
t'(\xi) & =q_{1}\left(\frac{\xi+\xi^{-1}}{2}\right)^{2}-q_{2}\left(\frac{\xi-\xi^{-1}}{2}\right)^{2}\\
u_{1}'(\xi) & =\sqrt{q_{1}-q_{2}}\cdot\frac{\xi-\xi^{-1}}{2}\\
u_{2}'(\xi) & =\sqrt{q_{1}-q_{2}}\cdot\frac{\xi+\xi^{-1}}{2}.
\end{align*}
Then differential forms $\psi^{*}\left[\substack{x,y\\
\alpha_{x}',\alpha_{y}'
}
\right]$ are given by the following table ($\alpha_{p_{1}}'=\alpha_{p_{2}}'=1$
and $\alpha_{q_{1}}'=\alpha_{q_{2}}'=1/2$): 
\par\end{flushleft}

\begin{flushleft}
\begin{tabular}{|c||c||c|c|}
\hline 
The table of $\psi^{*}\left[\substack{x,y\\
\alpha_{x}',\alpha_{y}'
}
\right]$ & $y\in\{p_{1},p_{2}\}$ & $y=q_{1}$ & $y=q_{2}$\tabularnewline
\hline 
\hline 
$x\in\{p_{1},p_{2}\}$ & $d\log\left(t'(\xi)-x\right)$ & $\epsilon_{x,1}'d\log\left(\frac{\left(\xi-\lambda_{x}'\right)\left(\xi+\lambda'{}_{x}^{-1}\right)}{\left(\xi+\lambda_{x}'\right)\left(\xi-\lambda'{}_{x}^{-1}\right)}\right)$ & $\epsilon_{x,2}'d\log\left(\frac{\left(\xi-\lambda_{x}'\right)\left(\xi-\lambda'{}_{x}^{-1}\right)}{\left(\xi+\lambda_{x}'\right)\left(\xi+\lambda'{}_{x}^{-1}\right)}\right)$\tabularnewline
\hline 
\hline 
$x=q_{1}$ & $\sqrt{q_{1}-q_{2}}d(\xi-\xi^{-1})$ & $2d\log(\xi-\xi^{-1})$ & $2d\log\xi$\tabularnewline
\hline 
$x=q_{2}$ & $\sqrt{q_{1}-q_{2}}d(\xi+\xi^{-1})$ & $2d\log\xi$ & $2d\log(\xi+\xi^{-1})$\tabularnewline
\hline 
\end{tabular}
\par\end{flushleft}

Here, $\pm\lambda_{x}'$ and $\pm\lambda'{}_{x}^{-1}$ are the four
solutions to $t'(\xi)=x$, and
\[
\epsilon_{x,1}'=\frac{2}{(\lambda_{x}'-\lambda_{x}'{}^{-1})\sqrt{q_{1}-q_{2}}}\in\left\{ \frac{\pm1}{\sqrt{x-q_{1}}}\right\} ,
\]
\[
\epsilon_{x,2}'=\frac{2}{(\lambda_{x}'+\lambda_{x}'{}^{-1})\sqrt{q_{1}-q_{2}}}\in\left\{ \frac{\pm1}{\sqrt{x-q_{2}}}\right\} .
\]
 The inverse images of $x=p_{i}$, $x=q_{i}$ and $x=\infty$ under
$\psi$ are given by
\[
\psi^{-1}(x)=\begin{cases}
\left\{ \pm\lambda_{x}',\pm\lambda'{}_{x}^{-1}\right\}  & x\in\{p_{1},p_{2}\}\\
\left\{ \pm1\right\}  & x=q_{1}\\
\left\{ \pm\sqrt{-1}\right\}  & x=q_{2}\\
\left\{ 0,\infty\right\}  & x=\infty.
\end{cases}
\]

Using the above information, we can rewrite both sides of the following
special case of Theorem \ref{thm: translation invariance}
\[
\hat{B}_{\gamma}^{\mathrm{f},\mathrm{f}}\bigl(\left.\substack{x_{0}\\
\alpha_{0}
}
\right|\left.\substack{x_{1}\\
\alpha_{1}
}
\right|\dots\left|\substack{x_{n+1}\\
\alpha_{n+1}
}
\right.\bigr)=\hat{B}_{\gamma}^{\mathrm{f},\mathrm{f}}\bigl(\left.\substack{x_{0}\\
\alpha_{0}'
}
\right|\left.\substack{x_{1}\\
\alpha_{1}'
}
\right|\dots\left|\substack{x_{n+1}\\
\alpha_{n+1}'
}
\right.\bigr)\qquad(x_{j}\in\{p_{1},p_{2},q_{1},q_{2}\})
\]
 as iterated integrals of rational differential forms on $\mathbb{P}^{1}$,
and obtain the following theorem (stated in a slightly restricted
form for the sake of clarity).
\begin{thm}
Let $x_{0},\ldots,x_{n+1}\in\left\{ p_{1},p_{2},q_{1},q_{2}\right\} $.
Assume that $x_{0},x_{1}\in\{p_{1},p_{2}\}$ and $x_{n},x_{n+1}\in\{q_{1},q_{2}\}$.
Assume that $x_{0}\neq x_{1}$ and $x_{n}\neq x_{n+1}$. Let $\gamma$
be a path from $x_{0}$ to $x_{n+1}$. Then
\[
(x_{n+1}-x_{0})^{1/2}I_{\varphi^{-1}(\gamma)}\left(y;f_{x_{0},x_{1}},f_{x_{1},x_{2}},\ldots,f_{x_{n-1},x_{n}};z\right)=(x_{0}-x_{n+1})^{1/2}I_{\psi^{-1}(\gamma)}\left(y';f_{x_{1},x_{2}}',f_{x_{2},x_{3}}',\ldots,f_{x_{n},x_{n+1}}';z'\right)
\]
where $y\in\varphi^{-1}(x_{0})$, $z\in\varphi^{-1}(x_{n+1})$, $y'\in\psi^{-1}(x_{0})$,
$z'\in\psi^{-1}(x_{n+1})$ and $f_{x,y}=\varphi^{*}\left[\substack{x,y\\
\alpha_{x},\alpha_{y}
}
\right]$, $f_{x,y}'=\psi^{*}\left[\substack{x,y\\
\alpha_{x}',\alpha_{y}'
}
\right]$.
\end{thm}

\end{document}